\documentclass[11pt]{amsart}
\usepackage[letterpaper,margin=1in]{geometry}

\usepackage[dvipsnames]{xcolor}

\newcommand\myshade{85}
\colorlet{mylinkcolor}{Red}
\colorlet{mycitecolor}{Cerulean}
\colorlet{myurlcolor}{Plum}

\usepackage{lmodern}

\usepackage[utf8]{inputenc}

\usepackage[inline,shortlabels]{enumitem}

\usepackage{amssymb,mathrsfs,stmaryrd}
\usepackage{mathtools}
\usepackage{thmtools,thm-restate}
\usepackage{colonequals}
\usepackage{tikz}
\usetikzlibrary{cd}

\usepackage{verbatim}

\mathchardef\mhyphen="2D

\usepackage[normalem]{ulem} 

\usepackage{relsize}
\usepackage[bbgreekl]{mathbbol}
\usepackage{amsfonts}
\DeclareSymbolFontAlphabet{\mathbb}{AMSb} 
\DeclareSymbolFontAlphabet{\mathbbl}{bbold}
\newcommand{\Prism}{{\mathlarger{\mathbbl{\Delta}}}}

\newcommand{\suchthat}{\;\ifnum\currentgrouptype=16 \middle\fi\vert\;}

\newcommand\restr[2]{{\left.\kern-\nulldelimiterspace#1\vphantom{\big|}\right|_{#2}}}

\usepackage{slashed}

\usepackage[pdfusetitle,pagebackref,linktocpage]{hyperref}
\usepackage[capitalize,noabbrev]{cleveref}

\hypersetup{
	linkcolor  = mylinkcolor!\myshade!black,
	citecolor  = mycitecolor!\myshade!black,
	urlcolor   = myurlcolor!\myshade!black,
	colorlinks = true
}

\DeclareMathOperator*{\colim}{colim}

\DeclareMathOperator{\Coh}{Coh}
\DeclareMathOperator{\coker}{coker}

\DeclareMathOperator{\Crys}{Crys}

\DeclareMathOperator{\Fil}{{Fil}}

\DeclareMathOperator{\Hom}{Hom}
\DeclareMathOperator{\id}{id}

\DeclareMathOperator{\Kos}{Kos}

\DeclareMathOperator{\Loc}{Loc}
\newcommand{\perf}{{\mathrm{perf}}}
\DeclareMathOperator{\Map}{Map}
\DeclareMathOperator{\Mod}{Mod}

\newcommand{\proet}{{\mathrm{pro\acute et}}}

\DeclareMathOperator{\Rep}{Rep}

\DeclareMathOperator{\Shv}{Shv}
\DeclareMathOperator{\Spec}{Spec}

\newcommand{\tf}{\mathrm{tf}}

\DeclareMathOperator{\Vect}{Vect}
\newcommand{\an}{{\mathrm{an}}}

\newcommand{\Spf}{{\mathrm{Spf}}}
\newcommand{\Spa}{{\mathrm{Spa}}}

\newcommand{\dR}{{\mathrm{dR}}}
\newcommand{\pe}{{\mathrm{pro\acute{e}t}}}
\newcommand{\Ainf}{{\mathbb{A}_{\mathrm{inf}}}}       
\newcommand{\Acrys}{{\mathbb{A}_{\mathrm{crys}}}}

\newcommand{\OB}{{\mathcal{O}\mathbb{B}}}

\newcommand{\OBdR}{{\mathcal{O}\mathbb{B}_{\mathrm{dR}}}}

\newcommand{\crys}{{\mathrm{crys}}}

\newcommand{\et}{{\mathrm{\acute{e}t}}}
\newcommand{\Gal}{{\mathrm{Gal}}}

\newcommand{\Isoc}{\mathrm{Isoc}}
\newcommand{\Perfd}{{\mathrm{Perfd}}}

\newcommand{\rAinf}{{\mathrm{A_{inf}}}}    

\newcommand{\rAcrys}{{\mathrm{A_{crys}}}}

\newcommand{\gr}{{\mathrm{gr}}}
\newcommand{\HT}{{\mathrm{HT}}}
\newcommand{\conj}{{\mathrm{conj}}}
\newcommand{\cont}{{\mathrm{cont}}}
\newcommand{\coh}{{\mathrm{coh}}}

\newcommand{\Prismsp}{{\Prism^{\mathrm{sp}}}}
\newcommand{\refl}{{\mathrm{refl}}}
\newcommand{\Ass}{{\mathrm{Ass}}}
\newcommand{\sat}{{\mathrm{sat}}}

\newcommand{\pre}{{\mathrm{pre}}}
\newcommand{\cofib}{{\mathrm{cofib}}}
\newcommand{\wCrys}{{\mathrm{wCrys}}}
\newcommand{\wt}{{\mathrm{wt}}}

\newcommand{\wh}{\widehat}
\newcommand{\pd}{{\mathrm{pd}}}
\newcommand\almosteq{\mathrel{\stackrel{\makebox[0pt]{\mbox{\normalfont\tiny a}}}{=}}}

\definecolor{mb}{rgb}{0.36, 0.54, 0.66}

\theoremstyle{plain}
\newtheorem{theorem}{Theorem}[section]

\newtheorem{maintheorem}{Theorem}

\newtheorem{proposition}[theorem]{Proposition}
\newtheorem{conjecture}[theorem]{Conjecture}
\newtheorem{lemma}[theorem]{Lemma}
\newtheorem{claim}[theorem]{Claim}
\newtheorem{corollary}[theorem]{Corollary}
\newtheorem{maincorollary}[maintheorem]{Corollary}

\theoremstyle{definition}
\newtheorem{definition}[theorem]{Definition}
\newtheorem{example}[theorem]{Example}
\newtheorem{construction}[theorem]{Construction}

\newtheorem{convention}[theorem]{Convention}

\newtheorem{remark}[theorem]{Remark}

\title{Ogus's conjecture on $F$-isocrystals}

\author{Haoyang Guo}

\begin{document}
	
	\begin{abstract}
In 1984, Ogus conjectured the existence of a canonical $F$-isocrystal that enhances the Gauss--Manin connection, for a proper relative rigid space with analytically good reduction.
We give a positive answer to this conjecture in full generality, through $p$-adic local systems and prismatic methods.
Along the way, we introduce a prismatic refinement of the $p$-adic Riemann--Hilbert functor and prove a primitive purity theorem for Frobenius modules.
\end{abstract}
	
	\maketitle
	\tableofcontents

\section{Introduction}
\label{sec:intro}

\subsection{Ogus's conjecture}
\label{introsub:Ogus}
Let $K$ be a $p$-adic local field with perfect residue field $k$, and let $\mathcal{O}_K$ be the ring of integers.
Recall that a smooth rigid space over $K$ has \emph{good reduction} if it admits a smooth $p$-adic formal scheme model over $\mathcal{O}_K$.
More generally, we let $X$ be a $p$-adic formal scheme over $\mathcal{O}_K$, and let $X_\eta$ be its generic fiber over $K$.
Then a smooth rigid space $Y_\eta$ over $X_\eta$ has \emph{good reduction} (with respect to $X$) if it is the generic fiber of a smooth morphism of $p$-adic formal schemes $f:Y\to X$.
As a given rigid space can have many distinct integral models, which are
far from isomorphic to one another and can have different types of singularity,
it is in principle difficult to determine whether a given rigid space admits good reduction. 

A fundamental insight of Grothendieck (\cite{Gro68}) and Berthelot (\cite{Ber74}) is that when a proper smooth rigid space $Y_\eta$ over $X_\eta$ has good reduction model $Y\to X$, its special fiber, which lives in characteristic $p$, could yield well-behaved algebraic invariants in characteristic $0$: the (relative) crystalline cohomology.
However, when the base space $X_\eta$ has positive dimension, additional complexities arise
when considering the relative crystalline cohomology due to the possible singularity of the base formal scheme $X$.
Moreover, it is very demanding, in practice,
to ensure that both the integral morphism $f:Y\to X$ and the target formal scheme $X$ to be smooth simultaneously. 
Relatedly, it is not clear whether the algebraic invariants over one integral model $X$ can yield companion objects over a
different integral model $X'$.

The above were some of the main focuses
of Ogus’s work in \cite{Ogu84}, where he introduced the notion of convergent $F$-isocrystals to enhance crystalline cohomology in the relative setting and to connect it with the Gauss--Manin connections in characteristic
$0$. 
To this end, the following notion was introduced in \cite[\S 5.1]{Ogu84}.
\begin{definition}
\label{intro:def:an_good}
A rigid space  $Y_\eta$ over $X_\eta$ has \emph{analytically good reduction} if $Y_\eta$ is the generic fiber of a smooth morphism of $p$-adic formal schemes $f:Y
\to X'$ for some integral model $X'$ of $X_\eta$.
\end{definition}
In particular, Ogus formulated a conjecture regarding the existence of canonical $F$-isocrystals associated with a relative rigid space with analytically good reduction, which we recall it below.
\begin{conjecture}\cite[Conj.\ 5.1]{Ogu84}
\label{conj:Ogus}
Let $X$ be a regular $p$-adic formal scheme, and let $Y_\eta \to X_\eta$ be a proper smooth rigid space with analytically good reduction over $X_\eta$. 
For each $i\in \mathbb{N}$, there exists a canonical $F$-isocrystal $\mathcal{E}_{\crys,i}$ on $X_k$ that enhances the
$i$-th Gauss--Manin connection of $Y_\eta/X_\eta$.
\end{conjecture}

To illustrate the difficulty of \Cref{conj:Ogus}, we first note that in the statement,
the base $X'$ of the integral smooth morphism $f:Y\to X'$ may be different from the original integral model $X$. 
Even
when the latter assumption is true, the anticipated $F$-isocrystal, which is the relative crystalline cohomology, should be both \emph{independent} of the reduction $f$
and \emph{functorial} with respect to the generic fiber $f_\eta:Y_\eta\to X_\eta$. 
Here the Frobenius structure in the statement is crucial: by analyzing the geometric properties of convergent isocrystals, it was shown in \cite[Thm.\ 5.7]{Ogu84} that the underlying isocrystal always exists when $X$ is smooth.
However, the existence of $F$-isocrystals as in \Cref{conj:Ogus} was only verified when $\dim X\leq 2$.
Its caninicity, on the other hand, was proved via the cycle
class map of the crystalline cohomology and the crystalline Riemann--Roch theorem, the main result of Gillet--Messing in \cite{GM87} (cf. \cite[\S 5]{Ogu84}). 
Crucially, both the existence and the caninicity made essential uses of the resolution of singularities for integral models of $X_\eta$, which is available if $\dim X\leq 2$.
We also mention that even though this conjecture was
proved in special cases, it found an immediate application: Ogus used it to study the crystalline realization
of the Kuga--Satake correspondence (\cite[Thm.\ 7.3]{Ogu84}), and the 
latter result 
played a central role in his joint work with Nygaard \cite{NO85}, where they proved the Tate conjecture for K3
surfaces of finite heights over finite fields for $p\geq 5$.
The main result of this article settles \Cref{conj:Ogus} in full generality (see \Cref{intro:cor:Ogus}).

\subsection{Prismatic Riemann--Hilbert functor}
\label{introsub:RH}
To explain our strategy of approaching \Cref{conj:Ogus}, we give some background on the general ideas.
As the input of \Cref{conj:Ogus} are rigid spaces in characteristic $0$, it is natural to consider if the algebraic invariants of special fibers or of integral models can be determined solely by generic fibers.
In the 1970s, Grothendieck raised the question on
whether algebraic invariants in characteristic $0$ can provide a necessary condition for determining when a rigid space $Y_\eta$ over $K$
has good reduction (\cite{Gro71}). 
After much effort, this question was precisely formulated in terms of $p$-adic \'etale cohomology by Fontaine (\cite{Fon82b}), known as the \emph{$\mathrm{C_{crys}}$-conjecture}.
Specifically, Fontaine introduced a functor 
\begin{equation}
\label{eq:Fontaine'sD_crys}
\mathrm{D}_\crys(-): \Rep_{\mathbb{Z}_p}(\Gal_K) \longrightarrow \Isoc^\varphi(k),
\end{equation}
where $\Rep_{\mathbb{Z}_p}(\Gal_K)$ is the category of continuous representations of the absolute Galois group $\Gal_K$ on finite rank $\mathbb{Z}_p$-modules, and $\Isoc^\varphi(k)$ is the category of $F$-isocrystals over $k$.
It is then conjectured in \cite{Fon82b} that the crystalline cohomology of the special fiber $Y_k$ can be recovered from the \'etale cohomology of $Y_\eta$ through the functor $\mathrm{D}_\crys(-)$, where $Y$ is a proper and smooth $p$-adic formal scheme over $\mathcal{O}_K$.
The $\mathrm{C_{crys}}$-conjecture, first proved by Faltings \cite{Fal89} in the algebraic setting and then by Bhatt--Morrow--Scholze \cite{BMS1} for non-algebraic $Y$, has been a seminal theme in $p$-adic Hodge theory from its early stages to modern development.

Along with the $\mathrm{C_{crys}}$-conjecture, Fontaine also introduced the abstract notion of \emph{crystalline representations}, defined as those $p$-adic Galois representations $V$ of $\Gal_K$ such that $\mathrm{D}_\crys(V)$ can recover the representation $V$.
So a major consequence of $\mathrm{C_{crys}}$-conjecture is that the $p$-adic \'etale cohomology of a proper rigid space $Y_\eta$ is a crystalline representation whenever $Y_\eta$ has good reduction.
In particular, the rational crystalline cohomology of the special fiber of \emph{any} smooth integral model of $Y_\eta$ depends only on $Y_\eta$ itself. 

On the other hand, a recent breakthrough in $p$-adic geometry is the introduction
of prismatic cohomology by Bhatt and Scholze \cite{BS22}, building on their earlier work with
Morrow (\cite{BMS1}, \cite{BMS19}). 
In \cite{BS22}, they
attached to each $p$-adic formal scheme $X$ a novel Grothendieck site $X_\Prism$, called the \emph{prismatic
site} of $X$. 
The objects of $X_\Prism$ are given by maps $\Spf(A)\to X$ from \emph{prisms} $(A,I)$, that is, $\mathbb{Z}_p$-algebras $A$ which are $(p,I)$-complete for some invertible ideal $I\subset A$ and are equipped with an
endomorphism on $A$ lifting Frobenius on $A/p$.
The cohomology of the associated
structure sheaf is called the prismatic  cohomology.
One of the key features of prismatic cohomology is its universality, as it encompasses
several important $p$-adic cohomology theories, including \'etale, crystalline, and de Rham cohomologies. 
Moreover, the associated coefficient theory, known as the \emph{prismatic $F$-crystal}, enhances both the Galois representation of the generic fiber and the $F$-isocrystal over the special fiber.
In addition, a fundamental observation of Bhatt and Scholze \cite{BS23} is that crystalline $\mathbb{Z}_p$-representations of $\Gal_K$ are in fact equivalent to prismatic $F$-crystals over $\mathcal{O}_K$.
As a consequence, the $p$-adic \'etale cohomology of a smooth proper rigid space $Y_\eta$ over $K$ naturally produces prismatic $F$-crystals over $\mathcal{O}_K$ whenever $Y_\eta$ admits good reduction.

To fit into the following more geometric discussion, for a finite field extension $K'/K$, we recall that the category of $p$-adic Galois representations $\Rep_{\mathbb{Z}_p}(\Gal_{K'})$ is equivalent to the category of $p$-adic local systems over $\Spec(K')$, denoted as $\Loc_{\mathbb{Z}_p}(\Spec(K'))$.
Then we note that for an analytically good reduction rigid space $f_\eta:Y_\eta \to X_\eta=X'_\eta$, the fiber at any point $x\in X'_\eta$ has good reduction.
In particular, by the aforementioned results of Faltings and Bhatt--Morrow--Scholze, the stalk of the relative \'etale cohomology $R^if_{\eta,*}\mathbb{Z}_p$ at any point $x=\Spec(K')$ is a crystalline representation modulo torsion.
This motivates us to consider the following natural notion of the crystallinity in the higher dimensional setting through a fiberwise condition.
\begin{definition}
\label{intro:def:crys}
For a rigid space $X_\eta$, we let $\Loc_{\mathbb{Z}_p}^\crys(X_\eta)$
\footnote{For the cautious reader on the possible ambiguity of this terminology, see \Cref{intro:thm:crys_loc_sys}.} be the full subcategory of \emph{pointwise crystalline local systems}, consisting of $\mathbb{Z}_p$-local systems $T$ over $X_\eta$ such that for each finite extension $K'/K$ and each point $x=\Spec(K')\in X_\eta$, the stalk $T|_x$ is a crystalline representation.
\end{definition}
By rephrasing the above discussion, we then obtain the following observation, where the subscript $(-)_\tf$ denotes the torsionfree quotient:
\begin{theorem}[Faltings, Bhatt--Morrow--Scholze]
\label{intro:thm:pc_of_coh}
The local system $(R^if_{\eta,*}\mathbb{Z}_p)_\tf$ belongs to $\Loc_{\mathbb{Z}_p}^\crys(X_\eta)$ if $f_\eta:Y_\eta \to X_\eta$ has analytically good reduction.
\end{theorem}

The first main result of this article is that any $p$-adic local system $T$ over the rigid space $X_\eta$ gives rise to a canonical prismatic invariant over a regular integral model $X$, and this prismatic invariant becomes an $F$-crystal of appropriate size when $T$ is pointwise crystalline.
To explain the construction, for a regular $p$-adic formal scheme $X$, we consider a full subcategory of \emph{special prisms} in $X_\Prism$, denoted as $X_\Prismsp$.
The category $X_\Prismsp$ consists of many regular objects (including Breuil--Kisin prisms, $q$-crystalline prisms, and their higher dimensional analogues), covers the original $X_\Prism$, and is closed under finite coproducts.
It in particular has the same crystal category as that of $X_\Prism$ (cf. \Cref{def:flat_prismatic_site}).
To each prism $(A,I)$, we consider the notion \emph{weak Frobenius module}, defined as a pair $(M,\varphi_{M})$, where $M$ is an $A$-module and $\varphi_{M}$ is an $A[1/I]$-linear map $\varphi_A^*M[1/I]\to M[1/I]$.
We call $(M,\varphi_{M})$ a \emph{Frobenius module} if $\varphi_{M}$ is an isomorphism, and denote the category of (weak) Frobenius modules over $A$ as $\Mod^{(w)\varphi}(A)$.
We then define the category $\wCrys^\varphi(X_\Prismsp)$ of \emph{weak prismatic $F$-crystals} over $X_\Prismsp$ such that its objects are the functors that assign to each $(A,I)\in X_\Prismsp$ a weak Frobenius module over $A$ (cf. \Cref{def:family_of_wFrob_mod}).
Our main result can be stated as follows.
\begin{maintheorem}[Canonical weak prismatic $F$-crystal]
	\label{intro:thm:RH}
	Let $X$ be a $p$-torsionfree regular $p$-adic formal scheme over $\mathcal{O}_K$.
	There is a canonical functor
	\[
	\Loc_{\mathbb{Z}_p}(X_\eta) \longrightarrow \wCrys^\varphi(X_\Prismsp),\quad T \longmapsto \mathcal{E}_{\Prism,T}.
	\]
	It satisfies the following properties:
	\begin{enumerate}[label=\upshape{(\roman*)}]
		\item\label{intro:thm:RH:fin} \emph{Finiteness (\Cref{prop:weak_prismatic_F_crystal:finiteness}):} The $A$-module $\mathcal{E}_{\Prism,T}(A,I)$ is finitely presented if the ring $A$ is regular and $A/I$ is $p$-torsionfree.
		\item\label{intro:thm:RH:crys_real} \emph{Crystalline realization (\Cref{sub:crystalline}):} By evaluating at special crystalline prisms, the functor $\mathcal{E}_{\Prism,T}$ induces a weak isocrystal with a weak Frobenius structure over $X_k$.
		\footnote{
		Here we recall from \cite{GY24} that a \emph{weak isocrystal} on the special fiber $X_k$ is a functor that assigns to each affine open subvariety $U_k$ an isocrystal $\mathcal{E}_{U_k}$ such that the pullback morphism $\mathcal{E}_{U_k}|_{V_k}\to \mathcal{E}_{V_k}$ is an injection for each open immersion $V_k\to U_k$. }
		\item\label{intro:thm:RH:Nygaard} \emph{Nygaard filtration (\Cref{sub:weak_prismatic_F_crystal:filtration}):} There is a natural descending and exhaustive $\mathbb{Z}$-filtration on $\varphi^*\mathcal{E}_{\Prism,T}$ which is compatible with the weak Frobenius structure and is eventually $\mathcal{I}_\Prism$-adic.
		\item\label{intro:thm:RH:dR} \emph{de Rham realization (\Cref{sub:dR}):} The twisted reduction $\varphi^*\mathcal{E}_{\Prism,T}\otimes_{\mathcal{O}_\Prism} \mathcal{O}_\Prism/\mathcal{I}_\Prism$ induces a presheaf of filtered $\mathcal{O}_X$-modules $\mathcal{E}_{\dR,T}$, and is equipped with a flat connection with Griffiths transversality either after inverting $p$. 
		Moreover, the flat connection can be defined integrally if $X$ is smooth.
		\item\label{intro:thm:RH:crys_loc} \emph{Crystallinity (\Cref{thm:prismatic_RH_for_crystalline_local_system}):} Assume $T\in \Loc^\crys_{\mathbb{Z}_p}(X_\eta)$.
		The output $\mathcal{E}_{\Prism,T}$ is a coherent and reflexive prismatic $F$-crystal whose \'etale realization is canonically isomorphic to $T$, and the Nygaard filtration enhances it into a coherent $F$-gauge.
		\item\label{intro:thm:RH:pullback} \emph{Pullbacks (\Cref{cor:weak_prismatic_F_crystal:pull_back}):} Let $f:Y\to X$ be a map of regular $p$-adic formal schemes.
		Assume either $T\in \Loc^\crys_{\mathbb{Z}_p}(X_\eta)$ or $f$ is finite \'etale.
		Then $\mathcal{E}_{\Prism,-}$ is compatible with pullbacks.
		\item\label{intro:thm:RH:pushforward} \emph{Pushforwards (\Cref{prop:weak_prismatic_F_crystal:direct_image}):} Let $g:X\to Z$ be a proper smooth map of regular $p$-adic formal schemes, and let $T\in \Loc^\crys_{\mathbb{Z}_p}(X_\eta)$.
		Then $\mathcal{E}_{\Prism,-}$ is compatible with the higher direct images up to bounded $p$-power torsions.
		\item\label{intro:thm:RH:compatible} \emph{Compatibility (\Cref{prop:dR:finiteness_&_injective}, \Cref{cor:weak_prismatic_F_crystal:crystalline:local_gluing}):} The crystalline realization \ref{intro:thm:RH:crys_real} recovers Guo--Yang's crystalline Riemann--Hilbert functor $\mathcal{E}_{\crys,T}$ in \cite{GY24}, and the de Rham realization \ref{intro:thm:RH:dR} admits a natural injection into Liu--Zhu's $p$-adic Riemann--Hilbert functor $\mathrm{D}_\dR(T)$ in \cite{LZ17}.
	\end{enumerate}
\end{maintheorem}

In light of \Cref{intro:thm:RH}.\ref{intro:thm:RH:compatible}, we also call the functor $T\mapsto \mathcal{E}_{\Prism,T}$ the \emph{prismatic Riemann--Hilbert functor}.
By combining \Cref{intro:thm:RH}.(\ref{intro:thm:RH:crys_real}, \ref{intro:thm:RH:crys_loc}, \ref{intro:thm:RH:compatible}), the pointwise crystallinity of \'etale cohomology in \Cref{intro:thm:pc_of_coh}, together with the de Rham comparison theorem of Scholze \cite{Sch13}, we give a positive answer to Ogus's \Cref{conj:Ogus} in full generality.
\begin{maincorollary}
\label{intro:cor:Ogus}
Let $X$ be a regular $p$-adic formal scheme, and let $Y_\eta \to X_\eta$ be a proper smooth rigid space with analytically good reduction over $X_\eta$. 
The crystalline realization of $\mathcal{E}_{\Prism,(R^if_{\eta,*}\mathbb{Z}_p)_\tf}$ is an $F$-isocrystal on $X_k$ that enhances the $i$-th Gauss--Manin connection of $Y_\eta/X_\eta$, is independent of the good reduction morphism $f:Y\to X'$, and is functorial in the rigid space $Y_\eta$ over $X_\eta$.
\end{maincorollary}

To explain the construction of \Cref{intro:thm:RH}, for simplicity, we consider the special case when $X=\Spf(\mathcal{O}_K)$ is a mixed characteristic point; the general construction for an arbitrary special prism $(A,I)$ over regular $X$ is similar (cf. \Cref{sub:F-crys_const}).
Recall, it was first observed by Breuil \cite{Bre98} and then extended by Kisin \cite{Kis06} that a crystalline representation $T$ gives rise to integral linear algebraic datum, the so-called \emph{Breuil--Kisin module}.
In our terminology, a Breuil--Kisin module is a finitely generated Frobenius module over the prism $(\mathfrak{S}, (E(u)))$, where $\mathfrak{S}\colonequals W(k)\llbracket u \rrbracket$ such that the Frobenius structure sends $u$ to $u^p$, and $E(u)$ is an Eisenstein polynomial.
Despite being simple itself, the construction of the Breuil--Kisin module associated with $T$ is however very technical and non-explicit.
This is in contrast to Fontaine's original functor $\mathrm{D}_{\crys}(T)$ of (\ref{eq:Fontaine'sD_crys}) in the 1980s, defined as the Galois invariant of the tensor product $T\otimes_{\mathbb{Z}_p} \mathrm{B_{crys}}$, where $\mathrm{B}_\crys$ is Fontaine's crystalline period ring.
So it is natural to ask if a Fontaine-style formula but for Breuil--Kisin modules exists.

A pivotal observation of Beilinson (\cite{Bei12}) and Bhatt (\cite{Bha12}) (which was extended to general perfectoid algebras by our joint work with Li in \cite{GL21}) is that the period ring $\mathrm{B}_{\crys}$ admits a geometric interpretation: it is a localization of the $p$-adic derived de Rham cohomology of $\mathcal{O}_{C}$, where the latter is the completion of the ring of $p$-adic algberaic integers.
It in particular suggests a new perspective on Fontaine's functor, and leads to a much larger pool of period rings by considering other cohomologies of $\mathcal{O}_{C}$.
Our article then offers one such example that is new:
We let $\Prism_{\mathcal{O}_{C}/\mathfrak{S}}$ be the relative prismatic cohomology of $\mathcal{O}_{C}$ over the Breuil--Kisin prism $(\mathfrak{S},E(u))$ introduced by Bhatt and Scholze (\cite{BS22}), which naturally admits a Frobenius structure, a $\Gal_K$-action, and Nygaard filtration after the Frobenius twist.
Let $\mu=[\epsilon] -1 \in \rAinf =\Prism_{\mathcal{O}_\mathcal{C}}\subset \Prism_{\mathcal{O}_{C}/\mathfrak{S}}$ be a canonical element.
Then the construction of the $A$-module $\mathcal{E}_{\Prism,T}(A,I)$ for $(A,I)=(\mathfrak{S},E(u))$ in \Cref{intro:thm:RH} is in fact a one-line formula, and yields the Breuil--Kisin module when $T$ is crystalline.
\begin{maincorollary}
\label{intro:cor:BK}
Let $T$ be a $\mathbb{Z}_p$-representation of $\Gal_K$ of rank $n$.
The Galois invariant 
\[
(T\otimes_{\mathbb{Z}_p} \Prism_{\mathcal{O}_{C}/\mathfrak{S}} [1/\mu])^{\Gal_K}
\]
is a finite free weak Frobenius module over $\mathfrak{S}$ of rank $\leq n$, and is canonically isomorphic to the associated Breuil--Kisin module when $T\in \Rep^\crys_{\mathbb{Z}_p}(\Gal_K)$.
\end{maincorollary}
Here we remind the reader that the relative prismatic cohomology $\Prism_{\mathcal{O}_{C}/\mathfrak{S}}$ lives in cohomological degree zero and can be identified with a coproduct of prisms (\Cref{cor:coproduct:perfect_with_framed}.)
Moreover, since the mod $u$ reduction of $\Prism_{\mathcal{O}_{C}/\mathfrak{S}} [1/\mu]$ is Fontaine's period ring $\mathrm{B_{crys}}$, the above $\mathfrak{S}$-module has rank $n$ if and only $T\in \Rep^\crys_{\mathbb{Z}_p}(\Gal_K)$, which yields a new and equivalent definition of the crystalline representation that is of integral nature.
It also worth mentioning that by replacing $(A,I)$ with other prisms, it leads to a formula of constructing the prismatic $F$-crystals \emph{directly} out of crystalline $p$-adic local systems, answering a question from Esnault.
This is in contrast to the \'etale realization functor of prismatic $F$-crystals which is of opposite (or precisely downstream) direction.

To prove \Cref{intro:cor:BK} and its higher dimensional generalization, the idea is to reduce the problem to the finiteness and the vanishing results of Tate--Sen (for $C$-linear representations) and more generally of Liu--Zhu \cite{LZ17} (cf. \Cref{thm:LZ_finiteness}).
However, this reduction process is one of the most technical components of the article.
Indeed, by extending an observation of Bhatt--Morrow--Scholze in \cite{BMS1}, the period ring $\Prism_{\mathcal{O}_{C}/\mathfrak{S}} [1/\mu]$ can be identified with an (Frobenius, Galois, and filtered) equivariant colimit of double-twisted relative prismatic cohomology $R\Gamma((\mathcal{O}_{C}/\mathfrak{S})_\Prism, \mathcal{O}\{n\})\otimes_{\mathbb{Z}_p}\mathbb{Z}_p(-n)$ for $n\in \mathbb{Z}$.
Then the essential work is to understand the Galois structure on the Nygaard filtration, where we make crucial use of a result from Bhatt--Lurie \cite{BL22a} on the relationship between the absolute and the relative de Rham realizations (cf. \Cref{thm:abs_and_rel_prismatic}), together with a careful d\'evissage argument that refines the difference of $\mu^n \Prism^{(1)}_{\mathcal{O}_{C}/\mathfrak{S}}$ and $\mu^{n+1}\Prism^{(1)}_{\mathcal{O}_{C}/\mathfrak{S}}$ through the Frobenius structure (\Cref{thm:structure_of_rational_period_sheaf}).

\begin{remark}[Pullback injectivity and the rank stratification]
As we discussed in \Cref{intro:thm:RH}.\ref{intro:thm:RH:pullback}, the prismatic Riemann--Hilbert functor interacts well with the pullbacks when the local system is crystalline.
In fact, the pullback compatibility extends to arbitrary local systems.
Specifically, given a map of regular $p$-adic formal schemes $f:X_2\to X_1$ and a map of prisms $(B_1,J_1)\to (B_2,J_2)$ such that $(B_i,J_i)\in (X_i)_\Prism$, there is an induced map of weak Frobenius modules over $B_2$:
\[
\mathcal{E}_{\Prism,T}(B_1,J_1)\otimes_{B_1} B_2 \longrightarrow \mathcal{E}_{\Prism,f^{-1}T}(B_2,J_2),
\]
where $T\in \Loc_{\mathbb{Z}_p}(X_{1,\eta})$.
Analogously to the crystalline Riemann--Hilbert functor in \cite{GY24}, we show in \Cref{thm:weak_prismatic_F_crystal:pull_back_wrt_prisms} that the map is generically injective when $(B_i,J_i)$ are special and regular.
In particular, as was first discovered in our previous work \cite[Def.\ 5.11]{GY24}, there is a rank function on the Zariski site of $X$ together with a canonical \emph{rank stratification}, which measures how the reduction (in the sense of monodromy) varies within the family.
This will be studied in depth in a subsequent work with Ziquan Yang.
\end{remark}

 \subsection{A primitive purity theorem for Frobenius modules}
\label{introsub:purity}
In the above discussions, one of the major ingredients in the proof of Ogus's \Cref{conj:Ogus} (i.e. \Cref{intro:cor:Ogus}) is the behavior of the prismatic Riemann--Hilbert functor at the local system $T=(R^if_{\eta,*} \mathbb{Z}_p)_\tf$, where the latter by \Cref{intro:thm:pc_of_coh} is pointwise crystalline.
In particular, our approach relies crucially on the understanding of the pointwise crystalline local systems.
We recall from our previous joint work with Yang \cite{GY24} that when the rigid space $X_\eta$ has a smooth integral model $X$, a pointwise crystalline local system $T$ over $X_\eta$ is crystalline \emph{globally}: there is an $F$-isocrystal over the entire special fiber $X_k$ that is associated to the local system $T$.
However, the smoothness of $X$ is crucial in \cite{GY24}, where we conducted a series of descending arguments for the $F$-isocrystal, and made essential use of its description in terms of a Frobenius module over a noetherian ring, a property that is only true when the special fiber $X_k$ is non-singular.
In fact, even the notion of (global) crystalline local system was traditionally defined only for good reduction setting, until in \cite[Def.\ 4.4]{GY24} where we extended it to general reduction.
So as theoretic advancements, our \Cref{intro:thm:RH}.\ref{intro:thm:RH:crys_loc} pushes the good behaviors of the pointwise crystalline local system to the regular reduction setting.
In fact, we show that several variants of the crystallinity for $p$-adic local systems are equivalent to each other.
\begin{maintheorem}[\Cref{cor:various_notion_of_crystallinity}]
\label{intro:thm:crys_loc_sys}
	Let $X$ be a $p$-torsionfree regular $p$-adic formal scheme over $\mathcal{O}_K$, and let $T\in \Loc_{\mathbb{Z}_p}(X_\eta)$.
	The following conditions are equivalent: 
	\begin{enumerate}[label=\upshape{(\alph*)}]
		\item The weak prismatic $F$-crystal $\mathcal{E}_{\Prism,T}$ is a prismatic $F$-crystal whose \'etale realization is isomorphic to $T$.
		\item The local system $T$ is crystalline with respect to the integral model $X$: there is an $F$-isocrystal on the special fiber $X_k$ that is associated to $T$, in the sense of \cite[Def.\ 4.4]{GY24}.
		\item There is a dense open smooth subscheme $U\subseteq X$ such that $T|_{U_\eta}$ is a crystalline local system with respect to $U$.
		\item The restriction $T|_{\Spa(L)}$ is a crystalline representation for each Shilov point $\Spa(L)$ of $X_\eta$.
		\item There is a dense open smooth subscheme $U\subseteq X$ such that the restriction $T|_x$ is a crystalline representation for each classical point $x\in U_\eta$.
	\end{enumerate}
	In particular, the above conditions are independent of the regular integral model $X$.
\end{maintheorem}
Building on our previous work with Yang \cite{GY24}, to obtain \Cref{intro:thm:crys_loc_sys}, we prove that crystalline local systems with respect to regular integral models satisfy the \emph{purity}, an observation first made by Tsuji in \cite{Tsu08}, reproved by Moon in \cite{Moo24} for crystalline local systems and smooth integral models, and then extended to the semi-stable local systems by Du--Liu--Moon--Shimizu \cite{DLMS24b}.
To prove this property for local systems, the essential ingredient is a statement on Frobenius modules over regular prisms.
To explain, we recall that a given prism $(A,I)$ is called transversal if $A/I$ is $p$-torsionfree, and is called regular if $A$ is a regular ring.
For an ideal $J\subset A$, we let $\mathrm{Ass}(J)$ be the set of associated prime ideals of $J$ in $A$.
When $A$ is noetherian, for each associated prime $\mathfrak{p}\in \mathrm{Ass}(p,I)$, we let $A_{L_\mathfrak{p}}$ be the $p$-complete localization of $A$ at $\mathfrak{p}\in \mathrm{Ass}(p,I)$ with the induced Frobenius structure.
\begin{maintheorem}[Primitive Purity, \Cref{thm:primitive_purity_inf}, \Cref{def:extendable_tuples}]
\label{intro:thm:purity}
Let $(A,I)$ be a transversal and regular prism.
Assume $(M_\et; (M_\mathfrak{p})_{\mathfrak{p}\in \mathrm{Ass}(p,I)})$ is a tuple of Frobenius modules, where $M_\et \in \Vect^\varphi(A[1/I]^\wedge_p)$, $M_\mathfrak{p}\in \Vect^\varphi(A_{L_\mathfrak{p}})$, and they are isomorphic after base changing to $A_{L_\mathfrak{p}}[1/I]^\wedge_p$.
There is a canonical Frobenius module $M\in \Mod^\varphi_\sat(A)$ such that 
\[
M[1/I]^\wedge_p\simeq M_\et, ~ \text{and}~M\otimes_A A_{L_\mathfrak{p}}\simeq M_\mathfrak{p}~ \text{up to bounded $I^\infty$-torsion}.
\]
\end{maintheorem}
Here $\Mod_\sat^\varphi(A)$ denotes the subcategory of Frobenius modules over $A$ that are finite presented and saturated; cf. \Cref{def:ref_Frob_mod}.
The complete localization $A_{L_\mathfrak{p}}$ is a two-dimensional regular $\delta$-ring with $(A_{L_\mathfrak{p}},IA_{L_\mathfrak{p}})$ a transversal prism, and reduction $\mathcal{O}_{L_\mathfrak{p}}\colonequals A_{L_{\mathfrak{p}}}/IA_{L_\mathfrak{p}}$ is a $p$-adic discrete valuation ring with possibly imperfect residue field (\Cref{prop:localization_injection}).
We also note that the analogous statement holds true for Frobenius modules over the mod $p^n$-reductions, which we prove in \Cref{thm:primitive_purity_reduction}.

The general strategy of proving \Cref{thm:primitive_purity_inf} is similar to that of Tsuji, Moon, and Du--Liu--Moon--Shimizu and is of elementary nature, in the sense that one constructs the Frobenius module over $A$ via the intersections, and proves that it satisfies the expected structural results.
However, unlike in previous work where the reduction $A/I$ is smooth or strictly semi-stable, several major properties of Breuil--Kisin modules fail dramatically.
These include the fully faithfulness of the \'etale realization functor $\Vect^\varphi(A)\to \Vect^\varphi(A[1/I]^\wedge_p)$ (proved by Kisin in \cite[Prop.\ 2.1.12]{Kis06} and Bhatt--Scholze in \cite[Thm.\ 7.2]{BS23}), together with the projectivity of the $p$-inverted localizations which is implicitly pivotal in the purity of local systems.
Indeed, as we illustrate in \Cref{eg:Frob_mod_over_general_Frob-reg_prism}, these pathological behaviors are essentially due to the vast range of possible Frobenius structures on an arbitrary regular ring.
It is in contrast to the case of Breuil--Kisin ring and its higher dimensional analogues in all the prior work, where the Frobenius structure sends a coordinate variable to its $p$-th power and in particular is much simpler.
We however mention that as we prove in \Cref{thm:Frob_mod_is_analyticall_loc_free}, when the regular prism $A$ is special (or precisely is \emph{framed}), the $p$-inverted localizations of a Frobenius module over $A$ are finite projective, provided an appropriate pointwise locally freeness condition holds.
Our proof of this locally freeness, on the other hand, is through a new strategy, which in spirit is closer to our previous work on the Frobenius isogeny of the relative prismatic cohomology in \cite[Thm.\ 8.1, Prop.\ 8.9]{GR24}.

\begin{remark}[Primitivity]
We name \Cref{intro:thm:purity} a \emph{primitive} purity result, with two main reasons:
One is due to the elementary nature of its statement, in comparison to the purity of $p$-adic local systems or of prismatic $F$-crystals.
Another is the observation that the purity results of those more complicated objects are often the consequences of the purity for Frobenius modules.
The latter is illustrated in our proof of \Cref{thm:prismatic_RH_for_crystalline_local_system}.
In addition, in an upcoming joint work with Inoue and Yang \cite{GIY26}, we will apply \Cref{intro:thm:purity} to prove the pointwise criteria and the purity for the log-crystalline local systems that encompass both geometric and arithmetic monodromy.
\end{remark}

Finally, we give another application to the extension problem for $p$-divisible groups.
Let $X_0$ be a finite type regular and flat scheme over $\mathcal{O}_K$, let $Z_0\subset X_0$ be a closed subscheme of codimension $\geq 2$ (in $X_0$) that is supported in the special fiber, and let $G$ be a $p$-divisible group over $U_0\colonequals X_0\backslash Z_0$.
It was known by examples of Raynaud--Ogus--Gabber (cf.  \cite[\S 6]{dJO97}) that $G$ may not extend to a $p$-divisible group over entire $X_0$; in other words, the purity for $p$-divisible groups over $X_0$ could fail.
This phenomenon was later analyzed by Vasiu--Zink \cite{VZ10}, where they studied the extension problem of \emph{all} the $p$-divisible groups over a fixed base.
Yet it remains a question on whether a \emph{given} $p$-divisible group could extend or not.

Using the results in this article, we provide a method to answer the above question.
We let $X$ be the $p$-adic completion of $X_0$, which is a regular and flat $p$-adic formal scheme and is geometrically the tubular neighborhood of the special fiber.
By the assumption of $G$ and \Cref{intro:thm:crys_loc_sys}, the relative Tate module $\mathrm{T}_p(G)$, regarded as a local system over $X_\eta$, is crystalline with respect to $X$.
So by \Cref{intro:thm:RH}, it yields a canonical coherent prismatic $F$-crystal $\mathcal{E}_{\Prism, \mathrm{T}_p(G)}$ over $X$ of Hodge--Tate weight $[0, 1]$ whose \'etale realization is $\mathrm{T}_p(G)$.
Using the prismatic Dieudonne theory (\cite{ALB23}), we then obtain the following linear-algebraic criterion on the extension problem of $G$.
\begin{maincorollary}[Extension of $p$-divisible group]
\label{intro:cor:p-divisible}
In the above setting, the $p$-divisible group $G$ extends to $X_0$ (equivalently, to $X$) if and only if the coherent prismatic $F$-crystal $\mathcal{E}_{\Prism, \mathrm{T}_p(G)}$ is locally free.
\end{maincorollary}
The latter condition admits a very explicit description: by choosing finitely many framed regular prisms $(A_i,I_i)$ whose reductions $A_i/I_i$ cover $X$ jointly (which always exist by \Cref{thm:regular prism}), it suffices to check that each Frobenius module $\mathcal{E}_{\Prism,\mathrm{T}_p(G)}(A_i,I_i)$ is finite projective.

\begin{remark}
Finally, it worth mentioning that the prismatic site is indispensable in the proof of Ogus's \Cref{conj:Ogus}.
An elementary yet fundamental reason is that for any given regular $p$-adic formal scheme $X$, its prismatic site admits regular, hence noetherian, covering objects (\Cref{thm:regular prism}), which fails very often for the crystalline site of the special fiber.
Relatedly, unlike the special case when $X$ is smooth, we are unaware of any method that can prove the purity result for $p$-adic local systems in the generality of \Cref{intro:thm:crys_loc_sys} without using the prismatic theory.
\end{remark}

\subsection{Convention}
\label{sub:convention}
Unless otherwise mentioned, we will fix a discretely valued $p$-adic field $K$ with the ring of integers $\mathcal{O}_K$ and the perfect residue field $k$.
We assume any prism used in the article is an object in $W(k)_\Prism$.
For a prism $(A,I)$, we use $\overline{A}$ to denote the reduction $A/I$.
We say that a map of prisms $(A_1,I_1)\to (A_2,I_2)$ is completely (faithfully) flat if the map of the underlying rings $A_1\to A_2$ is $(p,I_1)$-completely (faithfully) flat.
We also consider a more general notion called the $\delta$-pair $(A,I)$, which consists of a $\delta$-ring $A$ together with an ideal $I$.
Given a map of rings $A\to B$ such that $A$ underlies a $\delta$-pair $(A,I)$, for simplicity we may use $(B,(I))$ to denote the induced $\delta$-pair $(B,IB)$.

We let $C$ be a fixed complete algebraic closure of $K$, let $\mathcal{O}_C$ be its ring of integers, so that the absolute Galois group $\Gal_K$ of the field $K$ acts continuously on them.
We choose a compatible system $\zeta_{p^n}$ of $p^n$-th roots of unity in $\mathcal{O}_C$ and let $\epsilon=(1,\zeta_p,\ldots)$ be the corresponding element of $\mathcal{O}_C^\flat=\lim_{x\mapsto x^p} \mathcal{O}_C$.
Let $\rAinf = \rAinf (\mathcal{O}_C) \colonequals W(\mathcal{O}_C^\flat)$ and $q \colonequals [\epsilon] \in \rAinf$ be the Teichm\"uller lift.
Then the Galois group $\Gal_K$ acts on $q$ via the cyclotomic character $\chi:\Gal_K\to \mathbb{Z}_p^*$, such that $g(q)=q^{\chi(g)}$ for $g\in \Gal_K$.
We set $\mu \colonequals q - 1$ and $\tilde{\xi} \colonequals [p]_q = \frac{q^p-1}{q-1}$.
So by the above explicit formula, the invertible ideal $\mu\rAinf$ is invariant under the action of $\Gal_K$.

Given a ring $R$ and an ideal $I$, unless otherwise mentioned, we use $R^\wedge_I$ to denote the derived $I$-adic completion of $R$.
If $R$ is complete under the $I$-adic topology, we use $R\langle t_1,\ldots,t_n \rangle$ to denote the complete polynomial ring over $R$, namely the $I$-adic completion of $R[t_1,\ldots,t_n]$.
If $R$ is in addition equipped with a $\delta$-structure, we use $R\{t_1,\ldots,t_n\}$ to denote the complete $\delta$-polynomial ring over $R$  (i.e. the complete free $\delta$-ring of $n$ variables over $R$, in the sense of \cite[Lem.\ 2.11]{BS22}).
We also use $R\{t_1,\ldots,t_n\}^{\pd}$ to denote the $p$-complete divided power polynomial ring over $R$.
To simplify the notation, we sometimes use $R\langle t_i \rangle$ to abbreviate $R\langle t_1,\ldots,t_n \rangle$ when the context of variables is clear, and similarly for the other variants.
In the special case when $(A,I)$ is a bounded prism, we use $A\langle 1/I \rangle$ to denote the $p$-complete localization $(A[1/I])^\wedge_p$.

For a characteristic $p$ point $x$ of a $p$-adic formal scheme, we use $k(x)$ to denote the residue field of $x$.
On the other hand, for a characteristic $0$ point $x_\eta$ of a rigid space, we use $K(x_\eta)$ to denote its characteristic zero complete residue field in the sense of rigid analytic geometry.

\subsection{Acknowledgement}
One of the major insights that make our approach to Ogus's \Cref{conj:Ogus} possible is the pointwise criteria of $p$-adic local systems, which was first discovered jointly with Ziquan Yang in \cite{GY24}. 
I thank Ziquan heartfully for the enjoyable collaborations and look forward to more to come.
The question of constructing the prismatic $F$-crystal directly from the local system was first mentioned to me by H\'el\`ene Esnault, during my visit to the University of Copenhagen in January 2023. I am grateful to H\'el\`ene for the question and encouragement at that time.
I am grateful to Bhargav Bhatt for several discussions on regular prisms and on Galois cohomology, both of which play pivotal roles in this article.
I also thank Hui Gao for the helpful correspondence on Hodge--Tate period ring, Shizhang Li for suggesting the argument on the seperatedness of de Rham cohomology, Alexander Petrov for inspiring discussions in early stage, and Emanuel Reinecke for helpful discussion on the higher direct image.
Further thanks go to Christophe Breuil, Toby Gee, Mark Kisin, Tong Liu, and Peter Scholze for their interest in this work, encouraging communications, and/or helpful suggestions on the manuscript.
Finally, I thank Arthur Ogus for posting \Cref{conj:Ogus} in his article \cite{Ogu84}, which naturally connects with many of my interests and expertises, and motivates the results in the prismatic theory developed in this article.

I thank the University of Chicago and the University of Minnesota for the financial support and the nice working environments in the past three years, during which the project was worked out.

\section{Regular prisms}
\label{sec:prism}
In this section, we analyze the prisms whose underlying rings are regular.
As we will see, such objects cover the prismatic site of a regular $p$-adic formal scheme, and play a crucial role throughout the article.
In addition, we study the localizations and the coproducts of regular prisms in the prismatic site.

We assume the basics of the $\delta$-ring and the prism as in \cite[\S2, \S3]{BS22}.
\subsection{Regular prism and framed regular prism}
\label{sub:regular prism}
In this subsection, we introduce the notion of \emph{regular prism}, together with a better-behaved notion called the \emph{framed regular prism}.
We show that up to \'etale localization, there always exists a framed regular prism that lifts the given regular $p$-adic formal scheme.

We first setup a convention.
\begin{definition}
	\label{def:regular_p-adic_formal_sch}
	We say a $p$-adic formal scheme $X$ over $\mathcal{O}_K$ is \emph{regular} if it is $p$-torsionfree and topologically of finite type over $\mathcal{O}_K$, such that for each affine open subscheme $\Spf(R)\subset X$, the ring $R$ is regular.
\end{definition}
The main construction of this subsection is the following.
\begin{theorem}
	\footnote{We thank Bhargav Bhatt for suggesting to us the construction.}
	\label{thm:regular prism}
	Let $X$ be a regular $p$-adic formal scheme of relative dimension $n-1$ over $\mathcal{O}_K$.
	For each closed point $x\in X$, there is a transversal and orientable prism $(A,I)\in X_\Prism$, such that:
	\begin{enumerate}[label=\upshape{(\alph*)}]
		\item The $p$-adic formal scheme $\Spf(\overline{A})$ is a $p$-completely \'etale neighborhood of $x\in X$.
			\item\label{thm:regular prism framing} We equip the ring $W(k) \langle t_1^{\pm1},\ldots,t_{n}^{\pm1} \rangle$ with the trivial $\delta$-structure on $t_i$. There is a map of $p$-completely \'etale map of $\delta$-rings $W(k) \langle t_1^{\pm1},\ldots,t_{n}^{\pm1} \rangle\to A_0$ together with a principal invertible ideal $I_0\subset A_0$ such that $(A,I)=((A_0)^\wedge_{I_0},I_0(A_0)^\wedge_{I_0})$.
	\end{enumerate} 
\end{theorem}
Before we discuss the proof, we mention that \Cref{thm:regular prism} in particular implies that any regular $p$-adic formal scheme over $\mathcal{O}_K$ is a local complete intersection, a folklore statement that we cannot find in literature.
\begin{corollary}
	\label{cor:regular_implies_lci}
	Any regular $p$-adic formal scheme over $\mathcal{O}_K$ is also quasi-syntomic over $W(k)$.
\end{corollary}	
In the proof below, by \emph{shrinking} a $p$-adic formal scheme $X$ (around a point $x\in X$), we mean replacing $X$ by an open dense $p$-adic formal scheme of $X$ (that contains $x\in X$).
\begin{proof}[Proof of \Cref{thm:regular prism}]
	We may assume $X=\Spf(R)$ is affine and the point $x$ corresponds to a maximal ideal $m\subset R$.
	By assumption, since the ring $R$ is regular, the cotangent space $m/m^2$ is a $k(x)$-vector space of dimension $n$, where $k(x)$ is the residue field of $X$ at $x$ and is in particular a finite extension of $k$.
	So by taking the finite \'etale base change along $W(k)\to W(k(x))$ if necessary, we may assume $k(x)=k$ and the map $W(k)\to R \to k(x)$ induces an isomorphism on residue fields.
	By lifting elements from $m/m^2$ to $R$, we can then produce a map of $p$-complete $W=W(k)$-algebras
	\[
	\psi:W\langle t_1^{\pm1},\ldots,t_n^{\pm1} \rangle \longrightarrow R,
	\]
	where the image of $\{t_1-1,\ldots,t_n-1\}$ forms a $k$-basis of $m/m^2$.
	Moreover, by taking the base change along the complete localization map $W\langle t_1^{\pm1}, \ldots,t_n^{\pm1} \rangle \to W\llbracket t_1-1,\ldots,t_n-1 \rrbracket$ and by Cohen's structure theorem \cite[\href{https://stacks.math.columbia.edu/tag/00NO}{Tag 00NO}]{stacks-project}, the above induces a surjection 
	\[
	\wh{\psi}: W\llbracket t_1-1,\ldots,t_n-1 \rrbracket \longrightarrow R^\wedge_m,
	\]
	such that the ideal $\ker(\wh{\psi})$ is principally generated by a regular element $f\in(p, t_1-1,\ldots,t_n-1)\subset W\llbracket t_1-1,\ldots,t_n-1 \rrbracket$ (\cite[\href{https://stacks.math.columbia.edu/tag/00NR}{Tag 00NR}]{stacks-project}, \cite[\href{https://stacks.math.columbia.edu/tag/00NQ}{Tag 00NQ}]{stacks-project}).
	Here we note that the ring $W\llbracket t_1-1,\ldots,t_n-1 \rrbracket$ naturally admits a $\delta$-structure, defined by $\delta(t_i)=0$ (or equivalently, its Frobenius endomorphism $\varphi_{W\llbracket t_1-1,\ldots,t_n-1 \rrbracket}$ sends $t_i$ onto $t_i^p$ and extends Witt vector Frobenius of $W$).
	\begin{claim}
		\label{claim:W[[t]] is a prism}
		The $\delta$-pair $(W\llbracket t_1-1,\ldots,t_n-1 \rrbracket, \ker(\wh{\psi}))$ is a prism in $X_\Prism$.
	\end{claim}
\begin{proof}[Proof of \Cref{claim:W[[t]] is a prism}]
	Since any pseudo-uniformizer of the ring $R$ is contained in the maximal ideal $m\subset R$, the choice of $t_i$ implies that there is an element $g\in (t_1-1,\ldots,t_n-1)\subset W\llbracket t_1-1,\ldots,t_n-1 \rrbracket$ such that $\wh{\psi}(g)=p$.
	This shows that the element $p-g\in \ker(\wh{\psi})$.
	Note that since $g\in (t_1-1,\ldots,t_n-1)$, the element $\delta(p-g)=\frac{(p-\varphi(g))-(p-g)^p}{p}$, regarded as a function in $t_1,\ldots,t_n$, satisfies the condition that its evaluation at $(t_1=1,\ldots,t_n=1)$ is the number $(1-p^{p-1})$.
	In particular, the element $\delta(p-g)$ is a unit in $W\llbracket t_1-1\,\ldots,t_n-1 \rrbracket$, and $p-g\in \ker(\wh{\psi})$ is a distinguished element (in the sense of \cite[Def.\ 2.19]{BS22}) in the ring $W\llbracket t_1-1,\ldots,t_n-1 \rrbracket$.
	As a consequence, since the generator $f$ is contained in $\mathrm{rad}(W\llbracket t_1-1,\ldots,t_n-1 \rrbracket)$, by \cite[Lem.\ 2.24]{BS22} the element $f$ is a distinguished element as well, and the pair $(W\llbracket t_1-1,\ldots,t_n-1 \rrbracket, (f)= \ker(\wh{\psi}))$ is a prism in $X_\Prism$.
\end{proof}

	We then claim that $\psi$ is a surjection after restricted onto a $p$-complete open neighborhood of $x\in X$.
	To see this, by applying Zariski's Main Theorem (\cite[\href{https://stacks.math.columbia.edu/tag/05K0}{Tag 05K0}]{stacks-project}) at the mod $p$ reduction of the map $\psi$, there is an open neighborhood of the point $(t_1-1,\ldots,t_n-1)$, which is a $p$-adic formal scheme $\Spf(S)\subseteq  \Spf(W\langle t_1^{\pm1},\ldots,t_n^{\pm1} \rangle)$, such that the restriction $\psi|_{\Spf(S)}$ is a finite morphism.
	We let $C$ be the cokernel of the map $\psi|_{\Spf(S)}\colon S\to (S\otimes_{W\langle t_1^{\pm1},\ldots,t_n^{\pm1} \rangle} R)^\wedge_p$.
	Then by the finiteness of $\psi|_{\Spf(S)}$, we see $C$ is a finite (and thus $p$-complete) $S$-module.
	Notice that by the first paragraph, the flat base change $C\otimes_S W\llbracket t_1-1,\ldots,t_n-1 \rrbracket$ vanishes.
	As a consequence, since $W\llbracket t_1-1,\ldots,t_n-1 \rrbracket$ is faithfully flat over the localization $S_{(t_1-1,\ldots,t_n-1)}$, we see there is a Zariski open neighborhood $V$ of the point $(t_1-1,\ldots,t_n-1)$ in $\Spec(S)$ such that $C|_V=0$ and the map $\psi|_V$ is a surjection.
	By taking the $p$-completion of the scheme $V$, we get a $p$-complete open neighborhood $\Spf(A_0)$ of $(t_1-1,\ldots,t_n-1)$ inside $\Spf(W\langle t_1^{\pm1},\ldots,t_n^{\pm1} \rangle)$ such that the restriction $\psi|_{\Spf(A_0)}:A_0\to (R\otimes_{W\langle t_1^{\pm1},\ldots,t_n^{\pm1} \rangle} A_0)^\wedge_p$ is a surjection.
	Moreover, its complete localization at $(t_1-1,\ldots,t_n-1)$ is a regular surjection of codimension one onto the complete local ring of $R^\wedge_m$.

	Next we show that by shrinking $\Spf(A_0)$ around the point $(t_1-1,\ldots,t_n-1)$ in $\Spf(W\langle t_1^{\pm1},\ldots,t_n^{\pm1} \rangle)$ if necessary,
	the kernel ideal $I_0\colonequals \ker(\psi|_{\Spf(A_0)})$ is an invertible sheaf generated by a regular element.
	For this, we first note that up to shrinking $A_0$ by a $p$-complete open neighborhood of $(t_1-1,\ldots,t_n-1)$ if necessary, the ideal $I_0$ is invertible: this is because its base change along the complete localization map $A_0\to W\llbracket t_1-1,\ldots,t_n-1 \rrbracket$ is generated by one regular element and is in particular invertible, so by the similar approximation mentioned in the first paragraph above,	\footnote{More concretely, both the kernel and the cokernel of the map $I_0\otimes \Hom_{A_0}(I_0,A_0)\to A_0$ vanish after the complete localization, so they vanish after shrinking around the point $(t_1-1,\ldots,t_n-1)$.}
	there is a $p$-complete open neighborhood of  $(t_1-1,\ldots,t_n-1)$ in $\Spf(A_0)$ such that $I_0$ is invertible.
	As the ideal $I_0$ is Zariski locally principal, we may further assume the invertible sheaf $I_0$ is generated by one element, which we assume for now.
	
	Now, by inheriting the $\delta$-structure along the $p$-completely \'etale map from $W\langle t_1^{\pm1},\ldots,t_n^{\pm1} \rangle$ (\cite[Lem.\ 2.18]{BS22}), we obtain an orientable $\delta$-pair $(A_0,I_0)$ together with a $p$-completely \'etale map of $\delta$-rings $W\langle t_1^{\pm1},\ldots,t_n^{\pm1} \rangle\to A_0$.
	We let  $(A,I\colonequals I_0A)$ be the $\delta$-pair defined by the $I_0$-adic completion of $(A_0,I_0)$.
	Since the completion preserves the regularity, we know the ring $A$ is regular.
	So to show that $(A,I)$ is in fact a prism in $X_\Prism$, by \cite[Lem.\ 3.1, Def.\ 3.2]{BS22}, it suffices to show that the ideal $p\cdot A/(I,\varphi(I))$ vanishes up to shrinking around the point $(t_1-1,\ldots,t_n-1)$.
	This again follows from the observation that its flat base change along the complete localization $A\to W\llbracket t_1-1,\ldots,t_n-1 \rrbracket$ vanishes by \Cref{claim:W[[t]] is a prism}.	
\end{proof}
Argued verbatically as the proof of \Cref{thm:regular prism}, we obtain the following analogue of framed regular prism over a complete regular local ring of mixed characteristic.
\begin{corollary}[Regular prism of complete regular local ring]
Let $R$ be a $p$-torsionfree complete regular local ring with residue field $k$.
There is a prism $(W\llbracket t_1-1,\ldots,t_n-1 \rrbracket, (f))$ with $\delta(t_i)=0$, such that $\overline{A}$ is a finite \'etale extension of $R$.
\end{corollary}

We give some names for the prisms whose underlying rings are regular.
\begin{definition}
	\label{def:regular_prism} 
	Let $X$ be a $p$-adic formal scheme over $\mathcal{O}_K$, and let $(A,I)$ be a bounded prism in $X_\Prism$.
	\begin{enumerate}[label=\upshape{(\roman*)}]
		\item We say $(A,I)$ is a \emph{regular prism} if $A$ is a regular ring.
		\item We say $(A,I)$ is a \emph{framed regular prism} if there is a map of $\delta$-pairs $(A_0,I_0)\to (A,I)$ together with a finite set of elements $\Sigma\subset A$, such that
		\begin{itemize}
			\item $I_0$ is a principal invertible ideal, with $\Spf(A_0/I_0)$ being $p$-completely \'etale over $X$.
			\item $A_0$ is $p$-completely \'etale over $W(k)\langle t_a|~a\in \Sigma \rangle$ and is equipped with a $\delta$-structure that is compatible with $W(k)$ with $\delta(t_a)=0$.
			\item The prism $(A,I)$ is the $I_0$-completion of $(A_0,I_0)$.
		\end{itemize}  
		Under the assumption, we call the data $((A_0,I_0),\Sigma)$ a \emph{framing} of $(A,I)$.
	\end{enumerate}
\end{definition}
By definition, we know the prism $(A,I)$ in \Cref{thm:regular prism} is a framed regular prism and admits a framing from $(A_0,I_0)$ with $\Sigma\subset A^\times$.

Here we notice that the regularity of a prism puts a very strong constrain on its Frobenius structure.
\begin{proposition}[Kunz's theorem for prisms]\footnote{The statement was proved independently in a recent work of Ishizuka--Nakazato \cite[Prop.\ 3.4]{IN24}, and we refer the reader to \cite[\S 3]{IN24} for related discussions.}
	\label{prop:reg_imply_Frob-regular}
	Let $(A,I)$ be a transversal prism such that $A$ is noetherian.
	The ring $A$ is regular if and only if the map $\varphi_A:A\to A$ is flat.
\end{proposition}
\begin{proof}
	We first assume $A$ is regular.
	The assumptions implies that the ring $A$ is noetherian, $p$-torsionfree, and classically $p$-complete.
	So by \cite[\href{https://stacks.math.columbia.edu/tag/0912}{Tag 0912}]{stacks-project}, it is enough to prove that the mod $p^n$ reduction of the map $\varphi_A$ is faithfully flat.
	Moreover, as observed in \cite[Rmk.\ 4.2]{BMS19}, it suffices to consider the case for $n=1$.
	Note that the mod $p$ reduction of $\varphi_A$ is the Frobenius endomorphism of the noetherian $\mathbb{F}_p$-algebra $A/pA$.
	So by Kunz' theorem (\cite[\href{https://stacks.math.columbia.edu/tag/0EC0}{Tag 0EC0}]{stacks-project}), it remains to show that $A/pA$ is a regular ring.
	
	As the regularity is a local property and can be checked at maximal ideals (\cite[\href{https://stacks.math.columbia.edu/tag/00OD}{Tag 00OD}, \href{https://stacks.math.columbia.edu/tag/0AFS}{Tag 0AFS}]{stacks-project}), it suffices to check the regularity of $(A/pA)_{\overline{\mathfrak{m}}}$, where $\overline{\mathfrak{m}}$ is any maximal ideal of $A/pA$.
	We let $\mathfrak{m}$ be the unique maximal ideal of $A$ that lifts $\overline{m}$.
	Then by \cite[\href{https://stacks.math.columbia.edu/tag/00NQ}{Tag 00NQ}]{stacks-project}, since $A_\mathfrak{m}$ is a regular local ring, it suffice to prove that the element $p$, which is contained in $\mathfrak{m}A_\mathfrak{m}$ by the $p$-completeness of $A$, is not in $\mathfrak{m}^2A_\mathfrak{m}$.
	Moreover, by the faithful flatness of the complete localization $A_\mathfrak{m} \to A^\wedge_\mathfrak{m}$, it suffice to check the above by replacing $A_\mathfrak{m}$ by the ring $A^\wedge_\mathfrak{m}$.
	
	To continue, by \cite[Lem.\ 2.18]{BS22}, the complete local ring $A^\wedge_\mathfrak{m}$ admits a unique $\delta$-structure compatible with that of $A$. 
	If $p$ is contained in $\mathfrak{m}^2A^\wedge_\mathfrak{m}$, then since the $\delta$-structure sends the subset $\mathfrak{m}^2A^\wedge_\mathfrak{m}$ into the subset $\mathfrak{m}A^\wedge_\mathfrak{m}$ (\cite[Def.\ 2.1]{BS22}), we know $\delta(p)$ is an element in $\mathfrak{m}A^\wedge_\mathfrak{m}$.
	However, as $A$ is $p$-torsionfree, we also know that $\delta(p)=\frac{\varphi_A(p)-p^p}{p}=1-p^{p-1}$, which is a unit in $A^\wedge_\mathfrak{m}$ and is contradicting to the above.
	Hence $p$ cannot be contained in the ideal $\mathfrak{m}^2A^\wedge_\mathfrak{m}$.
	
	Now assume that $\varphi_A:A\to A$ is flat, which, by the surjectivity of the Frobenius endomorphism of $\Spec(A/pA)$, is automatically faithfully flat.
	Then by Kunz's theorem (\cite[\href{https://stacks.math.columbia.edu/tag/0EC0}{Tag 0EC0}]{stacks-project}) we know $A/pA$ is regular.
	So the claim follows from \cite[\href{https://stacks.math.columbia.edu/tag/00NU}{Tag 00NU}]{stacks-project}, and the fact that $A$ is $p$-complete and $p$-torsionfree.
\end{proof}

The following statement summarizes the properties of the Frobenius structure of framed regular prisms.
\begin{proposition}[Frobenius structure of framed regular prism]
	\label{prop:regular prism Frobenius}
	Let $X$ be a regular $p$-adic formal scheme over $\mathcal{O}_K$, and let $(A,I)$ be a framed regular prism in $X_\Prism$.
	\begin{enumerate}[label=\upshape{(\roman*)}]
		\item The Frobenius morphism $\varphi_A:A\to A$ is finite and faithfully flat.
		\item The mod $I$ reduction of the perfection map $(A,I)\to (A_\perf, I_\perf)$ is a quasi-syntomic cover with the shifted $p$-complete cotangent complex $\mathbb{L}_{\overline{A}_\perf/\overline{A}}[-1]$ isomorphic to a finite projective module over $\overline{A}_\perf$.
	\end{enumerate}
In particular, a framed regular prism is a regular prism.
\end{proposition}
\begin{proof}
	Let $(A_0,I_0)$ be a framing of $(A,I)$, where $A_0$ is a $p$-completely \'etale $\delta$-algebra over $W\langle t_1, \ldots, t_n\rangle$.
	We first note that by the assumption, the mod $p$ reduction of $A_0$ is a smooth $k$-algebra.
	In particular, the Frobenius endomorphism on $A_0/p$ is finite surjective and is flat, where the latter follows from Kunz's theorem (\cite[\href{https://stacks.math.columbia.edu/tag/0EC0}{Tag 0EC0}]{stacks-project}).
	So by the $p$-completeness and $p$-torsionfreeness of $A_0$, since the endomorphism $\varphi_{A_0}$ lifts the Frobenius endomorphism of $A_0/p$, we see $\varphi_{A_0}$ is finite and $p$-completely faithfully flat.
	Here we note that by the noetherianity of the ring $A_0$ and \cite[\href{https://stacks.math.columbia.edu/tag/0912}{Tag 0912}]{stacks-project}, the map $\varphi_{A_0}$ is in fact faithfully flat.
	
	We then note that as the ring $A$ is the $I_0$-adic completion of $A_0$, the Frobenius structure $\varphi_A$ is defined by the completion of $\varphi_{A_0}$.
	By the finiteness of the homomorphism $\varphi_{A_0}$ and the noetherianity of the ring $A_0$, we know $\varphi_A$ is isomorphic to the base change morphism $\varphi_{A_0}\otimes_{A_0} A$.
	Hence the map $\varphi_A$ is finite and faithfully flat as well.
	
	Finally, to calculate the $p$-complete cotangent complex $\mathbb{L}_{\overline{A}_\perf/\overline{A}}$, we recall that the perfection $(A_\perf, I_\perf)$ is defined as the $(p,I)$-completion of $\colim_{\varphi_A} A$.
	By the paragraph above, since the Frobenius map $\varphi_A$ is the $I_0$-completion of $\varphi_{A_0}$, the perfection map $A\to A_\perf$ can be identified with the $(p,I_0)$-completion of the map
	\[ 
	A_0 \longrightarrow \colim_{\varphi_{A_0}} A_0.
	\]
	In particular, the map of $p$-complete algebras $\overline{A}\to \overline{A}_\perf$ can be identified with the mod $I_0$ reduction map $\overline{A}_0 \to \overline{(\colim_{\varphi_{A_0}} A_0)^\wedge_p}$, and we get
	\[
	\mathbb{L}_{\overline{A}_\perf/\overline{A}} \simeq \mathbb{L}_{\overline{(\colim_{\varphi_{A_0}} A_0)^\wedge_p} / \overline{A}_0} \simeq  \mathbb{L}_{(\colim_{\varphi_{A_0}} A_0)^\wedge_p / A_0} \otimes_{A_0}^L \overline{A}_0.
	\]
	So to finish the proof, it suffice to show that $\mathbb{L}_{(\colim_{\varphi_{A_0}} A_0)^\wedge_p / A_0}$ has Tor amplitude $[-1,-1]$, which follows from the distinguished triangle
	\[
	\mathbb{L}_{(\colim_{\varphi_{A_0}} A_0)^\wedge_p/W} \longrightarrow \mathbb{L}_{(\colim_{\varphi_{A_0}} A_0)^\wedge_p / A_0} \longrightarrow \mathbb{L}_{A_0/W}[1] \otimes^L_{A_0} (\colim_{\varphi_{A_0}} A_0)^\wedge_p.
	\]
	Here the first term vanishes thanks to the relative perfectness of $(\colim_{\varphi_{A_0}} A_0)^\wedge_p$ over $W$, and the last term is a finite projective module over $(\colim_{\varphi_{A_0}} A_0)^\wedge_p$ that lives in cohomological degree $(-1)$, thanks to the smoothness of $A_0$ over $W$.
\end{proof}

A quick consequence of \Cref{prop:regular prism Frobenius} and Andr\'e's lemma is that framed regular prisms cover the absolute prismatic site.
\begin{corollary}[Weak initiality]
	\label{cor:regular prism cover}
	Let $X$ be a regular $p$-adic formal scheme over $\mathcal{O}_K$, and let $(A,I)$ be a framed regular prism such that $\Spf(\overline{A})$ is an \'etale cover of $X$.
	Then the prism $(A,I)$ covers the absolute prismatic site $X_\Prism$.
\end{corollary}
\begin{proof}
	Let $(A_\perf, IA_\perf)$ be the perfection of $(A,I)$, and let $(B,J)$ be any bounded prism in $X_\Prism$.
	By \Cref{prop:regular prism Frobenius}, we know the reduction $\overline{A}_\perf$ is a quasi-syntomic cover over $X$, so the $p$-complete base change $(\overline{B}\otimes_{\mathcal{O}_X} \overline{A}_\perf)^\wedge_p$ is also a quasi-syntomic cover.
	Thus by \cite[Prop.\ 7.11]{BS22}, we know there is a map of the prisms $(B,J)\to (C,L)$ such that $\overline{C}$ is a $p$-complete flat cover over $(\overline{B}\otimes_{\mathcal{O}_X} \overline{A}_\perf)^\wedge_p$ and hence over $\overline{B}$.
	Moreover, by the initial property of the perfect prism $(A_\perf,IA_\perf)$, the prism $(C,L)$ admits a map from $(A_\perf,IA_\perf)$ and hence from $(A,I)$.
	So in summary, a flat cover of the prism $(B,J)$ admits a map from $(A,I)$, which finishes the proof.
\end{proof}

By taking the $p$-adic divided power envelope for the surjection $A\to \overline{A}$, one obtains a crystalline prism of the mod $p$ fiber $X_{p=0}$ that is also a covering object, as we shall explain.
\begin{example}[Framed crystalline prism]
	\label{eg:crystalline_prism}
	Let $X=\Spf(R)$ be an affine regular $p$-adic formal scheme over $\mathcal{O}_K$, and let $(A,I)$ be a framed regular prism, such that $\overline{A}=R$.
	We let $D=D_A(I)$ be the $p$-complete $p$-adic divided power envelope for the surjection $A\to \overline{A}$, and let $i:A\to D$ be the canonical map.
	So we get a $\delta$-pair $(D,(p))$, equipped with the induced $\delta$-structure from $A$.
	In addition, by precomposing with the faithfully flat Frobenius structure $\varphi_A:A\to A$ (\Cref{prop:regular prism Frobenius}) and by applying the mod $p$ reduction, we get the following diagram of solid arrows:
	\begin{equation}
		\label{eq:crys_prism_1}
		\begin{tikzcd}
			A/p \arrow[r, "\varphi_A"] \arrow[d, "\text{mod~}I"'] & A/p \arrow[r, "i"] & D/p\\
			R/p\arrow[rru, dashed, "f"'].
		\end{tikzcd}
	\end{equation}
    Notice that for any generator $x\in I$, we have $\varphi_A(x)=x^p+p\delta(x)$, which implies that the image of the ideal $I$ along the composition $i\circ \varphi_A$ in (\ref{eq:crys_prism_1}) is zero.
    Hence the diagram in (\ref{eq:crys_prism_1}) induces a natural dashed arrow $f:R/p\to D/p$.
    Here we note that the above in particular produces a crystalline prism $(D,pD)$ with the structure map $f\colon R/p\to D/pD)$ in $(X_{p=0})_\Prism$.
    
    Moreover, we claim that the map of $\mathbb{F}_p$-algebras $f:R/p\to D/p$ is faithfully flat.
    To see this, we fix a generator $x$ of the invertible ideal $I\subset A$, whose mod $p$ reduction is a non-zero-divisor in $A/p$ (by the assumption of $A$).
    Then by taking the pushout of the left corner, the diagram (\ref{eq:crys_prism_1}) can be enlarged into the following
    \begin{equation}
    	\label{eq:crys_prism_2}
    	\begin{tikzcd}
    		A/p \arrow[r, "\varphi_A"]\arrow[d, "\text{mod~}I"'] & A/p \arrow[r, "i"] \ar[d] & D/p,\\
    		R/p \arrow[r] & A/(p,I^p)  \arrow[ul, phantom, "\ulcorner"]  \arrow[ru, "g"'] &
    	\end{tikzcd}
    \end{equation}
    and by the faithful flatness of $\varphi_A$ in \Cref{prop:regular prism Frobenius}, it is left to check that the map $g:A/(p,I^p)\to D_A(I)/p$ is faithful flat.
    
    To proceed, by the regularity of the element $x$ in $A/p$, we notice that the map $g$ is equal to the base change of $g_0:\mathbb{F}_p[T]/T^p \to D_{\mathbb{F}_p[T]}(T)$ along the map $\mathbb{F}_p[T]\xrightarrow{T\mapsto x}A/p$.
    The latter in particular allows us to reduce the question to checking the faithful flatness of the map $g_0$.
    Notice that since $T$ is a non-zero-divisor in $\mathbb{F}_p[T]$, by \cite[Lem.\ 3.42]{Bha12}, there is an increasing exhaustive $\mathbb{N}$-indexed filtration on $D_{\mathbb{F}_p[T]}(T)$, where the $i$-th graded piece is $\varphi_{\mathbb{F}_p[T]}^*(\Gamma^i_{\mathbb{F}_p}\mathbb{F}_p)$ and is in particular finite projective over the ring $\mathbb{F}_p[T]/T^p=\varphi_{\mathbb{F}_p[T]}^*(\mathbb{F}_p)$.
    This shows that $D_{\mathbb{F}_p[T]}(T)$ is a faithfully flat $\mathbb{F}_p[T]/T^p$-module, which finishes the proof.
\end{example}

Recall that for a given uniformizer $\pi\in \mathcal{O}_K$, there is a natural \emph{Breuil--Kisin} prism $(\mathfrak{S}\colonequals W(k)\llbracket u \rrbracket, E(u))\in (\mathcal{O}_K)_\Prism$.
Its $\delta$-structure extends that of $W(k)$ and vanishes at $u$, and its reduction map sends the element $u$ onto $\pi\in \mathcal{O}_K$.
So a Breuil--Kisin prism is in particular framed.
As we shall recall from \cite{DLMS24} and \cite{DLMS24b}, the construction naturally extends to the case when $X$ is strictly semi-stable or even smooth, with concrete presentations.
\begin{example}[Strictly semi-stable prism]
	\label{eg:semi-stable prism}
	Let $R$ be the ring $\mathcal{O}_K\langle x_1,\ldots,x_n \rangle/(\prod_i x_i-\pi)$, so that $R$ is strictly semi-stable and thus regular.
	Under the assumption, the tangent space of $R$ at the origin is generated by the image of $\{x_1,\ldots,x_n\}$ in $m/m^2$.
	We can then form a natural map 
	\[
	\psi:W\langle s_1,\ldots,s_n \rangle \longrightarrow R,
	\]
	sending $\{s_1,\ldots,s_n\}$ onto the elements $\{x_1,\ldots,x_n\}$.
	By rewriting the ring $R$ as 
	\[
	W\langle u,s_1,\ldots,s_n \rangle/(E(u), \prod_i s_i-u),
	\] 
	the map $\psi$ identifies the ring $R$ as the quotient of $W\langle s_1,\ldots,s_n \rangle$ by the regular element $E(\prod_i s_i)$.
	Moreover, we can equip $W\langle s_1,\ldots,s_n \rangle$ with the $\delta$-structure extending that of $W$ such that $\delta(s_i)=0$, thus yielding the $\delta$-pair $(W\langle s_1,\ldots,s_n \rangle, (E(\prod_i s_i)))$ over $(\mathfrak{S},E(u))$ whose reduction is $R$.
	Since $E(u)$ is distinguished and has image $E(\prod_i s_i)$ in $W\langle s_1,\ldots,s_n \rangle$, we in particular know that $E(\prod_i s_i)$ is a distinguished element as well.
	Hence the $E(\prod_i s_i)$-completion of $(W\langle s_1,\ldots,s_n \rangle, E(\prod_i s_i))$ is a framed regular prism with $\Sigma=\{s_1,\ldots,s_n\}$.
	Here we note that the construction naturally extends to $p$-complete \'etale $R$-algebras by \cite[Lem.\ 2.18]{BS22}.
\end{example}

At the end of this subsection, we prove that regular transversal prisms of Krull dimension two admits a uniform and explicit description.
\begin{proposition}
	\label{prop:Frob_reg_prism_of_dim_2}
	Let $(A,I)$ be a regular and transversal prism such that $A$ has Krull dimension two.
	There is an isomorphism of prisms $(V\llbracket t \rrbracket, (E(t)))\xrightarrow{\sim} (A,I)$, where 
	\begin{itemize}
		\item the element $E(t)$ is an Eisenstein polynomial in $t$;
		\item the ring $V$ is a mixed characteristic Cohen ring (\cite[\href{https://stacks.math.columbia.edu/tag/0327}{Tag 0327}]{stacks-project}) and is a Frobenius equivariant subring in $A$;
		\item the element $\varphi_{V\llbracket t \rrbracket}(t)$ is divisible by $t$.
	\end{itemize}
\end{proposition}
\begin{remark}[Difference with the Breuil--Kisin prism]
	\label{rmk:Frob_reg_prism_of_dim_2}
	It is tempting to claim that the prism $(A,I)$ in \Cref{prop:Frob_reg_prism_of_dim_2} is more or less a special case of the Breuil--Kisin prism (with possibly imperfect residue field).
	For the underlying rings, this is indeed the case by \textit{loc.\ cit.}.
	However, one cannot always arrange $(A,I)$ to the case such that $\varphi_{V\llbracket t \rrbracket}(t)=t^p$, either via a change of variable or up to a finite extension.
	The latter can be seen for example by taking (the reduction of) the perfection of the $q$-de Rham prism $(\mathbb{Z}_p\llbracket q-1 \rrbracket,[p]_q)$.
\end{remark}
\begin{proof}[Proof of \Cref{prop:Frob_reg_prism_of_dim_2}]
	We first notice that the flatness of $\varphi_A$ implies the flatness of the Frobenius endomorphism of $A/pA$.
	So by Kunz's theorem \cite[\href{https://stacks.math.columbia.edu/tag/0EC0}{Tag 0EC0}]{stacks-project}, the ring $A/pA$ is a complete regular local ring of dimensino one in characteristic $p$.
	Thus by Cohen's structural theorem \cite[\href{https://stacks.math.columbia.edu/tag/0C0S}{Tag 0C0S}]{stacks-project}, we have $A/pA\simeq k'\llbracket t' \rrbracket$ for a field $k'$ in characteristic $p$.
	Hence by taking the power series ring over any Cohen ring $V'$ of the residue field $k'$, we get an isomorphism of $p$-complete rings $A\simeq V'\llbracket t' \rrbracket$.
	Here we note that the latter isomorphism may not be Frobenius equivariant.
	
	To proceed, since the ring $A$ is of dimension two, the transversal assumption implies that $A$ is a complete regular local ring.
	So $A$ is in particular a unique factorization domain (\cite[\href{https://stacks.math.columbia.edu/tag/0AG0}{Tag 0AG0}]{stacks-project}), and thus the invertible ideal $I$ is principal.
	Moreover, by the transversal condition again, the quotient $\overline{A}$ is a mixed characteristic regular ring in dimension one, and hence a mixed characteristic complete discrete valuation ring.
	As the residue field of $\overline{A}$ is $k'$, the induced composition $V'\to V'\llbracket t' \rrbracket\simeq A \to \overline{A}$ is a finite totally ramified extension.
	So the image of $t'$ in $\overline{A}$, which generates $\overline{A}$ as a $V'$-algebra, is a uniformizer.
	In particular, the kernel of the surjection $V'\llbracket t \rrbracket\to \overline{A}$ is generated by a Eisenstein polynomial $E(t')$.
	
	At this moment, we note that for each $n\geq 1$, the pair $(A,\varphi^n(I))$ with the same delta structure is also a regular prism of dimension two.
	In particular, the sequence $\{A/\varphi^n(I)\}$ is a sequence of extensions of discrete valuation rings.
	We let $(A_\perf,IA_\perf)$ be the perfection of $(A,I)$.
	Then the reduction $\overline{A}_\perf$ is the $p$-adic completion of the colimit $\colim_{n\geq 0} A/\varphi^n(I)$.
	Hence the perfectoid ring $S\colonequals\overline{A}_\perf$ is the ring of integers of a perfectoid field extension of $\overline{A}[1/p]$,
	and the $\delta$-ring $A_\perf$ is isomorphic to $W(S^\flat)$.
	
	We now let $\mathfrak{f}$ be the residue field of $S$ and thus of $S^\flat$, which is perfect and naturally contains both $k'$ and thus the perfection $k'_\perf$.
	Here by the perfectness of $k'_\perf$, the inclusion $k'_\perf\to S^\flat$ induces a Frobenius equivariant injection $W(k'_\perf)\to W(S^\flat)=A_\perf$.
	Moreover, there is a Frobenius equivariant surjection $A_\perf=W(S^\flat)\to W(\mathfrak{f})$.
	Now we let $V$ be the intersection $W(k'_\perf)\cap A$ inside $A_\perf$, which by construction is preserved by the endomorphism $\varphi_A$.
	Let $\iota:V\to A$ be the inclusion map.
	Then we can fit $A$ into the Frobenius equivariant commutative diagram 
	\begin{equation}
		\label{eq:diagram_of_V}
		\begin{tikzcd}
			W(k'_\perf) \arrow[r, hook] & W(S^\flat) \arrow[r, hook] &W(\mathfrak{f}) \\
			V \arrow[u, hook] \arrow[r, hook, "\iota"'] & A. \arrow[u, hook] &
		\end{tikzcd}
	\end{equation}
On the other hand, by the explicit presentation $A/pA\simeq k'\llbracket t' \rrbracket$, we know the intersection $k'_\perf \cap A/pA$ inside $S^\flat=A_\perf/pA_\perf\simeq k'_\perf\llbracket t^{1/p^\infty} \rrbracket$ is naturally isomorphic to the subfield $k'$.
As a consequence, by the limit presentation $V\simeq \lim_{n\geq 1} \bigl( W(k'_\perf)/p^nW(k'_\perf) \cap A/p^nA \bigr)$, we know $V$ is also a Cohen ring of the field $k'$.
Therefore, via the inclusion map $\iota:V\to A$, we can now replace the previous presentation $V'\llbracket t' \rrbracket\simeq A$ by $V\llbracket t' \rrbracket\simeq A$, so that the subring $V\subset A$ is preserved by $\varphi_A$.

Finally, let $A'$ be the image of the Frobenius equivariant composition $A\to A_\perf\simeq W(S^\flat) \to W(\mathfrak{f})$ in (\ref{eq:diagram_of_V}).
Then by taking the mod $p$ reduction, we know the composition $V\to A\to A'$ is an isomorphism.
We let $f$ be the image of $t'\in V\llbracket t' \rrbracket\simeq A$ in $V\simeq A'$.
By construction, the element $(t'-\iota(f))$ in $A$ generates the kernel of the Frobenius equivariant surjection $A\to A'$.
Hence by assigning $t\colonequals t'-f$, we see the ring $A$ is isomorphic to $V\llbracket t \rrbracket$, with the induced Frobenius structure sending $t$ onto a $t$-divisible element.	
\end{proof}

\subsection{Localizations and completions}
\label{sub:localization_of_prism}
In this subsection, we consider some localizations and completions of a regular prism that will be used later.

We start with an observation on the localization of the reduction of a regular ring.
Below for a finite module $M$ over a ring $R$, we let $\mathrm{Ass}_R(M)$ be the set of associated prime ideals of $M$, and we drop the subscript of the notation when the meaning is clear.
\begin{lemma}
	\label{lem:localization_injection}
	Let $(A,I)$ be a regular prism.
	\begin{enumerate}[label=\upshape{(\roman*)}]
		\item\label{lem:localization_injection_minimal} The reduction $\overline{A}/p\overline{A}=A/(p,I)A$ has no embedded prime ideals, and hence the localization $\overline{A}_\mathfrak{p}/p\overline{A}_\mathfrak{p}$ is an artinian local ring for each $\mathfrak{p}\in \mathrm{Ass}(\overline{A}/p\overline{A})$.
		\item\label{lem:localization_injection_inj} The localization map $\overline{A}/p\overline{A} \to \prod_{\mathfrak{p}\in \mathrm{Ass}(\overline{A}/p\overline{A})} \overline{A}_\mathfrak{p}/p\overline{A}_\mathfrak{p}$ is an injection.
	\end{enumerate}
\end{lemma}
\begin{proof}
	Part \ref{lem:localization_injection_inj} follows by applying \cite[\href{https://stacks.math.columbia.edu/tag/0311}{Tag 0311}]{stacks-project} to $M=\overline{A}/p\overline{A}$.
	It is left to show that every prime ideal in $\mathrm{Ass}_{\overline{A}}(\overline{A}/p\overline{A})$ is minimal.
	The statement is local and can be checked around each closed point of $V(p,I)\subset \Spec(A)$, hence we may assume $A$ is a $p$-torsionfree complete regular local ring with a flat Frobenius structure $\varphi_A:A\to A$ (\Cref{prop:reg_imply_Frob-regular}).
	Moreover, since $\varphi_A$ is $\mathbb{Z}_p$-linear, by taking the base change along the mod $p$ reduction, we know the Frobenius endomorphism $\varphi_{A/pA}:A/pA\to A/pA$ is finite flat as well.
	So by Kunz's theorem \cite[\href{https://stacks.math.columbia.edu/tag/0EC0}{Tag 0EC0}]{stacks-project}, the ring $A/pA$ is a complete regular ring in characteristic $p$.
	Thus by Cohen's structural theorem \cite[\href{https://stacks.math.columbia.edu/tag/0C0S}{Tag 0C0S}]{stacks-project}, we have $A/pA\simeq k\llbracket x_1,\ldots,x_n \rrbracket$ for a field $k$ in characteristic $p$.
	If the ideal $I\cdot A/pA$ is the zero ideal (so that $I=pA$ and $(A,I)$ is a crystalline prism), then the claim follows from the above structural result.
	Otherwise, note that since the power series ring is a unique factorization domain, the non-zero invertible ideal $I\cdot A/pA$ is principal, and its generator $\overline{f}$ is a finite product of prime elements $\overline{g}_j$ in $k\llbracket  x_1,\ldots, x_n\rrbracket$.
	We let $g_j\in \overline{A}$ be any lift of $\overline{g}_j$ along the surjection $\overline{A}\to \overline{A}/p=k\llbracket x_1,\ldots, x_n \rrbracket/(\overline{f})$.
	Then by \cite[\href{https://stacks.math.columbia.edu/tag/00LB}{Tag 00LB}]{stacks-project}, we have $\mathrm{Ass}_{\overline{A}}(\overline{A}/p\overline{A})\subset \{(p,g_j)_j\}$.
	Notice that each prime factor $\overline{g}_j$ generates a height $1$ prime ideal in $k\llbracket x_1,\ldots, x_n \rrbracket$.
	So for each $j$, the Krull dimension of the quotient ring $\overline{A}/(p,g_j)\overline{A}=k\llbracket x_1,\ldots, x_n \rrbracket/\overline{g}_j$ is equal to $n-1$.
	Hence by the dimensional reason, every element in $\mathrm{Ass}_{\overline{A}}(\overline{A}/p\overline{A})$ is minimal.
\end{proof}

For the convenience of our discussion, we introduce the following notations on the complete localizations.
\begin{construction}
	\label{const:localization_of_regular_prism}
	Let $(A,I)$ be a transversal regular prism.
	\begin{enumerate}
		\item	\label{const:localization_of_regular_prism_dvr} 	
		For each prime ideal $\mathfrak{p}\in \mathrm{Ass}_{\overline{A}}(\overline{A}/p\overline{A})$, we let $\mathcal{O}_{L_{\mathfrak{p}}}$ be the $p$-completion of the localization $\overline{A}_{\mathfrak{p}}$, and let $L_{\mathfrak{p}}$ be the ring $\mathcal{O}_{L_\mathfrak{p}}[1/p]$.
		
		\item \label{const:localization_of_regular_prism_prism}
		By \cite[\href{https://stacks.math.columbia.edu/tag/05DZ}{Tag 05DZ}]{stacks-project} and the equality $\overline{A}/p\overline{A}=A/(p,I)A$, the closed immersion $\Spec(\overline{A}) \to \Spec(A)$ induces a natural bijection between the two sets of primes ideals 
		\[
		\mathrm{Ass}_{\overline{A}}(\overline{A}/p\overline{A}) \longrightarrow \mathrm{Ass}_A(\overline{A}/p\overline{A}),~\mathfrak{p} \longmapsto \mathfrak{P}.
		\]
		We let $A_{L_\mathfrak{p}}$ be the $(p,I)$-completion of the localization $A_\mathfrak{P}$.
	\end{enumerate}
\end{construction}

\begin{proposition}
	\label{prop:localization_injection}
	Let $(A,I)$ be a transversal regular prism, and let $n\in \mathbb{N}\cup\{\infty\}$.
	\begin{enumerate}[label=\upshape{(\roman*)}]
		\item\label{prop:localization_injection_R} The natural map $\overline{A}/p^n\overline{A} \longrightarrow \prod_{\mathfrak{p}\in \mathrm{Ass}_{\overline{A}}(\overline{A}/p\overline{A})} \mathcal{O}_{L_\mathfrak{p}}/p^n\mathcal{O}_{L_\mathfrak{p}}$ is an injection. 
		When $n=\infty$ the target is a finite product of complete discrete valuation rings $\mathcal{O}_{L_\mathfrak{p}}$ of mixed characteristic $(0,p)$.
		\item\label{prop:localization_injection_A} For each $\mathfrak{p}\in \mathrm{Ass}_{\overline{A}}(\overline{A}/p\overline{A})$, the pair $(A_{L_\mathfrak{p}}, IA_{L_\mathfrak{p}})$ naturally admits a structure of a regular prism over $\mathcal{O}_{L_\mathfrak{p}}$ that is compatible with $(A,I)$.
		Moreover, the canonical map $A/p^nA\longrightarrow \prod_{\mathfrak{p}\in \mathrm{Ass}_{\overline{A}}(\overline{A}/p\overline{A})} A_{L_\mathfrak{p}}/p^nA_{L_\mathfrak{p}}$ is an injection.
	\end{enumerate}
\end{proposition}
\begin{remark}
	Here we note that by \Cref{prop:localization_injection}.\ref{prop:localization_injection_R}, each $A_{L_\mathfrak{p}}$ is in particular a complete regular local ring of dimension $2$.
\end{remark}
\begin{proof}
	We start with \ref{prop:localization_injection_R}.
	By assumption and \Cref{lem:localization_injection}.\ref{lem:localization_injection_inj}, each localization $\overline{A}_\mathfrak{p}$ is a $p$-torsionfree regular local ring of Krull dimension one such that $p$ is contained in the maximal ideal.
	So by taking the $p$-adic completion, the ring $\mathcal{O}_{L_{\mathfrak{p}}}$ is complete discrete valuation ring of mixed characteristic $(0,p)$.
	By the $p$-completeness and $p$-torsionfreeness of $\overline{A}$, the injectivity of the map $\overline{A}/p^n\overline{A} \longrightarrow \prod_{\mathfrak{p}\in \mathrm{Ass}_{\overline{A}}(\overline{A}/p\overline{A})} \mathcal{O}_{L_\mathfrak{p}}/p^n\mathcal{O}_{L_\mathfrak{p}}$ follows from \Cref{lem:localization_injection}.\ref{lem:localization_injection_inj}.
	
	For \ref{prop:localization_injection_A}, as the map $A\to A_{L_\mathfrak{p}}$ is a $(p,I)$-completion of a Zariski localization, 
	the ring $A_{L_\mathfrak{p}}$ is regular.
	Moreover, by \cite[Lem.\ 2.18]{BS22}, the ring $A_{L_\mathfrak{p}}$ admits a unique $\delta$-structure that is compatible with $A$.
	The induced $\delta$-pair $(A_{L_\mathfrak{p}}, IA_{L_\mathfrak{p}})$ is then a transversal prism, thanks to \cite[Lem.\ 3.1]{BS22}.
	In particular, the injectivity of the map $A/p^nA\longrightarrow \prod_{\mathfrak{p}\in \mathrm{Ass}_{\overline{A}}(\overline{A}/p\overline{A})} A_{L_\mathfrak{p}}/p^nA_{L_\mathfrak{p}}$  follows from that in \ref{prop:localization_injection_R}.
\end{proof}

Let $(A,I)$ be a transversal orientable regular prism, with $d$ a generator of $I$.
By assumption, the sequence $(d,p)$ is regular in the ring $A$.
In particular, by \cite[\href{https://stacks.math.columbia.edu/tag/062F}{Tag 062F}]{stacks-project}, since $\mathrm{H}_1(\mathrm{Kos}_A(d,p))=0$, we know there is no element $x\in A\backslash pA$ such that $dx\in pA$.
Hence the ring $A/pA$ has no $I$-torsion and the canonical map $A/pA \to A/pA[1/I]$ is injective; similarly for $A_{L_{\mathfrak{p}}}/p^nA_{L_{\mathfrak{p}}}$, where $\mathfrak{p}\in \mathrm{Ass}_{\overline{A}}(\overline{A}/p\overline{A})$.
So by combining this with \Cref{prop:localization_injection}, we get the following natural commutative diagram of injections into various (complete) localizations
\begin{equation}
	\label{diagram:inj_of_prisms}
	\begin{tikzcd}[column sep=small,row sep=small]
		&A/p^nA[1/I] \arrow[rd, hook]&\\
		A/p^n \arrow[ru, hook] \arrow[rd, hook] && \prod_{\mathfrak{p}\in \mathrm{Ass}(\overline{A}/p\overline{A})} A_{L_{\mathfrak{p}}}/p^nA_{L_{\mathfrak{p}}}[1/I]\\
		& \prod_{\mathfrak{p}\in \mathrm{Ass}(\overline{A}/p\overline{A})} A_{L_{\mathfrak{p}}}/p^nA_{L_{\mathfrak{p}}}, \arrow[ru, hook]&
	\end{tikzcd}
\end{equation}
which naturally extends to general transversal regular prisms by flat descent.
As we shall see below, the diagram induces an intersection formula of the prism $A$.
\begin{corollary}
	\label{cor:intersection_prism}
	Let $(A,I)$ be a transversal regular prism, and let $n\in \mathbb{N}\cup \{\infty\}$.
	The canonical maps in diagram (\ref{diagram:inj_of_prisms}) induce an isomorphism
	\begin{equation}
		\label{eq:inj_of_prisms_and_localizations}
		A \xrightarrow{\sim} \bigl( \prod_{\mathfrak{p}\in \mathrm{Ass}_{\overline{A}}(\overline{A}/p\overline{A})} A_{L_{\mathfrak{p}}}/p^nA_{L_{\mathfrak{p}}} \bigr) \bigcap (A/p^nA[1/I]).
	\end{equation}
\end{corollary}
\begin{proof}
	By induction and by taking the limits, it suffices to assume $n=1$.
	By taking a flat cover if necessary, we may assume $I=(d)$ is principal.
	We denote the right hand side as $N$.
	Then we notice that the injection (\ref{eq:inj_of_prisms_and_localizations}) becomes an isomorphism after inverting the ideal $I$.
	If the map (\ref{eq:inj_of_prisms_and_localizations}) is not surjective, then there is a non-zero element $x\in N$ such that $x\notin A$ but $dx\in A$.
	We let $\overline{dx}$ be the image of $dx$ in $\overline{A}=A/IA$, which is nonzero by assumption.
	Then the image of $\overline{dx}$ in $N/IN$ is 
	\[
	d\cdot (\text{image of}~x~\text{in}~N/IN),
	\]
	which in particular vanishes.
	So we conclude the proof by noticing that from \Cref{lem:localization_injection}.\ref{lem:localization_injection_inj}, we know the map
	\[
	\overline{A}/p\overline{A} \to N/IN=\prod_{\mathfrak{p}\in \mathrm{Ass}(A)} \overline{A}_{L_\mathfrak{p}}/p\overline{A}_{L_\mathfrak{p}}
	\]
	is an injection, which contradicts the existence of the element $x$.
\end{proof}

To prepare for the analysis of Frobenius modules, it is useful to equip ourselves with the explicit formula of the complete localization of the framed regular prism.
\begin{proposition}
	\label{prop:competion_of_A_at_max_ideal}
	Let $(A,I)$ be a framed regular prism, let $\mathfrak{m}$ be a maximal ideal of $A$ with its residue field $k'\colonequals k(\mathfrak{m})=A/\mathfrak{m}$.
	Let $\mathfrak{m}'$ be the maximal ideal of $A'\colonequals A\otimes_{W(k)} W(k')$ defined by its canonical surjection onto $k'$.
	There exists an isomorphism of regular prisms 
	\[
	(W(k')\llbracket x_1,\ldots,x_n \rrbracket, I') \longrightarrow ((A')^\wedge_{\mathfrak{m'}},I(A')^\wedge_{\mathfrak{m'}}),
	\] 
	with the following properties:
	\begin{itemize}
		\item for each $i\in \{1,\ldots, n\}$, the endomorphism $\varphi_{W(k')\llbracket x_1, \ldots, x_n \rrbracket}$ on $W(k')\llbracket x_1, \ldots, x_n \rrbracket$ sends $x_i$ to a monic polynomial in $x_i$ of degree $p$ such that
		\[
		x_i|\varphi_{W(k')\llbracket x_1, \ldots, x_n \rrbracket}(x_i),~\text{and}~\delta\bigl( \frac{\varphi_{W(k')\llbracket x_1, \ldots, x_n \rrbracket}(x_i)}{x_i}\bigr)~\text{is zero or a unit};
		\]
		\item the ideal $I'$ is generated by an element $d\in W(k')\llbracket x_1, \ldots, x_n \rrbracket$ whose constant term is $p$.
	\end{itemize}
\end{proposition}
\begin{proof}
	We let $(A_0, I_0)$ be the framing of $(A,I)$, with a $p$-completely \'etale map of $\delta$-rings $W(k)\langle t_1, \ldots, t_n \rangle \to A_0$, where the $\delta$-structure vanishes on $t_i$.
	As the ring $A$ is $(p,I)$-complete, the ideal $(p,I)$ is contained in the maximal ideal $\mathfrak{m}$.
	Moreover, since the mod $(p,I)$ reduction of $A$ is a finite type $k$-algebra, the residue field $k'$ is a finite extension of $k$.
	So we get a commutative diagram of Frobenius equivariant maps
	\[
	\begin{tikzcd}
		W(k) \ar[r] \ar[d] & W(k) \langle t_1, \ldots, t_n \rangle \ar[r] & A \ar[rd] & \\
		W(k') \ar[rrr] &&& k'=A/\mathfrak{m},
	\end{tikzcd}
	\]
	which naturally extends to the Frobenius equivariant maps of $W(k')$-algebras
	\begin{equation}
		\label{eq:lem:competion_of_A_at_max_ideal_1}
		W(k') \langle t_1, \ldots, t_n \rangle \longrightarrow A'=A\otimes_{W(k)} W(k') \longrightarrow k'.
	\end{equation}
	We let $a_i$ be the image of $t_i$ in $k'$.
	Then by taking the formal completion for the maps in (\ref{eq:lem:competion_of_A_at_max_ideal_1}) with respect to the surjections onto $k'$, we obtain the Frobenius equivariant maps as below
	\begin{equation}
		\label{eq:lem:competion_of_A_at_max_ideal_2}
		W(k') \llbracket x_1, \ldots, x_n \rrbracket \longrightarrow (A')^\wedge_\mathfrak{m'} \longrightarrow k',
	\end{equation}
	where each $x_i$ is the image of the element $t_i-[a_i]\in W(k')\langle t_1, \ldots, t_n \rangle$.
	Note that since the map $W(k) \langle t_1, \ldots, t_n \rangle \to A_0$ is $p$-completely \'etale, the associated map $W(k')\llbracket x_1, \ldots, x_n \rrbracket \to A'^\wedge_\mathfrak{m'}$ is formally \'etale and induces an isomorphism of cotangent spaces at the maximal ideals.
	Hence by checking the graded pieces, the map $W(k')\llbracket x_1, \ldots, x_n \rrbracket \to A'^\wedge_\mathfrak{m'}$  is a Frobenius equivariant isomorphism.
	Here we also note that the Frobenius structure sends each $x_i$ onto $(x_i+[a_i])^p-\varphi_{W(k')}([a_i])$, where the latter is a degree $p$ monic polynomial in $x_i$  whose constant term is zero (thanks to the equality $\varphi_{W(k')}([a_i])=[a_i]^p$).
	The latter in particular implies that Frobenius endomorphism on $W(k')\llbracket x_1, \ldots, x_n \rrbracket$ is a finite flat cover.
	Moreover, by expanding the polynomial and the binomial formula, the element $\frac{(x_i+[a_i])^p-\varphi_{W(k')}([a_i])}{x_i}$ is of the form $x_ih(x_i)+p[a]^{p-1}$.
	In particular, the element $\delta(\frac{(x_i+[a_i])^p-\varphi_{W(k')}([a_i])}{x_i})$ is of the form $x_ig(x_i)+(1-p^{p-1})[a_i]^{p(p-1)}$, which is a unit if $a_i\neq 0$.
	If $a_i=0$, then we have $x_i=t_i$, and thus the $\delta$-stricture of $\frac{\varphi(x_i)}{x_i}=x_i^{p-1}$ is zero.
	
	The above produces a Frobenius equivariant isomorphism between two regular rings $W(k') \llbracket x_1, \ldots, x_n \rrbracket \longrightarrow (A')^\wedge_\mathfrak{m'}$, and we then claim that this map naturally underlies a morphism of prisms.
	First of all, since both rings are $p$-torsionfree, their Frobenius endomorphisms uniquely determine $\delta$-structures on them respectively, which are compatible under the map.
	Moreover, since $(A,I)$ is a prism, we know the element $p$ is contained in the ideal $I+\varphi(I)$, and hence in $I(A')^\wedge_\mathfrak{m'} + \varphi(I) (A')^\wedge_\mathfrak{m'}$.
	Thus the $\delta$-pair $((A')^\wedge_\mathfrak{m'}, I(A')^\wedge_\mathfrak{m'})$ is a prism as well (\cite[Def.\ 3.2.(1)]{BS22}).

	We let $I'$ be the unique invertible ideal of $W(k') \llbracket x_1, \ldots, x_n \rrbracket \longrightarrow (A')^\wedge_\mathfrak{m'}$ that is induced by $I(A')^\wedge_\mathfrak{m'}$ via the isomorphism $W(k') \llbracket x_1, \ldots, x_n \rrbracket \longrightarrow (A')^\wedge_\mathfrak{m'}$.
	Then the $\delta$-equivariant isomorphism makes $(W(k') \llbracket x_1, \ldots, x_n \rrbracket,I')$ a prism.
	Here by the orientability of the prism $(A,I)$ (\Cref{thm:regular prism}), the ideal $I'$ is principal and can be generated by an element $d$. 
	We write the element $d$ as $up^n+g$, with $u\in W(k')^\times$, $n\geq 0,$ and $g\in (x_1,\ldots,x_n)W(k')\llbracket x_1, \ldots, x_n \rrbracket$.
	If $n=0$, then by construction the element $d$ is a unit in $W(k')\llbracket x_1, \ldots, x_n \rrbracket$, which is impossible since the ring $W(k')\llbracket x_1, \ldots, x_n \rrbracket$ is also $d$-adically complete.
	On the other hand, by taking the surjection along the equivariant map $W(k')\llbracket x_1, \ldots, x_n \rrbracket \xrightarrow{x_i\mapsto 0} W(k')$, the containment $p\in (I',\varphi(I'))$ implies that $p\in (up^n, \varphi(u)p^n)\subset W(k')$, which is possible only when $n=1$.
	Hence by multiplying with the unit $u^{-1}$, we see the ideal $I'$ can be generated by a power series in $x_i$ whose constant term is the element $p$.
\end{proof}


\subsection{Coproduct of prisms and relative prismatic cohomology}
\label{sub:coproduct of prisms}
In this subsection, we show that the coproduct exists for a large collection of prisms over a regular $p$-adic formal scheme.
In addition, we consider the special case when one of the factors is perfect, where we prove that the relative prismatic cohomology often coincides with the coproduct.

For the convenience of our discussion, we sort out a subcategory of bounded prisms that satisfy some flatness conditions.
\begin{definition}[Flatness assumption]
	\label{ass:flat}
	Let $X$ be a $p$-adic formal scheme, and let $(B,J)$ be a bounded prism in $X_\Prism$.
	\begin{itemize}
		\item We say $(B,J)$ is a \emph{(faithfully) flat prism} if $\overline{B}$ is $p$-completely (faithfully) flat over $X$.
		\item We say $(B,J)$ is a \emph{flat crystalline prism} if $J=pB$ and $\overline{B}=B/pB$ is flat over $X_{p=0}$.
	\end{itemize}
\end{definition}
Note that a flat prism over $X$ is in particular transversal if $X$ is $p$-torsionfree.

\begin{theorem}[Existence of coproducts]
	\label{thm:coproduct_in_general}
	Let $X$ be a regular $p$-adic formal scheme, and let $(B_i,J_i)$ for $i=1,2$ be two bounded prisms such that each of them admit a map from a flat prism.
	\begin{enumerate}[label=\upshape{(\roman*)}]
		\item\label{thm:coproduct of prism general} The coproduct  $(B_1,J_1)\coprod (B_2,J_2)$ exists in $X_\Prism$.
		\item\label{thm:coproduct of prism faithful} If $(B_1,J_1)$ is a (faithfully) flat prism in $X_\Prism$, then the coproduct is completely (faithfully) flat over $(B_2,J_2)$.
	\end{enumerate}
	Assume $(B_2,J_2)$ is a framed regular prism $(A,I)$ for a framing $\Sigma=\{t_1,\ldots,t_n\}$, and $(B_1,J_1)$ admits a map of prisms $f:(A,I)\to (B_1,J_1)$.
	\begin{enumerate}[resume, label=\upshape{(\roman*)}]
		\item\label{thm:coproduct of prism regular and regular}  		
		The underlying $\delta$-ring of $(B_1,J_1) \coprod (A,I)$ is isomorphic to the complete $\delta$-envelope 
		\[
		B_1\{\frac{f(t_1)\otimes 1 - 1\otimes t_1}{J_1}, \ldots, \frac{f(t_n)\otimes 1 - 1\otimes t_n}{J_1}\},
		\] and its mod $I$ reduction is isomorphic to the complete divided power polynomial $\overline{B}_1\{u_1,\ldots,u_n\}^{\mathrm{pd}}$, where $u_i$ is the image of $ \frac{f(t_n)\otimes 1 - 1\otimes t_n}{d}$ and $d$ is a generator of $I$.
		In particular, the coproduct is completely free over $(B_1,J_1)$ (cf. \Cref{def:completely_projective}).
	\end{enumerate}
\end{theorem}
\begin{proof}
	By the categorical description of the coproduct and the fiber product, to prove \ref{thm:coproduct of prism general} and \ref{thm:coproduct of prism faithful}, it suffices to assume that both prisms are flat prisms.
	In addition, by the uniqueness of the coproduct (if exists), it suffices to prove its existence up to Zariski localization.
	So we assume $X=\Spf(R)$ is affine, and the maps to the reductions $R\to \overline{B}_i$ is ($p$-completely) faithfully flat.
	In addition, by the weak initiality of the framed regular prism and by further shrinking $X$ Zariski locally if necessary, we may assume that 
	\begin{itemize}
		\item there is a framed regular prism $(A,I)$ with $\overline{A}$ finite \'etale over $X$, and
		\item there are completely faithfully flat covers $(B_i,J_i)\to (C_i,L_i)$ for $i=1,2$, such that $(C_i,L_i)$ admit maps from $(A,I)$.
	\end{itemize}
Here we note that by assumption, since $\overline{A}$ is finite \'etale over $X$, and since the composition $R\to \overline{B}_i\to \overline{C}_i$ is $p$-completely faithfully flat and factors through the finite flat map $R\to \overline{A}$, the map $\overline{A}\to \overline{C}_i$ is also $p$-completely faithfully flat for each $i$.

To continue, we consider the following commutative diagram:
\begin{equation}
	\label{eq:diagram_of_self_tensor_product_of_prisms}
	\begin{tikzcd}
		(B_1\otimes_W B_2)^\wedge_{(p, 1\otimes J_2)} \ar[r] \arrow[d,"\alpha"]& (C_1,\otimes_W C_2)^\wedge_{(p,1\otimes L_2)} \arrow[d,"\beta"] & (A\otimes_W A)^\wedge_{(p,1\otimes I)} \arrow[d,"\gamma"] \ar[l]\\
		(\overline{B}_1\otimes_R \overline{B}_2)^\wedge_p \ar[r] & (C_1\otimes_R C_2)^\wedge_p & \overline{A} \ar[l],
	\end{tikzcd}
\end{equation}
where the horizontal arrows are completely faithfully flat, each square is a complete (derived) pushout diagram, and the arrows in the top row are compatible with the $\delta$-structures.
Moreover, we notice that the kernel of the map $\gamma$ is the ideal $(I\otimes 1, t_1\otimes 1 - 1\otimes t_1,\ldots, t_n\otimes 1 - 1\otimes t_n)$ and is generated by a $(p,1\otimes I)$-regular sequence in $(A\otimes_W A)^\wedge_{(p,1\otimes I)}$.
Thus by the complete faithful flatness of the top row in (\ref{eq:diagram_of_self_tensor_product_of_prisms}), the ideal $\ker(\alpha)$ is also generated by a $(p,1\otimes J_2)$-regular sequence as well.
Hence by \cite[Prop.\ 3.13]{BS22}, the prismatic envelope $\bigl((B_1\otimes_W B_2)^\wedge_{(p,1\otimes J_2)} \{\frac{\ker(\alpha)}{1\otimes J_2}\}, (1\otimes J_2) \bigr)$ for the surjection $\alpha$ exists, which by its categorical description is the coproduct of $(B_i,J_i)$ over $X_\Prism$.
In addition, under the complete faithful flatness assumption of $(B_i,J_i)$ and by \textit{loc. cit.}, we know the coproduct is completely faithfully flat over each of $(B_i,J_i)$.

	Finally we prove \ref{thm:coproduct of prism regular and regular}, for which we may assume $(B_1,J_1)=(A,I)$ as well.
	We let $J_0$ be the kernel ideal for the diagonal surjection $(A_0\otimes_W A_0)^\wedge_p \to R$.
	So the image of $J_0$ in $(A\otimes_W A)^\wedge_{(p,I\otimes 1,1\otimes I)}$ generates the kernel ideal $J$ for the surjection $(A\otimes_W A)^\wedge_{(p,I\otimes 1,1\otimes I)} \to R$.
	To see the explicit formula of the coproduct in \ref{thm:coproduct of prism regular and regular}, we consider an auxiliary construction. 
	As in the previous paragraphs, since the $\delta$-ring $A_0$ admits a map from the $\delta$-ring $W\langle t_1,\ldots,t_n\rangle$, there is a natural map of $\delta$-pairs
	\begin{equation}
		\label{eq:thm:coproduct 3}
		\left( \bigl( A_0\otimes_W W\langle t_1, \ldots, t_n\rangle \bigr)^\wedge_p, (t_i\otimes 1-1\otimes t_i) \right) \longrightarrow \bigl( (A_0\otimes_W A_0)^\wedge_p, J_0 \bigr).
	\end{equation}
	So by taking the induced map of the $(p,I_0\otimes 1)$-complete prismatic envelopes as in the proof of \ref{thm:coproduct of prism general}, we get a map of prisms
	\begin{equation}
		\label{eq:thm:coproduct 4}
		\alpha\colon \left( A\{\frac{t_i\otimes 1-1\otimes t_i}{I_0\otimes 1}\} ,(I_0\otimes 1) \right) \longrightarrow \left( \bigl( (A_0\otimes_W A_0)^\wedge_p\{\frac{J_0}{I_0\otimes 1}\} \bigr)^\wedge_{(p,I_0\otimes 1)}, (I_0\otimes1) \right).
	\end{equation}

	To proceed, we let $v_i= t_i\otimes 1 -1\otimes t_i$.
	Then we notice that the images of $v_1,\ldots,v_n$ in $W\langle t_1, \ldots, t_n\rangle\otimes_W W\langle t_1, \ldots, t_n\rangle$ form a basis of the relative differential $\Omega^1_{W\langle t_1, \ldots, t_n \rangle/W}$.
	In addition, since $W\langle t_1, \ldots, t_n\rangle\to A_0$ is $p$-completely \'etale, the natural map of coherent $R$-modules $A_0\otimes_{W\langle t_1, \ldots, t_n\rangle} (\Omega^1_{W\langle t_1, \ldots, t_n\rangle/W})^\wedge_p\to (\Omega^1_{A_0/W})^\wedge_p$ is an isomorphism.
	As a consequence, by checking it via the mod $(I_0\otimes 1)$ reduction, we see the map $\overline{\alpha}$ is an isomorphism of $p$-complete rings, and so is the map $\alpha$.
	
	Now consider the explicit presentation of the mod $I=(d)$ reduction.
	By unwinding the construction of the prismatic envelope in \cite[Prop.\ 3.13]{BS22}, the coproduct $A^1$ is obtained from the free $\delta$-ring over $(A_0\otimes_W W\langle t_1,\ldots,t_n \rangle)^\wedge_p$ on symbols $u_i$ by imposing the $\delta$-relations
		\[
		du_i=v_i.
		\]
		The higher $\delta$-relations are therefore generated by the identities
		\[
		\delta^m(du_i)=\delta^m(v_i)\qquad (m\geq 0).
		\]
		
			We claim that $\delta^m(v_i)\in (v_i)$ for all $m\geq 0$.
			Indeed, since the framing map $W(k)\langle t_1,\ldots,t_n\rangle\to A_0$ is a map of $\delta$-rings and we have $\delta(t_i)=0$, the $\delta$-structure vanishes on both $t_i\otimes 1$ and $1\otimes t_i$.
			Therefore
			\[
			p\delta(v_i)=\varphi(v_i)-v_i^p=(t_i^p\otimes 1-1\otimes t_i^p)-(t_i\otimes 1-1\otimes t_i)^p.
			\]
			On the other hand, the polynomial
			$\frac{X^p-Y^p-(X-Y)^p}{p}\in \mathbb{Z}[X,Y]$
			vanishes after setting $X=Y$, hence is divisible by $X-Y$.
			So by evaluating at $X=t_i\otimes 1$ and $Y=1\otimes t_i$, we obtain $\delta(v_i)\in (v_i)$.
			Now given an element $r=av_i\in (v_i)$, the identity 
			$\delta(av_i)=a^p\delta(v_i)+v_i^p\delta(a)+p\delta(a)\delta(v_i)$
			shows that $\delta(r)\in (v_i)$ as well.
			Hence the ideal $(v_i)$ is preserved by the $\delta$-structure, and the claim follows by induction.
			Since $v_i=du_i$, we conclude that
		\[
		\delta^m(du_i)\equiv 0 \pmod{d}
		\qquad \text{for all }m\geq 0.
		\]
		In addition, notice that $\delta(d)\in A^\times$.
			At this point the argument of \cite[Lem.\ 2.2.5, Prop.\ 2.2.8(1)]{DL26} applies verbatim: its proof only uses the relations $du_i=v_i$, the congruences $\delta^m(du_i)\equiv 0 \pmod{d}$, and the invertibility of $\delta(d)$.
			Concretely, modulo $d$ one gets a recursive system of congruences of the form
			\[
			p\mu_{i,0}\,\delta(u_i)\equiv u_i^p \pmod{d},
			\]
			and for $m\geq 1$, there exists $\mu_{i,m}\in \overline{A}^\times$ such that 
			\[
			p\mu_{i,m}\,\delta^{m+1}(u_i)\equiv (\delta^m(u_i))^p+\text{(lower divided-power terms)} \pmod{d}.
			\]
			In particular the subring generated by the higher $\delta^m(u_i)$ is the same as the one generated by the $p$-adic divided powers of the $u_i$.
			Hence the quotient modulo $d$ is exactly the ordinary divided power polynomial algebra on the classes of the $u_i$.
\end{proof}

Here we note that we can in particular obtain an explicit presentation of the \v{C}ech nerve for a given framed regular prism.
\begin{corollary}[\v{C}ech nerve of a regular prism]
	\label{cor:Cech_nerve_of_framed}
	Let $X$ be a regular $p$-adic formal scheme of relative dimension $d-1$ over $\mathcal{O}_K$, and let $(A,I)$ be a framed regular prism.
	Then the \v{C}ech nerve of the $(A,I)$ is the cosimplicial object of the prisms
	\[
	\Delta \ni [n] \longmapsto (A^n,IA^n) \colonequals \bigl( A\{\frac{\delta_{ij}}{I}\}_{0\leq i\leq n,1\leq j\leq d}, (I) \bigr).
	\]
	Here $\delta_{ij}$ is the image of the element $t_j\otimes 1 \otimes \cdots \otimes 1 - 1\otimes \cdots \otimes t_j \otimes \cdots\otimes 1$, where $t_j$ in the second term is at the $i$-th tensor factor.
\end{corollary}

For the later applications, we consider the coproduct of prisms such that one of the factor is perfect.
Our first observation is a reinterpretation of the relative prismatic cohomology using coproducts.
\begin{proposition}[Coproduct = Relative prismatic cohomology]
	\label{prop:equivalence_of_cat_for_relative_prismatic_site}
	Let $X$ be a bounded $p$-adic formal scheme, let $S$ be a $p$-torsionfree perfectoid algebra over $X$, and let $(B,J)$ be a bounded prism in $X_\Prism$.
	We let $S_{\overline{B}}$ be the $p$-completion of $S\otimes_{\mathcal{O}_X} \overline{B}$.
	The following two categories are naturally equivalent to each other
	\begin{itemize}
		\item $\{ ((C,L),~f,~g)~|~(C,L)\in X_\Prism, ~f\in \Map_{X_\Prism}\bigl((\rAinf(S),\ker(\tilde{\theta})), (C,L) \bigr),~g\in \Map_{X_\Prism}\bigl((B,J),(C,L)\bigr)\}$;
		\item $\bigl(S_{\overline{B}}/(B,J) \bigr)_\Prism$.
	\end{itemize}
	In particular, the coproduct $(\rAinf(S),\ker(\tilde{\theta}))\coprod (B,J)$ exists in $X_\Prism$ if and only if the prismatic site $\bigl(S_{\overline{B}}/(B,J) \bigr)_\Prism$ admits an initial prism.
\end{proposition}
\begin{proof}
	We let $(C,L)$ be any prism in $\bigl(S_{\overline{B}}/(B,J)\bigr)_\Prism$, which by definition is equivalent to a commutative diagram of rings 
	\begin{equation}
		\label{eq:claim:initial_prism_diagram}
		\begin{tikzcd}
			C \ar[r] & \overline{C} & S_{\overline{B}}= (S \otimes_{\mathcal{O}_X} \overline{B})^\wedge_p \ar[l] & S \ar[l]\\
			B \ar[u] \ar[r] & \overline{B} \ar[u] \ar[ur] && \mathcal{O}_X \ar[ll] \ar[u] \arrow[lu, phantom, "\lrcorner"],
		\end{tikzcd}
	\end{equation}
	such that the left vertical map is compatible with the $\delta$-structures and the right square is the prescribed pushout diagram.
	Then since $S$ is a perfectoid ring, by the initiality of the perfect prism $(\rAinf(S),\ker(\widetilde{\theta}))$ in $S_\Prism$ (\cite[Thm.\ 3.10, Prop.\ 7.2]{BS22}), 
	there is an unique map of prisms $(\rAinf(S),\ker(\widetilde{\theta})) \to (C,L)$ that enlarges the diagram (\ref{eq:claim:initial_prism_diagram}) to the following 
	\begin{equation}
		\label{eq:claim:initial_prism_diagram_large}
		\begin{tikzcd}
			\rAinf(S) \ar[d] \ar[r]& \overline{A_{\inf}(S)}=S \ar[d] &&\\
			C \ar[r] & \overline{C} & S_{\overline{B}}= (S \otimes_{\mathcal{O}_X} \overline{B})^\wedge_p \ar[l] & S \ar[l] \arrow[llu, equal]\\
			B \ar[u] \ar[r] & \overline{B} \ar[u] \ar[ur] && \mathcal{O}_X \ar[ll] \ar[u] \arrow[lu, phantom, "\lrcorner"].
		\end{tikzcd}
	\end{equation}
	In this way, in the category $X_\Prism$, we obtain the maps of prisms from $(\rAinf(S),\ker(\widetilde{\theta}))$ and $(B,J)$ to $(C,L)$ respectively, which are determined uniquely by the diagram in (\ref{eq:claim:initial_prism_diagram}). 
	Conversely, given maps of prisms from $(\rAinf(S),\ker(\widetilde{\theta}))$ and $(B,J)$ to $(C,L)$ respectively in $X_\Prism$, we obtain equivalently the following commutative diagram
	\begin{equation}
		\label{eq:claim:initial_prism_diagram_converse}
		\begin{tikzcd}
			\rAinf(S) \ar[d] \ar[r]& \overline{A_{\inf}(S)}=S \ar[d] &&\\
			C \ar[r] & \overline{C} &  & S \ar[ll] \arrow[llu, equal]\\
			B \ar[u] \ar[r] & \overline{B} \ar[u] && \mathcal{O}_X \ar[ll] \ar[u],
		\end{tikzcd}
	\end{equation}
	where the left vertical maps are compatible with the $\delta$-structures.
	The diagram in (\ref{eq:claim:initial_prism_diagram_converse}) uniquely extends to a diagram as in (\ref{eq:claim:initial_prism_diagram_large}), hence produce a prism $(C,L)\in (S_{\overline{B}}/(B,J))_\Prism$.
\end{proof}

Next, we recall the Faltings extension, its splitting, and its relation to the relative prismatic cohomology.
\begin{proposition}
	\label{prop:relative_prismatic_coh_of_perfectoid}
	Let $X$ be a regular $p$-adic formal scheme over $\mathcal{O}_K$ and let $S$ be a perfectoid algebra over $X$.
	\begin{enumerate}[label=\upshape{(\roman*)}]
		\item\label{prop:relative_prismatic_coh_of_perfectoid_cotangent_complex}
		The $p$-complete cotangent complex $\mathbb{L}_{S/X}$ fits into a canonical fiber sequence
		\begin{equation}
			\label{eq:Faltings_extension}
			S\{1\}[1] \longrightarrow \mathbb{L}_{S/X} \longrightarrow \bigl( \mathbb{L}_{X/W}\otimes_{\mathcal{O}_X} S \bigr)^\wedge_p[1].
		\end{equation}
	\end{enumerate}
	Assume there is a map of prisms $f:(A,I)\to (\rAinf(S), \ker(\widetilde{\theta}))$ in $X_\Prism$, where $(A,I)$ is a framed regular prism with a framing $(A_0,I_0)$.
	\begin{enumerate}[label=\upshape{(\roman*)}]
		\setcounter{enumi}{1}
		\item\label{prop:relative_prismatic_coh_of_perfectoid_cotangent_complex_2}
		The map $f$ induces an isomorphism between the shifted $p$-complete cotangent complex $\mathbb{L}_{S/X}[-1]$ and the tensor product $\mathbb{L}_{A_0/W}\otimes_{A_0} S[-1]$. 
		In particular, $\mathbb{L}_{S/X}[-1]$ is a finite projective $S$-module.
		\item\label{prop:relative_prismatic_coh_of_perfectoid_initial_prism}
		Let $(B,J)\in X_\Prism$ be a flat prism or a flat crystalline prism.
		The derived relative prismatic cohomology $\Prism_{S_{\overline{B}}/B}$ coincides with the site theoretic relative prismatic cohomology of $S_{\overline{B}}$ over $(B,J)$ and lives in cohomologically degree zero, with its mod $I$ reduction being of Tor amplitude $[0,0]$ over $S_{\overline{B}}$.
	\end{enumerate}
\end{proposition}
\begin{proof}
	Part \ref{prop:relative_prismatic_coh_of_perfectoid_cotangent_complex} follows from the canonical fiber sequence of the cotangent complexes with respect to the maps $W\to \mathcal{O}_X\to S$, together with the natural isomorphisms
	\[
	\mathbb{L}_{S/W} \simeq \mathbb{L}_{S/\rAinf(S)} \simeq S\{1\}[1].
	\]	
	
	For \ref{prop:relative_prismatic_coh_of_perfectoid_cotangent_complex_2}, we first recall that $\Spf(A_0)$ is a smooth $p$-adic formal scheme over $W$.
	The maps of rings $W\to A_0\to \rAinf(S)$ then produces a fiber sequence of $p$-complete cotangent complexes
	\[
	\mathbb{L}_{\rAinf(S)/W} \longrightarrow \mathbb{L}_{\rAinf(S)/A_0} \longrightarrow \bigl( \mathbb{L}_{A_0/W}\otimes_{A_0}\rAinf(S) \bigr)^\wedge_p[1].
	\]
	Notice that since $\rAinf(S)$ is $p$-completely relative perfect over $W$, the first term in the sequence vanishes.
	On the other hand, by the $p$-complete smoothness of $A_0$ over $W$, we know the $\mathbb{L}_{A_0/W}$ is a finite projective $A_0$-module.
	Thus by taking the derived base change of the above sequence along the surjection $A_0\to \overline{A}_0=\overline{A}$ and by the base change formula of the cotangent complex, we get
	\[
	\mathbb{L}_{S/\overline{A}}\simeq \mathbb{L}_{\rAinf(S)/A_0} \otimes^L_{A_0} \overline{A}_0 \simeq \mathbb{L}_{A_0/W} \otimes_{A_0} S[1].
	\]
	Furthermore, we notice that the assumption on the map $f$ implies that the structure map $\Spf(S)\to X$ factors through $\Spf(\overline{A}_0)$, where the latter is \'etale over $X$.
	So we have $\mathbb{L}_{S/\overline{A}}\simeq \mathbb{L}_{S/X}$, which concludes the proof of \ref{prop:relative_prismatic_coh_of_perfectoid_cotangent_complex_2}.
	
	For \ref{prop:relative_prismatic_coh_of_perfectoid_initial_prism}, we notice that the flatness assumption implies that $S_{\overline{B}}\simeq (S\otimes^L_{\mathcal{O}_X} \overline{B})^\wedge_p$.
	Moreover, there are natural isomorphisms of $p$-complete cotangent complexes
	\[
	\mathbb{L}_{S_{\overline{B}}/\overline{B}} \simeq \bigl( \mathbb{L}_{S/X}\otimes^L_{\mathcal{O}_X} \overline{B} \bigr)^\wedge_p\simeq \bigl( \mathbb{L}_{S/X} \otimes^L_S S_{\overline{B}} \bigr)^\wedge_p,
	\]
	where the latter by \ref{prop:relative_prismatic_coh_of_perfectoid_cotangent_complex_2} is isomorphic to the shifted finite projective $S_{\overline{B}}$-module $\mathbb{L}_{S/X}\otimes_S S_{\overline{B}}$.
	So the statement follows from \cite[Var.\ 4.1.19, Thm. 4.3.6]{BL22a} (see also \cite[Rmk.\ 4.3.9]{BL22a} for a related discussion), where the condition $(\ast^+)$ of loc.\ cit.\ is satisfied thanks to the aforementioned finite projectivity.
\end{proof}

Summarizing the properties of the coproduct where one of the factors being perfect, we obtain the following consequence.
\begin{corollary}[Coproduct with a perfect prism]
	\label{cor:coproduct:perfect_with_framed}
	Let $X$ be a regular $p$-adic formal scheme over $\mathcal{O}_K$, let $S$ be a perfectoid ring over $X$, and let $(B,J)$ be a flat or a flat crystalline prism.
	\begin{enumerate}[label=\upshape{(\roman*)}]
		\item\label{cor:coproduct:perfect_with_framed_general} 
		There is a natural isomorphism of prisms
		\[
		(\rAinf(S), \ker(\widetilde{\theta}) )\coprod (B,J) \simeq \bigl(\Prism_{S_{\overline{B}}/B}, I\Prism_{S_{\overline{B}}/B}\bigr),
		\]
		which lives in cohomological degree zero with its mod $I$ reduction completely (faithfully) flat over $S_{\overline{B}}$. In addition, it is completely (faithfully) flat over $\overline{B}$ if $S$ is $p$-completely (faithfully) flat over $X$.
		\item\label{cor:coproduct:perfect_with_framed_special} 	Assume there is a map of prisms $f:(A,I)\to (\rAinf(S), \ker(\widetilde{\theta}))$, where $(A,I)$ is framed regular.
		Then $f$ induces an isomorphism of prisms
		\[
		(\rAinf(S), \ker(\widetilde{\theta}) )\coprod (A,I) \simeq \bigl( \rAinf(S_{\overline{A}})\{\frac{f(t_i)\otimes 1-1\otimes t_i}{I}\} , (I) \bigr),
		\]
	\end{enumerate}
\end{corollary}
\begin{proof}
	Part \ref{cor:coproduct:perfect_with_framed_general} follows from the discreteness of the cotangent complex in \Cref{prop:relative_prismatic_coh_of_perfectoid}.\ref{prop:relative_prismatic_coh_of_perfectoid_initial_prism} and the categorical description of the relative prismatic cohomology in \Cref{prop:equivalence_of_cat_for_relative_prismatic_site}, together with the flatness in \Cref{thm:coproduct_in_general}.\ref{thm:coproduct of prism faithful}.
	Part \ref{cor:coproduct:perfect_with_framed_special} follows from the explicit formula in \Cref{thm:coproduct_in_general}.\ref{thm:coproduct of prism regular and regular}.
\end{proof}

Finally, to assist the calculation of the global section for the rational prismatic period sheaf, we record the following formulae on intersections and Frobenius structures.
\begin{lemma}
	\label{lem:intersection_of_prismatic_coh_with_base}
	Let $X$ be a regular $p$-adic formal scheme, and let $(B,J)\in X_\Prism$ be a flat prism.
	Assume $S$ is a perfectoid ring that is $p$-completely faithfully flat over $X$.
	Then the canonical map below is an equality
	\begin{equation}
		\label{eq:lem:intersection_of_prismatic_coh_with_base}
			B \longrightarrow (B[1/p])^\wedge_J \cap \Prism_{S_{\overline{B}}/B}.
	\end{equation}
\end{lemma}
\begin{proof}
	By assumption and \Cref{cor:coproduct:perfect_with_framed}.\ref{cor:coproduct:perfect_with_framed_general}, the sequence $(J,p)$ is regular in both $B$ and $\Prism_{S_{\overline{B}}/B}$, and the both rings are $p$-torsionfree.
	So by induction and by taking the limit, the claim follows if the mod $J$ reduction
		\begin{equation}
		\label{eq:lem:intersection_of_prismatic_coh_with_base_reduction}
				\overline{B} \longrightarrow \overline{B}[1/p] \cap \overline{\Prism}_{S_{\overline{B}}/B}
	\end{equation}
	is an equality.
	We claim that the cokernel $C\colonequals\coker(\overline{B}\to \overline{\Prism}_{S_{\overline{B}}/B})$ is $p$-torsionfree.
	Granting the claim, if the map in (\ref{eq:lem:intersection_of_prismatic_coh_with_base_reduction}) is not surjective, there exists an element $x\in \overline{B}\backslash p \overline{B}$ such that $\frac{x}{p}\in \overline{\Prism}_{S_{\overline{B}}/B}$.
	The image of $\frac{x}{p}$ in $C$ is then a non-zero $p$-torsion element, leading to a contradiction.
	
	To see the claim, we first notice that since $\overline{B}$ and $\overline{\Prism}_{S_{\overline{B}}/B}$ are derived $p$-complete, so is the cokernel $C$.
	We consider the complete base change of the fiber sequence
	\[
	\overline{\Prism}_{S_{\overline{B}}/B} \longrightarrow (\overline{\Prism}_{S_{\overline{B}}/B}\otimes_{\overline{B}} \overline{\Prism}_{S_{\overline{B}}/B})^\wedge_{p} \longrightarrow (C\otimes_{\overline{B}} \overline{\Prism}_{S_{\overline{B}}/B})^\wedge_p.
	\]
	By \Cref{cor:coproduct:perfect_with_framed}.\ref{cor:coproduct:perfect_with_framed_general}, we know map $\overline{B}\to \overline{\Prism}_{S_{\overline{B}}/B}$ is $p$-completely faithfully flat.
	So to show the $p$-torsionfreeness of $C$, it is equivalent to showing that the complete base change $(C\otimes_{\overline{B}} \overline{\Prism}_{S_{\overline{B}}/B})^\wedge_{p}$ is $p$-torsionfree.
	Note that the first arrow in the sequence above (namely $\overline{\Prism}_{S_{\overline{B}}/B} \to (\overline{\Prism}_{S_{\overline{B}}/B}\otimes_{\overline{B}}\overline{\Prism}_{S_{\overline{B}}/B})^\wedge_{p}$) is split injective, since the arrow admits a canonical section by the multiplication.
	Thus $(C\otimes_{\overline{B}} \overline{\Prism}_{S_{\overline{B}}/B})^\wedge_{p}$ is a direct summand of $(\overline{\Prism}_{S_{\overline{B}}/B}\otimes_{\overline{B}}\overline{\Prism}_{S_{\overline{B}}/B})^\wedge_{p}$.
	Hence the $p$-torsionfree of $(C\otimes_{\overline{B}} \overline{\Prism}_{S_{\overline{B}}/B})^\wedge_{p}$ follows from that of $(\overline{\Prism}_{S_{\overline{B}}/B}\otimes_{\overline{B}}\overline{\Prism}_{S_{\overline{B}}/B})^\wedge_{p}$, where the latter follows from the $p$-torsionfreeness of $\overline{\Prism}_{S_{\overline{B}}/B}$ together with the complete flatness of the morphism $\overline{B}\to \overline{\Prism}_{S_{\overline{B}}/B}$.
\end{proof}

\begin{proposition}
\label{prop:intersection_of_Frobenius_of_base_with_coproduct}
Let $X$ be a regular $p$-adic formal scheme, let $S$ be a $p$-completely flat perfectoid algebra over $X$, and let $(B,J)\in X_\Prism$.
Suppose $(B,J)$ is a coproduct of finitely many framed regular prisms.
\begin{enumerate}[label=\upshape{(\roman*)}]
\item\label{prop:intersection_of_Frobenius_of_base_with_coproduct:base} The mod $J$-reduction of the Frobenius structure $\varphi_B$ is injective.
\item\label{prop:intersection_of_Frobenius_of_base_with_coproduct:cohomology} For each $a\geq 1$, the mod $\varphi_{\Prism_S}^{-a-1}(I_{\Prism_S})$ reductions of the absolute Frobenius maps $\Prism_{S_{\overline{B}}/B}\to \Prism_{S_{\overline{B}}/B}$ and $\Prism^{(1)}_{S_{\overline{B}}/B}\to \Prism^{(1)}_{S_{\overline{B}}/B}$ are injective.
\item\label{prop:intersection_of_Frobenius_of_base_with_coproduct:intersection} The following complete pushout diagram is also a pullback diagram of injections
\[
\begin{tikzcd}
B \ar[r] \arrow[d, "\varphi_B"'] & \Prism_{S_{\overline{B}}/B} \ar[d]\\
B \ar[r] & \Prism^{(1)}_{S_{\overline{B}}/B}.
\end{tikzcd}
\]
\end{enumerate}
\end{proposition}
\begin{proof}
We first assume that $(B,J)=(A,I)$ is a framed regular prism.
In this case, \ref{prop:intersection_of_Frobenius_of_base_with_coproduct:base} follows from the fact that the Frobenius structure $\varphi_A$ is finite flat (\Cref{prop:regular prism Frobenius}).
For \ref{prop:intersection_of_Frobenius_of_base_with_coproduct:cohomology}, we let $A_0=W \langle t_1,\ldots,t_n \rangle$ be a framing, and let $d_S$ be a generator of $I_{\Prism_S}$ in $\Prism_S$.
	We first show that both the source and the target are isomorphic to the complete divided power polynomial rings over $\Prism_S/\varphi^{-a-1}(d_S)$ and over $\Prism_S/\varphi^{-a}(d_S)$ respectively.
	Recall from \Cref{thm:coproduct_in_general} and its proof that the relative prismatic cohomology $\Prism_{S_{\overline{A}}/A}$ is isomorphic to the prismatic envelope $\Prism_S \{\frac{t_1\otimes 1 - 1\otimes t_i}{d_S}\}$, which can represented as the $\delta$-quotient
	\[\Prism_S \{u_1,\ldots,u_n\} /(u_id-(t_i\otimes 1 - 1\otimes t_i))_\delta.\]
	Now by repeating the arguments for \Cref{thm:coproduct_in_general}.\ref{thm:coproduct of prism regular and regular}, we have
		\begin{align*}
		\Prism_{S_{\overline{A}}/A}/\varphi^{-a-1}(d_S) & \simeq \Prism_S/\varphi^{-a-1}(d_S) \{ u_1,\ldots,u_n\}^{\pd},\\
		\Prism_{S_{\overline{A}}/A} /\varphi^{-a}(d_S) & \simeq  \Prism_S/\varphi^{-a}(d_S) \{ u_1,\ldots,u_n\}^{\pd},
		\end{align*}
	and the map in the statement extends the isomorphism $\Prism_S/\varphi^{-a-1}(d_S) \xrightarrow{\varphi_{\Prism_S}} \Prism_S/\varphi^{-a}(d_S)$ and sends $u_i$ onto $\varphi(u_i)$.
	Notice that in the ring $\Prism_{S_{\overline{A}}/A}$, by applying the Frobenius at the relation $1\otimes t_i=t_i\otimes 1 -u_id_S$ and using the assumption that the Frobenius sends both $1\otimes t_i$ and $t_i\otimes 1$ to their $p$-th powers, we have
		\[
		\varphi(d_S)\varphi(u_i) = -\sum_{j=1}^p \binom{p}{j} (-u_id_S)^j (t_i\otimes 1)^{p-j}.
		\]
	Moreover, the images of both $d_S$ and $\varphi(d_S)$ in $\Prism_S/\varphi^{-a}(d_S)$ are equal to products of $p$ by units (since $a\geq 1$).
	So the image of $\varphi(u_i)$ in the mod $\varphi^{-a}(d_S)$ reduction $ \Prism_S/\varphi^{-a}(d_S) \{ u_1,\ldots,u_n\}^{\pd}$ is a degree $p$ polynomial of $u_i$ over $\Prism_S/\varphi^{-a}(d_S)$ and is in particular non-zero.
	In particular, the induced map $\Prism_{S_{\overline{A}}/A}/\varphi^{-a-1}(d_S) \to \Prism_{S_{\overline{A}}/A}/\varphi^{-a}(d_S)$ identifies the source as the complete divided power polynomial subring generated by the variables $\varphi(u_i)$ over $\Prism_{S}/\varphi^{-a}(d_S)$.
	As a consequence, by considering the mod $\varphi^{-a}(d_S)$ reduction of the image, the absolute Frobenius map $\Prism_{S_{\overline{A}}/A} \to \Prism_{S_{\overline{A}}/A}$ is non-zero on each free divided power variable $u_i$ and its divided powers, and there are no algebraic relations for different $i$, which implies the injectivity of the map.
	The same extends to the Frobenius twisted version by raising $1\otimes t_i$ to its $p$-th power.

For \ref{prop:intersection_of_Frobenius_of_base_with_coproduct:intersection}, we notice that the diagram for $(B,J)=(A,I)$ naturally admits a compatible injection into the diagram but for $(B,J)=(A_\perf,IA_\perf)$, where the claim for the latter is automatic since $\varphi_{A_\perf}$ is an isomorphism.
So we reduce the question into showing the following being cartesian:
\[
\begin{tikzcd}
A \ar[r] \arrow[d, "\varphi_A"']  & A_\perf \arrow[d, "\varphi_{A_\perf}"]\\
A \ar[r]& A_\perf,
\end{tikzcd}
\]
which is clear thanks to the explicit coordinate from the framing $W \langle t_1,\ldots,t_n \rangle\to A_0\to A$.

Now we consider the case for general $(B,J)$.
We then notice that the claims can be fitted into a natural induction: if $(B,J)$ is a multiple coproduct of framed regular prisms, Part \ref{prop:intersection_of_Frobenius_of_base_with_coproduct:base} for $(B,J)$ follows from Part \ref{prop:intersection_of_Frobenius_of_base_with_coproduct:cohomology} for $S=\overline{A}_\perf$ and $(B',J')$, where $(B',J')$ is the other factor such that $(B,J)=(A,I)\coprod (B',J')$, thanks to the complete projectivity of $A\to A_\perf$ together with the injectivity of the linearization maps in \Cref{lem:inj_and_completely_proj_mod}.
For \ref{prop:intersection_of_Frobenius_of_base_with_coproduct:cohomology}, the arguments in the first paragraph above work the same, since the coproduct $(B,J)$ admits a similar presentation as a prismatic envelope (cf. \Cref{cor:Cech_nerve_of_framed}).
For \ref{prop:intersection_of_Frobenius_of_base_with_coproduct:intersection}, we notice that the diagram for $(B,J)=(A,I)\coprod (B',J')$ admits a compatible injection into the diagram for $(B,J)=(A_\perf,IA_\perf)\coprod (B',J')$.
So it reduces to check that the following diagram is cartesian (where we slightly abuse the notations for coproducts)
\[
\begin{tikzcd}
A\coprod B' \ar[r] \arrow[d, "\varphi_{A\coprod B'}"']  & A_\perf\coprod B' \arrow[d, "\varphi_{A_\perf\coprod B'}"]\\
A\coprod B' \ar[r]& A_\perf\coprod B',
\end{tikzcd}
\]
which follows from the explicit formula for the reduction of $\varphi_{A_\perf\coprod B'}=\varphi_{\Prism_{(\overline{A}_\perf)_{\overline{B}'}/B'}}$ in the first paragraph above.
\end{proof}


\section{Frobenius modules over regular prisms}
\label{sec:Frob_mod}
In this section, we analyze the structure of a Frobenius module over a regular prism.
We prove that certain torsionfree Frobenius modules over a framed regular prism are locally free after restricted onto the analytic locus.
In addition, we prove that Frobenius modules over regular prisms (or their reductions) admit a primitive purity result.

\subsection{Frobenius modules, and generalities on the reflexivity}
\label{sub:Frob_mod_and_ref}
To facilitate the discussion, we introduce a natural extension of Breuil--Kisin(--Fargues) modules, and discuss some general results on the reflexivity.

\begin{definition}
	\label{def:Frobenius_mod}
	Let $(A,I)$ be a prism, let $n\in \mathbb{N}$, and let $a\leq b$ be two integers.
	Let $A'$ be any of the rings in the collection $\{A, A\langle 1/I \rangle, A/p^nA, A/p^nA[1/I]\}$.
	\begin{enumerate}
		\item A \emph{weak Frobenius module over $A'$} is a pair $(M,\varphi_M)$, where $M$ is an $A'$-module, and $\varphi_M$ is an $A'[1/I]$-linear map $\varphi_A^*M[1/I] \to M[1/I]$. 
		We also require $M$ to be finitely generated if $A$ is a noetherian ring.
		\item A weak Frobenius module $(M,\varphi_M)$ over $A'$ is called a \emph{Frobenius module over $A'$} if the map $\varphi_M$ is an isomorphism.
		\item An $I$-torsionfree Frobenius module $(M,\varphi_M)$ over $A'$ has \emph{Frobenius height in $[a,b]$} if there are containments of $A'$-submodules in $M[1/I]$\footnote{Note that the condition is trivially satisfied for any $a,b$ if $A'$ is either $A\langle 1/I \rangle$ of $A/p^nA[1/I]$.}
		\begin{equation}
			\label{eq:def_Frob_height}
					I^b\cdot M \subseteq \varphi_M(\varphi_A^*M)) \subseteq I^a\cdot M.
		\end{equation}
	\end{enumerate}
	We use $\Coh^{w\varphi}(A')$ to denote the category of weak Frobenius modules over $A'$, use $\Coh^\varphi(A')$ to denote the subcategory of Frobenius modules over $A'$, and use $\Coh^{\varphi, [a,b]}(A')$ to denote the subcategory of $I$-torsionfree Frobenius modules of Frobenius height $[a,b]$.
\end{definition}
\begin{remark}
	\label{rmk:Frob_mod_vs_weak_Frob_mod}
	The category of Frobenius modules naturally embeds fully faithfully into the category of weak Frobenius modules.
\end{remark}
We sort out a subcategory of (weak Frobenius) modules as below.
\begin{definition}
	\label{def:ref_Frob_mod}
	Let $(A,I)$ be a transversal prism, let $n\in \mathbb{N}$.
	\begin{enumerate}
		\item A finitely generated module $M$ over $A$ is called 
		\begin{itemize}
			\item \emph{analytically locally free} if the restriction of $M$ onto $\Spec(A)\backslash V(p,I)$ is locally free;
			\item \emph{saturated} if it is $p$-torsionfree, $I$-torsionfree, and satisfies the equality $M=M[1/p]\cap M[1/I]$;
			\item \emph{reflexive} if it is analytically locally free, and the Koszul complex $\Kos_M(p,d)$ is discrete, where $d$ is any generator of $I$ up to passing to a Zariski cover.
		\end{itemize}
		\item A finitely generated module $M$ over $A/p^nA$ is called \emph{reflexive} if $M$ is $I$-torsionfree, the localization $M[1/I]$ is finite projective over $A/p^nA[1/I]$, and the reduction $M/IM$ is generically locally free over $A/(p^n,I)A$.
	\end{enumerate}
	For $A'\in \{A, A/p^nA\}$ and $*\in \{\refl,\sat\}$, we use $\Coh^{w\varphi}_*(A')$ to denote the full subcategory of reflexive/saturated weak Frobenius modules over $A'$, and similarly use $\Coh^{w\varphi}_*(A')$ and $\Coh^{\varphi,[a,b]}_*(A')$ for the full subcategories of weak Frobenius modules or Frobenius modules with prescribed heights.
\end{definition}

As we see below, the reflexivity defined here is a special case of the reflexivity in commutative algebra.
Roughly speaking, the condition on the Koszul complex in the reflexivity is equivalent to the torsionfreeness and the saturatedness, and can be detected using the reduction.
\begin{lemma}
	\label{lem:reflexive_is_saturated}
	Let $(x,y)$ be a regular sequence in a noetherian ring $A$.
	Let $M$ be a finitely genereated $A$-module.
	Then the following conditions are equivalent:
	\begin{enumerate}[label=\upshape{(\roman*)}]
		\item\label{lem:reflexive_is_saturated_ref} the Koszul complex $\Kos_M(x,y)$ is discrete;
		\item\label{lem:reflexive_is_saturated_sat} $M$ is both $x$-torsionfree and $y$-torsionfree, and the natural injection $M\to M[1/x]\cap M[1/y]$ is an isomorphism.
		\item\label{lem:reflexive_is_saturated_red} $M$ is both $x$-torsionfree and $y$-torsionfree, the reduction $M/xM$ is $y$-torsionfree, and the reduction $M/yM$ is $x$-torsionfree.
	\end{enumerate}
\end{lemma}
\begin{proof}
	We first prove the equivalence between \ref{lem:reflexive_is_saturated_ref} and \ref{lem:reflexive_is_saturated_sat}.
	Assume \ref{lem:reflexive_is_saturated_ref} holds for $M$, namely the homology of the following chain complex vanishes for  positive homological degrees
	\[
	\Kos_M(x,y)=M\xrightarrow{m\mapsto (xm,ym)} M\oplus M \xrightarrow{(m_1,m_2)\mapsto ym_1-xm_2} M.
	\]
	If $M[x^\infty]\neq 0$, by the finiteness assumptions, there is an element $m\in M$ such that $xm=0$ and $y\nmid m$.
	In particular, the pair $(0,m)$ produces a non-zero element in $\mathrm{H}_1(\Kos_M(x,y))$, a contradiction.
	By symmetry, the torsion submodule $M[y^\infty]$ vanishes as well.
	If the map $M\to  M[1/x]\cap M[1/y]$ is not surjective, then there are elements $m_1, m_2\in M$ such that $x\nmid m_1$ and $y\nmid m_2$, together with the equality that $\frac{m_1}{x^r}=\frac{m_2}{y^s}\notin M$ for some pair of positive integers $(r,s)$.
	So we get an equality $m_1y^s=m_2x^r$ of elements in $M$, and the image of $m_1$ in $M/xM$ is a non-zero and $y^s$-torsion element.
	In particular, the subset $(M/xM)[y]\subset (M/xM)[y^s]$ is non-empty 
	\footnote{If not, the multiplication-by-$y$ map $M/xM[y^{i+1}]\to M/xM[y^{i}]$ is injective, and by induction we know $M/xM[y^s]=0$, a contradiction.}.
	Hence by picking an element $m_2'\in M$ whose image in $M/xM$ is a non-zero element in $M/xM[y]$, we know there is another element $m_1'\in M$ such that $ym_2'=xm_1'$ in $M$.
	Note that the element $m_2'\in M$ is not divisible by $x$.
	As a consequence, we see the pair $(m_1',m_2')$ represents a non-zero class in $\mathrm{H}_1(\Kos_M(x,y))$, contradicting to our assumption.
	
	Conversely, assume \ref{lem:reflexive_is_saturated_sat} holds for $M$.
	The torsionfreeness assumption of $M$ implies that the vanishing of $\mathrm{H}_2(\Kos_M(x,y))$.
	In addition, if there is a non-zero class in $\mathrm{H}_1(\Kos_M(x,y))$, then by definition there are elements $m_1,m_2\in M$ such that $ym_1=xm_2$, yet the element $m\colonequals \frac{m_1}{x}=\frac{m_2}{y}\in M[1/xy]$ is not contained in the submodule $M$.
	Notice that the element $m$ in contained in both $M[1/x]$ and $M[1/y]$.
	Thus we have $m\in (M[1/x]\cap M[1/y])\backslash M$, a contradiction.
	
	Finally we prove the equivalence between \ref{lem:reflexive_is_saturated_sat} and \ref{lem:reflexive_is_saturated_red}, and assume $M$ is $x$-torsionfree and $y$-torsionfree.
	If $M/xM$ has non-trivial $y$-torsion, then there is an element $m_1\in  M$ such that $x\nmid m_1$ but $ym_1=xm_2$ for some $m_2\in M$.
	This implies that the element $\frac{m_1}{x}=\frac{m_2}{y}$ is contained in the intersection $(M[1/x]\cap M[1/y])$ but not in $M$.
	Conversely, if $M\to M[1/x]\cap M[1/y]$ is not surjective, then we may find an element $\frac{m_1}{x}=\frac{m_2}{y}$, for some $m_1,m_2\in M$ such that $x\nmid m_1$ and $y\nmid m_2$.
	This implies that the image of $m_1$ in $M/xM$ is non-zero, yet the image of $ym_1=xm_2$ is zero.
\end{proof}
It worth to mentioning that the proof of the implication from \Cref{lem:reflexive_is_saturated}.\ref{lem:reflexive_is_saturated_ref} to \Cref{lem:reflexive_is_saturated}.\ref{lem:reflexive_is_saturated_sat} applies to more general settings.
\begin{corollary}
	\label{cor:Koszul_reg_imply_sat}
	Let $(x,y)$ be two elements in a ring $A$, and let $M$ be an $A$-module.
	Assume $M$ is $x$-torsionfree and $y$-torsionfree, and the Koszul complex $\Kos_M(x,y)$ is discrete.
	Then $M\to M[1/x]\cap M[1/y]$ is an isomorphism.
\end{corollary}
We also notice that by \cite[\href{https://stacks.math.columbia.edu/tag/0AY6}{Tag 0AY6}]{stacks-project}, the reflexivity in \Cref{def:ref_Frob_mod} often coincides with the usual notion in commutative algebra.
\begin{corollary}
\label{cor:reflexive_in_ca}
Let $(A,I)$ be a regular and transversal prism, and let $M$ be a finitely generated analytically locally free $A$-module.
Then $M$ is reflexive in the sense of \Cref{def:ref_Frob_mod} if and only if it is reflexive in the sense of commutative algebra (\cite[\href{https://stacks.math.columbia.edu/tag/0AVT}{Tag 0AVT}]{stacks-project}).
Moreover, the reflexive hull of $M$ is equal to the intersection $M[1/p]\cap M[1/I]$.
\end{corollary}

By specializing the above results to modules over a regular prism, we obtain the following observations.
\begin{corollary}
	\label{cor:reflexive_is_locally_free}
	Let $(A,I)$ be a transversal regular prism, and let $\mathfrak{p}$ be a prime ideal in $\Ass_{\overline{A}}(\overline{A}/p\overline{A})$.
	Then for any $n\in \mathbb{N}\cup \{\infty\}$, an object in $\Coh_\refl(A_{L_\mathfrak{p}}/p^nA_{L_\mathfrak{p}})$ is finite projective over $A_{L_\mathfrak{p}}/p^nA_{L_\mathfrak{p}}$.
\end{corollary}
\begin{proof}
	The statement for $n\in \mathbb{N}$ is trivial by \Cref{def:ref_Frob_mod}, and it is left to consider the case when $n=\infty$.
	In the latter case, by \Cref{def:ref_Frob_mod} and \Cref{lem:reflexive_is_saturated}, we know an object $M\in \Coh_\refl(A_{L_\mathfrak{p}})$ is in particular reflexive in the sense of \cite[\href{https://stacks.math.columbia.edu/tag/0AY6}{Tag 0AY6}]{stacks-project}.
	Hence we can use \cite[\href{https://stacks.math.columbia.edu/tag/0B3N}{Tag 0B3N}]{stacks-project} to conclude the finite projectivity of $M$ over $A$.
\end{proof}
\begin{corollary}
	\label{cor:reg_prism_is_sat}
	Let $(A,I)$ be a transversal regular prism.
	Then the inclusion map $A\to A[1/p]\cap A[1/I]$ is an isomorphism, and the completion map $A[1/I]\to A\langle 1/I \rangle$ is an injection.
\end{corollary}
\begin{proof}
	By the assumption of $(A,I)$, we know the Koszul complex $\Kos_A(d,p)$ is discrete, where $d$ is a local generator of the ideal $I$.
	So the first claim follows from the equivalence of \ref{lem:reflexive_is_saturated_ref} and \ref{lem:reflexive_is_saturated_sat} in \Cref{lem:reflexive_is_saturated}.
	For the second claim, it suffices to notice that by the equivalence of \ref{lem:reflexive_is_saturated_sat} and \ref{lem:reflexive_is_saturated_ref} in \Cref{lem:reflexive_is_saturated}, we know  for each $n\in \mathbb{N}$ that $A/p^nA$ is $d$-torsionfree and hence the map $A/p^nA\to A/p^nA[1/I]$ is injective.
	So by taking the limit, we get an injection $A\to A \langle 1/I \rangle$, where the source is $I$-torsionfree.
	Hence the map $A[1/I]\to A \langle 1/I \rangle$ is injective as well.
\end{proof}

\subsection{Analytic locally freeness}
\label{sub:Frobenius_module_ref}
We prove that a $p$-torsionfree Frobenius module over a framed regular prism is in fact analytically locally free, as long as it satisfies an appropriate pointwise assumption.

\begin{theorem}[Analytic locally freeness]
	\label{thm:Frob_mod_is_analyticall_loc_free}
	Let $(A,I)$ be a framed regular prism, and let $(M,\varphi_M)$ be a Frobenius module over $(A,I)$.
	Assume that for any surjection $(A,I)\to (A',IA')$ onto a transversal regular prism of Krull dimension two, the base change $M\otimes_A A'[1/p]$ is locally free over $A'[1/p]$.
	Then the restriction of $M/M[p^\infty]$ at the open subscheme $\Spec(A)\backslash V(p,I)$ is locally free.
\end{theorem}
Before the proof, it worth to mention that the analytic locally freeness in \Cref{thm:Frob_mod_is_analyticall_loc_free} does not hold for general torsionfree Frobenius modules over a regular prism or even a framed regular prism.
In the following, we give an example of a framed regular prism $(A,I)$ that shares the same underlying ring/ideal with that of the Breuil--Kisin prism, yet the structural results of the Breuil--Kisin modules as in \cite[Prop.\ 4.3]{BMS1} fail for those over $(A,I)$.
Relatedly, in contrast to the result of Kisin \cite[Prop.\ 2.1.12]{Kis06} and Bhatt--Scholze (\cite[Thm.\ 7.2]{BS23}) on Breuil--Kisin modules, the \'etale realization functor $\Vect^\varphi(A)\to \Vect^\varphi(A\langle 1/I \rangle)$ is not full.
\begin{example}
	\label{eg:Frob_mod_over_general_Frob-reg_prism}
	We let $A$ be the ring $W\llbracket t \rrbracket$, and let $\varphi_A:A\to A$ be the unique continuous extension of $\varphi_W$ that sends the element $t$ onto the polynomial $t(t-p)^{p-1}$.
	Since the mod $p$ reduction of the map $\varphi_A$ is finite flat, and since the ring $A$ is both $p$-complete and $p$-torsionfree, we know $\varphi_A$ is finite flat as well.
	We let $E(t)=t-p$.
	Then as the constant term of the element $\frac{(\varphi_A(t)-p)-(t-p)^p}{p}\in A$ is $\frac{-p-(-p)^p}{p}$ and in particular is a unit in $W$, we know $E(t)$ is distinguished in the $\delta$-ring $A$.
	Thus the pair $(A,(E(t)))$ is a prism (\cite[Lem.\ 3.1]{BS22}).	
	Now consider the $p$-torsionfree module $M\colonequals A/tA\simeq W$.
	The Frobenius pullback $\varphi_A^*M$ is naturally isomorphic to the $A$-module $A/\varphi_A(t)A=W\llbracket t \rrbracket/tE(t)^{p-1} W\llbracket t \rrbracket$.
	In particular, there is a natural isomorphism between the localizations
	\[
	\varphi_M:\varphi_A^*M [1/E(t)] =  \bigl(A/tE(t)^{p-1} A\bigr)[1/E(t)] \simeq \bigl(A/tA\bigr)[1/E(t)]=M[1/E(t)].
	\]
	Hence pair $(M,\varphi_M)$ forms a $p$-torsionfree Frobenius module that is not analytically locally free.
\end{example}
\begin{proof}[Proof of \Cref{thm:Frob_mod_is_analyticall_loc_free}]
	We let $J$ be a non-zero fitting ideal of $M$.
	We first claim that $M[1/p]$ is finite projective over $A[1/p]$.
	By \cite[\href{https://stacks.math.columbia.edu/tag/07ZD}{Tag 07ZD}]{stacks-project}, it suffices to show that the inclusion of $A$-modules $J\subset A$ becomes an equality after inverting by $p$.
	Here we implicitly use the fact that the formation of the fitting ideal commutes with arbitrary base changes of the ring (\cite[\href{https://stacks.math.columbia.edu/tag/07ZA}{Tag 07ZA}.(3)]{stacks-project}).
	Moreover, by the local property of quasi-coherent sheaves and the faithfully flatness of the complete localization, to check $J[1/p]=A[1/p]$, it suffices to do so by taking the base change along the map $A\to A^\wedge_{\mathfrak{m}}$ for a maximal ideal $\mathfrak{m}\subset A$.
	In addition, by taking the finite \'etale base change along $W(k)\to W(k(\mathfrak{m}))$ and apply the further faithfully flat base change along $A^\wedge_{\mathfrak{m}}\to A'^\wedge_{\mathfrak{m}}$ as in \Cref{prop:competion_of_A_at_max_ideal}, we may assume that there is a Frobenius equivariant isomorphism of regular prisms
	\[
	(W\llbracket x_1,\ldots,x_n \rrbracket, I') \simeq (A^\wedge_{\mathfrak{m}}, IA^\wedge_{\mathfrak{m}}),
	\]
	where $(W\llbracket x_1,\ldots,x_n \rrbracket, I')$ satisfies the condition as in \Cref{prop:competion_of_A_at_max_ideal}.
	So the claim follows from \Cref{thm:BK_mod_invert_p_explicit}, which will be proved in the rest of the subsection.
	
	We now grant the above claim and assume $M$ is $p$-torsionfree.
	By the Beauville--Laszlo gluing, to finish the proof, it suffices to show that $M\otimes_A A\langle 1/I \rangle$ is finite projective over $A\langle 1/I \rangle$.
	By \cite[\href{https://stacks.math.columbia.edu/tag/0F7M}{Tag 0F7M}]{stacks-project}, it suffices to show that the fitting ideal $J\otimes_A A\langle 1/I \rangle$ is equal to $p^nA\langle 1/I \rangle$ for some $n$.
	The latter follows from the more general observation on Frobenius equivariant ideals in a noetherian $\delta$-ring \Cref{Prop:fitting_ideal_invert_I}, applied at $A\langle 1/I \rangle$.
\end{proof}

\begin{proposition}[Local freeness of \'etale realization]
	\label{Prop:fitting_ideal_invert_I}
	Let $A$ be a $p$-complete $p$-torsionfree noetherian $\delta$-ring, and let $J$ be an ideal in $A$ such that $\varphi_A^*(J)=J$.
	Then $J=p^nA$ for some $n\in \mathbb{N}$.
\end{proposition}
\begin{proof}
	By assumption, the ideal $J$ is a $p$-complete $p$-torsionfree finitely presented $A$-module.
	We let $n\in \mathbb{N}$ be the largest integer such that $J\subseteq p^nA$.
	Then by the $p$-torsionfree assumption and by replacing $J$ with the ideal $\frac{J}{p^n}$, we may assume $J$ is not contained in the ideal $pA$ and it suffices to check that the ideal $J$ is equal to $A$.
	
	We assume that $J$ is not equal to $A$, and let $\overline{J}$ be the image of $J$ in $A/pA$, so that $\overline{J}$ is a non-zero ideal in $\overline{A}$.
	By assumption, we have the equality $\overline{J}=\varphi_A^*(\overline{J})$, as ideals in $A/pA$.
	We let $\{f_1,\ldots,f_m\}$ be a set of generators of $\overline{J}$, which is guaranteed by the noetherian assumption.
	Then the condition on Frobenius twist of $\overline{J}$ implies the equality of ideals 
	\[
	(f_1,\ldots,f_m)A/pA = (f_1^p,\ldots,f_m^p)A/pA,
	\]
	and thus $\{f_1^p,\ldots,f_m^p\}$ is another set of generators of $\overline{J}$.
	So by repeating this process, and by noticing that the systems of ideals $\{\overline{J}^n\}_n$ and $\{(f_1^{p^n},\ldots,f_m^{p^n})A/pA\}_n$ are cofinal to each other, 
	we get the equality of ideals
	\[
	\overline{J}= \overline{J}^n, ~\forall n\in \mathbb{N}.
	\]
	Hence by Krull's intersection theorem (\cite[\href{https://stacks.math.columbia.edu/tag/00IP}{Tag 00IP}]{stacks-project}), we know $\overline{J}=(0)$, a contradiction.	
\end{proof}
For the convenience of the discussion, we sort out the following explicit conditions on the prism as in \Cref{prop:competion_of_A_at_max_ideal}.
\begin{convention}
	\label{assumption:completion_of_framed_regular_prism}
	We let $(A,I)$ be the regular prism $(V\llbracket x_1,\ldots,x_n \rrbracket, (d))$, where
	\begin{itemize}
		\item the ring $V$ is an unramified Cohen ring of mixed characteristic that is preserved by $\varphi_A$;
		\item for each $1\leq i\leq n$, the element $\varphi_A(x_i)$ is a monic polynomial in $x_i$ of degree $p$ such that
		\[
		x_i~\text{divides}~\varphi_A(x_i),~\text{and}~\delta(\frac{\varphi_A(x_i)}{x_i})~\text{is either zero or a unit};
		\]
		\item the ideal $I$ is generated by a non-constant power series $d\in W\llbracket x_1,\ldots,x_n \rrbracket$ whose constant term is $p$.
	\end{itemize}
\end{convention}
\begin{theorem}
	\label{thm:BK_mod_invert_p_explicit}
	Let $(A,I)$ be the regular prism in \Cref{assumption:completion_of_framed_regular_prism}, and let $(M,\varphi_M)$ be a Frobenius module over $A$.
	Let $J$ be a non-zero ideal of $A$ satisfying the following two conditions:
	\begin{itemize}
		\item There is an equality of ideals in $A[1/d]$:
			\begin{equation}
			\label{eq:thm:BK_mod_invert_p_ideal}
			J [1/d] = \varphi_A^*(J) [1/d];
		\end{equation}
	\item for any surjection $(A,I)\to (A',IA')$ onto a regular prism of Krull dimension two, the base change $JA'[1/p]$ is either the zero ideal or the entire ring $A'[1/p]$.
	\end{itemize}
	Then $J[1/p]=A[1/p]$.
\end{theorem}
The rest of the subsection will be devoted to the proof of \Cref{thm:BK_mod_invert_p_explicit}, during which we explore various algebraic properties of the $\delta$-ring $A$ and the ideal $J$.
We start with the following simple calculations will be used frequently so we record them here.
\begin{lemma}
	\label{lem:power_of_Frobenius}
	Let $A$ be the $\delta$-ring $A$ in \Cref{assumption:completion_of_framed_regular_prism}, and let $g\in A$ be any element.
	\begin{enumerate}[label=\upshape{(\roman*)}] 
		\item For $s\in \mathbb{N}$, $r\in \mathbb{N}_{\geq 1}$, and $h\in A$, the element $(p^s\cdot h+g)^{p^r}$ is of the form $p^{sp^r}\cdot h^{p^r}+g^{p^r}+p^{s+r'}\cdot g\cdot h \cdot g_1$, where $g_1\in A$, and $r'=\begin{cases}
		r, ~\text{if}~s\geq 1;\\
		1,~\text{if}~s=0.
		\end{cases}$
		\item The element $\varphi_A(g)-g^p$ is of the form $p\cdot g_2$ for some $g_2\in A$, and $g_2$ is contained in the ideal $(x_1,\ldots,x_n)\subset A$ if $g$ is so.
	\end{enumerate}
\end{lemma}
\begin{proof}[Proof of \Cref{lem:power_of_Frobenius}]
	For the first claim, we use the binomial formula and consider each $\binom{p^r}{i}\cdot (p^s\cdot h)^i\cdot g^{p^r-i}$, which is divided by $g\cdot p^s\cdot h$ for any $1\leq i\leq p^r-1$.
	So the statement for $s=0$ follows from the divisibility $p|\binom{p^r}{i}$.
	If $s\geq 1$, we recall that the $p$-adic valuation of $\binom{p^r}{i}=\frac{p^r}{i}\binom{p^r-1}{i-1}$ is $r-v$, where $v$ is the $p$-adic valuation of $i$.
	So it is left to check that for each $1\leq i\leq p^r-1$, we have $si+r-v\geq s+r$, or equivalently $s(i-1)\geq v$.
	The latter follows from the inequalities of real functions $v\leq \log_p(i)\leq \log_2(i) \leq i-1$ for $i\geq 0$.
	
	For the first half of the second claim, it follows by the definition of $\varphi_A$.
	If $g\in (x_1,\ldots,x_n)\subset A$, since the setup of \Cref{assumption:completion_of_framed_regular_prism} implies that the Frobenius endomorphism $\varphi_A:A\to A$ sends the ideal $(x_1,\ldots,x_n)$ into itself, we see the element $\varphi_A^r(g)-g^{p^r}$ belongs to $(x_1,\ldots,x_n)$ as well.
\end{proof}
As quick consequences, we have the following results on certain ideals being the entire ring.
\begin{lemma}
	\label{lem:f_and_varphi(f)}
	Let $A$ be the $\delta$-ring in \Cref{assumption:completion_of_framed_regular_prism}, and let $f=p^r+f_1$ be an element in $A$, where $r\geq 1$ and $f_1\in (x_1,\ldots,x_n)\subset A$.
	Then the ideal $(f, \varphi_A(f),\ldots, \varphi_A^r(f))[1/p]$ is equal to $A[1/p]$.
\end{lemma}
\begin{proof}
	We prove this by induction on $r$.
	If $r=1$, then the element $\frac{\varphi_A(f)-f^p}{p}$ is of the form
	\begin{align*}
		\frac{p+ \varphi_A(f_1) - (p+f_1)^p}{p} & = \frac{p+ \varphi_A(f_1) - (p^p +f_1^p +p\cdot f_1\cdot f_2)}{p} \\
		& = (1-p^{p-1}) + \frac{\varphi_A(f_1)-f_1^p}{p} - f_1\cdot f_2,
	\end{align*}
	where $f_2$ is an element in $A$ and the first equality follows from \Cref{lem:power_of_Frobenius}.
	In particular, both $\frac{\varphi_A(f_1)-f_1^p}{p}$ and $f_1\cdot f_2$ belong to the ideal $(x_1,\ldots,x_n)\subset A$.
	Hence the element $\frac{\varphi_A(f)-f^p}{p}$ is a unit in $A[1/p]$.
	
	Now let $r>1$.
	We claim that the ideal $(f,\varphi_A(f))[1/p]$ contains an element $g\colonequals p^{r-1}+g_1$ in $A$ for some $g_1\in (x_1,\ldots,x_n)\subset A$.
	Granting the claim, by taking Frobenius twists, we know $(f, \varphi_A(f),\ldots, \varphi_A^r(f))[1/p]$ contains $(g,\varphi_A(g),\ldots, \varphi_A^{r-1}(g))[1/p]$, where the latter is equal to $A[1/p]$ by the induction hypothesis.
	To see the claim, we consider the element $\delta(f)=\frac{\varphi_A(f)-f^p}{p}$, which is of the form
	\begin{align*}
		\frac{p^r+ \varphi_A(f_1) - (p^r+f_1)^p}{p} & = \frac{p^r + \varphi_A(f_1) - (p^{rp} +f_1^p + p^r\cdot f_1\cdot f_2)}{p} \\
		& = \frac{(p^r-p^{rp}) + (\varphi_A(f_1)-f_1^p) - p^r\cdot f_1\cdot f_2}{f}\\
		& = p^{r-1}\cdot (1-p^{rp-1}) + \frac{\varphi_A(f_1)-f_1^p}{p} - p^{r-1} \cdot f_1\cdot f_2,
	\end{align*}
	where $f_2$ is an element in $A$ and the first equality follows from \Cref{lem:power_of_Frobenius}.
	In particular, we know the element $\bigl( \frac{\varphi_A(f_1)-f_1^p}{p} - p^{r-1} \cdot f_1\cdot f_2 \bigr)$ belong to the ideal $(x_1,\ldots,x_n)\subset A$.
	Hence the required element $g$ can be obtained by dividing $\delta(f)$ by the unit $(1-p^{rp-1})$, which finishes the proof of the claim.
\end{proof}

\begin{lemma}
	\label{lem:ideal_generated_by_Frobenius_powers_of_d}
	Let $A$ be the $\delta$-ring in \Cref{assumption:completion_of_framed_regular_prism}.
	For each $s\in \mathbb{N}$ and $N\in \mathbb{N}_{\geq 1}$, the ideal $(\varphi_A^{s+1}(d)^N,d^N)A[1/p]$ is equal to $A[1/p]$.
\end{lemma}
\begin{proof}
	We first notice that by the assumption of \Cref{assumption:completion_of_framed_regular_prism}, the element $d$ is of the form $p+g$, where $g\in (x_1,\ldots,x_n)\subset A$.
	So by \Cref{lem:power_of_Frobenius}, the element $\frac{\varphi_A^{s+1}(d)-d^{p^{s+1}}}{p}$ is of the form
	\begin{align*}
		\frac{p+\varphi_A^{s+1}(g)-(p+g)^{p^{s+1}}}{p} & = \frac{p-p^{p^{s+1}} + \varphi_A^{s+1}(g)- g^{p^{s+1}} - p\cdot g\cdot g_1}{p}  \\
		&= (1-p^{p^{s+1}-1}) + p^s\cdot g_2 - g\cdot g_1,
	\end{align*}
	where $g_1\in A$ and $g_2\in (x_1,\ldots,x_n)$.
	So the element $p^s\cdot g_2 - g\cdot g_1$ is contained in the ideal $(x_1,\ldots,x_n)$, and the element $\frac{p+\varphi_A^{s+1}(g)-(p+g)^{p^{s+1}}}{p}$ is invertible in $A[1/p]$.
	Hence the ideal $(\varphi_A^{s+1}(d), d)A[1/p]$ is equal to the ring $A[1/p]$.
	Finally, the claim for general $N$ follows by noticing that the ideal $(\varphi_A^{s+1}(d)^N,d^N)A[1/p]$ contains a large power of the ideal $(\varphi_A^{s+1}(d), d)A[1/p]$.
\end{proof}

The following observation narrows down the possibility of the ideal $J[1/p]$.
\begin{proposition}
	\label{prop:fitting_ideal_of_BK_mod}
	Let $(A,I)$ be the prism in \Cref{assumption:completion_of_framed_regular_prism}, and let $J$ be an ideal in $A$ satisfying (\ref{eq:thm:BK_mod_invert_p_ideal}).
	Then either $J[1/p]=A[1/p]$, or $J[1/p]$ is contained in $(x_1,\ldots,x_n)A[1/p]$.
\end{proposition}
\begin{proof}
	We assume $J[1/p]$ is not the entire ring $A[1/p]$, and prove the other condition by contradiction: namely we assume that there exists an element $f=p^r+f_1\in J$, where $r\geq 1$ and $f_1\in (x_1,\ldots,x_n)A$.
	We first notice that by (\ref{eq:thm:BK_mod_invert_p_ideal}), we know $\varphi_A(f)\in J[1/d]$.
	By taking its Frobenius twist, we get 
	\[
	\varphi_A^2(f) \in \varphi_A^*(J)[1/\varphi_A(d)] \subset \varphi_A^*(J)[1/\varphi_A(d)d] = J[1/\varphi_A(d)d].
	\]
	Continuing this process, we see the Frobenius twists of the equality (\ref{eq:thm:BK_mod_invert_p_ideal}) imply that for each $i\geq 1$, we have the containment
	\begin{equation}
		\label{eq:thm:BK_mod_invert_p_Frobenius_power_of_element}
		\{\varphi_A^i(f),\varphi_A^{i-1}(f),\ldots, f\} \subset J[\frac{1}{\varphi_A^{i-1}(d)\cdots d}].
	\end{equation}
	So combining \Cref{lem:f_and_varphi(f)} with (\ref{eq:thm:BK_mod_invert_p_Frobenius_power_of_element}), we get the following equality of ideals:
	\begin{equation}
		\label{eq:thm:BK_mod_invert_p_equality_of_ideals_after_localizing}
		J[\frac{1}{p\varphi_A^{r-1}(d)\cdots d}] = A[\frac{1}{p\varphi_A^{r-1}(d)\cdots d}].
	\end{equation}
	
	We now let $s\in \mathbb{N}$ be the least non-negative integer such that the equality below holds true.
	\begin{equation}
		\label{eq:thm:BK_mod_invert_p_equality_of_ideals_after_localizing_2}
		J[\frac{1}{p\varphi_A^{s}(d)\cdots d}] = A[\frac{1}{p\varphi_A^{s}(d)\cdots d}].
	\end{equation}
	Then by (\ref{eq:thm:BK_mod_invert_p_equality_of_ideals_after_localizing_2}) and (\ref{eq:thm:BK_mod_invert_p_ideal}) we have
	\begin{equation}
		\label{eq:thm:BK_mod_invert_p_twisted_ideal}
		\varphi^*(J) [\frac{1}{p\varphi_A^{s}(d)\cdots d}] = A[\frac{1}{p\varphi_A^{s}(d)\cdots d}].
	\end{equation}
	Hence by representing the identity $1_A$ as elements in the left hand side, there exists a large positive integer $N$ such that
	\[
	(\varphi_A^s(d)\cdots d)^N \in J[1/p],~ \text{and}~ (\varphi_A^s(d)\cdots d)^N \in \varphi_A^*(J) [1/p].
	\]
	
	Next we let $(A_{\perf}, IA_{\perf})$ be the perfection of the prism $(A,I)$, where $A_{\perf}$ by assumption is faithfully flat over $A$.
	Since $\varphi_{A_\perf}$ is an isomorphism, we get
	\begin{equation}
		\label{eq:thm:BK_mod_invert_p_power_of_d_in_J}
		(\varphi_A^s(d)\cdots d)^N \in J[1/p]\otimes_A A_{\perf},~ \text{and}~ (\varphi_{A_\perf}^{s-1}(d)\cdots \varphi_{A_\perf}^{-1}(d))^N \in J[1/p]\otimes_A A_{\perf}.
	\end{equation}
	
	Now we discuss the possible values of the integer $s$.
	If $s=0$, then we know from (\ref{eq:thm:BK_mod_invert_p_power_of_d_in_J}) that both $d^N$ and $\varphi_{A_{\perf}}^{-1}(d)^N$ belong to the ideal $J[1/p]\otimes_A A_{\perf}$.
	However, \Cref{lem:ideal_generated_by_Frobenius_powers_of_d} implies that the ideal of $A_{\perf}[1/p]$ that is generated by $(d^N, \varphi^{-1}(d)^N)$ is the entire ring $A_{\perf}[1/p]$.
	Thus the inclusion $J[1/p]\subset A[1/p]$ becomes an isomorphism after taking the base change along the faithfully flat cover $A\to A_\perf$, a contradiction to the assumption from the beginning that $J[1/p]\neq A[1/p]$.
	
	So let us assume $s\geq 1$ from now.
	Under the assumption, \Cref{lem:ideal_generated_by_Frobenius_powers_of_d} implies that the ideal of $A_{\perf}[1/p]$ generated by $(\varphi^s(d)^N, \varphi^{-1}(d)^N)$ is the ring $A_{\perf}[1/p]$.
	On the other hand, by (\ref{eq:thm:BK_mod_invert_p_power_of_d_in_J}), both $(\varphi_A^s(d)\cdots d)^N$ and  $(\varphi_{A_\perf}^{s-1}(d)\cdots \varphi_{A_\perf}^{-1}(d))^N$ belong to the ideal $J[1/p]\otimes_A A_{\perf}$.
	Hence by taking a linear combination, we know the element $(\varphi_{A_\perf}^{s-1}(d)\cdots d)^N=(\varphi_{A}^{s-1}(d)\cdots d)^N$ belong to $J[1/p]\otimes_A A_{\perf}$.
	The latter, by the faithful flatness of $A\to A_{\perf}$ again, implies the equality $J[\frac{1}{p\varphi_{A}^{s-1}(d)\cdots d}]=A[\frac{1}{p\varphi_{A}^{s-1}(d)\cdots d}]$, contradicting to the minimality of the positive integer $s$.
	As a consequence, we see the ideal $J[1/p]$ has to be contained in the ideal $(x_1,\ldots,x_n)A[1/p]$.
\end{proof}

To prepare for the induction procedure later, we notice that the equation (\ref{eq:thm:BK_mod_invert_p_ideal}) puts a very strong relationship between the divisibility of $J$ by $x_i$, and the formation of the distinguished element $d$.
\begin{lemma}
	\label{lem:dividibility_of_J_and_x}
	Let $(A,I)$ be the prism in \Cref{assumption:completion_of_framed_regular_prism}, and let $J\subset A$ be a non-zero ideal satisfying (\ref{eq:thm:BK_mod_invert_p_ideal}).
	Assume the ideal $J[1/p]$ is contained in the ideal $x_i\cdot A[1/p]$ for some integer $i$.
	Then $\delta(\frac{\varphi_A(x_i)}{x_i})$ is a unit, and there exists an integer $m\in \mathbb{N}$ and a unit $u\in A$ such that $d = u\cdot \varphi_A^m (\frac{\varphi_A(x_i)}{x_i})$.
	In particular, the element $du^{-1}$ is contained in the subring $V\llbracket x_i \rrbracket\subset V\llbracket x_1,\ldots,x_n \rrbracket$.
\end{lemma}
\begin{proof}
	We first note that the ideal $J$ is finitely generated.
	So for any fixed $m\in \mathbb{N}$, by taking the successive $m$-th Frobenius twists of the equality  (\ref{eq:thm:BK_mod_invert_p_ideal}), we have
	\begin{equation}
	\label{eq:dividibility_of_J_and_x}
	(\varphi^{m}_A(d)\cdots d)^N J[1/p] \subset \varphi_A^{m+1}(x_i)\cdot A[1/p],~\text{for}~N\gg0.
	\end{equation}
	Moreover, the element $\varphi_A^{m+1}(x_i)$ can be written as the product
	\[
	\varphi_A^{m+1}(x_i) = x_i \cdot \frac{\varphi_A(x_i)}{x_i} \cdots \varphi_A^m(\frac{\varphi_A(x_i)}{x_i}).
	\]
	
	We then discuss the value of $\delta( \frac{\varphi_A(x_i)}{x_i})$.
	If $\delta( \frac{\varphi_A(x_i)}{x_i})=0$, then $\varphi_A^{m+1}(x_i)=x_i^{p^{m+1}}$, which is not divisible by any of $\varphi^i_A(d)$.
	So by the fact that the ring $A$ is a unique factorization domain, the containment in (\ref{eq:dividibility_of_J_and_x}) implies that $J[1/p]$ is contained in $x_i^{p^{m+1}}\cdot A[1/p]$ for any $m\in \mathbb{N}$. 
	Hence $J[1/p]$ is in 
	\[
	\bigcap_{m>0} x_i^{p^m}\cdot A[1/p],
	\]
	where the latter is zero and contradicts to the assumption of $J[1/p]$.
	Thus $\delta( \frac{\varphi_A(x_i)}{x_i})$ can only be a unit.
	
	To continue, we assume the element $d$ (which is irreducible) is not divisible by any of the element $\varphi_A^m (\frac{\varphi_A(x_i)}{x_i})$.
	By the assumption of \Cref{assumption:completion_of_framed_regular_prism}, we know each $\varphi_A^j(\frac{\varphi_A(x_i)}{x_i})$ is an Eisinstein polynomial of degree $p^j(p-1)$ and is in particular irreducible.
	Moreover, the non-divisibility assumption on $d$ implies that the ideal $J[1/p]$ is contained in the ideal $\varphi_A^j(\frac{\varphi_A(x_i)}{x_i})\cdot A[1/p]$ for each $j\in \mathbb{N}$.
	In addition, for the degree reason, the Eisenstein polynomials $\varphi_A^j(\frac{\varphi_A(x_i)}{x_i})$ are pairwise non-divisible by each other.
	Hence the ideal $J[1/p]$ is contained in the ideal 
	\[
	\bigcap_{m>0} \varphi_A^m(\frac{\varphi_A(x_i)}{x_i})\cdot A[1/p],
	\]
	and each non-zero element in $J[1/p]$ can be divided by infinitely many irreducible elements $\varphi_A^m(\frac{\varphi_A(x_i)}{x_i})$, contradicting to the noetherianity of $A$.
	Thus $d$ has to be divided by one of $\varphi_A^m(\frac{\varphi_A(x_i)}{x_i})$.
\end{proof}

The following crystalline analogue of \Cref{thm:BK_mod_invert_p_explicit} will be used later.
\begin{proposition}[Frobenius modules over regular crystalline prism]
	\label{prop:BK_mod_crystalline}
	Let $A$ be the $\delta$ ring in \Cref{assumption:completion_of_framed_regular_prism}, and let $J$ be a non-zero ideal in $A$ such that 
	\[
	\varphi^*_A(J)[1/p]=J[1/p].
	\]
	Then $J[1/p]=A[1/p]$.
\end{proposition}
\begin{proof}
	We first notice that the ideal $J[1/p]$ is not contained in the ideal $(x_1,\ldots,x_n)[1/p]$.
	Otherwise, by taking the Frobenius twists, for each $m\in \mathbb{N}$, we have
	\[
	J\subset J[1/p]=\varphi^{m,*}_A(J)[1/p]\subset (\varphi_A^m(x_1), \ldots, \varphi_A^m(x_n))[1/p].
	\]
	On the other hand, since $J$ is contained in $A$, we get
	\[
	J\subset (\varphi_A^m(x_1), \ldots, \varphi_A^m(x_n))[1/p] \cap A.
	\]
	We then claim that the intersection on the right hand side is the ideal $(\varphi_A^m(x_1),\ldots,\varphi_A^m(x_n))$ in $A$.
	Granting the claim, since $(\varphi_A^m(x_1), \ldots, \varphi_A^m(x_n))$ is contained in $(p,x_1,\ldots,x_n)^m$, and intersection $\bigcap_{m>0} (p,x_1,\ldots,x_n)^m$ in $A$ is zero, we see $J=(0)$, contradicting to the non-zero assumption of this ideal.
	To see the claim, we prove by contradiction and assume that there is a non-zero element $f\in A$ and an integer $i\geq 1$, such that $p^if\in (\varphi_A^m(x_1), \ldots, \varphi_A^m(x_n))$ and $p^{i-1}f\notin (\varphi_A^m(x_1), \ldots, \varphi_A^m(x_n))$.
	Then the image of $f$ in the quotient $A/(\varphi_A^m(x_1), \ldots, \varphi_A^m(x_n))$ is a non-zero $p^i$-torsion element.
	Yet by the assumption in \Cref{assumption:completion_of_framed_regular_prism} that each $\varphi_A^m(x_j)$ is a monic polynomial in $x_j$, the quotient $A/(\varphi_A^m(x_1), \ldots, \varphi_A^m(x_n))$ is finite free over $V$ and is in particular $p$-torsionfree, a contradiction.
	
	To proceed, we notice that there exists an element $f=p^r+f_1$ in $J[1/p]$, where $r\in \mathbb{N}$ and $f_1\in (x_1,\ldots,x_n)$.
	Moreover, by taking Frobenius twists, we know the element $\varphi_A^m(f)$ belongs to $J[1/p]$ for each $m\in \mathbb{N}$.
	If $r=0$, then since the element $f_1$ is topologically nilpotent, the element $f$ is a unit in $A$, which implies that $J[1/p]=A[1/p]$.
	If $r>0$, the equality $J[1/p]=A[1/p]$ follows from \Cref{lem:f_and_varphi(f)}, which finishes the proof.
\end{proof}

Now assembling all the ingredients above, we are ready to prove \Cref{thm:BK_mod_invert_p_explicit}.
\begin{proof}[Proof of \Cref{thm:BK_mod_invert_p_explicit}]
	We prove the statement by induction on $n$.
	If $n=1$, as $V\llbracket x_1 \rrbracket$ itself is a transversal regular prism of Krull dimension two, there is nothing to prove.
	So we may fix an integer $n>1$, and assume that \Cref{thm:BK_mod_invert_p_explicit} is true for the setup of $(n-1)$ variables.
	We assume that the non-zero ideal $J[1/p]$ is not the entire ring $A[1/p]$.
	By \Cref{prop:fitting_ideal_of_BK_mod}, we know the ideal $J[1/p]$ is contained in $(x_1,\ldots,x_n)[1/p]$.
	We will show that the assumptions on $J[1/p]$ would eventually contradict the inductive hypothesis.
	
	To proceed, we first notice that by the setup in \Cref{thm:BK_mod_invert_p_explicit}, for each $1\leq i\leq n$, the surjection $(A,(d))\to (A/x_iA, (d))$ is Frobenius equivariant and induces a natural prism structure on the target.
	Then we discuss the following possibilities:
	\begin{enumerate}
		\item 	If the ideal $J[1/p]$ is contained in the ideal $x_n A[1/p]$, then by \Cref{lem:dividibility_of_J_and_x} we know both of the following two conditions are true:
		\begin{itemize}
			\item the ideal $J[1/p]$ is not contained in the ideal $x_1A[1/p]$;
			\item the element $d-p$ is not divided by $x_1$.
		\end{itemize} 
		We let $\overline{J}$ be the image of the ideal $J$ in $A/x_1A$, and let $\overline{d}$ be the reduction of the element $d$.
		Then the $\delta$-pair $(A/x_1A, (\overline{d}))$ satisfies the same assumptions of \Cref{assumption:completion_of_framed_regular_prism} but with $(n-1)$ variables.
		In addition, by the above two conditions, we know the ideal $\overline{J}[1/p]$ is non-zero, is contained in $(x_2,\ldots,x_n)A/x_1A[1/p]$ (hence a proper ideal in $A/x_1A[1/p]$), and satisfies the two assumptions in \Cref{thm:BK_mod_invert_p_explicit}.
		The latter is impossible, thanks to the inductive hypothesis of \Cref{thm:BK_mod_invert_p_explicit} as in the first paragraph.
		\item If the ideal $J[1/p]$ is not contained in $x_nA[1/p]$.
		We let $\overline{J}$ be the image of $J$ in $A/x_nA$, which is then non-zero, is contained in $(x_1,\ldots,x_{n-1})A/x_nA[1/p]$, and satisfies the two assumptions in \Cref{thm:BK_mod_invert_p_explicit}.
		At this moment, the mod $x_n$ reduction of the distinguished element (which we denote as $\overline{d}$) is of the form $p+h$, where $h$ is an element in the ideal $(x_1,\ldots,x_{n-1})A/x_nA$.
		If the element $h$ is equal to zero, then $(A/x_nA, IA/x_nA)$ is a regular crystalline prism.
		So by the crystalline analogue of the locally freeness in \Cref{prop:BK_mod_crystalline}, the ideal $\overline{J}[1/p]$ has to be the entire ring $A/x_nA[1/p]$, which contradicts the previous sentence that $\overline{J}[1/p]$ is a proper ideal.
		Otherwise, the element $h$ is non-zero in the ideal $(x_1,\ldots,x_{n-1})A/x_nA$, so the prism $(A/x_nA,IA/x_nA)$ satisfies the same assumptions of \Cref{assumption:completion_of_framed_regular_prism} but with $(n-1)$ variables.
		Thus we reduce the problem to \Cref{thm:BK_mod_invert_p_explicit} but for $(n-1)$-variables, under which there cannot exist an ideal $\overline{J}[1/p]$ that satisfies the aforementioned properties.
		Hence we again obtain a contradiction, which finishes the proof.
	\end{enumerate}
\end{proof}

Finally, we provide an alternative assumption of the regular prism that prevents the defect we saw in \Cref{eg:Frob_mod_over_general_Frob-reg_prism}.
Roughly speaking, it says that the aforementioned example would never happen if the reduction $\overline{A}$ is sufficiently ramified.
\begin{proposition}
	\label{prop:Frob_mod_when_the_reduction_is_mildly_ramified}
	Let $(A,I)$ be a regular prism of Krull dimension two, and assume the absolute ramification degree of the discrete valuation ring $\overline{A}$ is not dividing $p-1$.
	Then for any Frobenius module $M$ over $A$, we have $M[1/p]$ is free.
\end{proposition}
\begin{corollary}
	\label{cor:Frob_mod_when_the_reduction_is_mildly_ramified}
	Assume the absolute ramification degree of the $p$-adic local field $K$ is not dividing $(p-1)$.
	Let $(A,I)$ be a framed regular prism such that the reduction $A/I$ is an $\mathcal{O}_K$-algebra.
	For any Frobenius module $M$ over $A$, the quotient $M/M[p^\infty]$ is analytically locally free.
\end{corollary}
\begin{proof}[Proof of \Cref{prop:Frob_mod_when_the_reduction_is_mildly_ramified}]
	We recall the structural result of the regular prism of Krull dimension two in \Cref{prop:Frob_reg_prism_of_dim_2}, and identify $(A,I)$ with the pair $(V\llbracket t \rrbracket,(d))$, where the Frobenius structure $\varphi_{V\llbracket t \rrbracket}$ preserves the subring $V$.
	Then the argument is similar to that of \cite[Prop.\ 4.3]{BMS1}, which we now explain.
	Let $J$ be a fitting ideal of $M$, which under our assumption satisfies the equation in (\ref{eq:thm:BK_mod_invert_p_ideal}).
	By Weierstrass preparation theorem \cite[Page 17, Corollary]{Hoc12}, we know the ideal $J[1/p]$ is generated by an element $f(t)=\prod_{i=1}^m f_i(t)^{a_i}$, where each $f_i$ is a monic irreducible polynomial in $t$.
	Then by considering the $p$-adic valuation of the solution of each $f_i$, we know the only possible irreducible factor of $f(t)$ is $t$.
	Now we assume $J[1/p]$ is not equal to $A[1/p]$.
	Then the equality (\ref{eq:thm:BK_mod_invert_p_ideal}) implies that $\varphi_A(t)^{a_1}=t^{a_1}\cdot d^s\cdot u$ for some integer $s\in \mathbb{N}$ and some unit $u\in A^\times$.
	In particular, the irreducible element $d$ divides the element $\frac{\varphi_A(t)}{t}$.
	However, by \Cref{prop:Frob_reg_prism_of_dim_2}, since $d$ is an Eisinstein polynomial in $t$ and since the mod $p$ reduction of $\frac{\varphi_A(t)}{t}$ is $t^{p-1}$, the divisibility implies that $\deg(d)|(p-1)$.
	Hence the ramification degree of the field extension $\overline{A}[1/p]$ over $V[1/p]$ is dividing $(p-1)$, contradicting to the assumption on the ramification index of $\overline{A}$ in \Cref{thm:BK_mod_invert_p_explicit}.	
\end{proof}

\subsection{Primitive Purity theorem: mod $p^n$ reductions}
\label{sub:primitive_purity_reduction}
In this subsection, we show that Frobenius modules over the reduction of a transversal regular prism enjoy a similar purity result as that of crystalline or semi-stable local systems.
The proof strategy is inspired by the work of Du--Liu--Moon--Shimizu as in \cite{DLMS24}, \cite{DLMS24b}.

We start by introducing the category of tuples that consist of compatible Frobenius modules over various complete localizations.
Below we refer the reader to \Cref{def:Frobenius_mod} for the precise definition of Frobenius modules.
For the simplicity of the notations, we use $\mathcal{G}^{(\varphi)}_{(A,I),n}$ to denote either the category $\mathcal{G}_{(A,I),n}$ or the category $\mathcal{G}^{\varphi}_{(A,I),n}$. 
We also remind the reader of a convention that for a (derived) $p$-complete ring $A$, the quotient ring $A/p^nA$ is defined to be the ring $A$ itself when $n=\infty$.
\begin{definition}
	\label{def:extendable_tuples}
	Let $(A,I)$ be a transversal regular prism, and let $n\in \mathbb{N}\cup\{\infty\}$.
	We define the category $\mathcal{G}^{(\varphi)}_{(A,I),n}$ consisting of the tuples as below
	\[
	\bigl( M_\et; (M_{\mathfrak{p}}, \alpha_{\mathfrak{p}})_{\mathfrak{p}\in \mathrm{Ass}_{\overline{A}}(\overline{A}/p\overline{A})} \bigr),
	\]
	where \begin{itemize}
		\item $M_\et$ is a finite projective (Frobenius) module over $A\langle 1/I \rangle/p^nA\langle 1/I \rangle$;
		\item $M_{\mathfrak{p}}$ is a finite projective (Frobenius) module over $A_{L_{\mathfrak{p}}}/p^nA_{L_{\mathfrak{p}}}$, for each $\mathfrak{p}\in \mathrm{Ass}_{\overline{A}}(\overline{A}/p\overline{A})$;
		\item $\alpha_\mathfrak{p}:M_\et \otimes_{A\langle 1/I \rangle} A_{L_{\mathfrak{p}}}\langle 1/I \rangle \xrightarrow{\sim} M_{L_{\mathfrak{p}}} \otimes_{A_{L_{\mathfrak{p}}}}  A_{L_{\mathfrak{p}}}\langle 1/I \rangle$ is an isomorphism of (Frobenius) modules over $A_{L_{\mathfrak{p}}}\langle 1/I \rangle/p^nA_{L_{\mathfrak{p}}}\langle 1/I \rangle$, for each $\mathfrak{p}\in \mathrm{Ass}_{\overline{A}}(\overline{A}/p\overline{A})$.
	\end{itemize}
The morphisms are defined as maps of tuples that are compatible with the linear structures (and the Frobenius structures).
A tuple $\bigl( M_\et; (M_{\mathfrak{p}}, \alpha_{\mathfrak{p}})_{\mathfrak{p}\in \mathrm{Ass}_{\overline{A}}(\overline{A}/p\overline{A})} \bigr)$ in $\mathcal{G}^{\varphi}_{(A,I),n}$ has \emph{Frobenius height in $[a,b]$} if each $M_\mathfrak{p}$ has Frobenius height in $[a,b]$ for all $\mathfrak{p}\in \mathrm{Ass}_{\overline{A}}(\overline{A}/p\overline{A})$.
When $n=\infty$, we use $\mathcal{G}^{(\varphi)}_{(A,I)}$ to abbreviate the notation $\mathcal{G}^{(\varphi)}_{(A,I), \infty}$.
\end{definition}

Now we can present a primitive version of the purity result, for reflexive Frobenius modules over the mod $p^n$-reduction of an arbitrary transversal regular prism.
\begin{theorem}[Primitive Purity at the finite level]
	\label{thm:primitive_purity_reduction}
	Let $(A,I)$ be a transversal regular prism, and let $n\in \mathbb{N}$.
	There is a natural fully faithful functor 
	\begin{equation}
		\label{eq:primitive_purity_functor_reduction}
		\mathcal{G}^{(\varphi)}_{(A,I),n} \longrightarrow \Coh^{(\varphi)}_\refl(A/p^nA),
	\end{equation}
	The functor preserves the Frobenius heights, and its composition with the tensor product functor $\Coh^{(\varphi)}_\refl(A/p^nA) \to \mathcal{G}^{(\varphi)}_{(A,I),n}$ is equal to the identity functor of $\mathcal{G}^{(\varphi)}_{(A,I),n}$.
\end{theorem}
\begin{proof}
	We first give the construction of the functor.
	Let $\bigl( M_\et; (M_{\mathfrak{p}}, \alpha_{\mathfrak{p}})_{\mathfrak{p}\in \mathrm{Ass}_{\overline{A}}(\overline{A}/p\overline{A})} \bigr)\in \mathcal{G}_{(A,I),n}$.
	By the assumption on the finite projectivity and the injectivity of the localization maps in \Cref{prop:localization_injection}.\ref{prop:localization_injection_A}, both $M_\et$ and $\prod M_\mathfrak{p}$ are contained in $\widetilde{M}\colonequals M_\et \otimes_{A/p^nA[1/I]} \prod A_{L_\mathfrak{p}}/p^nA_{L_\mathfrak{p}}[1/I]\simeq \prod M_\mathfrak{p} [1/I]$, and it makes sense to consider the intersection $M\colonequals M_\et \cap \prod M_\mathfrak{p}$, as a sub $A/p^nA$-module in $M_\et$.
	It is then proved in \Cref{lem:intersection_is_fg} and \Cref{lem:intersection_is_refl} below that $M$ is finitely generated over $A/p^nA$ and is reflexive, hence is an object in $\Coh_\refl(A/p^nA)$.
	In addition, since $\varphi_A:A\to A$ is flat (\Cref{prop:reg_imply_Frob-regular}), by the fact that the flat base change commutes with the finite limit, we get
	\begin{align*}
		\varphi_A^*(M)[1/I] & = \varphi_A^*(M_\et \cap \prod M_\mathfrak{p}) [1/I] \\
		& \simeq (\varphi_A^* M_\et [1/I]) \cap ( \prod \varphi_A^* M_\mathfrak{p}[1/I] ).
	\end{align*}
So if we assume $\bigl( M_\et; (M_{\mathfrak{p}}, \alpha_{\mathfrak{p}})\bigr)\in \mathcal{G}^\varphi_{(A,I),n}$, then the Frobenius structure of $\bigl( M_\et; (M_{\mathfrak{p}}, \alpha_{\mathfrak{p}})\bigr)$ naturally equips the last term of the equation above with an isomorphism to $M_\et[1/I]\cap \prod M_\mathfrak{p}[1/I]=M[1/I]$.
Hence we obtain the Frobenius structure $\varphi_M$ on $M$.
If we further assume the tuple has Frobenius height in $[a,b]$, then we have the inclusions 
\[
I^b\cdot M_\et \subseteq \varphi_A^*M_\et \subseteq I^a \cdot M_\et, ~ \text{and}~I^b \cdot M_\mathfrak{p} \subseteq \varphi_A^* M_\mathfrak{p} \subseteq I^a\cdot M_\mathfrak{p}, ~\forall \mathfrak{p}.
\]
Notice that since $I$ is invertible and $\widetilde{M}$ is $I$-torsionfree, as $A/p^nA$-submodules in $\widetilde{M}$ we have $I^i\cdot M_\et\cap I^i \cdot \prod M_\mathfrak{p}= I^i\cdot M$ for each integer $i$.
Thus by taking the intersection we have that $(M,\varphi_M)$ has Frobenius height in $[a,b]$ as well.

	We then show that the following composition with the tensor product functor is the identity functor
	\[
		\mathcal{G}^{(\varphi)}_{(A,I),n} \longrightarrow \Coh^{(\varphi)}_\refl(A/p^nA) \longrightarrow \mathcal{G}^{(\varphi)}_{(A,I),n}.
	\]
	As the claim is Zariski local on $A$, we assume that the ideal $I=(d)$ is principal.
	Let $\bigl( M_\et; (M_{\mathfrak{p}}, \alpha_{\mathfrak{p}})\bigr)$ be an object in $\mathcal{G}_{(A,I),n}$, and let $M$ be the intersection $M_\et \cap \prod M_\mathfrak{p}$, which by the last paragraph is an object in $\Coh_\refl(A/p^nA)$.
	As the intersection commutes with the flat base change, the natural map $M\to M_\et$ becomes an isomorphism after inverting by $I$.
	In addition, by \Cref{lem:intersection_is_refl} we know the canonical linearlization $M\otimes_A \prod A_{L_\mathfrak{p}} \to \prod M_\mathfrak{p}$ is an isomorphism.
	Hence the tensor product functor naturally sends $M$ onto $\bigl(M_\et; (M_\mathfrak{p}, \alpha_\mathfrak{p})\bigr)$.
	If the input admits a Frobenius structure, then from the last paragraph, the Frobenius structure $\varphi_M$ on $M$ is defined via intersection.
	Noticing that since the images of $M$ in both $M_\et$ and $\prod M_\mathfrak{p}$ are generating, hence the base changes $\varphi_M\otimes_A A[1/I]$ and $\varphi_M \otimes_A \prod A_\mathfrak{p}$ coincide with $\varphi_{M_\et}$ and $\varphi_{\prod M_\mathfrak{p}}$ respectively.

	Finally, since the composition above is an equivalence of categories, the intersection functor $\mathcal{G}^{(\varphi)}_{(A,I),n} \longrightarrow \Coh^{(\varphi)}_\refl(A/p^nA)$ is faithful.
	To show the functor is full, we let $\bigl( M_\et; (M_{\mathfrak{p}}, \alpha_{\mathfrak{p}})\bigr)$ and $\bigl( N_\et; (N_{\mathfrak{p}}, \beta_{\mathfrak{p}}) \bigr)$ be two objects in the source, and let $M$ and $N$ be their corresponding intersection modules over $A/p^nA$.
	Then given a morphism $f:M\to N$ in $\Coh^{(\varphi)}(A/p^nA)$, we may define the natural base changes morphisms $g_\et\colonequals f[1/I]$ and $g_\mathfrak{p}\colonequals f \otimes_A A_{L_\mathfrak{p}}$.
	Then we obtain a natural arrow 
	\[
	g\colonequals (g_\et; (g_\mathfrak{p}))\colon \bigl( M_\et; (M_{\mathfrak{p}}, \alpha_{\mathfrak{p}})\bigr) \longrightarrow \bigl( N_\et; (N_{\mathfrak{p}}, \beta_{\mathfrak{p}}) \bigr).
	\]
	It is then left to check that under the functor (\ref{eq:primitive_purity_functor_reduction}), the image of the arrow $g$, which by definition is the restriction of $g$ on $M= M_\et \cap \prod M_\mathfrak{p}$, is equal to the arrow $f$.
	The latter follows quickly from the fact that $M$ is contained in $M[1/I]$, and by noticing that the two maps from $M$ to $N$ coincide after inverting by $I$.
\end{proof}

The following algebraic observations on the intersection module were used in the proof above.
\begin{lemma}[Finiteness]
	\label{lem:intersection_is_fg}
	Let $(A,I)$ be a transversal regular prism, and 
	let $\bigl( M_\et; (M_{\mathfrak{p}}, \alpha_{\mathfrak{p}})_{\mathfrak{p}\in \mathrm{Ass}_{\overline{A}}(\overline{A}/p\overline{A})} \bigr)\in \mathcal{G}_{(A,I),n}$ for some $n\in \mathbb{N}$.
	The intersection $M\colonequals M_\et \cap \prod M_\mathfrak{p}$ is finitely generated over $A$.
\end{lemma}
\begin{proof}
	As the statement is Zariski local with respect to $A$, by passing to an open subset of $\Spec(A)$ if necessary, we may assume the ideal $I=(d)$ is principal and $M_\et$ (and hence $\prod M_\mathfrak{p}[1/I]\simeq M_\et \otimes_A \prod A_{L_\mathfrak{p}}$) is finite free.
	We first notice that since $M[1/I]=M_\et$ is finite free over $A/p^n[1/I]$, there are finitely many elements $x_1,\ldots ,x_n\in M$ whose images in $M[1/I]$ form a basis.
	In addition, using the isomorphisms $\alpha_{\mathfrak{p}}$, the images of $(x_1,\ldots ,x_n)$ in $M_\et \otimes_{A/p^nA[1/I]} \prod A_{L_\mathfrak{p}}/p^nA_{L_\mathfrak{p}}[1/I] \simeq \prod M_\mathfrak{p}[1/I]$ form a basis of the latter over $\prod A_{L_\mathfrak{p}}/p^nA_{L_\mathfrak{p}}[1/I]$ as well.
	Thus by the finitely generatedness of $\prod M_\mathfrak{p}$, there is an integer $N\in \mathbb{N}$ such that $\prod M_\mathfrak{p}$ is contained in the $\prod A_{L_\mathfrak{p}}/p^nA_{L_\mathfrak{p}}$-submodule of $\prod A_{L_\mathfrak{p}}/p^nA_{L_\mathfrak{p}}[1/I]$ generated by $(\frac{x_1}{d^N},\ldots, \frac{x_n}{d^N})\in M_\et$.
	
	We now claim that, within $M_\et$, the $A/p^nA$-submodule $M$ is contained in the $A/p^nA$-submodule of $M_\et$ generated by $\frac{x_1}{d^N},\ldots, \frac{x_n}{d^N}$, which by the assumption that $A$ is noetherian would conclude the proof.
	To see the claim, we let $x=\sum_i a_i\cdot \frac{x_i}{d^N}$ be any element in $M$ for some $a_i\in A/p^nA[1/I]$.
	By construction, the image of the element $x$ in $\prod M_\mathfrak{p}[1/I]$ is contained in $\prod M_\mathfrak{p}$, and is in particular in contained in the free submodule $\bigoplus_i \prod A_{L_\mathfrak{p}}/p^nA_{L_\mathfrak{p}}\cdot \frac{x_i}{d^N} \subset \bigoplus_i \prod A_{L_\mathfrak{p}}/p^nA_{L_\mathfrak{p}}[1/I]\cdot \frac{x_i}{d^N}= \prod M_\mathfrak{p}[1/I]$.
	Thus the coefficients $a_i$, regarded as elements in $\prod A_{L_\mathfrak{p}}/p^nA_{L_\mathfrak{p}}[1/I]$, are also within the subring $\prod A_{L_\mathfrak{p}}/p^nA_{L_\mathfrak{p}}$.
	As a consequence, we get $a_i\in (A/p^nA[1/I])\cap (\prod A_{L_\mathfrak{p}}/p^nA_{L_\mathfrak{p}})$, which finishes the proof since the intersection ring by \Cref{cor:intersection_prism} is equal to $A/p^nA$.
\end{proof}
\begin{lemma}[Reflexivity and base change]
	\label{lem:intersection_is_refl}
	Let $(A,I)$ be a transversal regular prism, and let $\bigl( M_\et; (M_{\mathfrak{p}}, \alpha_{\mathfrak{p}})_{\mathfrak{p}\in \mathrm{Ass}_{\overline{A}}(\overline{A}/p\overline{A})} \bigr)\in \mathcal{G}_{(A,I),n}$ for some $n\in \mathbb{N}$.
	The intersection $M\colonequals M_\et \cap \prod M_\mathfrak{p}$ is reflexive over $A$ (in the sense of \Cref{def:ref_Frob_mod}), and the canonical map $M\otimes_A \prod A_\mathfrak{p} \to \prod M_\mathfrak{p}$ is an isomorphism.
\end{lemma}
\begin{proof}
	As the finite limit commutes with the flat base change, the localization $M[1/I]$ is naturally isomorphic to $M_\et$, which by assumption is finite projective over $A/p^nA[1/I]$.
	On the other hand, notice that by \Cref{prop:localization_injection}.\ref{prop:localization_injection_R}, the localization of $A/(p^n,I)A$ at a generic point is equal to one of the ring $A_{L_\mathfrak{p}}/(p^n,I)A_{L_\mathfrak{p}}$.
	Thus to check the generic locally freeness of $M/IM=M\otimes_{A/p^nA} A/(p^n,I)A$ (hence the reflexivity of $(M,\varphi_M)$), it suffices to show that the linearlization map $M\otimes_A A_{L_\mathfrak{p}}\to M_\mathfrak{p}$ is an isomorphism, for each $\mathfrak{p}\in \Ass_{\overline{A}}(\overline{A}/p\overline{A})$.
	Notice that using again the exactness of the flat base change along $A\to A_{L_\mathfrak{p}}$, we have 
	\begin{align*}
		M  \otimes_A A_{L_\mathfrak{p}} & = \bigl( \bigcap_{\mathfrak{p}'\in \Ass_{\overline{A}}(\overline{A}/p \overline{A})} (M_\et \cap M_\mathfrak{p'}) \bigr) \otimes_A A_{L_\mathfrak{p}} \\
		& \simeq \bigcap_\mathfrak{p'} \bigl( (M_\et \otimes_A A_{L_\mathfrak{p}}) \cap (M_{\mathfrak{p}'} \otimes_A A_{L_\mathfrak{p}}) \bigr)\\
		& \simeq \bigcap_\mathfrak{p'} \bigl( M_\mathfrak{p}[1/I] \cap (M_{\mathfrak{p}'} \otimes_A A_{L_\mathfrak{p}}) \bigr).
	\end{align*}
	Moreover, if $\mathfrak{p}'\neq \mathfrak{p}$, then since the mod $p^n$ reduction of them are distinct minimal prime ideals of the ideal $I\cdot A/(p^nA$ (\Cref{prop:localization_injection}.\ref{prop:localization_injection_R}), we know the ideal $I$ is invertible in the mod $p^n$ reduction of the ring $A_{L_\mathfrak{p'}}\otimes_A A_{L_{\mathfrak{p}}}$.
	Hence $M_{\mathfrak{p}'} \otimes_A A_{L_\mathfrak{p}}= M_{\mathfrak{p'}}\otimes_A (A_{L_\mathfrak{p'}}\otimes_A A_{L_\mathfrak{p}})$ is naturally isomorphic to $M_{\mathfrak{p'}}\otimes_A  A_{L_\mathfrak{p}}[1/I]$, which is further isomorphic to $M_\et \otimes_{A} A_{L_\mathfrak{p}}\otimes_A A_{L_{\mathfrak{p'}}}[1/I] \simeq M_\mathfrak{p}\otimes_A A_{L_\mathfrak{p'}}[1/I]$.
	As a consequence, continuing with the previous calculations, we get
	\begin{align*}
		M \otimes_A A_{L_\mathfrak{p}} & \simeq M_\mathfrak{p} \otimes_{A/p^nA} \bigl( A/p^nA[1/I] \cap (\prod_\mathfrak{p'} A_{L_{\mathfrak{p'}}}/p^n A_{L_{\mathfrak{p'}}}) \bigr) \\
		& \simeq M_{\mathfrak{p}} \otimes_A A=M_\mathfrak{p},
	\end{align*}
	where the second isomorphism is \Cref{cor:intersection_prism}.
\end{proof}

As we saw from the proof of \Cref{thm:primitive_purity_reduction}, the essential image of the functor (\ref{eq:primitive_purity_functor_reduction}) consists of the objects $(M,\varphi_M)\in \Coh^{(\varphi)}_\refl(A/p^nA)$ such that the natural injection $M\to (M[1/I] )\cap  (M\otimes_{A/p^nA} \prod A_{L_\mathfrak{p}}/p^nA_{L_\mathfrak{p}})$ is an isomorphism.
In general, the Frobenius module $(M,\varphi_M)$ is likely to be a proper submodule in the intersection, with the cokernel killed by a power of the ideal $I$. 
The last result in this subsection says that the difference of the above two Frobenius modules depends only on their Frobenius heights and the natural number $n$.
\begin{lemma}
	\label{lem:difference_of_two_Frob_mod}
	Let $(A,I)$ be a regular prism, and let $f:(M_1,\varphi_{M_1})\to (M_2,\varphi_{M_2})$ be an injection of $I$-torsionfree Frobenius modules in $\Coh^{\varphi,[a,b]}(A/p^nA)$ such that $M_1[1/I]=M_2[1/I]$.
	Then there is an integer $r=r(n;a,b)$, depending only on $n, a, b$, such that $\mathrm{Coker}(f)$ is $I^r$-torsion.
\end{lemma}
\begin{proof}
	We first assume the ideal $I$ admits a generator $d$.
	As the element $\varphi_A(d)-d^p$ is divided by $p$, by the binomial formula, we can find a large enough integer $s$ (of $p$-power) that depends only on $n$ such that $\varphi_A(d^s)-d^{ps}$ is divided by $p^n$ (\Cref{lem:power_of_Frobenius}).
	On the other hand, since both $M_1$ and $M_2$ have Frobenius height in $[a,b]$, their Frobenius structures induce $A/p^nA$-linear maps $\widetilde{\varphi}_{M_i}:\varphi_A^* M_i \to d^a\otimes M_i$.
	As the image of $\widetilde{\varphi}_{M_i}$ contains $d^b\otimes  M_i$, the inverse isomorphism $\varphi_{M_i}^{-1}$ can be extended to an $A/p^nA$-linear map $\widetilde{\psi}_{M_i}\colon d^a\otimes M_i\to \varphi_A^*M_i$, such that the composition $\varphi_A^*M_i \xrightarrow{\widetilde{\varphi}_{M_i}} d^a\otimes M_i \xrightarrow{\widetilde{\psi}_{M_i}} \varphi_A^*M_i$ is the multiplication by $d^{b-a}$, for $i=1,2$.
	
	We now assume the map $f$ is not an isomorphism, and let $M_3$ be $\mathrm{Coker}(f)$, which is a non-zero $I$-power torsion finitely generated module, and is equipped with the induced $A/p^nA$-linear morphisms $\widetilde{\varphi}_{M_3}:\varphi_A^* M_3 \to d^a\otimes M_3$ and $\widetilde{\psi}_{M_3}: d^a\otimes M_3\to \varphi_A^*M_3$, with the composition $\widetilde{\psi}_{M_3}\circ \widetilde{\varphi}_{M_3}=d^{b-a}$.
	We let $r$ be the smallest positive integer such that $M_3$ is equal to the torsion submodule $M_3[d^{sr}]$, so that the multiplication by $d^{s(r-1)}$ is non-zero on $M_3$.
	Then by taking the Frobenius twist, we have the equalities of torsion modules \[
	\varphi_A^*M_3=\varphi_A^*(M_3[d^{sr}])=(\varphi_A^*M_3)[\varphi_A(d)^{sr}],
	\] 
	which (by the choice of $s$) is further equal to $(\varphi_A^*M_3)[d^{psr}]$.
	Moreover, since $\varphi_A$ is faithfully flat (\Cref{prop:reg_imply_Frob-regular}), we know the multiplication by the element $\varphi_A(d)^{s(r-1)}=d^{ps(r-1)}$ on $(\varphi_A^*M_3)$ is non-zero.
	On the other hand, as the module $d^a\otimes M_3$ is killed by $d^{sr}$ and the composition $\widetilde{\psi}_{M_3}\circ \widetilde{\varphi}_{M_3}$ is equal to the multiplication by $d^{b-a}$ on $\varphi_A^*M_3$, we know the element $d^{sr}\cdot d^{b-a}=d^{sr+b-a}$ kills the module $\varphi_A^*M_3$ as well.
	Hence combining the above analysis, we get the inequality $sr+b-a > ps(r-1)$, which implies that $r \leq \frac{b-a+ps}{(p-1)s}$.
	Notice that the positive integer $s$ can be chosen to be only depending on $n$.
	Hence the claim follows by taking the roof function of the above positive rational number.
\end{proof}

\subsection{Primitive Purity theorem: infinite level}
\label{sub:primitive_purity_infinity}
In this subsection, we extend the purity result for mod $p^n$ modules to the infinite level, for Frobenius modules over regular prisms.

We start with the following limit presentation.
\begin{proposition}[Limit presentation of Frobenius modules]
	\label{prop:Frob_mod_and_its_reduction}
	Let $(A,I)$ be a transversal regular prism.
	Then the mod $p^n$ reduction functors induce a fully faithful functor as below
	\begin{equation}
		\label{eq:Frob_mod_and_its_reduction}
			\Coh^{(\varphi)}_\refl(A) \longrightarrow 
		\lim_n \Coh^{(\varphi)}_\refl(A/p^nA).
	\end{equation}
In addition, a Frobenius module $(M,\varphi_M)\in \Coh^{\varphi}_\refl(A)$ has Frobenius height $[a,b]$ if and only if its each mod $p^n$ reduction does so.
\end{proposition}
\begin{proof}
	To start, we claim that the image of the tensor product functor $\Coh_{\refl}(A) \to \Coh(A/p^nA)$ is contained in the subcategory $\Coh_{\refl}(A/p^nA)$.
	Namely for $M\in \Coh_\refl(A)$, we want to check that $M/p^nM$ is $I$-torsionfree, the localization $M/p^nM[1/I]$ is finite projective over $A/p^nA[1/I]$, and the reduction $M/(p^n,I)M$ is generically locally free over $A/(p^n,I)A$:
	By \Cref{def:ref_Frob_mod}, since $M[1/I]$ is finite projective over $A[1/I]$, its mod $p^n$-reduction is finite projective over $A/p^nA[1/I]$ as well.
	Moreover, by \Cref{lem:reflexive_is_saturated}.\ref{lem:reflexive_is_saturated_red}, the reflexivity assumption of $M$ implies that $M/p^nM$ is $I$-torsionfree.
	Furthermore, since $M$ is reflexive and the map $A\to A_{L_\mathfrak{p}}$ is flat for each $\mathfrak{p}\in \Ass_(\overline{A}/p\overline{A})$, the base change $M\otimes_A A_{L_\mathfrak{p}}$ is reflexive and hence finite projective (\Cref{cor:reflexive_is_locally_free}) over $A_{L_\mathfrak{p}}$.
	In particular, the mod $(I,p^n)$ reduction of $M\otimes_A A_{L_\mathfrak{p}}$ is finite projective over $A_{L_\mathfrak{p}}/(p^n,I)A_{L_\mathfrak{p}}$.
	Hence $M/(p^n,I)M$ is generically locally free.
	As a consequence, the limit of the base change functors naturally induce a functor as in (\ref{eq:Frob_mod_and_its_reduction}).
	Notice that since $A$ is $p$-complete and $p$-torsionfree, we naturally have $M=\lim_n M/p^nM$.
	In particular, the induced functor on the underlying modules $\Coh_{\refl}(A) \longrightarrow 
	\lim_n \Coh_{\refl}(A/p^nA)$ is fully faithful.
	
	Now we prove that the functor (\ref{eq:Frob_mod_and_its_reduction}) for Frobenius modules is fully faithful.
	Let $(M,\varphi_M)$ and $(N,\varphi_N)$ be two objects in $\Coh^\varphi_\refl(A)$, and let $f:M\to N$ be a map of $A$-modules, such that $f_n\colonequals f\otimes_A A/p^nA$ is Frobenius equivariant for each $n\in \mathbb{N}$.
	By the first paragraph above, it suffices to show that $f$ is Frobenius equivariant itself.
	Note that by assumption, each $f_n$ fits into the commutative diagram
	\begin{equation}
	\label{eq:diagram_of_Frob_mod_n}
	\begin{tikzcd}
	\varphi_A^* M/p^nM [1/I] \arrow[r,"\varphi_A^*(f_n){[1/I]}"] \arrow[d,"\varphi_M \otimes A/p^nA{[1/I]}"']  & \varphi_A^* N/p^nN [1/I] \arrow[d,"\varphi_N \otimes A/p^nA{[1/I]}"] \\
		M/p^nM [1/I] \arrow[r,"f_n{[1/I]}"] & N/p^nN[1/I].
	\end{tikzcd}
\end{equation}
By taking the limit, we get 
\begin{equation}
	\label{eq:diagram_of_Frob_limit}
	\begin{tikzcd}
		(\varphi_A^* M [1/I])^\wedge_p \arrow[rr, "{\lim \varphi_A^*(f_n){[1/I]}}"]  \arrow[d, "{\varphi_M \otimes A \langle 1/I \rangle}"']  &&(\varphi_A^* N [1/I])^\wedge_p  \arrow[d, "{\varphi_N \otimes (A{[1/I]})^\wedge_p}"] \\
		M \langle 1/I \rangle \arrow[rr, "{\lim (f_n[1/I])}"]&& N \langle 1/I \rangle.
	\end{tikzcd}
\end{equation}
Notice that the first vertical map in (\ref{eq:diagram_of_Frob_limit}) is the base change of the Frobenius structure $\varphi_M$.
So by precomposing with the Frobenius structure of $M$, we get the following enlarged commutative diagram
\begin{equation}
	\label{eq:diagram_of_Frob_larger}
	\begin{tikzcd}
		\varphi_A^*M[1/I] \ar[r] \arrow[d, "{\varphi_M}"'] & (\varphi_A^* M [1/I])^\wedge_p \arrow[rr, "{\lim \varphi_A^*(f_n){[1/I]}}"] \arrow[d, "\varphi_M \otimes (A{[1/I]})^\wedge_p"] &&(\varphi_A^* N [1/I])^\wedge_p \arrow[d, "{\varphi_N \otimes (A{[1/I]})^\wedge_p}"] \\
		M[1/I]   \ar[r]& M \langle 1/I \rangle \arrow[rr, "{\lim f_n{[1/I]}}"]  && N \langle 1/I \rangle.
	\end{tikzcd}
\end{equation}
On the other hand, the composition of top row in (\ref{eq:diagram_of_Frob_larger}), which has nothing to do with the Frobenius structures, can be also factored as the composition $\varphi_A^*M[1/I] \xrightarrow{\varphi_A^* f{[1/I]}} \varphi_A^* N[1/I] \to (\varphi_A^* N[1/I])^\wedge_p$.
Hence by combining the aforementioned factorization and the commutativity in (\ref{eq:diagram_of_Frob_larger}), we get a commutative diagram
\begin{equation}
	\label{eq:diagram_of_Frob_outcome}
	\begin{tikzcd}
		\varphi_A^* N[1/I] \arrow[r] \arrow[d]  & (\varphi_A^*N[1/I])^\wedge_p \arrow[d, "{\varphi_N \otimes (A{[1/I]})^\wedge_p}"]\\
		M[1/I] \ar[r] & N \langle 1/I \rangle,
	\end{tikzcd}
\end{equation}
where the right vertical map is an isomorphism, and the top map is injective, thanks to the analytic locally freeness of $N$ and \Cref{cor:reg_prism_is_sat}.
Notice that since the isomorphism $\varphi_N \otimes A \langle 1/I \rangle: (\varphi_A^*N[1/I])^\wedge_p \xrightarrow{\sim}  N \langle 1/I \rangle$ is equal to to the base change of the isomorphism $\varphi_N:  \varphi_A^*N[1/I] \to N[1/I]$, the intersection $N \langle 1/I \rangle \cap \varphi_A^* N [1/I]$ inside $N \langle 1/I \rangle\simeq (\varphi_A^*N[1/I])^\wedge_p$ is exactly the submodule $N[1/I]$.
Hence the commutativity of the diagram (\ref{eq:diagram_of_Frob_outcome}) implies that the $M[1/I]$ maps into $N[1/I]$, which finishes the proof.

Finally, we check the claim on Frobenius height.
If $(M,\varphi_M)\in \Coh^{\varphi,[a,b]}_\refl(A)$, then by definition we have the containments of the submodules in $M[1/I]$ as in (\ref{eq:def_Frob_height}).
Thus by taking the image of them into $M/p^nM$, each its mod $p^n$ reduction has the same Frobenius height as well.
Conversely, if each $(M/p^nM,\varphi_{M/p^nM})$ has Frobenius height $[a,b]$, then by taking the limit for the sequences (\ref{eq:def_Frob_height}) of $M/p^nM$, we see $(M,\varphi_M)$ admits the same type of sequence as well.
\end{proof}
\begin{corollary}
	\label{cor:gluing_tuple_and_their_reduction}
	Let $(A,I)$ be a transversal regular prism.
	The base change functor induces a fully faithful functor 
	\begin{equation}
		\label{eq:gluing_tuple_and_its_reduction}
		\mathcal{G}^{(\varphi)}_{(A,I)} \longrightarrow 
		\lim_n \mathcal{G}^{(\varphi)}_{(A,I),n},
	\end{equation}
which preserves the Frobenius heights.
\end{corollary}
\begin{proof}
	We notice that the each ring appeared in the definition of $\mathcal{G}_{(A,I)}$ is noetherian.
	Thus for the underlying categories of modules $\mathcal{G}_{(A,I),n}$, the fully faithfulness of the functor (\ref{eq:gluing_tuple_and_its_reduction}) follows from the locally freeness assumption of the reductions, together with the flatness of the limit in \cite[\href{https://stacks.math.columbia.edu/tag/0912}{Tag 0912}]{stacks-project}.
	For the finite projective Frobenius modules over the ring $A\langle 1/I \rangle$, the fully faithfulness of the limit functor (\ref{eq:gluing_tuple_and_its_reduction}) follows by the same proof as in the last paragraph of the proof for \Cref{prop:Frob_mod_and_its_reduction}.
	As a consequence, the fully faithfulness for the functor $\mathcal{G}^{\varphi}_{(A,I)} \longrightarrow 
	\lim_n \mathcal{G}^{\varphi}_{(A,I),n}$ follows by applying \Cref{prop:Frob_mod_and_its_reduction} at the regular prism $(A_{L_\mathfrak{p}}, IA_{L_\mathfrak{p}})$, together with the aforementioned fully faithfulness for Frobenius modules over $A\langle 1/I \rangle$.
\end{proof}
Now we present the primitive purity theorem for Frobenius modules over $A$.
\begin{theorem}[Primitive Purity at the infinite level]
	\label{thm:primitive_purity_inf}
	Let $(A,I)$ be a transversal regular prism.
	There is a natural faithful functor
		\begin{equation}
			\label{eq:primitive_purity_functor}
			\mathcal{G}^{\varphi}_{(A,I)} \longrightarrow \Coh^\varphi_\sat(A), \quad (M_\et; (M_\mathfrak{p}, \alpha_\mathfrak{p})) \longmapsto M.
		\end{equation}
	It satisfies the following properties:
	\begin{itemize}
		\item The functor preserves the Frobenius heights.
		\item The composition with the tensor product functor $\Coh^\varphi(A)\to \Coh^\varphi(A\langle 1/I \rangle)$ naturally sends $M$ onto $M_\et$.
		\item For each $\mathfrak{p}\in \Ass_{\overline{A}}(\overline{A}/p\overline{A})$, there is a natural injection of Frobenius modules $M\otimes_A A_{L_\mathfrak{p}}\hookrightarrow M_\mathfrak{p}$, which becomes isomorphic after base changing to $A_{L_\mathfrak{p}}\langle 1/I \rangle$.
	\end{itemize}
\end{theorem}
\begin{remark}[Refined versions]
	Before the proof, we comment on some special cases where the result admits automatic improvements.
	\begin{enumerate}
		\item If the \'etale realization functors $\Vect^\varphi(A_{L_\mathfrak{p}})\to \Vect^\varphi(A_{L_\mathfrak{p}}\langle 1/I \rangle)$ are fully faithful for each associated prime ideal $\mathfrak{p}$, then the composition of the functor (\ref{eq:primitive_purity_functor}) with the tensor product functor $\Coh^\varphi_\sat(A) \to \mathcal{G}^{\varphi}_{(A,I)}$ is naturally equivalent to the identity functor on $\mathcal{G}^{\varphi}_{(A,I)}$.
		\item If the analytic locally freeness in \Cref{thm:Frob_mod_is_analyticall_loc_free} holds true for all the $p$-torsionfree Frobenius modules over $A$ (for example when $(A,I)$ is Breuil--Kisin or satisfies the assumption on ramification in \Cref{cor:Frob_mod_when_the_reduction_is_mildly_ramified}), then the essential image of (\ref{eq:primitive_purity_functor}) is contained in the subcategory $\Coh_\refl^\varphi(A)$.
		In particular, each $M$ is analytically locally free.
	\end{enumerate}
\end{remark}
\begin{proof}[Proof of \Cref{thm:primitive_purity_inf}]
	\noindent \textit{Step 1: The underlying $A$-module.}
	Let $(M_\et; (M_\mathfrak{p}, \alpha_\mathfrak{p}))$ be an object in $\mathcal{G}_{(A,I)}$.
	By the finite projectivity assumption of objects in $\mathcal{G}_{(A,I)}$ and by the injectivity in \Cref{prop:localization_injection}.\ref{prop:localization_injection_A}, the isomorphisms $\alpha_{\mathfrak{p}}$ induce injections $M_\et \to \widetilde{M}\colonequals M_\et\otimes_{A\langle 1/I \rangle} \prod A_{L_\mathfrak{p}}\langle 1/I \rangle \simeq \prod M_\mathfrak{p}\otimes_{A_{L_{\mathfrak{p}}}} A_{L_\mathfrak{p}}\langle 1/I \rangle$, and in particular, it makes sense to consider the intersection $M\colonequals M_\et \cap \prod M_\mathfrak{p}$ within $\widetilde{M}$.
	Note that since both $M_\et$ and $\prod M_\mathfrak{p}$ are classical (and thus derived) $p$-complete, by \cite[Lem.\ 6.15]{BMS1}, we know $M$ is derived $p$-complete as well.
	Moreover, the finite projectivity assumption implies that $M_\et$ and $M_\mathfrak{p}$ are $p$-torsionfree.
	Thus the submodule $M$ is also $p$-torsionfree and is in particular classical $p$-complete, and we get the isomorphism $M=\lim M/p^nM$.
	\\
	
	\noindent \textit{Step 2: The finiteness and the limit presentation.}
	We let $(M_{\et,n}; (M_{\mathfrak{p},n}, \alpha_\mathfrak{p}))\in \mathcal{G}_{(A,I),n}$ be the mod $p^n$ reduction of $(M_\et; (M_\mathfrak{p}, \alpha_\mathfrak{p}))$, and as in the proof of \Cref{thm:primitive_purity_reduction} it makes sense to consider the intersection $M_n$, which by \textit{loc.\ cit.} is a reflexive finitely presented module over $A/p^nA$.
	Then we claim that the natural map $M/p^nM \to M_n$ is injective.
	Granting the claim, by the finiteness of $M_n$ and the $p$-completeness of $M$, we know $M$ is finitely generated over $A$.
	To see the claim, if the map is not injective, then there is an element $m\in M$ whose image in $M_\et \oplus \prod M_\mathfrak{p}$ is contained in $\ker\bigl( (M_\et \oplus \prod M_\mathfrak{p}) \to (M_{\et,n} \oplus \prod M_{\mathfrak{p},n}) \bigr)$.
	The finite projectivity assumption then implies that the image of $m$ in $M_\et \oplus \prod M_\mathfrak{p}$ is contained in $p^nM_\et \oplus \prod p^nM_\mathfrak{p}$, which we denote as $p^n\cdot (m'_\et, (m'_\mathfrak{p})_\mathfrak{p})$ for some elements $m'_\et \in M_\et$ and $m'_\mathfrak{p} \in M_\mathfrak{p}$.
	Then we get $p^n\cdot m'_\et - p^n(m'_\mathfrak{p})_\mathfrak{p}=0$ in $M_\et \otimes \prod A_{L_\mathfrak{p}}\langle 1/I \rangle \simeq \prod M_\mathfrak{p}\langle 1/I \rangle$.
	But notice that the latter is $p$-torsionfree.
	Hence we have $m'_\et = (m'_\mathfrak{p})_\mathfrak{p}$, which then comes from an element $m'$ in $M$ such that $p^n\cdot m'=m$.
	This shows that $m\in p^nM$, and thus the injectivity of $M/p^nM\to M_n$.
	
	In addition, notice that the finite projectivity assumption on $M_\et$ (resp. $M_\mathfrak{p}$) implies its limit presentation $M_\et=\lim_n M_{\et,n}$ (resp. $M_\mathfrak{p}=\lim_n M_{\mathfrak{p},n}$).
	Hence by commuting the limit with the intersection, we have 
	\begin{equation}
		\label{eq:thm:primitive_purity_inf_limit}
		M=M_\et \cap \prod M_\mathfrak{p} = (\lim M_{\et,n}) \cap (\lim \prod M_{\mathfrak{p},n}) = \lim (M_{\et,n} \cap \prod  M_{\mathfrak{p},n}) = \lim M_n.
	\end{equation}
	\\
	
	\noindent \textit{Step 3: The Frobenius structure $\varphi_M$.}
	We claim that the Frobenius structures $\varphi_{M_\et}$ and $\varphi_{M_\mathfrak{p}}$ naturally induce a Frobenius structure over $M$.
	By assumption that $(A,I)$ is regular, we know $\varphi_A:A\to A$ is flat (\Cref{prop:reg_imply_Frob-regular}).
	In particular, by the fact that flat base change commutes with the finite intersection, we have
	\begin{align*}
		\varphi_A^* M[1/I] & = \varphi_A^*(M_\et \cap \prod M_\mathfrak{p}) [1/I] \\
		& = (\varphi_A^*M_\et [1/I]) \cap (\varphi_A^* \prod M_\mathfrak{p} [1/I] \\
		& \xrightarrow{\sim} (M_\et[1/I]) \cap (\prod M_\mathfrak{p}[1/I]) \\
		& = M[1/I],
	\end{align*}
	where the isomorphism in the third line above is given by the isomorphisms $\varphi_{M_\et}$ and $\varphi_{M_\mathfrak{p}}$, which by assumption are compatible within their common ambient module $M_\et \otimes \prod A_{L_\mathfrak{p}} \langle 1/I \rangle \simeq \prod M_\mathfrak{p} \langle 1/I \rangle$.
	Thus we obtain a natural Frobenius structure on $M$, which we denote as $\varphi_M$.
	
	Here we note that as we saw from the proof of \Cref{thm:primitive_purity_reduction}, each individual $A/p^nA$-module $M_n$ admits a Frobenius structure as well, defined by the same formula as above (but for $M_{\et,n}$ and $M_{\mathfrak{p},n}$).
	So by their constructions, the mod $p^n$ reduction of $\varphi_M$ is naturally compatible with the Frobenius structure $\varphi_{M_n}$ of $M_n$ along the injection $M/p^nM\to M_n$.
	
	Moreover, notice that for each $i\in \mathbb{Z}$, as $A$-submodules of $M[1/I]$ we have $I^i\cdot M_\et \cap I^i \cdot \prod M_\mathfrak{p}=I^i\cdot (M_\et \cap \prod M_\mathfrak{p})=I^i\cdot M$.
	Thus if each $M_\mathfrak{p}$ (and trivially for $M_\et)$ has Frobenius height $[a,b]$, then by taking the intersection of (\ref{eq:def_Frob_height}), we see $(M,\varphi_M)$ has Frobenius height $[a,b]$ as well.
	\\
	
	\noindent \textit{Step 4: The saturatedness.}
	By assumption, since $M$ is a submodule in the finite projective $A\langle 1/I \rangle$-module $M_\et$, we know $M$ is torsionfree.
	we want to show that $M$ satisfies the condition in \Cref{lem:reflexive_is_saturated}.\ref{lem:reflexive_is_saturated_sat} with respect to $p$ and (local generators of) $I$, namely the injection $M\to M[1/p]\cap M[1/I]$ is an isomorphism.
	Notice that by commuting the flat base change with the finite intersection, we get
	\begin{align*}
		M[1/p]\cap M[1/I] & = (M_\et \cap \prod M_\mathfrak{p})[1/p] \cap (M_\et \cap \prod M_\mathfrak{p})[1/I]\\
		& =  (M_\et [1/p] \cap M_\et [1/I]) \cap \prod (M_\mathfrak{p} [1/p] \cap M_\mathfrak{p} [1/I]).
	\end{align*}
	Since each $M_\mathfrak{p}$ is locally free over $A_{L_\mathfrak{p}}$ and in particular is reflexive, we have $(M_\mathfrak{p} [1/p] \cap M_\mathfrak{p} [1/I]) = M_\mathfrak{p}$.
	On the other hand, since $A\langle 1/I \rangle = A\langle 1/I \rangle[1/p] \cap A\langle 1/I \rangle[1/I]$, by the finite projectivity of $M_\et$ over $A\langle 1/I \rangle$, we get $M_\et [1/p] \cap M_\et [1/I] = M_\et$.
	Hence continuing with the calculation above, we get $M[1/p]\cap M[1/I] = M_\et \cap \prod M_\mathfrak{p}= M$.\\

	\noindent \textit{Step 5: Composition with the tensor product functor $\Coh^\varphi_\refl(A)\to \mathcal{G}^{\varphi}_{(A,I)}$.}
	We now consider the composition of \ref{eq:primitive_purity_functor} with the tensor product functor $\Coh_{\refl}(A)\to \mathcal{G}_{(A,I)}$. 
	We first claim that the linearization map $M\otimes_A A\langle 1/I \rangle \to M_\et$ is an Frobenius equivariant isomorphism.
	As both objects are classically $p$-complete and $p$-torsionfree, it suffices to show the isomoprhims for the mod $p$ reductions, namely
	\[
	M/pM[1/I] \to M_{\et,1},
	\]
	where the latter by \Cref{thm:primitive_purity_reduction} is naturally identified with $M_1[1/I]$.
	From Step 1, we know $M/pM\to M_1$ is injective, which implies the injectivity of the localization $M/pM[1/I] \to M_{\et,1}$.
	Moreover, by the limit presentation $M=\lim M_n$ in Step 2, we know $M/pM$ is equal to the intersection of Frobenius modules	$\cap_{n>0} \mathrm{Im}(M_n\to M_1)$.
	Notice that \Cref{thm:primitive_purity_reduction} implies the equality $M_n[1/I]=M_{\et,n}$ for each $n$, and by our assumption the mod $p$ reduction of $M_{\et,n}$ is equal to $M_{\et,1}$.
	Hence each $\mathrm{Im}(M_n\to M_1)$ is a Frobenius submodule of $M_1$ with same Frobenius heights such that $\mathrm{Im}(M_n\to M_1)[1/I]=M_1[1/I]$.
	As a consequence, by \Cref{lem:difference_of_two_Frob_mod} we know there is a positive integer $r=r(1;a,b)$ such that the cokernel of the map $\mathrm{Im}(M_n\to M_1) \to M_1$ is killed by $I^r$, for each $n>0$.
	The latter implies that the submodule $I^rM_1$ is contained in the intersection $\cap_{n>0} \mathrm{Im}(M_n\to M_1)=M/pM$, and hence $M/pM[1/I]=M_1[1/I]$.
	
	We then check that each $M\otimes_A A_{L_\mathfrak{p}}\to M_\mathfrak{p}$ is an injection and induces an isomorphism after base changing to $A_{L_\mathfrak{p}}\langle 1/I \rangle$.
	We first notice that since the mod $p$ reduction of the map is $M/pM\otimes_A A_{L_\mathfrak{p}} \to M_\mathfrak{p}/pM_\mathfrak{p}$, where the target by \Cref{thm:primitive_purity_reduction} is equal to $M_1\otimes_A A_{L_\mathfrak{p}}$, its injectivity follows from that of $M/pM\to M_1$, which was proved in Step 4.
	So by the completeness and torsionfreeness, we know $M\otimes_A A_{L_\mathfrak{p}}\to M_\mathfrak{p}$ is injective as well.
	The isomorphism over $A_{L_\mathfrak{p}}\langle 1/I \rangle$ follows from the last paragraph together with the input isomorphism $\alpha_{\mathfrak{p}}:M_\et \otimes_{A\langle 1/I \rangle} A_{L_\mathfrak{p}}\langle 1/I \rangle \simeq M_{L_\mathfrak{p}}\langle 1/I \rangle$.
\end{proof}
\begin{remark}
	For the underlying $A$-modules without the Frobenius structures, it follows from the proof of \Cref{thm:primitive_purity_inf} that the functor (\ref{eq:primitive_purity_functor}) can be factorized as below:
	\[
	\begin{tikzcd}
		\mathcal{G}_{(A,I)} \ar[d] \ar[r] & \Coh_\sat(A) \\
		\lim_n \mathcal{G}_{(A,I),n} \ar[r]& \lim_n \Coh_\refl(A) \ar[u],
	\end{tikzcd}
\]
where the bottom functor is (\ref{eq:primitive_purity_functor_reduction}).
\end{remark}

\section{Integral pro-\'etale cohomology and its finiteness}
\label{sec:galois_coh}
In this section, we setup the basics of the pro-\'etale geometry needed in the article, give a projection formula for integral continuous group cohomology, and recollect various finiteness results of pro-\'etale cohomology of integral and rational coefficients.

We setup the convention of the pro-\'etale site used in the article. 
Fix a $p$-adic complete discrete valuation ring $\mathcal{O}_K$ with perfect residue field $k$.
Let $\Gal_K$ be the absolute Galois group of a fixed algebraic closure $\overline{K}$ over $K$, and let $C$ be the completion of $\overline{K}$, with $\mathcal{O}_C$ its ring of integers.
Here we will implicitly fix a cardinality and assume all the constructions are \emph{small} (cf. \cite[\S\ 4]{Sch22}).
\begin{definition}
	\label{def:proetale_site}
	\begin{enumerate}
		\item We let $\Perfd_\pe$ be the \emph{big pro-\'etale site} on the category of perfectoid spaces over $\mathbb{Q}_p$, together with the pro-\'etale topology, as in \cite[Def.\ 8.1]{Sch22}.
		\item For a locally spatial diamond $Z$, we let $Z_{\pe}$ be the \emph{pro-\'etale site of $Z$} of perfectoid spaces over $Z$, defined as the sliced site $\Perfd_\pe|_{Z}$.
	\end{enumerate}
\end{definition}
\begin{remark}
	The cautious reader may notice that the notion of the pro-\'etale site used in this article is different from the original construction in \cite{Sch13}, where various period sheaves were introduced.
	We remark that the only properties that are relevant to the discussion of the period sheaves are the sheafifiness of the (tilted) pro-\'etale structure sheaves and the acyclicity of their cohomology on affinoid perfectoid spaces, which are guaranteed by \cite[Prop.\ 8.5]{Sch22}.
\end{remark}

We also setup the convention on the generic fiber.
\begin{definition}
	Let $X$ be bounded $p$-adic formal scheme.
	We let $X_\eta$ be the \emph{generic fiber of $X$}, regarded as a locally special diamond in $\Shv(\Perfd_\pe)$ as in \cite[Def.\ 15.5]{Sch22}. 
\end{definition}
\begin{remark}
	Assume $X$ is a topologically of finite type $p$-adic formal scheme over $\mathcal{O}_K$, so that its Raynaud's generic fiber is a rigid space over $K$.
	Then thanks to \cite[Lem.\ 15.6]{BS22} and \cite[Prop.\ 10.2.3]{SW20}, the (finite) \'etale site of this rigid space is equivalent to the (finite) \'etale site of the diamond $X_\eta$.
	For this reason, we will identify the Raynaud's generic fiber of $X$ with its associated diamond $X_\eta$, and call both of them the \emph{generic fiber of $X$}.
\end{remark}

To calculate the pro-\'etale cohomology, it is often useful to translate it into certain continuous group cohomology.
As a preparation, we recall the following construction of the algebraic fundamental group.
\begin{construction}[Algebraic fundamental group]
	\label{const:algebrac_pi_1}
	Let $X$ be a connected normal $p$-adic formal scheme over $\mathcal{O}_K$, and let $X_\eta$ be its generic fiber.
	Let $\bar{x}_\eta$ be a geometric point of $X_\eta$.
	We let $\widetilde{X}_\eta$ be a fixed maximal connected pro-finite-\'etale cover of $X_\eta$, which is a perfectoid space in $X_{\eta,\pe}$.
	Then the canonical map $\widetilde{X}_\eta \to X_\eta$ is a Galois cover, and its associated Galois group is the \emph{algebraic fundamental group} $G_{X_\eta}\colonequals \pi_1^\mathrm{alg}(X_\eta, \bar{x}_\eta)$.	
	In addition, the map $\widetilde{X}_\eta \to X_\eta$ factors through the base extension $X_{\eta,C}\to X_\eta$, which corresponds to the surjection $G_{X_\eta}\to \Gal_K$ and we use $G_{X_{\eta,C}}$ to denote the kernel.
	In the special case when $X=\Spf(R)$ is affine, 
	the adic space $\widetilde{X}_\eta$ is affinoid perfectoid and is of the form $\Spa(S[1/p],S)$, where $S$ is the $p$-completion of the integral closure of $R$ in a fixed connected maximal pro-finite-\'etale extension of $R[1/p]$, and admits a map from the ring of integers $\mathcal{O}_C$.
\end{construction}


\subsection{Projection formula}
\label{sub:proj_formula}
We prove a projection formula for continuous group cohomology in this subsection.

To prepare, we start with a discreteness result of the completely flat tensor product.
\begin{lemma}
	\label{lem:complete_tensor_product_w_flat}
	Let $R$ be a ring, let $(f_1,\ldots,f_d)$ be a Koszul regular sequence in $R$, and assume $R$ is complete under the $I=(f_1,\ldots,f_d)$-adic topology.
	Assume $M$ and $N$ are two derived $I$-complete $R$-modules, such that 
	\begin{itemize}
		\item $M$ is $I$-completely flat over $R$;
		\item $(f_1,\ldots,f_d)$ is a Koszul regular sequence for $N$.
	\end{itemize}
	Then the derived $I$-adic complete tensor product $(M\otimes^L_R N)^\wedge_I$ is naturally isomorphic to the classical completion of the classical tensor product, namely $\lim_n (M\otimes_R N)/I^n$.
	In particular, it lives in cohomological degree zero.
\end{lemma}
\begin{proof}
	By the first assumption, the derived $I$-completion $(M\otimes^L_R N)^\wedge_I$ is $I$-completely flat over $N$.
	This implies that the Koszul complex $\mathrm{Kos}(M\otimes^L_R N; f_1^n,\ldots, f_d^n)$ is a flat complex over $\mathrm{Kos}(N; f_1^n,\ldots, f_d^n)$.
	On the other hand by the second assumption, the Koszul complex $\mathrm{Kos}(N; f_1^n,\ldots, f_d^n)$ is a discrete module.
	Note that by definition, the derived completion $(M\otimes^L_R N)^\wedge_I$ is $R\lim_n \mathrm{Kos}(M\otimes^L_R N; f_1^n,\ldots, f_d^n)$.
	As a consequence, since  $\{\mathrm{Kos}(M\otimes^L_R N; f_1^n,\ldots, f_d^n)\}_n$ is a pro-system of discrete modules with surjective transition maps, we see the derived complete tensor product $(M\otimes^L_R N)^\wedge_I$ lives in cohomological degree zero.
	
	Moreover, since $(f_1,\ldots,f_d)$ is Koszul regular in $R$, using the noetherian approximation and \cite[\href{https://stacks.math.columbia.edu/tag/0925}{Tag 0925}, \href{https://stacks.math.columbia.edu/tag/0921}{Tag 0921}]{stacks-project}, we see the pro-systems $\{\mathrm{Kos}(R, f_1^n,\ldots,f_d^n)\}_n$ and $\{R/I^n\}$ are pro-isomorphic to each other,
	and the derived completion $(M\otimes^L_R N)^\wedge_I$ is isomorphic to $R\lim_n \bigl( (M\otimes^L_R N) \otimes^L_R R/I^n \bigr)$.
	Finally, since the derived tensor product $(M\otimes^L_R N) \otimes^L_R R/I$ by assumptions is isomorphic to $M/IM \otimes_{R/IR} N/IN$, the discreteness of the derived completion $(M\otimes^L_R N)^\wedge_I$  implies that $R\lim_n \bigl( (M\otimes^L_R N) \otimes^L_R R/I^n \bigr)\simeq \lim_n (M\otimes_R N)/I^n$.
\end{proof}
Another preparation is the following vanishing result on the limit of the first continuous group cohomology.
\begin{lemma}
	\label{lem:vanishing_of_certain_limit_of_group_coh}
	Let $R$ be a $p$-complete $p$-torsionfree ring, let $G$ be a pro-finite group. 
	Assume $N$ is a $p$-complete $p$-torsionfree $R$-module with a continuous $G$-action.
	Then we have 
	\[
	R^1\lim_n (\mathrm{H}^0_\cont(G,N)/p^n) = \lim_n (\mathrm{H}^1_\cont(G, N)[p^n]) = 0.
	\]
\end{lemma}
\begin{proof}
	By assumption, we know $\mathrm{H}^0_\cont(G,N)\subseteq N$ is also a $p$-complete and $p$-torsionfree module.
	So we have $\mathrm{H}^0_\cont(G,N)\simeq R\lim_n \mathrm{H}^0_\cont(G,N)/p^n$, which implies the vanishing of $R^1\lim_n (\mathrm{H}^0_\cont(G,N)/p^n)$.
	In addition, by the isomorphism $N=R\lim_n N/p^nN\simeq \lim_n N/p^nN$ and the fact that limit commutes with continuous group invariant, we have 
	\[
	\mathrm{H}^0_\cont(G,N) =\mathrm{H}^0_\cont(G,\lim_n N) \simeq  \lim_n \mathrm{H}^0_\cont(G,N/p^nN).
	\]
	So by the universal coefficient theorem (with respect to $p$), we get the vanishing of $\lim_n (\mathrm{H}^1_\cont(G, N)[p^n])$.
\end{proof}
We also introduce a stronger flatness assumption.
\begin{definition}
	\label{def:completely_projective}
	Let $R$ be a ring, let $(f_1,\ldots,f_d)$ be a Koszul regular sequence in $R$, and assume $R$ is complete under the $I=(f_1,\ldots,f_d)$-adic topology.
	\begin{enumerate}
	\item We say an $R$-module $M$ is \emph{completely projective (resp. completely finite projective, resp. completely free)} over $R$ if there is a projective  (resp. finite projective, resp. free) module $M_0$ over $R$, such that $M\simeq (M_0)^\wedge_I$.
	\item We say a map of $R$-modules $M\to N$ is \emph{completely injective} if the derived mod $(f_1,\ldots,f_d)$ reduction of the map is an injection of $R/(f_1,\ldots,f_d)R$-modules.
	\end{enumerate}
\end{definition}
Note that the complete injectivity in particular implies the injectivity.
We also say an $R$-algebra $S$ (equivalently a homomorphism $R\to S$) is \emph{completely projective} if $S$ is completely projective over $R$, and similarly for the others.
So a completely projective homomorphism of complete rings is in particular completely faithfully flat.

The following lemmas on the complete projective base change of an injection will be used later.
\begin{lemma}[Injectivity of complete base change, two variables]
	\label{lem:inj_and_completely_proj_mod}
	Let $R$ be a ring and assume $R$ is complete under a Koszul regular sequence $(f_1,f_2)$ in $R$.
	Let $R\to S$ be a completely injective map of $(f_1,f_2)$-complete rings.
	Let $M\to N$ be a map of $(f_1,f_2)$-regular and $(f_1,f_2)$-complete $R$-modules.
	\begin{enumerate}
	\item Assume either $M$ is completely flat, or $S$ is completely faithfully flat over $R$. The complete linearization map $M\to (M\otimes_R S)^\wedge_{(f_1,f_2)}$ is injective.
	\item Assume $S$ is completely projective over $R$ and the mod $f_1$ reduction of $M\to N$ is injective. The complete base change $(M\otimes_R S)^\wedge_{(f_1,f_2)} \to (N\otimes_R S)^\wedge_{(f_1,f_2)}$ is injective.
	\end{enumerate}
\end{lemma}
\begin{proof}
    For the first part, it suffices to check that for each $n\in \mathbb{N}$, the linearization map 
    \[
    M/(f_1,f_2)^nM \to M/(f_1,f_2)^nM\otimes_{R/(f_1,f_2)^nR} S/(f_1,f_2)^nS
    \] is injective:
    this is clear if $M$ is completely flat since $R\to S$ is completely injective, so we assume $S/(f_1,f_2)^nS$ is faithfully flat over $R/(f_1,f_2)^nR$ for now.
    Then to check its injectivity, by the faithfully flatness it suffices to check it for the base change along $R/(f_1,f_2)^nR\to S/(f_1,f_2)^nS$, which admits a canonical section via the multiplication map of the two identical factors of $S/(f_1,f_2)^nS$.
    Hence by taking the inverse limit, we get the injectivity of $M\to (M\otimes_R S)^\wedge_{(f_1,f_2)}$.
    
	For the second part, we let $\overline{R}$ denote the mod $f_1$ reduction of $R$ and similarly for other objects.
	We let $C$ be the cokernel of the canonical injection $\overline{M} \to \overline{N}$, which is a $f_2$-complete module and by assumption has $\lim_{n\in \mathbb{N}} C[f_2^n]=0.$
	To show the claim, since both objects are derived $f_1$-complete, it suffices to check the injectivity of their mod $f_1$-reduction, namely the map
	\[
	( \overline{M}\otimes^L_{\overline{R}} \overline{S} )^\wedge_{f_2} \longrightarrow ( \overline{N}\otimes^L_R S)^\wedge_{f_2},
	\]
	which by the long exact sequence of the tensor product is equivalent to the vanishing of $\lim_{n\in \mathbb{N}} (C[f_2^n]\otimes^L_{\overline{R}} \overline{S})$.
	Notice that by the assumption, the reduction $\overline{S}$ is the $f_2$-completion of a projective module $P$ over $\overline{R}$, where we can embed $P$ into a free module $\bigoplus_i \overline{R}$.
	So we get
	\[
	\lim_{n\in \mathbb{N}} (C[f_2^n]\otimes^L_{\overline{R}} \overline{S}) =  \lim_{n\in \mathbb{N}} (C[f_2^n]\otimes_{\overline{R}}  P) \hookrightarrow \lim_{n\in \mathbb{N}} (\bigoplus_i C[f_2^n]) \subset \lim_{n\in \mathbb{N}} \prod_i C[f_2^n].
	\]
	So by switching the order of limits in the last term above, we see it is zero, which finishes the proof.
\end{proof}
Here we note that same arguments also show the following simpler case.
\begin{lemma}[Injectivity of complete base change, one variable]
\label{lem:inj_and_completely_proj_mod_one_variable}
Let $R$ be a $f$-complete $f$-torsionfree ring for an element $f\in R$, and let $R\to S$ be a completely injective map of $f$-complete rings.
Let $M\to N$ be a map of $f$-complete and $f$-torsionfree modules.
\begin{enumerate}
\item Assume either $M$ is completely flat, or $S$ is completely faithfully flat over $R$. The complete linearization map $M\to (M\otimes_R S)^\wedge_f$ is injective.
\item Assume $S$ is completely projective over $R$ and $M\to N$ is injective. The complete base change $(M\otimes_R S)^\wedge_f \to (N\otimes_R S)^\wedge_f$ is injective.
\end{enumerate}
\end{lemma}
It is tempting to hope that the injectivity can be preserved by an arbitrary completely faithfully flat base change, and even in the one-variable case we give a counterexample below.
Roughly speaking, it is related to the fact that the complete localization functor is not exact in the condensed world (cf. \cite[Exm.\ A.8, Rmk.\ A.9]{KP22}.)
\begin{example}
\label{eg:counter_eg_of_injectivity_preserved_by_bc}
Let $R=\mathbb{Z}_p\langle x \rangle$, let $f=p$, and let $S$ be the ring $R\times \mathbb{Z}_p\langle x^{\pm1} \rangle$.
Consider the map of $R$-modules $M_1=(\oplus_{i \in \mathbb{N}} R)^\wedge_p \to M_2 = (\oplus_{i \in \mathbb{N}} R)^\wedge_p$ that sends $(a_i)$ onto $(pa_i-xa_{i+1})$.
Then the kernel of the complete base change consists of all the elements $(0, (a_i))$ such that $a_i=(\frac{p}{x})^ia_0$ and $a_0\in S$.
\end{example}

We now give a projection formula on the continuous group invariant, for the complete tensor product with a flat module.
\begin{theorem}[Projection formula]
	\label{thm:projection_formula}
	Let $R$ be a $p$-complete $p$-torsionfree ring, let $G$ be a pro-finite group that acts on $R$ trivially.
	Let $M, N$ be two $p$-complete $p$-torsionfree continuous $R[G]$-modules such that the $G$-action on $M$ is trivial.
	Assume either of the following conditions:
	\begin{enumerate}[label=\upshape{(\alph*)}]
		\item\label{thm:projection_formula_H1} $M$ is $p$-completely flat over $R$ and $\mathrm{H}^{1}_\cont(G,N)$ has bounded $p^\infty$-torsion; or
		\item\label{thm:projection_formula_projective} $M$ is completely projective over $R$.
	\end{enumerate}
	Then we have natural isomorphisms
	\[
	\mathrm{H}^0_\cont(G, (M\otimes^L_R N)^\wedge_p) \simeq \bigl( M\otimes^L_R \mathrm{H}^0_\cont(G,N) \bigr)^\wedge_p \simeq \lim_n \bigl(M\otimes_R \mathrm{H}^0_\cont(G,N)\bigr) /p^n.
	\]
\end{theorem}
Here we mention that when $G$ is profinite, the statement and the proof can be adapted to the higher cohomology, once formulated in an appropriate derived category of continuous $G$-representations.
\begin{proof}
	We first notice that the second isomorphism in the statement follows from \Cref{lem:complete_tensor_product_w_flat}, since $\mathrm{H}^0_\cont(G,N)$ is a derived $p$-complete submodule of $N$, where the latter is $p$-torsionfree.
	It then remains to check that there is a natural isomorphism between the first and the last term.	
	By assumptions and \Cref{lem:complete_tensor_product_w_flat}, the complete derived tensor product $(M\otimes^L_R N)^\wedge_p$ is the  inverse limit of the discrete module $(M\otimes_R N)/p^n$.
	So by the fact that the inverse limit commutes with the group invariant, we get 
	\begin{equation}
		\label{eq:thm:projection_formula_1}
		\begin{multlined}
			\mathrm{H}^0_\cont(G, (M\otimes^L_R N)^\wedge_p)  \simeq \mathrm{H}^0_\cont\bigl(G, \lim_n (M/p^n\otimes_{R/p^n} N/p^n) \bigr) \\
			\simeq \lim_n \mathrm{H}^0_\cont (G, M/p^n\otimes_{R/p^n} N/p^n ),
		\end{multlined}
	\end{equation}
	where the latter naturally admits a map from $\lim_n \bigl(  M/p^n \otimes_{R/p^n}  \mathrm{H}^0_\cont(G,N/p^n) \bigr)$.
	We then claim that the map is an isomorphism: to see this, it suffices to check that for each $n\in \mathbb{N}$, the canonical map below is an isomorphism
	\begin{equation}
		\label{eq:thm:projection_formula_colimit}
		M/p^n \otimes_{R/p^n}  \mathrm{H}^0_\cont(G,N/p^n) \longrightarrow \mathrm{H}^0_\cont(G, M/p^n\otimes_{R/p^n} N/p^n).
	\end{equation}
	Since $M$ is $p$-completely flat over $R$ in both cases, the reduction $M/p^nM$ is flat over $R/p^nR$.
	In particular, by Lazard's theorem \cite[\href{https://stacks.math.columbia.edu/tag/058G}{Tag 058G}]{stacks-project}, we know $M/p^nM$ is isomorphic to a filtered colimit $\colim_k F_k$, where each $F_k$ is a finite free $R/p^nR$-module.
	So thanks to the fact that the filtered colimit commutes with the group invariant, we get natural isomorphisms that are compatible with (\ref{eq:thm:projection_formula_colimit})
	\begin{align*}
		M/p^n \otimes_{R/p^n}  \mathrm{H}^0_\cont(G,N/p^n) & \simeq \colim_k F_k \otimes_{R/p^n}  \mathrm{H}^0_\cont(G,N/p^n) \\
		&\simeq \colim_k  \mathrm{H}^0_\cont(G,F_k\otimes N/p^n) \\
		& \simeq \mathrm{H}^0_\cont \bigl(G, \colim_k (F_k\otimes_{R/p^n} N/p^n) \bigr) \\
		&\simeq \mathrm{H}^0_\cont (G, M/p^n\otimes_{R/p^n} N/p^n),
	\end{align*}
	where the second isomorphism uses the fact that the group invariant commutes with finite direct sums.
	So combine (\ref{eq:thm:projection_formula_1}) and the limit of (\ref{eq:thm:projection_formula_colimit}), we get
	\begin{equation}
		\label{eq:thm:projection_formula_2}
		\mathrm{H}^0_\cont(G, (M\otimes^L_R N)^\wedge_p)  \simeq \lim_n\bigl(  M/p^n \otimes_{R/p^n}  \mathrm{H}^0_\cont(G,N/p^n) \bigr).
	\end{equation}
	
	To obtain the formulae in the statement that involves $\mathrm{H}^0_\cont(G,N)/p^n$, we recall that by the universal coefficient theorem and the flatness of $M/p^n$ over $R/p^n$, we have a short exact sequence of inverse systems
	\begin{equation}
		\label{eq:projection_formula_univ_coef}
		0 \longrightarrow \{M/p^n\otimes \mathrm{H}^0_\cont(G,N)/p^n \}_n \longrightarrow \{M/p^n \otimes  \mathrm{H}^0_\cont(G,N/p^n)\}_n \longrightarrow \{ M/p^n \otimes \mathrm{H}^1_\cont(G,N)[p^n] \}_n \longrightarrow 0.
	\end{equation}
	Now we discuss the two scenarios as in the statement.
	If $M$ and $N$ satisfy the assumptions in \ref{thm:projection_formula_H1}, then the fourth term of (\ref{eq:projection_formula_univ_coef}) is equal to $\{ M/p^n \otimes_{R/p^n} \mathrm{H}^1_\cont(G,N)[p^m] \}_n$ for a fixed integer $m$.
	Since the transition morphism of $\{ M/p^n \otimes_{R/p^n} \mathrm{H}^1_\cont(G,N)[p^m] \}_n$ is the multiplication-by-$p$ map, the inverse system is pro-zero.
	Hence we get the isomorphism
	\[
	\lim_n  \bigl(M/p^n\otimes \mathrm{H}^0_\cont(G,N)/p^n \bigr) \simeq \lim_n  \bigl( M/p^n \otimes_{R/p^n}  \mathrm{H}^0_\cont(G,N/p^n) \bigr),
	\]
	which combining with (\ref{eq:thm:projection_formula_2}) finishes the proof.
	
	Next, we assume that $M$ is the $p$-completion of a projective module $M_0$ as in \ref{thm:projection_formula_projective}.
	We let $\bigoplus_i R$ be a free module that admits a split injection from $M_0$.
	Then there are natural injections of inverse systems
	\begin{align*}
			\{M/p^nM \otimes \mathrm{H}^1_\cont(G,N)[p^n]\}_n & = \{M_0/p^nM_0\otimes_{R/p^nR} \mathrm{H}^1_\cont(G,N)[p^n] )\}_n \\
			& \hookrightarrow	\{\bigoplus_i R/p^nR \otimes_{R/p^nR} \mathrm{H}^1_\cont(G,N)[p^n] )\} \\
		& \hookrightarrow \{\prod_i \mathrm{H}^1_\cont(G,N)[p^n] \},
	\end{align*}
	On the other hand, since the inverse limit commutes with the product, we have
	\begin{equation}
		\label{eq:projection_formula_limit}
		\lim_n \prod_i (\mathrm{H}^1_\cont(G,N)[p^n] ) \simeq \prod_i \lim_n  (\mathrm{H}^1_\cont(G,N)[p^n] ),
	\end{equation}
	where $\lim_n  (\mathrm{H}^1_\cont(G,N)[p^n]) =0$ thanks to the assumption that $N$ is $p$-complete and $p$-torsionfree and \Cref{lem:vanishing_of_certain_limit_of_group_coh}.
	Hence the limit of $\{M/p^n \otimes \mathrm{H}^1_\cont(G,N)[p^n]\}_n$, which is contained in the first limit of (\ref{eq:projection_formula_limit}), also vanishes.
\end{proof}
\begin{remark}[Analogue for pro-\'etale cohomology]
By translating pro-\'etale cohomology into continuous group cohomology, we naturally obtain an analogue of \Cref{thm:projection_formula} but for the global section of a $p$-complete pro-\'etale sheaf $\mathcal{F}$ on a rigid space $X_\eta$.
This is similar to the classical relationship between \'etale cohomology and group cohomology, and we will assume this fact without additional explanations.
\end{remark}


\subsection{Finiteness of arithmetic pro-\'etale cohomology}
\label{sub:finiteness}
In this subsection, we recollect the finiteness of rational arithmetic pro-\'etale cohomology using the decompletion result of Liu--Zhu \cite{LZ17}, and prove the finiteness of the integral arithmetic pro-\'etale cohomology, up to bounded torsions.

Recall from \cite[\S 6]{Sch13} and \cite[\S 2.1]{LZ17} that for a smooth rigid space $X_\eta$, there is a period sheaf $\mathcal{O}\mathbb{C}=\gr^0 \OBdR$ over $X_{\eta,\pe}$.
We let $\mathbb{L}^\an_{X_{\eta,\pe}/X_\eta}$ be the analytic cotangent complex of the rational complete structure sheaf $\widehat{\mathcal{O}}$ over $\mathcal{O}_{X_\eta}$, which is defined by inverting $p$ at the $p$-complete cotangent complex $\mathbb{L}_{\widehat{\mathcal{O}}^+/\mathcal{O}_X^+}$ (\cite[Rmk.\ 4.5]{GL21}); similarly for $\mathbb{L}^\an_{X_{\eta,\pe}/K}$.
Then by \cite[Def.\ 6.8.(iv)]{Sch13} and \cite[Thm.\ 4.21]{GL21}, there is a natural formula
\begin{equation}
	\label{eq:formula_of_OC}
	\mathcal{O}\mathbb{C} \simeq \colim_{n\in \mathbb{N}} \bigl( \wedge^i\mathbb{L}^\an_{X_{\eta,\pe}/X_\eta}(-i)[-i] \bigr),
\end{equation}
where the transition morphism $\widehat{\mathcal{O}}\to \mathbb{L}^\an_{X_{\eta,\pe}/X_\eta}(-1)[-1]$ and its wedge higher powers are induced by the Faltings extension, or more conceptually the canonical map
\[
\mathbb{L}^\an_{X_{\eta,\pe}/K} \longrightarrow \mathbb{L}^\an_{X_{\eta,\pe}/X_\eta}.
\]
In particular, the pro-\'etale sheaf $\wedge^i\mathbb{L}^\an_{X_{\eta,\pe}/X_\eta}(-i)[-i]$ is naturally contained in $\mathcal{O}\mathbb{C}$.
\begin{theorem}[Liu--Zhu]
	\label{thm:LZ_finiteness}
	Let $X=\Spf(R)$ be an affine topologically of finite type $p$-adic formal scheme over $\mathcal{O}_K$, and assume the generic fiber $X_\eta$ is smooth over $K$ of dimension $d$.
	Let $T$ be a $\mathbb{Z}_p$-local system over $X_\eta$.
	The cohomology $R\Gamma_\proet(X_\eta, T\otimes_{\mathbb{Z}_p} \widehat{\mathcal{O}}(j))$ lives in $D^{[0,d+1]}_{\coh}(R[1/p])$, and vanishes if $|j|\gg 0$.
\end{theorem}
\begin{proof}
	The statement can be proved using the decompletion result in \cite{LZ17}, as we shall explain.
	We first notice that the statement is Zariski local with respect to $X_\eta$.
	So as in the setup of \cite[\S 2.3]{LZ17}, we may assume $X_\eta$ is small and thus admits a toric chart, and let $\Gamma=\Gamma_\mathrm{geom} \rtimes \Gal(K_\infty/K)$ be the Galois group associated to the induced pro-finite Galois cover of $X_\eta$, where $K_\infty/K$ is the ($p$-completed) cyclotomic cover.
	Then by \cite[Prop.\ 2.8 (and the paragraph above for notations), Lem.\ 3.10]{LZ17}, we know that there is a large enough subextension $K_m/K$ in $K_\infty$, together with a finite projective $(R_{K_m}\colonequals R\otimes_{\mathcal{O}_K} K_m)$-module $M$ that is equipped with a continuous $\Gamma$-action, such that 
	\[
	R\Gamma_\cont(\Gamma, M(j)) \simeq R\Gamma_\proet(X_\eta, T\otimes_{\mathbb{Z}_p} \widehat{\mathcal{O}}(j)),
	\]
	where $M(j)=M\otimes_{\mathbb{Z}_p} \mathbb{Z}_p(j)$ is the $j$-th Tate twist of $M$.
	Notice that by construction, since $\Gamma_\mathrm{geom}$ (resp. $\Gal(K_\infty/K_m)$) is a $p$-completed free abelian group of $d$ topological generators (resp. $1$ generator), the continuous group cohomology $R\Gamma_\cont(\Gamma_m, M(j))$ is computed by the Koszul complexes, namely
	\[
	R\Gamma_\cont(\Gamma_m, M(j)) \simeq \bigl[ \Kos_{M(j)}(\Gamma_\mathrm{geom}) \xrightarrow{\gamma - \id} \Kos_{M(j)}(\Gamma_\mathrm{geom}) \bigr],
	\]
	where $\gamma$ is a topological generator of $\Gal(K_\infty/K_m)$ and $\Gamma_m$ is the preimage of $\Gal(K_\infty/K_m)\subset \Gal(K_\infty/K)$ in $\Gamma$.
	In addition, by the finiteness of $M$, each term in the complex $\Kos_{M(j)}(\Gamma_\mathrm{geom})$ above is a finite $R[1/p]$-module, and the arrows in the complex are $R[1/p]$-linear.
	Hence $R\Gamma_\cont(\Gamma_m, M(j))$ lives in the category $D^{[0,d+1]}_{\coh}(R[1/p])$.
	Finally, by the Hochschild--Serre spectral sequence together with the fact that the $\Gal(K_m/K)$-invariant functor is exact on $\mathbb{Q}_p$-vector spaces, we see $R\Gamma_\proet(X_\eta, T\otimes_{\mathbb{Z}_p} \widehat{\mathcal{O}}(j))$ lives in the category $D^{[0,d+1]}_{\coh}(R[1/p])$ as well.
	
	To continue, as the geometric Galois cohomology commutes with the Tate twist, we have $\mathrm{H}^i(\Kos_{M(j)}(\Gamma_\mathrm{geom}))= \mathrm{H}^i(\Kos_{M}(\Gamma_\mathrm{geom}))(j)$, as continuous representations of $\Gal(K_\infty/K_m)$.
	For each element $x\in \mathrm{H}^i(\Kos_{M}(\Gamma_\mathrm{geom}))$, we let $x^j$ be the associated element in the $j$-th Tate twist.
	Then we have $\gamma(x^j)=p^{jn}\cdot \gamma(x)^j$, where $n$ is a constant positive integer that depends on $K_m$.
	In particular, since $\mathrm{H}^i(\Kos_{M}(\Gamma_\mathrm{geom}))$ is finitely generated over $R[1/p]$, the action of $\gamma-\id$ on $\mathrm{H}^i(\Kos_{M}(\Gamma_\mathrm{geom}))(j)$ is invertible when $|j|$ is sufficiently large.
	As a consequence, the arithmetic Galois cohomology $R\Gamma_\cont\bigl(\Gal(K_\infty/K_m), \mathrm{H}^i(\Kos_{M}(\Gamma_\mathrm{geom}))(j)\bigr)$ vanishes when $|j|$ is sufficiently large.
	Hence by the boundedness of $\Kos_{M}(\Gamma_\mathrm{geom}))(j)$, we see the entire cohomology $R\Gamma_\proet(X_\eta, T\otimes_{\mathbb{Z}_p} \widehat{\mathcal{O}}(j))$ vanishes for $|j|\gg 0$.
\end{proof}
\begin{proposition}[Rational finiteness]
	\label{prop:LZ_finiteness}
	Let $X=\Spf(R)$ be an affine topologically of finite type $p$-adic formal scheme over $\mathcal{O}_K$, and assume the generic fiber $X_\eta$ is smooth over $K$ of dimension $d$.
	Let $T$ be a $\mathbb{Z}_p$-local system $T$ over $X_\eta$.
	\begin{enumerate}[label=\upshape{(\roman*)}]
		\item\label{prop:LZ_finiteness_general} For each $i\in \mathbb{N}$, the cohomology $R\Gamma_\proet(X_\eta, T\otimes_{\mathbb{Z}_p} \wedge^i \mathbb{L}^\an_{X_{\eta,\pe}/X_\eta}[-i])$ lives in $D^{[0,d+1]}_{\coh}(R[1/p])$.
		\item\label{prop:LZ_finiteness_H0} The zero-th cohomology $\mathrm{H}^0_\proet(X_\eta,  T\otimes_{\mathbb{Z}_p}\wedge^i \mathbb{L}^\an_{X_{\eta,\pe}/X_\eta}[-i])$ is contained in $\mathrm{H}^0_\proet(X_\eta, T\otimes_{\mathbb{Z}_p} \widehat{\mathcal{O}}_X(i))$ and vanishes when $i\gg 0$. 
		If $T$ is a constant local system, then the zero-th cohomology vanishes for $i>0$.
	\end{enumerate}
\end{proposition}
\begin{proof}
	Part \ref{prop:LZ_finiteness_general} is a consequence of \Cref{thm:LZ_finiteness} and the Faltings extension (\cite[Thm.\ 4.9]{GL21}): namely the analytic cotangent complex $\mathbb{L}^\an_{X_{\eta,\pe}/X_\eta}$ fits into the fiber sequence
	\[
	\widehat{\mathcal{O}}(1)[1] \longrightarrow \mathbb{L}^\an_{X_{\eta,\pe}/X_\eta} \longrightarrow \widehat{\mathcal{O}}\otimes_{R[1/p]} \mathbb{L}^\an_{X_\eta/K}[1],
	\]
	where $\mathbb{L}^\an_{X_\eta/K}$ is a finite projective $R[1/p]$-module.
	So the statement follows by taking the associated long exact sequence of the wedge power and \Cref{thm:LZ_finiteness}.
	
	For \ref{prop:LZ_finiteness_H0}, we notice that by (\ref{eq:formula_of_OC}), the pro-\'etale sheaf $T\otimes_{\mathbb{Z}_p}\wedge^i \mathbb{L}^\an_{X_{\eta,\pe}/X_\eta}[-i]$ is naturally contained in $T\otimes_{\mathbb{Z}_p}\mathcal{O}\mathbb{C}(i)$.
	In addition, by \cite[Lem.\ 2.9]{LZ17} (see \cite[Prop.\ 2.8]{LZ17} for the notations), we know 
	\[
	\mathrm{H}^0_\proet(X_\eta, T\otimes_{\mathbb{Z}_p}\mathcal{O}\mathbb{C}(i)) \simeq \mathrm{H}^0_\proet(X_\eta, T\otimes_{\mathbb{Z}_p} \widehat{\mathcal{O}}(i)).
	\]
	Thus the vanishing of the zero-th cohomology of $T\otimes_{\mathbb{Z}_p}\wedge^i \mathbb{L}^\an_{X_{\eta,\pe}/X_\eta}[-i]$ follows from the aforementioned inclusion and \Cref{thm:LZ_finiteness}.
	For the constant local system, as mentioned above, the vanishing of the zero-th cohomology of $\wedge^i \mathbb{L}^\an_{X_{\eta,\pe}/X_\eta}[-i]$ for $i>0$ follows from that of $\widehat{\mathcal{O}}(i)$.
	The latter is a consequence of Tate's calculation of Galois cohomology (\cite{Tat67}) together with the fact that $\mathrm{H}^0_\proet(X_{K_{\infty}}, \widehat{\mathcal{O}})=R_{K_\infty}$ (a consequence of for example the Hodge--Tate decomposition \cite{Sch13}).
\end{proof}

Now we turn to the integral arithmetic pro-\'etale cohomology.
Different from the rational case, the integral pro-\'etale cohomology is not finitely generated due to many torsion classes.
However, we prove below that the integral cohomology has bounded $p^\infty$-torsion and the torsionfree quotient is finitely generated.
\begin{proposition}
	\label{prop:finiteness_integral_general}
	Let $R$ be a $p$-torsionfree and $p$-complete noetherian ring, and let $\mathrm{H}$ be a derived $p$-complete $R$-module.
	Assume $\mathrm{H}[1/p]$ is finitely generated over $R[1/p]$.
	\begin{enumerate}[label=\upshape{(\roman*)}]
		\item\label{thm:finiteness_integral_general_bounded} The submodule $\mathrm{H}[p^\infty]$ is bounded, namely $\mathrm{H}[p^\infty]=\mathrm{H}[p^N]$ for some $N\in \mathbb{N}$.
		\item\label{thm:finiteness_integral_general_tf} The torsionfree quotient $\mathrm{H}_\tf\colonequals \mathrm{H}/\mathrm{H}[p^\infty]$ is a finitely presented $R$-module.
	\end{enumerate}
\end{proposition}
\begin{proof}
	First we notice that the image of $\mathrm{H}$ in $\mathrm{H}[1/p]$ is naturally isomorphic to the torsionfree quotient $\mathrm{H}_\tf$.
	Since $\mathrm{H}[1/p]$ is a finitely generated $R[1/p]$-module, we let $\{\overline{m}_1,\ldots, \overline{m}_n\}$ be a finite set of elements in $\mathrm{H}_\tf$ whose image in $\mathrm{H}[1/p]$ generates the latter.
	We let $\{m_1,\ldots, m_n\}$ be a set of lifts of the aforementioned elements in $\mathrm{H}$.
	Then the choice of elements induces a map $f:R^{\oplus n} \to \mathrm{H}$ of $R$-modules such that its image in $\mathrm{H}[1/p]$ generates the entire $\mathrm{H}[1/p]$.
	We let $M$ (resp. $\overline{M}$) be the image of $R^{\oplus n}$ in $\mathrm{H}$ (resp. in $\mathrm{H}_\tf$), and let $g:M\to \mathrm{H}$ and $\overline{g}: \overline{M}\to \mathrm{H}_\tf$ be the induced injections.
	Then we obtain the following commutative diagram:
	\begin{equation}
		\label{eq:thm:finiteness_integral_general_1}
		\begin{tikzcd}
			R^{\oplus n} \arrow[d, two heads] & \\
			M \arrow[r, hook, "g"] \arrow[d, two heads] & \mathrm{H} \arrow[d, two heads]\\
			\overline{M} \arrow[r, hook, "\overline{g}"] & \mathrm{H}_\tf \arrow[d, hook]\\
			& \mathrm{H}[1/p].
		\end{tikzcd}
	\end{equation}
	
    Notice that since $M$ is finitely generated over $R$, we know $M$ is derived $p$-complete (\cite[\href{https://stacks.math.columbia.edu/tag/0EEU}{Tag 0EEU}]{stacks-project}).
    So by the derived $p$-completeness of $\mathrm{H}$, we know $\mathrm{cofib}(g)=\coker(g)$ is also derived $p$-complete.
    On the other hand, by the construction of $M$ above, we know $M[1/p]=\mathrm{H}[1/p]$ which implies that $\mathrm{cofib}(g)$ is $p^\infty$-torsion.
    So by Bhatt's boundedness result of torsion complete modules (\cite[\href{https://stacks.math.columbia.edu/tag/0CQY}{Tag 0CQY}]{stacks-project}), we know $\mathrm{cofib}(g)$ is bounded.
    
    To continue, we consider the map of short exact sequences induced from (\ref{eq:thm:finiteness_integral_general_1})
    \begin{equation}
    	\label{eq:thm:finiteness_integral_general_ses}
    	\begin{tikzcd}
    		0 \ar[r] &M \arrow[r, hook, "g"] \arrow[d, two heads] & \mathrm{H} \arrow[d, two heads] \ar[r] & \mathrm{cofib}(g) \ar[r] \ar[d]& 0\\
    		0 \ar[r] & \overline{M} \arrow[r, hook, "\overline{g}"] & \mathrm{H}_\tf \ar[r] & \mathrm{cofib}(\overline{g}) \ar[r] &0. \\
    	\end{tikzcd}
    \end{equation}
    From the diagram, we know $\ker(\mathrm{H}\to \mathrm{H}_\tf)$, which is equal to $\mathrm{H}[p^\infty]$, is sandwiched by $\ker(M\to \overline{M})$ and $\cofib(g)$.
    By the noetherian assumption of $R$ and the finiteness of $M$, we know $\ker(M\to \overline{M})=M[p^\infty]$ is bounded.
    Hence by combining the latter with the boundedness of $\cofib(g)$, we know  $\mathrm{H}[p^\infty]$ is also bounded, which finishes \ref{thm:finiteness_integral_general_bounded}.
    
    To prove \ref{thm:finiteness_integral_general_tf}, we notice that (\ref{eq:thm:finiteness_integral_general_ses}) implies that the map $\cofib(g)\to \cofib(\overline{g})$ is surjective.
    In particular $\cofib(\overline{g})$ is bounded and is killed by $p^{N}$ for some $N\in \mathbb{N}$.
    So as submodules inside $\mathrm{H}[1/p]$, we have
    \[
    \overline{M} \subseteq \mathrm{H}_\tf \subseteq \frac{1}{p^{N}}\overline{M}.
    \]
    Hence by the noetherian assumption of $R$ and the finiteness of $M$, we see $\mathrm{H}_\tf$ is finitely presented over $R$ as well.
\end{proof}
\begin{corollary}[Integral finiteness]
	\label{thm:finiteness_integral_proet_coh}
	Let $X=\Spf(R)$ be an affine topologically of finite type $p$-adic formal scheme over $\mathcal{O}_K$, and assume the generic fiber $X_\eta$ is smooth over $K$ of dimension $d$.
	Let $T$ be a $\mathbb{Z}_p$-local system $T$ over $X_\eta$, let $i,j\in \mathbb{N}$, and let $\mathrm{H}^j$ be the $j$-th cohomology of $R\Gamma_\proet\bigl(X_\eta, T\otimes_{\mathbb{Z}_p} \wedge^i \mathbb{L}_{\widehat{\mathcal{O}}^+/\mathcal{O}_X}[-i] \bigr)$.
	\begin{enumerate}[label=\upshape{(\roman*)}]
		\item The submodule $\mathrm{H}^j[p^\infty]$ is bounded.
		\item The torsionfree quotient $\mathrm{H}^j_\tf$ is a finitely presented $R$-module, and vanishes if $j\notin [0,d+1]$.
		\item The module $\mathrm{H}^0$ vanishes for $i\gg 0$.
	\end{enumerate}
\end{corollary}

Finally, we give the calculation of the global section for the integral pro-\'etale structure sheaf.
\begin{lemma}
	\label{lem:global_sec_of_Ohat}
	Let $X=\Spf(R)$ be a normal topologically of finite type $p$-adic formal scheme over $\mathcal{O}_K$, and let $X_\eta$ be its generic fiber.
	The canonical map below is an isomorphism
	\[
	R \longrightarrow \mathrm{H}^0_\pe(X_\eta, \widehat{\mathcal{O}}^+).
	\]
\end{lemma}
\begin{proof}
	Note that there are natural maps of ringed sites $(X_{\eta,\pe}, \widehat{\mathcal{O}}^+) \to (X_{\eta, \mathrm{an}}, \mathcal{O}^+) \to (X_{\mathrm{Zar}}, \mathcal{O}_X)$, inducing maps of rings
	\begin{equation}
		\label{eq:lem:global_sec_of_Ohat}
		\mathcal{O}_X(X)=R \longrightarrow \mathcal{O}^+(X_\eta) \longrightarrow \widehat{\mathcal{O}}^+(X_\eta).
	\end{equation}
	It is left to show that the canonical maps above are all equalities.
	By assumption, $X_\eta$ is the affinoid rigid space $\Spa(R[1/p],R)$, with $R$ an integrally closed subring in $R[1/p]$.
	On the other hand, by \cite[Prop.\ 10.2.3]{SW20} and the normality of $R[1/p]$, we know $\widehat{\mathcal{O}}(X_\eta)=\mathcal{O}_{X_\eta}(X_\eta)=R[1/p]$.
	So the maps in (\ref{eq:lem:global_sec_of_Ohat}) are all inclusions, and it is left to show that $R\to \widehat{\mathcal{O}}^+(X_\eta)$ is surjective.
	Note that by \cite[Lem.\ 4.2.(ii)]{Sch13}, the subring $\widehat{\mathcal{O}}^+(X_\eta)$ is contained in the ring of power-bounded elements in $R[1/p]$.
	So the claim follows from \cite[Thm.\ 7.4.1]{dJ95}, which shows that $R$ is the subring of power-bounded elements in $R[1/p]$.
\end{proof}

\section{Prismatic family of period sheaves}
\label{sec:period sheaves}
Let $X$ be a regular $p$-adic formal scheme over $\mathcal{O}_K$, and let $X_\eta$ be its generic fiber.
In this section, we introduce a prismatic family of period sheaves: a family of period sheaves over the rigid space $X_{\eta,\pe}$ that are parametrized by the prisms in $X_\Prism$.
We then calculate their global sections.

\subsection{Period presheaves}
\label{sub:construction_of_period}
We start with the construction of the period presheaves on the pro-\'etale site of the rigid space $X_\eta$.

To prepare, we recall the construction of the infinitesimal period sheaf as below.
\begin{definition}
	\label{def:Ainf}
	Let $S$ be a $p$-torsionfree perfectoid ring.
	\begin{enumerate}
		\item We let $\rAinf(S)$ be the ring $W(S^\flat)$, where $S^\flat=\lim_{x\mapsto x^p} S$ is the tilt of $S$.
		It admits an automorphism $\varphi_{\rAinf(S)}$ induced by the Witt vector Frobenius of the perfect ring.
		Moreover, there is a natural surjection $\widetilde{\theta}:\rAinf(S)\to S$, sending a Teichm\"uller lift $[(x_0,x_1,\ldots)]$ onto $x_1$.
		\item Let $X$ be a topologically finite type $p$-adic formal scheme.
		We define a sheaf of complete rings $\Ainf$ on $X_{\eta,\pe}$ by sending an affinoid perfectoid space $\Spa(S[1/p],S)\in X_{\eta,\pe}$ onto the $(p, \ker(\widetilde{\theta}))$-complete ring $\rAinf(S)$. 
		\footnote{Here we mention that the sheaf condition of (the zero-th cohomology of) $\Ainf$ follows from that of $\widehat{\mathcal{O}}^+$ in \cite[Prop.\ 8.5]{Sch22}.}
		We let $\varphi_{\Ainf}$ be the induced automorphism on $\Ainf$, and let $\widetilde{\theta}:\Ainf\to \widehat{\mathcal{O}}^+$ be the induced surjection.
	\end{enumerate}
\end{definition}
Now we introduce the period presheaves that are parametrized by prisms, using the derived relative prismatic cohomology as in \cite[Const.\ 7.6]{BS22}.
\begin{definition}[Prismatic period presheaf]
	\label{def:prismatic_period_sheaf}
	Let $X$ be a bounded $p$-adic formal scheme, and let $(B,J)\in X_\Prism$ be a bounded prism.
	The \emph{prismatic period presheaf} over $(B,J)$, denoted as $\Prism^\pre_{(-)_{\overline{B}}/B}$, is a presheaf of complexes that sends an affinoid perfectoid space $U=\Spa(S[1/p],S)$ in the pro-\'etale site $X_{\eta,\pe}$ onto the derived relative prismatic cohomology $\Prism_{S_{\overline{B}}/B}\in D_{(p,I)\text{-comp}}(B)\subset D(B)$.
	The presheaf $\Prism^\pre_{(-)_{\overline{B}}/B}$ is equipped with a Frobenius endomorphism $\varphi_\Prism \colonequals \varphi_{\Prism^\pre_{(-)_{\overline{B}}/B}}$ that is compatible with $\varphi_B$.
\end{definition}
Analogously, we can define the Frobenius twisted prismatic period presheaf $\Prism^{(1),\pre}_{(-)_{\overline{B}}/B}\colonequals \varphi_B^*\Prism^\pre_{(-)_{\overline{B}}/B}$ and its Nygaard (filtered) completion $\widehat{\Prism}^{(1),\pre}_{(-)_{\overline{B}}/B}$.
\begin{remark}
	By \cite[Thm.\ 6.5]{Sch13}, the prismatic period presheaf is naturally isomorphic to the presheaf of complete rings $\Prism^\pre_{\widehat{\mathcal{O}}^+_{\overline{B}}/B}$ for affinoid perfectoid objects over $X_{\eta,\pe}$.
	To lighten the notations, we use $\Prism^\pre_{(-)_{\overline{B}}/B}$ instead, and similarly for the related constructions later.
\end{remark}
\begin{remark}[Base change property]
	\label{rmk:base_change_presheaf}
	We recall from \cite{BS22} that there is a natural base change formula for the relative prismatic cohomology.
	Namely given a map of bounded prisms $(B_1,J_1)\to (B_2,J_2)$ in $X_\Prism$, there is a natural isomorphism of presheaves on $X_{\eta,\pe}$
	\begin{equation}
		\label{eq:base_change_of_period_sheaf_general}
		\begin{tikzcd}
			\bigl( \Prism^\pre_{(-)_{\overline{B_1}}/B_1} \otimes^L_{B_1} B_2 \bigr)^\wedge_{(p,J_2)} \arrow[r, "\sim"] &\Prism^\pre_{(-)_{\overline{B_2}}/B_2}.
		\end{tikzcd}
	\end{equation}
\end{remark}
\begin{remark}[Relation with $\Ainf$]
	\label{rmk:Ainf_to_relative_prismatic_coh}
	For a perfectoid ring $S$, the relative prismatic site $(S_{\overline{B}}/(B,J))_\Prism$ admits a forgetful functor to the absolute prismatic site $S_\Prism$.
	So by the initiality of the perfect prism $(\rAinf(S), \ker(\widetilde{\theta}))$, we get a natural Frobenius equivariant map of presheaves on $X_{\eta,\pe}$
	\begin{equation}
		\label{eq:Ainf_to_relative_prismatic_coh}
		\Ainf \longrightarrow \Prism^\pre_{(-)_{\overline{B}}/B},
	\end{equation}
	which is compatible with the base prisms $(B,J)\in X_\Prism$.
\end{remark}

Thanks to the analysis on the coproduct in \Cref{cor:coproduct:perfect_with_framed}, the prismatic period presheaf becomes a sheaf and admits simpler and more explicit presentations in favorable situations.
To see this, we first note the following observation on the reduced perfection of a framed regular prism.
\begin{lemma}[Perfection and pro-\'etale cover]
	\label{lem:perfection_of_framed_regular_prism}
	Let $X$ be a regular $p$-adic formal scheme over $\mathcal{O}_K$, let $(A,I)\in X_\Prism$ be a framed regular prism with the framing $\Sigma\subset A^\times$, and let $(A_\perf,I_\perf)$ be its perfection.
	The generic fiber $\Spf(\overline{A}_\perf)_\eta$ is an affinoid perfectoid space and is a pro-finite-\'etale cover of $\Spf(\overline{A})_\eta$
\end{lemma}
\begin{proof}
	As $(A_\perf,I_\perf)$ is a perfect prism, by \cite[Thm.\ 3.10]{BS22}, the reduction $\overline{A}_\perf$ is a perfectoid ring. 
	Thus the generic fiber $\Spf(\overline{A}_\perf)_\eta$ is an affinoid perfectoid space that admits a map to $\Spf(\overline{A})_\eta$.
	So it remains to prove that the latter map of adic spaces is pro-finite-\'etale.
	
	We let the $\delta$-pair $(A_0,I_0)$ be the framing of the prism $(A,I)$, so $A_0$ is $p$-completely \'etale over $W\langle t_1^{\pm1}, \ldots, t_n^{\pm1} \rangle$ and $\Spf(A_0/I_0)=\Spf(\overline{A})$.
	By construction, the Frobenius morphism $\varphi_{A_0}$ is a finite flat cover that is compatible with $\varphi_{W\langle  t_1^{\pm1}, \ldots, t_n^{\pm1} \rangle}:W\langle  t_1^{\pm1}, \ldots, t_n^{\pm1} \rangle \to W\langle  t_1^{\pm1}, \ldots, t_n^{\pm1} \rangle$, $t_i\mapsto t_i^p$.
	Moreover, since the map $\varphi_{W\langle  t_1^{\pm1}, \ldots, t_n^{\pm1} \rangle}$ induces a finite \'etale cover on the generic fiber and since the map of rigid spaces $\Spf(A_0)_\eta\to \Spf(W\langle  t_1^{\pm1}, \ldots, t_n^{\pm1} \rangle)_\eta$ is \'etale, the compatibility of Frobenii implies that the map $\varphi_{A_0}$ induces a finite \'etale endomorphism $(\varphi_{A_0})_\eta : \Spf(A_0)_\eta\to \Spf(A_0)_\eta$ on the smooth affinoid rigid space $\Spf(A_0)_\eta$.
	Hence by taking the limit of $\Spf(A_0)_\eta$ with respect to a countable amount of $(\varphi_{A_0})_\eta$ and by \cite[Lem.\ 4.5]{Sch13}, we see $\Spf\bigl( (\colim_{\varphi_{A_0}} A_0)^\wedge_p \bigr)_\eta$ is a pro-finite-\'etale cover of $\Spf(A_0)_\eta$.
	Finally, by the last paragraph of the proof for \Cref{prop:regular prism Frobenius}, the reduced perfection $\overline{A}_\perf$ is the mod $I_0$ reduction of the ring $(\colim_{\varphi_{A_0}} A_0)^\wedge_p$.
	Hence the claim follows from the following pullback diagram
	\[
	\begin{tikzcd}
		\Spf(\overline{A}_\perf)_\eta \arrow[r, hook] \ar[d] & \Spf\bigl( (\colim_{\varphi_{A_0}} A_0)^\wedge_p \bigr)_\eta \arrow[d, "\text{pro-finite-\'etale}"] \\
		\Spf(\overline{A})_\eta \arrow[r, hook] & \Spf(A_0)_\eta. \arrow[lu, phantom, "\lrcorner"]
	\end{tikzcd}
	\]
\end{proof}
Recall from \cite[\S 2]{Sch13} that for a perfectoid ring $S$ and a pseudo-uniformizer $\pi^\flat\in S^\flat$, one can define the notion of the almost zero module over $\rAinf(S)$ with respect to the ideal $([\pi^\flat]^{1/p^n}, n>0)\subset \rAinf(S)$.
Namely an $\rAinf(S)$-module $M$ is \emph{almost zero} if $M$ is killed by $([\pi^\flat]^{1/p^n}, n>0)$.
The notion is independent of the choice of $\pi^\flat$, and we use $M \almosteq 0$ to denote the condition that $M$ is almost zero.
\begin{theorem}[Sheaf property]
	\label{thm:prismatic_period_sheaf}
	Let $X$ be a regular $p$-adic formal scheme, let $(A,I)\in X_\Prism$ be a framed regular prism for a framing $\Sigma\subset A^\times$, with $\overline{A}_\perf$ the reduction of its perfection.
	Assume $(B,J)\in X_\Prism$ is a flat prism or a flat crystalline prism.
	\begin{enumerate}[label=\upshape{(\roman*)}]
		\item\label{thm:prismatic_period_sheaf_property}  The restriction of the presheaf $\mathcal{H}^0(\Prism^\pre_{(-)_{\overline{B}}/B})$ onto $X_{\eta,\pe}|_{\Spf(\overline{A}_\perf)_\eta}$ is a Frobenius equivariant $B$-linear sheaf of complete modules and is $(p,J)$-completely flat over $\Ainf$.
		\item\label{thm:prismatic_period_sheaf_coh} For any affinoid perfectoid space $U=\Spa(S[1/p],S)\in X_{\eta,\pe}|_{\Spf(\overline{A}_\perf)_\eta}$, one has
		\[
		\mathrm{H}^0_\pe(U, \Prism^\pre_{(-)_{\overline{B}}/B}) = \Prism_{S_{\overline{B}}/B}; \quad \mathrm{H}^i_\pe(U ,\Prism^\pre_{(-)_{\overline{B}}/B}) \almosteq 0,~\forall i>0.
		\]	
	\end{enumerate}
\end{theorem}

\begin{proof}
	We let $U=\Spa(S[1/p],S)$ be an affinoid perfectoid space over $X$ that admits a map from $\Spf(\overline{A}_\perf)_\eta$.
	By \Cref{cor:coproduct:perfect_with_framed}.\ref{cor:coproduct:perfect_with_framed_general}, we know the relative prismatic cohomology $\Prism_{S_{\overline{B}}/B}\in D_{(p,I)\text{-comp}}(B)$ lives in cohomological degree zero and is completely flat over $\rAinf(S)$.
	So it is left to check the sheaf property and calculate the higher cohomology.
	
	We let $U^0=\Spa(S^0[1/p],S^0)$ be any pro-\'etale cover of $U$, and let $U^n=\Spa(S^n[1/p],S^n)$ be the $(n+1)$-th self product of $U^0$ over $U$.
	So the cosimplicial object $U^\bullet$ forms the \v{C}ech nerve of the pro-\'etale covering map $U^0\to U$.
	To check the sheaf condition, it is equivalent to showing that the sequence below is left exact
	\begin{equation}
		\label{eq:thm:prismatic_period_sheaf:1}
		0 \longrightarrow \Prism_{S_{\overline{B}}/B} \longrightarrow \Prism_{S^0_{\overline{B}}/B} \xrightarrow{p_0-p_1} \Prism_{S^1_{\overline{B}}/B}.
	\end{equation}
    Here we note that by the coproduct presentation in \Cref{cor:coproduct:perfect_with_framed}.\ref{cor:coproduct:perfect_with_framed_general}, the sequence in (\ref{eq:thm:prismatic_period_sheaf:1}) is naturally equivalent to the sequence
    \begin{equation}
    	\label{eq:thm:prismatic_period_sheaf:2}
    			\left( \Prism_{S_{\overline{B}}/B} \otimes_{\rAinf(S)} \bigl( 0 \longrightarrow  \rAinf(S) \longrightarrow \rAinf(S^0) \xrightarrow{p_0-p_1} \rAinf(S^1) \bigr) \right)^\wedge_{(p,I)}.
    \end{equation}
	Note that each term in the sequence above is $I$-complete and $I$-torsionfree.
	So by the left exactness of the inverse limit functor, it suffices to check the left exactness of (\ref{eq:thm:prismatic_period_sheaf:2}) at their mod $(I,p)$-reduction.
	In the latter setting, the claim follows from the sheaf property of $\rAinf(-)/(p,I)=\widehat{\mathcal{O}}^+/p$ on $X_{\eta,\pe}$ (\cite[Prop.\ 8.5]{Sch22}), together with the flatness of $\Prism_{S_{\overline{B}}/B}/(p,I)$ over $\rAinf(S)/(p,I)=S/p$.
	
	We let $\mathfrak{m}$ be the ideal $([\pi^\flat]^{1/p^n}, n>0)$ in $\rAinf(S)$.
	To show the almost vanishing of the higher cohomology, by the similar reasoning as above, it is equivalent to checking that for any $n>0$, the $n$-th cohomology group of the following complex is almost zero:
	\begin{equation}
		\label{eq:thm:prismatic_period_sheaf:Cech-Alex}
		C\colonequals \left( \Prism_{S_{\overline{B}}/B} \otimes^L_{\rAinf(S)} \bigl( \rAinf(S^0) \xrightarrow{p_0-p_1} \rAinf(S^1) \xrightarrow{p_0-p_1+p_2} \cdots \bigr) \right)^\wedge_{(p,I)}.
	\end{equation}
    Note that by the exactness of the mod $(p,I)^n$ reduction of (\ref{eq:thm:prismatic_period_sheaf:2}) as in the last paragraph above, we have $\mathrm{H}^0(C)\simeq R\lim_n \mathrm{H}^0(C\otimes^L_{\rAinf(S)} \rAinf(S)/(p,I)^n)$.
    In particular, the canonical truncation of $C$ satisfies the isomorphism
    \[
     \tau^{\geq 1} C \simeq R\lim_n \tau^{\geq 1} (C\otimes^L_{\rAinf(S)} \rAinf(S)/(p,I)^n).
    \]
    We then claim that each $M_n\colonequals \tau^{\geq 1}(C\otimes^L_{\rAinf(S)} \rAinf(S)/(p,I)^n)$ is almost zero.
    Granting the claim, since each reduction $M_n$ is a complex over $\rAinf(S)/(\mathfrak{m}, (p,I)^n)$, and by taking the limit we see $\tau^{\geq 1}C=R\lim_n M_n$ is a complex over $\rAinf(S)/ \mathfrak{m}^\wedge_{(p,I)}$.
    Since the surjection $\rAinf(S)\to  \rAinf(S)/\mathfrak{m}^\wedge_{(p,I)}$ factors through $\rAinf(S)\to \rAinf(S)/\mathfrak{m}$, we know $\tau^{\geq 1}C$ is also a complex over $\rAinf(S)/\mathfrak{m}$, and hence almost zero with respect to $\mathfrak{m}$.

	Now we check that $\tau^{\geq 1}\bigl(C\otimes^L_{\rAinf(S)} \rAinf(S)/(p,I)^n\bigr)$ is almost zero, and by induction we let $n=1$.
	In this case, we recall that the ring $\Prism_{S_{\overline{B}}/B}/(p,I)$ is flat over $\rAinf(S)/(p,I)=S/p$.
	Since being almost zero is preserved under flat base change, it amounts to show that the $\rAinf(S)$-complex below is almost zero with respect to $\mathfrak{m}$
	\[
	\tau^{\geq 1}\left( \bigl( \rAinf(S^0) \xrightarrow{p_0-p_1} \rAinf(S^1) \xrightarrow{p_0-p_1+p_2} \cdots \bigr) \otimes^L_{\rAinf(S)} S/p \right).
	\]
	The latter was shown in \cite[Prop.\ 8.5]{BS22}, which finishes the proof.
\end{proof}
Note that in the special case when $(B,J)=(A,I)$ is a framed regular prism, we obtain the following explicit description of the prismatic period sheaf.
\begin{corollary}
	\label{cor:explicit_desciption_of_period_sheaf}
	Let $X$ be a regular $p$-adic formal scheme, let $(A,I)\in X_\Prism$ be a framed regular prism with framing $\Sigma\subset A^\times$, and let $(A_\perf,I_\perf)$ be the perfection.
	There is an isomorphsm of Frobenius equivalent $(p,I)$-complete $\Ainf$-linear sheaves on $X_{\eta,\pe}|_{\Spf(\overline{A}_\perf)_\eta}$
	\[
	\Ainf\{\frac{\iota(t_1)\otimes 1-1\otimes t_1}{I},\ldots,\frac{\iota (t_d)\otimes 1-1\otimes t_d}{I}\} \longrightarrow \mathcal{H}^0(\Prism^\pre_{(-)_{\overline{A}}/A}).
	\]
	In particular, $\mathcal{H}^0(\Prism^\pre_{(-)_{\overline{A}}/A})$ is completely free over $\Ainf$.
\end{corollary}
\begin{proof}
	By assumption, for each affinoid perfectoid space $U=\Spa(S[1/p],S)\in X_{\eta,\pe}|_{\Spf(\overline{A}_\perf)_\eta}$, there is a canonical commutative diagram of rings that is functorial in $U$:
	\[
	\begin{tikzcd}
		A \arrow[r,"\text{perfection}"] \arrow[d, "\text{mod}~I"'] & A_\perf \arrow[d, "\text{mod}~I"']  \arrow[r, dashed] & \rAinf(S) \arrow[d, "\widetilde{\theta}"'] \\
		R \ar[r] & \overline{A}_\perf \ar[r] & S,
	\end{tikzcd}
	\]
	where the existence of the dashed arrow follows from the equivalence of the categories between perfectoid rings and perfect prisms (\cite[Thm.\ 3.10]{BS22}).
	In particular, we see the coproduct $(\rAinf(S),\ker(\widetilde{\theta}))\coprod (A,I)$ is the base change of $(A,I)\coprod (A,I)$ along the map $f:(A,I)\to (\rAinf(S),\ker(\widetilde{\theta}))$ as above.
	Hence the claim follows from \Cref{cor:coproduct:perfect_with_framed}.\ref{cor:coproduct:perfect_with_framed_special}.
\end{proof}


\subsection{Hodge--Tate and de Rham specializations}
\label{sub:specialize}
Similar to the prismatic cohomology, one of the central features of the prismatic period sheaf is that it admits various realizations, as we shall explain in this subsection.

We start with the Hodge--Tate specialization map.
\begin{definition}[Hodge--Tate specialization]
	\label{def:HT_period_sheaf}
	Let $X$ be a bounded $p$-adic formal scheme, and let $(B,J)\in X_\Prism$ be a bounded prism.
	\begin{enumerate}[label=\upshape{(\roman*)}]
		\item The  \emph{Hodge--Tate period presheaf} on $X_{\eta,\pe}$, denoted as $\overline{\Prism}^\pre_{(-)_{\overline{B}}/B}$, is an $\widehat{\mathcal{O}}^+$-linear presheaf of complexes that sends an affinoid perfectoid space $U=\Spa(S[1/p],S) \in X_{\eta,\pe}$ onto the derived Hodge--Tate prismatic cohomology $\overline{\Prism}_{S_{\overline{B}}/B}\in D_{p\text{-comp}}(\overline{B})\subset D(B)$.
		\item The \emph{Hodge--Tate specialization map} of presheaves on $X_{\eta,\pe}$
		\[
		\gamma^\HT\colon \Prism^\pre_{(-)_{\overline{B}}/B} \longrightarrow \overline{\Prism}^\pre_{(-)_{\overline{B}}/B},
		\]
		is defined by the derived mod $J$ reduction on the source.
	\end{enumerate}
\end{definition}
The conjugate filtration on the relative Hodge--Tate cohomology (\cite{BS22}) naturally extends to a filtration on the Hodge--Tate period presheaf.
\begin{corollary}[Conjugate filtration]
	\label{cor:conjuage_fil}
	Let $X$ be a bounded $p$-adic formal scheme, and let $(B,J)\in X_\Prism$ be a bounded prism.
	\begin{enumerate}[label=\upshape{(\roman*)}]
		\item\label{cor:conjuage_fil_graded} There is an $\mathbb{N}$-indexed ascending and exhaustive \emph{conjugate filtration} $\Fil^\conj_\bullet$ on $\overline{\Prism}^\pre_{(-)_{\overline{B}}/B}$, with its $i$-th graded piece being the presheaf of complexes $\wedge^i \mathbb{L}^{\pre}_{(-)_{\overline{B}}/\overline{B}}\{-i\}[-i]$.
		\item\label{cor:conjuage_fil_non-canonical} Assume $(A,I)\in X_\Prism$ is a framed regular prism with the framing $\Sigma\subset A^\times$, and let $\overline{A}_\perf$ be the reduction of the perfection of $A$.
		In addition, we assume $(B,J)$ is a flat prism or a flat crystalline prism.
		The restriction of $\gr^\conj_i \overline{\Prism}^\pre_{(-)_{\overline{B}}/B}$ on $X_{\eta,\pe}|_{\Spf(\overline{A}_\perf)_\eta}$ is isomorphic to the presheaf of $p$-complete $\overline{B}$-modules
		\[
		\left(  \Gamma^i \mathbb{L}_{A_0/W}\{-i\} \otimes^L_{A_0} \widehat{\mathcal{O}}^+|_{\Spf(\overline{A}_\perf)_\eta}\otimes^L_{\mathcal{O}_X} \overline{B} \right)^\wedge_p.
		\]
		In particular, the zero-th cohomology sheaves of both $\wedge^i \mathbb{L}^{\pre}_{(-)_{\overline{B}}/\overline{B}}\{-i\}[-i]|_{\Spf(\overline{A}_\perf)_\eta}$ and $\Fil^\conj_i\overline{\Prism}^\pre_{(-)_{\overline{B}}/B}|_{\Spf(\overline{A}_\perf)_\eta}$ are sheaves of finite projective modules over $\widehat{\mathcal{O}}^+_{\overline{B}}|_{\Spf(\overline{A}_\perf)_\eta}$.
	\end{enumerate}
\end{corollary}
Here we emphasize that the isomorphism in \Cref{cor:conjuage_fil}.\ref{cor:conjuage_fil_non-canonical} is \emph{not} Galois equivariant with respect to the pro-\'etale cover $\Spf(\overline{A}_\perf)_\eta \to \Spf(\overline{A})_\eta$.
\begin{proof}
	Part \ref{cor:conjuage_fil_graded} follows from the $p$-complete left Kan extension of the usual conjugate filtration as in \cite[Thm.\ 1.14]{BS22}.
	The identification of the cotangent complex in \ref{cor:conjuage_fil_non-canonical} follows from \Cref{prop:relative_prismatic_coh_of_perfectoid}.(\ref{prop:relative_prismatic_coh_of_perfectoid_cotangent_complex_2} and \ref{prop:relative_prismatic_coh_of_perfectoid_initial_prism}).
	The sheaf property follows from that of $\widehat{\mathcal{O}}^+$ (\cite[Prop.\ 8.5]{Sch22}) together with the flatness.
\end{proof}

Another useful construction is the de Rham specialization constructed below.
\begin{definition}[de Rham specialization]
	\label{def:dR_period_sheaf}
		Let $X$ be a bounded $p$-adic formal scheme, and let $(B,J)\in X_\Prism$ be a bounded prism.
	\begin{enumerate}[label=\upshape{(\roman*)}]
		\item The  \emph{de Rham period presheaf} on $X_{\eta,\pe}$, denoted as  $\dR^{\pre}_{(-)_{\overline{B}}/\overline{B}}$, is an $\overline{B}$-linear presheaf of $p$-complete complexes that sends an affinoid perfectoid space $U=\Spa(S[1/p],S) \in X_{\eta,\pe}$ onto the $p$-complete derived de Rham complex $\dR_{S_{\overline{B}}/\overline{B}}\in D_{p\text{-comp}}(\overline{B})\subset D(B)$.
		\item The \emph{de Rham specialization map} of presheaves on $X_{\eta,\pe}$ is the canonical composition
		\[
		\gamma^\dR\colon \Prism^\pre_{(-)_{\overline{B}}/B} \longrightarrow \Prism^{(1),\pre}_{(-)_{\overline{B}}/B}  \longrightarrow \dR^\pre_{(-)_{\overline{B}}/\overline{B}}.
		\]
	\end{enumerate}
\end{definition}
		By the relative prismatic--de Rham comparison theorem (\cite[Thm.\ 1.8.(3)]{BS22}, \cite[Prop.\ 5.2.3]{BL22a}), the de Rham realization map is equal to the $p$-complete derived pullback of the source along the composition $B \xrightarrow{\varphi_B} B \xrightarrow{\text{mod}~J~} \overline{B}$.
\begin{remark}[Hodge filtration]
	\label{rmk:Hodge_fil}
	Let $X$ be a bounded $p$-adic formal scheme, and let $(B,J)\in X_\Prism$ be a bounded prism.
	There is an $\mathbb{N}$-indexed descending \emph{Hodge filtration} on $\dR^\pre_{(-)_{\overline{B}}/\overline{B}}$, with its $i$-th graded piece being the presheaf of the complex $\wedge^i \mathbb{L}_{(-)_{\overline{B}}/\overline{B}}[-i]$.
	The zero-th graded piece then produces a natural map of presheaves of rings
	\[
	\dR^\pre_{(-)_{\overline{B}}/\overline{B}} \longrightarrow (\widehat{\mathcal{O}}^+\otimes_{\mathcal{O}_X} \overline{B})^\wedge_p.
	\]
	Moreover, we can define the Hodge-filtered completed de Rham period presheaf $\widehat{\dR}^{\pre}_{(-)_{\overline{B}/\overline{B}}}$ over $X_{\eta,\pe}$.
\end{remark}

Analogous to the Hodge--Tate period sheaf, we also obtain the following property of the de Rham period sheaf when the base prism admits a flatness assumption.
Note that this in particular generalizes the construction of \cite{GL21} to the regular setting.
\begin{proposition}
	\label{prop:property_of_dR}
	Let $X$ be a regular $p$-adic formal scheme, and let $(A,I)\in X_\Prism$ be a framed regular prism with framing $\Sigma\subset A^\times$, with $\overline{A}_\perf$ the reduction of its perfection.
	Assume $(B,J)\in X_\Prism$ is a flat prism or a flat crystalline prism.
	\begin{enumerate}[label=\upshape{(\roman*)}]
		\item\label{prop:property_of_dR_sheaf} The restriction of the presheaf $\mathcal{H}^0(\dR^\pre_{(-)_{\overline{B}}/\overline{B}})$ on $X_{\eta,\pe}|_{\Spf(\overline{A}_\perf)_\eta}$ is a sheaf in $\overline{B}$-modules.
		\item\label{prop:property_of_dR_explicit} Assume $(B,J)=(A,I)$.
		Then the restriction of the presheaf $\dR^\pre_{(-)_{\overline{A}}/\overline{A}}$ onto $X_{\eta,\pe}|_{\Spf(\overline{A}_\perf)_\eta}$ is naturally isomorphic to the complete divided power polynomial ring over $\widehat{\mathcal{O}}^+$ as below
		\[
		\dR^\pre_{(-)_{\overline{A}}/\overline{A}}|_{\Spf(\overline{A}_\perf)_\eta} \simeq \widehat{\mathcal{O}}^+|_{\Spf(\overline{A}_\perf)_\eta}\{\overline{t}_i\otimes1 -1\otimes t_i\}^\pd,
		\]
		where $\overline{t}_i$ is the image of $t_i\in A$ along the map $A\to \overline{A} \to \widehat{\mathcal{O}}^+|_{\Spf(\overline{A}_\perf)_\eta}$, and the base map $\overline{A}\to \dR^\pre_{(-)_{\overline{A}}/\overline{A}}$ sends $t_i$ onto $1\otimes t_i$.
	\end{enumerate}
\end{proposition}
\begin{proof}
	By \Cref{cor:explicit_desciption_of_period_sheaf}, we know the restriction of $\mathcal{H}^0(\Prism^\pre_{(-)_{\overline{A}}/A})$ on $X_{\eta,\pe}|_{\Spf(\overline{A}_\perf)_\eta}$ is a sheaf in $(p,I)$-complete modules and is completely flat over $\Ainf$.
	On the other hand, by \Cref{prop:regular prism Frobenius} the map $\varphi_A:A\to A$ is finite flat.
	So the zero-th cohomology sheaf of the Frobenius pullback $\Prism^{(1),\pre}_{(-)_{\overline{A}}/A}=\varphi_A^* \Prism^\pre_{(-)_{\overline{A}}/A}$ is a sheaf and the composition $\Ainf\to \Prism^\pre_{(-)_{\overline{A}}/A} \to \Prism^{(1),\pre}_{(-)_{\overline{A}}/A}$ is completely flat as well.
	In particular, thanks to the $I$-completeness and $I$-torsionfreeness of $\mathcal{H}^0(\Prism^{(1),\pre}_{(-)_{\overline{A}}/A})$, the mod $I$ reduction is a sheaf in $p$-completely flat modules over $\Ainf/\varphi^{-1}(I)$.
	The latter, combined with the $p$-complete sheaf property of $\Ainf/\varphi^{-1}(I)$, implies the sheaf property of $\mathcal{H}^0(\dR^\pre_{(-)_{\overline{A}}/\overline{A}})$ over $X_{\eta,\pe}|_{\Spf(\overline{A}_\perf)_\eta}$, together with its $p$-complete flatness over $\Ainf/\varphi^{-1}(I)$.
	Here we note that since $\Spf(\overline{A})\to X$ is $p$-completely \'etale and the map $\Spf(\overline{A}_\perf)_\eta\to \Spf(\overline{A})_\eta$ is pro-finite-\'etale, for any affinoid perfectoid space $U=\Spf(S[1/p],S)\in X_{\eta,\pe}|_{\Spf(\overline{A}_\perf)_\eta}$, the structural map $\Spf(S)\to X$ factors through the \'etale morphism $\Spf(\overline{A})\to X$.
	In particular, the restriction $\dR^\pre_{(-)_{\overline{A}}/\overline{A}}$ on $X_{\eta,\pe}|_{\Spf(\overline{A}_\perf)_\eta}$ is naturally isomorphic to that of $\bigl(\dR^\pre_{(-)/X}\otimes_{\mathcal{O}_X} \overline{A}\bigr)^\wedge_p$.
	As a consequence, the presheaf of $p$-complete modules $\mathrm{H}^0(\dR^\pre_{(-)/X})$ is a sheaf of $p$-complete $p$-torsionfree modules on $X_{\eta,\pe}|_{\Spf(\overline{A}_\perf)_\eta}$.
	
	For more general base, if $(B,J)\in X_\Prism$ is a transversal flat prism, then \Cref{ass:flat} implies that the map of $p$-adic formal scheme $\Spf(\overline{B})\to X$ is $p$-completely flat.
	So the sheaf property of $\mathcal{H}^0(\dR^\pre_{(-)_{\overline{B}}/\overline{B}})$ follows from that of the $p$-complete $p$-torsionfree sheaf $\mathcal{H}^0(\dR^\pre_{(-)/X})$, together with the $p$-completely flat base change along $\Spf(\overline{B})\to X$.
	For the second scenario when $(B,J)$ is a crystalline flat prism, the $\mathbb{F}_p$-scheme $\Spec(\overline{B})$ is flat over the reduction $X_{p=0}$ by \Cref{ass:flat}.
	So the sheaf property of the $\mathbb{F}_p$-linear presheaf $\dR^\pre_{(-)_{\overline{B}}/\overline{B}}$ follows from that of the mod $p$ reduction of the sheaf $\dR^\pre_{(-)/X}$ on $X_{\eta,\pe}|_{\Spf(\overline{A}_\perf)_\eta}$, together with the flat base change of the algebraic derived de Rham complex.
	This finishes the proof of Part \ref{prop:property_of_dR_sheaf}.
	
	For part \ref{prop:property_of_dR_explicit}, the claim follows from the explicit presentation of the relative prismatic cohomology in \Cref{cor:explicit_desciption_of_period_sheaf}, together with the formula that $\dR_{(-)_{\overline{A}}/\overline{A}}= (\varphi_A^*\Prism^\pre_{(-)_{\overline{A}}/A})\otimes_A \overline{A}$ (\Cref{def:dR_period_sheaf}).	
\end{proof}


\subsection{Period sheaves and their global sections}
\label{sub:coh_of_sheaf}
In this subsection, we introduce the integral prismatic period sheaves and calculate their global sections.

\begin{definition}[Prismatic period sheaves]
	\label{def:period_sheafification}
	Let $X$ be a regular $p$-adic formal scheme over $\mathcal{O}_K$, and let $(B,J)\in X_\Prism$ be a flat prism or a flat crystalline prism.
	We define the following sheaves of complete complexes over $X_{\eta,\pe}$:
	\begin{enumerate}
		\item The \emph{prismatic period sheaf} over $(B,J)$, denoted as $\Prism_{(-)_{\overline{B}}/B}$, is the $(p,J)$-complete hypercomplete sheafification of the presheaf $\Prism^\pre_{(-)_{\overline{B}}/B}$.
		\item The \emph{Hodge--Tate period sheaf} over $(B,J)$, denoted as $\overline{\Prism}_{(-)_{\overline{B}}/B}$, is the $p$-complete hypercomplete sheafification of the presheaf $\overline{\Prism}^\pre_{(-)_{\overline{B}}/B}$.
		\item The \emph{de Rham period sheaf} over $(B,J)$, denoted as $\dR_{(-)_{\overline{B}}/\overline{B}}$, is the $p$-complete hypercomplete sheafification of the presheaf $\dR^\pre_{(-)_{\overline{B}}/\overline{B}}$.
		\item We use $\mathbb{L}_{(-)_{\overline{B}}/\overline{B}}$ to denote the $p$-complete hypercomplete sheafification of the presheaf of $p$-complete cotangent complex $\mathbb{L}^\pre_{(-)_{\overline{B}}/\overline{B}}$.
	\end{enumerate}
\end{definition}
Analogously, we may define the Frobenius twisted version $\Prism^{(1)}_{(-)_{\overline{B}}/B}$ and its Nygaard-filtered completion $\widehat{\Prism}^{(1)}_{(-)_{\overline{B}}/B}$, together with the Hodge-filtered completion $\widehat{\dR}_{(-)_{\overline{B}}/\overline{B}}$, which we would not repeat.

By \Cref{lem:perfection_of_framed_regular_prism}, given a regular $p$-adic formal scheme $X$ over $\mathcal{O}_K$ and a framed regular prism $(A,I)$ of $X$ with $\Sigma\subset A^\times$, the generic fiber of the reduction of the perfection $\Spf(\overline{A}_\perf)$ forms a pro-\'etale cover of the rigid space $\Spf(\overline{A})_\eta$.
Moreover, since each point in $X$ admits an \'etale neighborhood that lifts to a framed regular prism (\Cref{thm:regular prism}), the above allows us to produce a covering of the pro-\'etale site of $X_\eta$.
Namely, we let $\{(A_i,I_i)\}_i$ be a collection of framed regular prisms in $X_\Prism$ such that $\{\Spf(\overline{A}_i)\}_i$ covers $X$ in the $p$-completely \'etale topology, then the affinoid perfectoid space $Y_\eta=\coprod_i \Spf(\overline{A}_{i,\perf})_\eta$ is a covering object in $X_{\eta,\pe}$.
On the other hand, by \Cref{thm:prismatic_period_sheaf}.\ref{thm:prismatic_period_sheaf_property} (resp. \Cref{cor:conjuage_fil}.\ref{cor:conjuage_fil_non-canonical}, resp. \Cref{prop:property_of_dR}.\ref{prop:property_of_dR_sheaf}), we know the zero-th cohomology presheaf of $\Prism^\pre_{(-)_{\overline{B}}/B}|_{Y_\eta}$ (resp. $\overline{\Prism}^\pre_{(-)_{\overline{B}}/B}|_{Y_\eta}$, resp. $\dR^\pre_{(-)_{\overline{B}}/\overline{B}}|_{Y_\eta}$) satisfies the sheaf condition over $X_{\eta,\pe}|_{Y_\eta}$.
\begin{corollary}
	\label{cor:sheafification_is_unfolding}
	Let $X$ be a regular $p$-adic formal scheme over $\mathcal{O}_K$, let $(B,J)\in X_\Prism$ be a flat prism or a flat crystalline prism, and let $Y_\eta$ be the affinoid perfectoid space mentioned above.
	The restriction of the period sheaf $\mathcal{H}^0(\Prism_{(-)_{\overline{B}}/B})$ on $Y_{\eta,\pe}$ is naturally isomorphic to the presheaf $\Prism^\pre_{(-)_{\overline{B}}/B}|_{Y_\eta}$.
	Similarly for $\overline{\Prism}_{(-)_{\overline{B}}/B}$ and $\dR_{(-)_{\overline{B}}/\overline{B}}$.
\end{corollary}

We now calculate the global section of various period sheaves.
We start with the de Rham cohomology.
\begin{proposition}
	\label{prop:global_section_dR}
	Let $X=\Spf(R)$ be an affine regular $p$-adic formal scheme over $\mathcal{O}_K$.
	\begin{enumerate}[label=\upshape{(\roman*)}]
		\item\label{prop:global_section_dR_0}
		The structural map below is an isomorphism
		\[
		R \longrightarrow \mathrm{H}^0_\proet(X_\eta, \gr^0_H \dR_{(-)/X}),
		\]
		and $\mathrm{H}^0_\pe(X_\eta, \gr^i_H \dR_{(-)/X})=0$ for $i\geq 1$.
		\item\label{prop:global_section_dR_inj}
		Let $\widehat{\dR}_{(-)/X}$ be the Hodge-filtered-completion of $\dR_{(-)/X}$.
		The canonical map 
		\[
		\mathrm{H}^0_\pe(X_\eta, \dR_{(-)/X}) \longrightarrow \mathrm{H}^0_\pe(X_\eta, \widehat{\dR}_{(-)/X})
		\]
		is an injection of $p$-complete modules.
		In particular, we have a natural isomorphism $R\xrightarrow{\sim} \mathrm{H}^0_\pe(X_\eta, \dR_{(-)/X})$.
	\end{enumerate}
\end{proposition}
\begin{proof}
	Notice that each graded piece $\gr^i_H \dR_{(-)/X}$ is naturally isomorphic to the sheaf of $p$-complete K\"ahler differentials $\wedge^i\mathbb{L}_{\widehat{\mathcal{O}}^+/\mathcal{O}_X}[-i]$ over $X_{\eta, \pe}$. 
	The latter by \Cref{prop:relative_prismatic_coh_of_perfectoid}.\ref{prop:relative_prismatic_coh_of_perfectoid_cotangent_complex_2} is an $\widehat{\mathcal{O}}^+$-vector bundle and in particular is $p$-torsionfree.
	So Part \ref{prop:global_section_dR_0} follows from \Cref{lem:global_sec_of_Ohat} and the vanishing result in 
	\Cref{prop:LZ_finiteness}.\ref{prop:LZ_finiteness_general}.
	
	The injectivity in Part \ref{prop:global_section_dR_inj} is equivalent to the separatedness of the Hodge filtration on $\mathrm{H}^0_\pe(X_\eta, \dR_{(-)/X})$.
	To prove the latter, we notice that the map $R\to \widehat{\mathcal{O}}^+$ factors through $R\to (\mathbb{A}_{\inf}\otimes_W R)^\wedge_p \to \widehat{\mathcal{O}}^+$, where the first arrow is relatively perfect and the second arrow is a surjection.
	In particular, the maps induce a natural filtered isomorphism
	\[
	\dR_{(-)/X} \simeq \dR_{\widehat{\mathcal{O}}^+/(\mathbb{A}_{\inf}\otimes_W R)^\wedge_p}.
	\]
	
	To analyze the above de Rham cohomology, we may assume $X$ (hence $X_\eta$) is connected.
	By shrinking $X$ in \'etale topology if necessary, we let $(A,I=(d))$ be a framed regular prism with a framing $A_0=W\langle t_1^{\pm1},\ldots,t_n^{\pm1} \rangle$ and assume $\overline{A}=R$.
	Assume $S$ is the perfectoid ring $\overline{A}_\perf$, so that there is an induced map of prisms $f:(A,I)\to (\rAinf(S)=\Prism_S, I_{\Prism_S})$, which by the noetherian property of $A$ and \Cref{prop:regular prism Frobenius} is a faithfully flat cover.
	Here we notice that by \Cref{lem:perfection_of_framed_regular_prism} the generic fiber of $\Spf(S)$ is a pro-finite-\'etale cover of $X_\eta$, and by the sheaf property in \Cref{prop:property_of_dR} it suffices to show that the Hodge filtration is separated on $\dR_{S/R}$.
	
	Now we consider the following cartesian diagram of rings
	\begin{equation}
		\begin{tikzcd}
			(A\otimes_W A)^\wedge_p / (I\otimes 1) \ar[r] \ar[d] & (A\otimes_W A)^\wedge_p / (I\otimes 1, \Delta) \simeq R \ar[d]\\
			(A\otimes_W \rAinf(S))^\wedge_p / ( I\otimes 1) \ar[r]& (A\otimes_W \rAinf(S))^\wedge_p / (I\otimes 1, \Delta) \simeq (R\otimes_W \rAinf(S))^\wedge_p / (\Delta) \simeq S,
		\end{tikzcd}
	\end{equation}
    where $\Delta = (t_1\otimes 1 - 1\otimes t_1, \ldots, t_n\otimes 1 - t_n \otimes 1)$ is the diagonal ideal of $(A\otimes_W A)^\wedge_p$, and the vertical maps are faithfully flat.
    Notice that since the sequence of elements $\{t_1\otimes 1 - 1\otimes t_1, \ldots, t_n\otimes 1 - t_n \otimes 1, d\otimes 1\}$ is regular (and thus Koszul-regular) in the noetherian ring $(A\otimes_W A)^\wedge_p$, by \cite[\href{https://stacks.math.columbia.edu/tag/09CC}{Tag 09CC}]{stacks-project} we can switch the positions without affecting the regularity.
    In particular, the elements $\{t_1\otimes 1 - 1\otimes t_1, \ldots, t_n\otimes 1 - t_n \otimes 1\}$ is a regular sequence in the ring $(A\otimes_W A)^\wedge_p / (I\otimes 1)$, and hence also in the ring $(A\otimes_W \rAinf(S))^\wedge_p / ( I\otimes 1)$ by the faithfully flatness.
    As a consequence, the kernel of the canonical surjection $(R\otimes_W \rAinf(S))^\wedge_p / (I\otimes 1, \Delta) \to S$ is generated by a regular sequence.
    Hence the claim follows from the more general observation on the relative de Rham cohomology as in \Cref{lem:dR_is_sep}.	
\end{proof}
The following general result on the separatedness of the relative de Rham cohomology was used in the proof, and we thank Shizhang Li for the related discussions.
\begin{lemma}[Separatedness of the Hodge filtration]
	\label{lem:dR_is_sep}
	Let $A\to B$ be a surjection of $p$-complete and $p$-torsionfree rings with the kernel generated by a regular sequence.
	The Hodge filtration on the $p$-complete de Rham cohomology $\dR_{B/A}$ is separated.
\end{lemma}
\begin{proof}
	We let $J$ be the kernel ideal of the surjection $A\to B$.
	We first notice that by \cite{Bha12}, the relative de Rham cohomology $\dR_{B/A}$ is naturally filtered isomorphic to the $p$-complete divided power envelope $D_A(I)$, where the Hodge filtrations are sent onto the ($p$-complete) divided power filtrations.
	As each $\Fil_H^j\dR_{B/A}=\Fil^j D_A(I)$ is a derived $p$-complete submodule of $D_A(I)$, the intersection $M\colonequals\cap_{j\in \mathbb{N}} \Fil^j D_A(I)$ is also $p$-complete and $p$-torsionfree.
	Thus by the derived Nakayama lemma, to show that $M=0$, it suffices to check it after the derived mod $p$ reduction.
	So from now on we assume $A\to B$ is a regular surjection of $\mathbb{F}_p$-algebras.
	Moreover, by the K\"unneth formula in \cite[Prop.\ 2.7, Lem.\ 3.37]{Bha12} and the aforementioned description of the Hodge filtration, we may assume the kernel is generated by one regular element $f\in A$.
	Under the latter assumption, the ring $D_A(I)$ is naturally isomorphic to the quotient of the divided power polynomial ring, namely 
	\[
	D_A(I)\simeq A\{x\}^{\pd}/(x-f) = (\bigoplus_j A\cdot \gamma_j(x))/(f-x),
	\]
	where $\gamma_j(x)$ is the $j$-th divided power structure of the element $x$.
	In addition, the $j$-th Hodge filtration of $D_A(I)$ is the submodule generated by $\{\gamma_i(x)~|~i\geq j\}$.
	
	To continue, we notice that since $A$ is $f$-torsionfree, its $f$-adic completion is flat over $\mathbb{F}_p\llbracket x \rrbracket$.
	In particular, there is an affine open neighborhood $\Spec(A')$ of $\Spec(B)=V(f)$ in $\Spec(A)$ such that $A'$ is flat over $\mathbb{F}_p[x]$.
	Note that the \'etaleness of $A\to A'$ implies that  $\dR_{B/A}=\dR_{B/A'}$.
	So by replacing $A$ with $A'$ if necessary, we may assume $A$ is flat over $\mathbb{F}_p[x]$.
	
	Furthermore, we recall that the ring $\dR_{B/A}$ naturally admits an ascending and exhaustive conjugate filtration $\Fil^\conj_i \dR_{B/A}$.
	In particular, each element $x\in M\subseteq  \dR_{B/A}$ is contained in $\Fil^\conj_n\dR_{B/A}$ for some $n\in \mathbb{N}$.
	We then claim that for each $n$, there is an integer $m\in \mathbb{N}$ such that $\Fil^\conj_n \dR_{B/A} \cap \Fil^m_H \dR_{B/A} =0$.
	Granting the claim, if the intersection $M\neq 0$, we may pick a nonzero element $x\in M$, which is contained in $\Fil^\conj_n \dR_{B/A}$ for some $n$.
	The latter however is impossible, since $x \notin \Fil^m_H\dR_{B/A}\supseteq M$.
	
	Finally we check the claim.
	By the flat base change formulae of the conjugate and the Hodge filtrations, it suffices to assume $A=\mathbb{F}_p[x]$ and $f=x$.
	Then the divided power envelope $D_A(I)$ is equal to the divided power polynomial ring $\mathbb{F}_p\{x\}^{\pd}$, where $\Fil^j_H D_A(I)$ is the submodule $\bigoplus_{i\geq j} \mathbb{F}_p\cdot \gamma_i(x)$.
	On the other hand, by \cite[Proof of Lem.\ 3.42]{Bha12}, we know the conjugate filtration $\Fil^\conj_n D_A(I)$ is the $\mathbb{F}_p[x]/x^p$-submodule  $\bigoplus_{i=0}^n \mathbb{F}_p[x]/x^p\cdot \gamma_{ip}(x)$.
	In particular, by comparing the descriptions of the Hodge filtration with that of the conjugate filtration, we notice that for a fixed integer $n\in \mathbb{N}$, the intersection of $\Fil^\conj_n D_A(I)$ and $\Fil^m_H D_A(I)$ is zero if $m>pn+p-1$.
\end{proof}
\begin{corollary}
	\label{cor:Nygarrd_filtration_separated}
	Let $X$ be a regular $p$-adic formal scheme over $\mathcal{O}_K$, let $(B,J)\in X_\Prism$ be a flat prism, and let $S$ be the perfectoid ring over $X$ given by the reduction of the perfection of a framed regular prism.
	Then the Nygaard filtration on $\Prism^{(1)}_{S_{\overline{B}}/B}$ is separated.
\end{corollary}
\begin{proof}
	Notice that the Nygaard filtration is finer than the $J$-adic filtration on $\Prism^{(1)}_{S_{\overline{B}}/B}$, where the latter is classically $J$-adically complete and is in particular $J$-adically separated.
	So it suffices to check that the mod $J$-reduction of the Nygaard filtration, which by \cite{BS22} is the Hodge filtration on $\dR_{S_{\overline{B}}/\overline{B}}$, is separated.
	Hence the claim follows from  \Cref{prop:regular prism Frobenius}, the flatness assumption of the prism $(B,J)$ in $X_\Prism$, and \Cref{lem:dR_is_sep}.
\end{proof}

Below for a bounded prism $(B,J)$ and a $p$-adic formal scheme $Y$ over $\overline{B}$, we call the tautological map $f:B\to \Prism_{Y/B}$ the structure map, and call its Frobenius pullback $\varphi_B^*(f):B\to \Prism^{(1)}_{Y/B}$ the twisted structure map.
\begin{theorem}
	\label{thm:global_section}
	Let $X$ be a regular $p$-adic formal scheme over $\mathcal{O}_K$, and let $(B,J)\in X_\Prism$ be a flat prism.
	\begin{enumerate}[label=\upshape{(\roman*)}]
		\item\label{thm:global_section_Hodge-Tate} The structural map induces an isomorphism
		\[
		\overline{B}\xrightarrow{\sim} \mathrm{H}^0_\pe(X_\eta, \gr^\conj_0 \overline{\Prism}_{(-)_{\overline{B}}/B}),
		\]
		and $\mathrm{H}^0_\pe(X_\eta, \gr^\conj_i \overline{\Prism}_{(-)_{\overline{B}}/B})=0$ for $i\geq 1$.
		\item\label{thm:global_section_twisted_prismatic} The twisted structural map induces an isomorphism
		\[
		B \xrightarrow{\sim} \mathrm{H}^0_\pe(X_\eta, \Prism^{(1)}_{(-)_{\overline{B}}/B}).
		\]
		\item\label{thm:global_section_prismatic} Assume either $(B,J)$ is a coproduct of framed regular prisms, or the map $\varphi_B$ is completely faithfully flat.
		Then the structural map induces an isomorphism
		\[
		B \xrightarrow{\sim} \mathrm{H}^0_\pe(X_\eta, \Prism_{(-)_{\overline{B}}/B}).
		\]
	\end{enumerate}
\end{theorem}
\begin{proof}
	For Part \ref{thm:global_section_Hodge-Tate}, we notice that each graded piece $\gr^\conj_i \overline{\Prism}_{(-)_{\overline{B}}/B}$ is naturally isomorphic to the sheaf of $p$-complete K\"ahler differentials $\wedge^i\mathbb{L}_{(-)_{\overline{B}}/\overline{B}}\otimes_{\overline{B}} \overline{B}\{-i\}[-i]$ over $X_{\eta, \pe}$.
	By the flat base change property of the cotangent complex, we further have
	\[
	\gr^\conj_i \overline{\Prism}_{(-)_{\overline{B}}/B} \simeq \left( \bigl( \wedge^i\mathbb{L}_{\widehat{\mathcal{O}}^+/\mathcal{O}_X}[-i] \bigr) \otimes_{\mathcal{O}_X} \overline{B}\otimes_{\overline{B}}\overline{B}\{-i\} \right)^\wedge_p.
	\]
	By \Cref{prop:relative_prismatic_coh_of_perfectoid}.\ref{prop:relative_prismatic_coh_of_perfectoid_cotangent_complex_2}, we know $\wedge^i\mathbb{L}_{\widehat{\mathcal{O}}^+/\mathcal{O}_X}[-i]$ is an $\widehat{\mathcal{O}}^+$-vector bundle and in particular is $p$-torsionfree.
	So Part \ref{thm:global_section_Hodge-Tate} follows from \Cref{lem:global_sec_of_Ohat} (for $i=0$) and the vanishing result in 
	\Cref{prop:LZ_finiteness}.\ref{prop:LZ_finiteness_H0} (for $T=\mathbb{Z}_p$ and $i>0$), together with the projection formula \Cref{thm:projection_formula}.\ref{thm:projection_formula_H1} (which is applicable thanks to the flatness of $\overline{B}$ over $X$ and \Cref{thm:finiteness_integral_proet_coh}).
	
	For Part  \ref{thm:global_section_twisted_prismatic}, we notice that \Cref{prop:global_section_dR} and the filtered version of the prismatic-de Rham comparison theorem (\cite{BS22}) implies the natural map below is an isomorphism
	\[
	B \longrightarrow \mathrm{H}^0_\pe(X_\eta, \widehat{\Prism}^{(1)}_{(-)_{\overline{B}}/B}),
	\]
	where $\widehat{\Prism}^{(1)}_{(-)_{\overline{B}}/B}$ is the sheaf associated to the Nygaard-completed twisted relative prismatic cohomology.
	The above map can be naturally factored as
	\begin{equation}
		\label{eq:global_sections_of_twisted_prismatic_cohomology}
		B \longrightarrow \mathrm{H}^0_\pe(X_\eta, \Prism^{(1)}_{(-)_{\overline{B}}/B}) \longrightarrow \mathrm{H}^0_\pe(X_\eta, \widehat{\Prism}^{(1)}_{(-)_{\overline{B}}/B}),
	\end{equation}
	which implies that the first arrow is an injection.
	So to prove \ref{thm:global_section_twisted_prismatic}, it suffices to check that the second arrow above is also injective; or in other words the Nyggard filtration on $\mathrm{H}^0_\pe(X_\eta, \Prism^{(1)}_{(-)_{\overline{B}}/B})$ is separated.
	
	As the statement is $p$-completely \'etale local with respect to $X$, by shrinking $X$ in \'etale topology if necessary, we assume $X=\Spf(R)$ is affine and admits an auxiliary framed regular prism $(A,I=(d))$ with a framing $A_0=W\langle t_1^{\pm1},\ldots,t_n^{\pm1} \rangle$ and the equality $\overline{A}=R$.
	Assume $S$ is the perfectoid ring $\overline{A}_\perf$.
	Then by \Cref{lem:perfection_of_framed_regular_prism}, we know the generic fiber of $\Spf(S)$ is a pro-finite-\'etale cover of $X_\eta$, and by \Cref{thm:prismatic_period_sheaf} it suffices to show that the Nygaard filtration is separated on $\Prism^{(1)}_{S_{\overline{B}}/B}$, which was verified in \Cref{cor:Nygarrd_filtration_separated}.

	Finally, we consider Part \ref{thm:global_section_prismatic}, where we denote the tautological arrow $B\to \mathrm{H}^0_\pe(X_\eta, \Prism_{(-)_{\overline{B}}/B})$ by $g$.
	Let $S$ be the perfectoid ring considered above, and we first assume $\varphi_B$ is completely faithfully flat.
	There is a natural injection of derived $(p,J)$-complete modules $\mathrm{H}^0_\pe(X_\eta, \Prism_{(-)_{\overline{B}}/B}) \to \Prism_{S_{\overline{B}}/B}$, such that $(J,p)$ forms a regular sequence for each of them.
	Now we notice that by applying \Cref{lem:inj_and_completely_proj_mod} at the homomorphism $\varphi_B:B\to B$ and the map $M\colonequals \mathrm{H}^0_\pe(X_\eta, \Prism_{(-)_{\overline{B}}/B}) \hookrightarrow N\colonequals \Prism_{S_{\overline{B}}/B}$, we know the completed base change below is injective:
	\begin{equation}
		\label{eq:global_sections_of_twisted_prismatic_cohomology_2}
		\bigl(\mathrm{H}^0_\pe(X_\eta, \Prism_{(-)_{\overline{B}}/B}) \otimes_{B, \varphi_B} B \bigr)^\wedge_{(p,J)} \longrightarrow \Prism^{(1)}_{S_{\overline{B}}/B}.
	\end{equation}
	In addition, the map (\ref{eq:global_sections_of_twisted_prismatic_cohomology_2}) naturally factors through the subring $\mathrm{H}^0_\pe(X_\eta, \Prism^{(1)}_{(-)_{\overline{B}}/B})$.
    So by combining with (\ref{eq:global_sections_of_twisted_prismatic_cohomology}), we obtain the natural $B$-linear injective arrows as below
    \[
    B \hookrightarrow \bigl(\mathrm{H}^0_\pe(X_\eta, \Prism_{(-)_{\overline{B}}/B}) \otimes_{B, \varphi_B} B \bigr)^\wedge_{(p,J)} \hookrightarrow \mathrm{H}^0_\pe(X_\eta, \Prism^{(1)}_{(-)_{\overline{B}}/B}) \hookrightarrow \Prism^{(1)}_{S_{\overline{B}}/B},
    \]
    where the first arrow is the base change of the map $g$ by $\varphi_B$, and the composition of the first two arrows by Part \ref{thm:global_section_twisted_prismatic} is an isomorphism.
    Thus the canonical map $B\to \bigl(\mathrm{H}^0_\pe(X_\eta, \Prism_{(-)_{\overline{B}}/B}) \otimes_{B, \varphi_B} B \bigr)^\wedge_{(p,J)}$ is isomorphic.
    As a consequence, the complete faithfully flatness of the map $\varphi_B$ now implies that the derived complete complex $\cofib(g)$ vanishes itself, and thus $g$ is an isomorphism.
    
    If $(B,J)$ is a coproduct of framed regular prisms, then we can consider the following commutative diagram
    \[
    \begin{tikzcd}
    B \arrow[d, "\varphi_B"'] \arrow[r,"g"] &  \mathrm{H}^0_\pe(X_\eta, \Prism_{(-)_{\overline{B}}/B}) \ar[r] \ar[d] & \Prism_{S_{\overline{B}}/B} \ar[d]\\
    B \arrow[r] & \mathrm{H}^0_\pe(X_\eta, \Prism^{(1)}_{(-)_{\overline{B}}/B}) \ar[r] & \Prism^{(1)}_{S_{\overline{B}}/B},
    \end{tikzcd}
    \]
    where the vertical arrows are the linearization maps along $\varphi_B$, and all the arrows are injective (\Cref{prop:intersection_of_Frobenius_of_base_with_coproduct}).
    Moreover, the bottom left horizontal arrow by \ref{thm:global_section_twisted_prismatic} is an isomorphism.
    Hence the surjectivity of the map $g$ follows from the fact that the outer square is cartesian, as in \Cref{prop:intersection_of_Frobenius_of_base_with_coproduct}.\ref{prop:intersection_of_Frobenius_of_base_with_coproduct:intersection}.
\end{proof}

\subsection{Rational period sheaves}
\label{sub:rat_period_sheaf}
Similarly to the various period rings in the classical $p$-adic Hodge theory, the prismatic period sheaves defined in \Cref{def:prismatic_period_sheaf} admit rational analogues.
However, to yield an integral Galois invariant, we consider the localization by inverting the element $\mu\in \rAinf$, instead of the element $p$ in the classical situations.

We start by recalling the $F$-gauge structure on the Breuil--Kisin twist.
\begin{construction}[Breuil--Kisin twist as an $F$-gauge]
	\label{const:filtered_BK_twist}
	Let $(A,I)$ be a transversal prism.
	The first Breuil--Kisin twist over $A$ is the invertible $A$-module
	\[
	A\{1\} \colonequals \lim_{\cdot \frac{1}{p}} I_r/I_r^2,
	\]
	where $I_r$ is the ideal $I\cdot \varphi_A(I)\cdots \varphi_A^{r-1}(I)$ in $A$ (\cite[Const.\ 2.2.11]{BL22a}).
	It satisfies the natural base change formula with respect to maps of prisms, making it an invertible prismatic crystal $\mathcal{O}_\Prism\{1\}$ over $(\mathbb{Z}_p)_\Prism$.
	Moreover, as in \cite[Const.\ 2.2.14]{BL22a}, Frobenius endomorphism $\varphi_A$ induces a natural isomorphism
	\begin{equation}
		\label{eq:BK_twist_Frob}
			\varphi_{A\{1\}}: \varphi_A^* (A\{1\}) \longrightarrow \frac{1}{I}\cdot A\{1\}.
	\end{equation}
	The latter allows us to enhance $\mathcal{O}_\Prism\{1\}$ with a canonical filtration by taking the Frobenius preimage as in \cite[Thm.\ 1.2]{GL23}.
	Concretely, assume $(A,I)=(\Prism_S, I)$ for a quasiregular semiperfectoid ring $S$.
	we then equip $\mathcal{O}_{\Prism}\{1\}$ with the filtration such that its value at $(A,I)$ is
	\[
	\Fil^i  (A\{1\}) = \{x\in A\{1\}~|~\widetilde{\varphi}_{A\{1\}}(x) \in I^i A\{1\} \subset \frac{1}{I}\cdot A\{1\} \},
	\]
	where $\widetilde{\varphi}_A$ is the precomposition of $\varphi_A$ with the Frobenius-linearization $A\{1\}\to  \varphi_A^* (A\{1\})$.
	So for $i\leq -1$, we have $\Fil^i (A\{1\}) = A\{1\}$.
	For example, in the special case when $S$ is perfectoid, by (\ref{eq:BK_twist_Frob}) the submodule $\Fil^i A\{1\}$ is isomorphic to $\varphi_A^{-1}(I)^{i+1}A\{1\}$ through the Frobenius structure for $i\geq -1$.
	As shown in \cite{GL23}, each filtered module $\Fil^\bullet (A\{1\})$ satisfies the filtered base change formula among prisms arising from the syntomic site $(\mathbb{Z}_p)_\mathrm{qsyn}$, making $\mathcal{O}_\Prism\{1\}$ an invertible $F$-gauge.
	In addition, the similar construction extends to $\mathcal{O}_\Prism\{n\}=(\mathcal{O}_\Prism\{1\})^{\otimes n}$ for any $n\in \mathbb{Z}$.
\end{construction}
To analyze the rational prismatic period sheaf, the following construction plays an important role.
\begin{construction}[The twist by $\mu$]
	\label{const:Ainf_twist_by_mu}
	We let $K_\infty$ be a perfectoid Galois extension of $K$ that contains all the $p$-power roots of unity $\zeta_{p^n}$.
	Then by choosing a compatible system of $\zeta_{p^n}$, we can define the element $\mu=[\epsilon]-1\in \rAinf(\mathcal{O}_{K_\infty})$, which satisfies the formula
	\begin{equation}
		\label{eq:equation_of_mu}
			\varphi_{\rAinf(\mathcal{O}_{K_\infty})}(\mu)=\mu\cdot \frac{[\epsilon]^p-1}{[\epsilon]-1}.
	\end{equation}
	The latter implies that the invertible ideal $\mu\rAinf(\mathcal{O}_{K_\infty}) \subset \rAinf(\mathcal{O}_{K_\infty})$ is naturally a Frobenius submodule of height $[1,1]$.
	In particular, we can define a natural filtration using the Frobenius structure, namely
	\[
	\Fil^i \mu\rAinf(\mathcal{O}_{K_\infty}) \colonequals \{x\in \mu\rAinf(\mathcal{O}_{K_\infty})~|~\varphi_{\rAinf}(x)\in I^i\mu\rAinf(\mathcal{O}_{K_\infty})\}.
	\]
	Here we note that the submodule $\mu\rAinf(\mathcal{O}_{K_\infty}) \subset \rAinf(\mathcal{O}_{K_\infty})$ and its filtrations are preserved under the natural action of $\Gal(K_\infty/K)$, and the inclusion map is filtered.
	
	Conceptually, as explained in \cite[Ex.\ 4.24]{BMS1} (cf. \cite[\S 2.6]{BL22a}), the submodule $\mu\rAinf(\mathcal{O}_{K_\infty})$ admits a canonical (Galois and Frobenius) equivariant isomorphism with the double twist
	\[
	\rAinf(\mathcal{O}_{K_\infty})\{-1\}(1)\colonequals \rAinf(\mathcal{O}_{K_\infty})\{-1\}\otimes_{\mathbb{Z}_p} \mathbb{Z}_p(1),
	\] where $\rAinf(\mathcal{O}_{K_\infty})\{-1\}$ is the Breuil--Kisin twist and $\mathbb{Z}_p(1)$ is the Tate twist.
	Moreover, the $i$-th filtration on $\mu\rAinf(\mathcal{O}_{K_\infty})$ is naturally isomorphic to $\Fil^i (\rAinf(\mathcal{O}_{K_\infty})\{-1\} ) \otimes_{\mathbb{Z}_p} \mathbb{Z}_p(1)$, which we abbreviate as $\Fil^i \bigl(\rAinf(\mathcal{O}_{K_\infty})\{-1\} (1) \bigr)$.
	As before, the similar constructions naturally extend to $\mu^n\rAinf(\mathcal{O}_{K_\infty})$ for any $n \in \mathbb{Z}$.
\end{construction}
The following statement describes the difference among the different twists $\mu^m\rAinf(\mathcal{O}_{K_\infty})$.
\begin{proposition}
	\label{prop:Ainf_twist_difference}
	Let $n\geq m$ be two integers, and let $\iota_{n,m}$ be the canonical filtered inclusion $\mu^n \rAinf(\mathcal{O}_{K_\infty}) \to \mu^m \rAinf(\mathcal{O}_{K_\infty})$.
	We have the following descriptions of Galois modules:
	\begin{enumerate}[label=\upshape{(\roman*)}]
		\item\label{prop:Ainf_twist_difference_Fil} As submodules in $\rAinf(\mathcal{O}_{K_\infty})[\frac{1}{\mu}]$, we have 
		\[
		\Fil^i \mu^m \rAinf(\mathcal{O}_{K_\infty}) = \varphi_{\rAinf}^{-1}(I)^{\max\{i,m\}}\varphi_{\rAinf}^{-1}(\mu)^m\rAinf(\mathcal{O}_{K_\infty}).
		\]
		\item\label{prop:Ainf_twist_difference_graded}
		The associated graded of $\mu^m \rAinf(\mathcal{O}_{K_\infty})$ is given by
		\[
		\gr^i \mu^m \rAinf(\mathcal{O}_{K_\infty})  = \begin{cases}
		(\zeta_p-1)^m\cdot \varphi_{\rAinf}^{-1}(I)^{i}/\varphi_{\rAinf}^{-1}(I)^{i+1},~i\geq m;\\
		0,~i<m,
		\end{cases}
		\]
		which is functorial in $i$ and $m$.
		In particular, the entire associated graded $\gr^\bullet \mu^m \rAinf(\mathcal{O}_{K_\infty})$ is an invertible module over $\gr^\bullet_N \Prism_{\mathcal{O}_{K_\infty}}=\gr^\bullet_{\varphi_{\rAinf}^{-1}(I)}\rAinf(\mathcal{O}_{K_\infty})$, and the induced map of graded pieces $\gr^\bullet(\iota_{n,m})$ is injective.
		\item\label{prop:Ainf_twist_difference_cokernel_graded} The $i$-th graded module of the cokernel of $\iota_{n,m}$ is given by
		$$\gr^i \coker(\iota_{n,m}) = 
		\begin{cases}
				\gr^i\bigl(  \mu^m \rAinf(\mathcal{O}_{K_\infty}) \bigr)/(\zeta_p-1)^{n-m} \simeq (\zeta_p-1)^m\mathcal{O}_{K_\infty}\{i\}\cdot (\zeta_p-1)^m/(\zeta_p-1)^n,~i\geq n;\\
				\gr^i\bigl(  \mu^m \rAinf(\mathcal{O}_{K_\infty}) \bigr) \simeq (\zeta_p-1)^m\varphi^{-1}(I)^i/\varphi^{-1}(I)^{i+1} = (\zeta_p-1)^m\mathcal{O}_{K_\infty}\{i\},~m \leq i < n;\\
				0,~i<m.
			\end{cases}$$
	\end{enumerate}
\end{proposition}
Here we note that the above descriptions extend to $p$-torsionfree perfectoid $\mathcal{O}_{K_\infty}$-algebras $S$, and we use $\iota_{n,m,S}$ to denote the induced map.
\begin{proof}
	Part \ref{prop:Ainf_twist_difference_Fil} follows from (\ref{eq:equation_of_mu}) together with the fact that the map $\varphi_{\rAinf}:\rAinf(\mathcal{O}_{K_\infty}) \to \rAinf(\mathcal{O}_{K_\infty})$ is an isomorphism.
	The description and the invertibility of the graded piece in Part \ref{prop:Ainf_twist_difference_graded} are direct consequences of Part \ref{prop:Ainf_twist_difference_Fil}, together with the observation that the image of $\varphi_{\rAinf}^{-1}(\mu)$ in $\rAinf(\mathcal{O}_{K_\infty})/\varphi_{\rAinf}^{-1}(I)$ is $\zeta_p-1$.
	For the injectivity of the map $\gr^\bullet(\iota_{n,m})$, it suffices to notice that by \Cref{const:Ainf_twist_by_mu} we have 
	\[
	\Fil^i  \mu^n \rAinf(\mathcal{O}_{K_\infty}) = \bigl( \mu^n \rAinf(\mathcal{O}_{K_\infty}) \bigr) \bigcap \Fil^i \mu^m \rAinf(\mathcal{O}_{K_\infty}),~\forall i.
	\]
	Part \ref{prop:Ainf_twist_difference_cokernel_graded} follows from Part \ref{prop:Ainf_twist_difference_graded}.
\end{proof}

Another preparation is a description of the canonical map between the absolute and the relative prismatic cohomology, as analyzed by Bhatt--Lurie in \cite{BL22a}.
\begin{theorem}[Bhatt--Lurie]
	\label{thm:abs_and_rel_prismatic}
Let $S$ be a ring, and let $(A,I)$ be a bounded prism.
\begin{enumerate}[label=\upshape{(\roman*)}]
	\item\label{thm:abs_and_rel_prismatic_Fil} There is a canonical commutative diagram of filtered $\mathrm{E}_\infty$-algebras
\begin{equation}
	\label{eq:abs_to_rel_prismatic_Fil}
	\begin{tikzcd}
		\Fil^\bullet_N \Prism_S \ar[d] \arrow[r, "{\gamma_\Prism^\dR}"] & \Fil^\bullet_H \dR_{S/\mathbb{Z}_p} \ar[d] \\
		\Fil^\bullet_N \Prism^{(1)}_{(S\otimes_{\mathbb{Z}_p}{\overline{A}})^\wedge_p/A} \arrow[r] & \bigl(\Fil^\bullet_H \dR_{S/\mathbb{Z}_p}\otimes_{\mathbb{Z}_p} \overline{A}\bigr)^\wedge_p,
	\end{tikzcd}
\end{equation} 
where the horizontal arrows are the absolute and the relative prismatic--de Rham realization morphisms respectively, as in \cite[Const.\ 5.5.3, Cor.\ 5.2.8]{BL22a},
and the right vertical arrow is the tensor product of $\mathbb{Z}_p\to \overline{A}$.
\item\label{thm:abs_and_rel_prismatic_gr} Assume $S$ is quasiregular semiperfectoid.
For the given $i\in \mathbb{N}$, the induced map 
\begin{equation}
	\label{eq:abs_prism_to_dR_graded}
	\gr^i(\gamma_\Prism^\dR)\colon \gr^i_N \Prism_S \longrightarrow \gr^i_H \dR_S
\end{equation}
is an inclusion of $S$-modules with the cokernel killed by a bounded power of $p$.
\end{enumerate}
\end{theorem}
\begin{proof}
	For Part \ref{thm:abs_and_rel_prismatic_Fil}, we recall from \cite[Const.\ 5.5.3, Cor.\ 5.6.3]{BL22a} that the Nygaard filtered absolute prismatic cohomology can be defined using the pullback diagram
	\begin{equation}
		\label{eq:abs_Nyggard_def}
		\begin{tikzcd}
			\Fil^\bullet_N \Prism_S \ar[d] \arrow[r, "{\gamma_\Prism^\dR}"] & \Fil^\bullet_H \dR_{S/\mathbb{Z}_p} \ar[d] \\
		\varprojlim_{(B,J)\in (\mathbb{Z}_p)_\Prism} \Fil^\bullet_N \Prism^{(1)}_{S_{\overline{B}}/B} \arrow[r] & \varprojlim_{(B,J)\in (\mathbb{Z}_p)_\Prism} \bigl(\Fil^\bullet_H \dR_{S/\mathbb{Z}_p}\otimes_{\mathbb{Z}_p} \overline{B}\bigr)^\wedge_p,
		\end{tikzcd}
	\end{equation}
	where the right vertical map is the limit of the tensor product maps, and the bottom arrow is the relative prismatic--de Rham comparison.
	So Part \ref{thm:abs_and_rel_prismatic_Fil} follows by combining (\ref{eq:abs_Nyggard_def}) with the tautological projection diagram as below
	\[
	\begin{tikzcd}
		\varprojlim_{(B,J)\in (\mathbb{Z}_p)_\Prism} \Fil^\bullet_N \Prism^{(1)}_{S_{\overline{B}}/B} \arrow[r] \ar[d] & \varprojlim_{(B,J)\in (\mathbb{Z}_p)_\Prism} \bigl( \Fil^\bullet_H \dR_{S/\mathbb{Z}_p}\otimes_{\mathbb{Z}_p}\overline{B}\bigr)^\wedge_p \ar[d] \\
		\Fil^\bullet_N \Prism^{(1)}_{S_{\overline{A}}/A} \arrow[r] & \bigl(\Fil^\bullet_H \dR_{S/\mathbb{Z}_p}\otimes_{\mathbb{Z}_p} \overline{A}\bigr)^\wedge_p.
	\end{tikzcd}\]

	To see Part \ref{thm:abs_and_rel_prismatic_gr}, we recall from the proof of \cite[Prop.\ 5.5.12]{BL22a} that the fiber of the morphism $\gr^i(\gamma_\Prism^\dR)$ is isomorphic to the complex 
	\[
	C\colonequals \bigl(\Fil^\conj_{i-1} \widehat{\Omega}_S^\slashed{D} \bigr)^{\Theta=-i},
	\]
	where $\widehat{\Omega}_S^\slashed{D}$ is the diffracted Hodge complex as in \cite[\S 4.7]{BL22a}.
	By \cite[Not.\ 4.7.2]{BL22a}, the complex $C$ admits a finite filtration with its non-zero graded pieces given by  
	\[
	C_j\colonequals \wedge^j \mathbb{L}_{S/\mathbb{Z}_p}[-j] \otimes_{\mathbb{Z}_p}(\mathbb{Z}_p \xrightarrow{\cdot (j-i)} \mathbb{Z}_p)
	\] for $j\in \{0,1,\ldots, i-1\}$.
	On the other hand, since $S$ is quasiregular semiperfectoid (\cite[Not.\ 7.1]{BS22}), we know each $\wedge^j \mathbb{L}_{S/\mathbb{Z}_p}[-j]$ is a $p$-torsionfree module in cohomological degree zero.
	Hence the complex $C_j$ lives in cohomological degree $1$ and is killed by the integer $(j-i)$.
	As a consequence, the map (\ref{eq:abs_prism_to_dR_graded}) is injective with the cokernel killed by a power of $p$.
\end{proof}

We then extend the analysis of \Cref{prop:Ainf_twist_difference} and \Cref{thm:abs_and_rel_prismatic} and consider the base change of the morphism $\iota_{n,m}$ to the relative prismatic cohomology of a perfectoid algebra.
\begin{proposition}
	\label{prop:gr_of_base_change_of_iota}
	Let $X=\Spf(R)$ be a regular $p$-adic formal scheme over $\mathcal{O}_K$, let $(B,J)$ be a flat prism, and let $S$ be a perfectoid ring over $X_{\mathcal{O}_{K_\infty}}$.
	For $n\geq m$, let $\widetilde{\iota}_{n,m}$ be the $(p,J)$-complete filtered base change of $\iota_{n,m,S}$ along the morphism 
	$\Fil^\bullet_{\varphi^{-1}(I)} \rAinf(S) = \Fil^\bullet_N \Prism_S \to \Fil^\bullet_N \Prism^{(1)}_{S_{\overline{B}}/B}$.
	There is a natural graded isomorphism of graded abelian groups (in cohomological degree zero)
	\begin{equation}
		\label{eq:graded_of_base_change_of_iota}
		\gr^\bullet \cofib(\widetilde{\iota}_{n+1,n}) [1/p] \simeq M_{S_{\overline{B}}}^\bullet\otimes_{S} S(n),
	\end{equation}
where $M_{S_{\overline{B}}}^\bullet$ is a graded and functorial $S_{\overline{B}}[1/p]$-module of degree $\geq n$, such that each $M_{S_{\overline{B}}}^i$ admits a natural finite filtration whose graded piece is of the form $\bigl( \wedge^i \mathbb{L}_{R/\mathbb{Z}_p}\otimes_R \overline{B}\{j\}\bigr) \otimes_{\overline{B}} S_{\overline{B}}[1/p]$.
\end{proposition}
Before the proof, we remind the reader the convention that $S_{\overline{B}}$ is defined as $(S\otimes_R \overline{B})^\wedge_p$.
\begin{proof}
	By switching the order of taking the associated graded and the filtered base change, we obtain an identification of the graded maps
	\[
	\gr^\bullet (\widetilde{\iota}_{n+1,n}) \simeq  (\gr^\bullet \iota_{n+1,n,S}) \otimes_{\gr^\bullet \rAinf(S)} \gr^\bullet_N \Prism^{(1)}_{S_{\overline{B}}/B} ,
	\]
	and similarly their cofibers.
	Here we notice that by \Cref{prop:Ainf_twist_difference}.\ref{prop:Ainf_twist_difference_graded}, the cofiber (which coincides with the cokernel) of the map $\gr^\bullet (\iota_{n,m,S})$ is a perfect graded complex over $\gr^\bullet \rAinf(S)$.
	So by \Cref{prop:Ainf_twist_difference}.\ref{prop:Ainf_twist_difference_cokernel_graded} and by killing the $p$-torsions, we get
	\[
	\gr^\bullet \coker(\widetilde{\iota}_{n+1,n})[1/p] \simeq \bigl( S\{n\} \otimes_{\gr^\bullet \rAinf(S)} \gr^\bullet_N \Prism^{(1)}_{S_{\overline{B}}/B} \bigr) [1/p],
	\]
	where the action of $\gr^\bullet \rAinf(S)$ on $S\{n\}$ factors through the graded surjection $\gr^\bullet \rAinf(S)\to \gr^0\rAinf(S)\simeq S$ and the module $S\{n\}$ lives in the graded degree $n$.
	In particular, the right hand side above is naturally isomorphic to 
	\[
	\bigl( S\{n\} \otimes_S S \otimes_{\gr^\bullet \rAinf(S)} \gr^\bullet_N \Prism^{(1)}_{S_{\overline{B}}/B} \bigr) [1/p].
	\]
	So to prove the statement, it amounts to decomposing the graded ring $S \otimes_{\gr^\bullet \rAinf(S)} \gr^\bullet_N \Prism^{(1)}_{S_{\overline{B}}/B}[1/p]$, which we denote by $Q$.
	
	Next, we notice that the graded ring $\gr^\bullet_N \Prism^{(1)}_{S_{\overline{B}}/B}[1/p]$ lives in cohomological degree zero: by Nygaard--conjugate comparison (\cite[Thm.\ 15.2]{BS22}), each $\gr^\bullet_i \Prism^{(1)}_{S_{\overline{B}}/B}[1/p]$ admits a finite filtration whose formation is $\wedge^i\mathbb{L}_{S_{\overline{B}}/\overline{B}}\otimes_{\overline{B}}\overline{B}\{-i\}[-i][1/p]$, which, thanks to the Faltings extension and the smoothness of the generic fiber $X_\eta$, is a finite projective $S_{\overline{B}}[1/p]$-module.
	In addition, as $\Fil^\bullet \rAinf(S)$ is the $\varphi^{-1}(I)$-adic filtration for an invertible ideal $\varphi^{-1}(I)$, the kernel of $\gr^\bullet \rAinf(S)\to S$ is a principal ideal generated by $\gr^1 \rAinf(S)$.
	Thus the ring $Q$ is naturally isomorphic to the quotient ring of $\gr^\bullet_N \Prism^{(1)}_{S_{\overline{B}}/B}[1/p]$ by the ideal generated by the image of $\gr^1 \rAinf(S)$.
	
	To understand the aforementioned image, we let $C$ be the $S_{\overline{B}}[1/p]$-submodule generated by the image of $\gr^1 \rAinf(S)[1/p]$ in $\gr^1_N \Prism^{(1)}_{S_{\overline{B}}/B}[1/p]$. 
	Consider the natural commutative diagram as below:  
	\begin{equation}
		\label{eq:prism_dR_gr_1}
		\begin{tikzcd}
					\gr^1 \rAinf(S)=\gr^1_N \Prism_S \ar[d] \ar[r] & \gr^1_H\dR_S= \mathbb{L}_{S/\mathbb{Z}_p}[-1] \ar[d]\\
			\gr^1_N \Prism^{(1)}_{(S\otimes_{\mathbb{Z}_p} \overline{B})^\wedge_p/B} \ar[d] \ar[r] & \gr^1_H \dR_{(S\otimes_{\mathbb{Z}_p} \overline{B})^\wedge_p/B}=\bigl( \mathbb{L}_{S/\mathbb{Z}_p} \otimes_{\mathbb{Z}_p} \overline{B}\bigr)^\wedge_p[-1] \ar[d]\\
			\gr^1_N \Prism^{(1)}_{S_{\overline{B}}/B} \ar[r] & \gr^1_H \dR_{S_{\overline{B}}/B}=\bigl( \mathbb{L}_{S/R} \otimes_R \overline{B}\bigr)^\wedge_p [-1],
		\end{tikzcd}
	\end{equation}
where the left column factorizes the morphism $\gr^1 \rAinf(S) \to \gr^1_N \Prism^{(1)}_{S_{\overline{B}}/B}$, the horizontal arrows are absolute and relative prismatic--de Rham realizations, and the right column follows from the functoriality of de Rham cohomology.
    By \Cref{thm:abs_and_rel_prismatic}.\ref{thm:abs_and_rel_prismatic_gr}, the top horizontal arrow becomes an isomorphism after inverting $p$.
	On the other hand, by the proof of \cite[Cor.\ 15.4]{BS22}, the bottom horizontal arrow in (\ref{eq:prism_dR_gr_1}) fits naturally into a fiber sequences
	\begin{equation}
		\label{eq:gr_N_and_gr_H}
		\begin{tikzcd}
			S_{\overline{B}} \otimes_{\overline{B}} \overline{B}\{1\}= \gr^0_N \Prism^{(1)}_{S_{\overline{B}}/B} \otimes_{\overline{B}} \overline{B}\{1\} \ar[r] & \gr^1_N \Prism^{(1)}_{S_{\overline{B}}/B} \ar[r] & \gr^1_H \dR_{S_{\overline{B}}/\overline{B}} = \mathbb{L}_{S/R} \otimes_S S_{\overline{B}} [-1],
		\end{tikzcd}
	\end{equation}
which becomes a short exact sequence of finite projective $S_{\overline{B}}[1/p]$-modules after inverting $p$.
Hence by combining (\ref{eq:prism_dR_gr_1}), (\ref{eq:gr_N_and_gr_H}) and the aforementioned isomorphism, we see the image of $C$ in $\gr^1_H \dR_{S_{\overline{B}}/\overline{B}}[1/p]$ coincides with the image of $\mathbb{L}_{S/\mathbb{Z}_p}\otimes_S S_{\overline{B}}[1/p][-1]$ in $\mathbb{L}_{S/R} \otimes_S S_{\overline{B}} [1/p] [-1]$.
As a consequence, the quotient $\gr^1_N \Prism^{(1)}_{S_{\overline{B}}/B}[1/p]/C$ functorially fits into a canonical short exact sequence
\[
S_{\overline{B}}\otimes_{\overline{B}} \overline{B}\{1\}[1/p]  = \gr^0_N \Prism^{(1)}_{S_{\overline{B}}/B} \otimes_{\overline{B}} \overline{B}\{1\} [1/p]\longrightarrow \gr^1_N \Prism^{(1)}_{S_{\overline{B}}/B}[1/p]/C \longrightarrow \mathbb{L}_{R/\mathbb{Z}_p}\otimes_R S_{\overline{B}}[1/p].
\]
This finishes the proof for the lowest graded piece.

To extend the above to the higher graded pieces of the ring $Q$, we recall that the Nygaard graded pieces admit a canonical fiber sequence that is compatible with the prismatic--de Rham comparison (\cite[Cor.\ 5.2.8]{BL22a})
\[
\gr^{n-1}_N \Prism^{(1)}_{S_{\overline{B}}/B}\otimes_{\overline{B}} \overline{B}\{1\} \longrightarrow \gr^n_N \Prism^{(1)}_{S_{\overline{B}}/B}\otimes_{\overline{B}} \longrightarrow \gr^n_H \dR_{S_{\overline{B}}/\overline{B}}.
\]
On the other hand, by the multiplicativity, the graded pieces of the quotient ring $\gr^\bullet_H \dR_{S_{\overline{B}}/\overline{B}}[1/p]/(\mathbb{L}_{S/\mathbb{Z}_p}[-1])$ is of the form $\wedge^i \mathbb{L}_{R/\mathbb{Z}_p}\otimes_R S_{\overline{B}}[1/p]$.
Hence the structure of the general graded piece of the ring $Q$ follows by the above short exact sequence and induction.	
\end{proof}

Now we are ready to define the rational prismatic period sheaf.
\begin{construction}[Rational prismatic period sheaves]
	\label{const:rational_period_sheaves}
	Let $X$ be a regular $p$-adic formal scheme over $\mathcal{O}_K$, and let $(B,J)$ be a prism over $X_\Prism$.
	\begin{enumerate}
	\item Consider the following tensor product over $(X_{\eta, K_{\infty}})_\pe$
	\[
	\Prism_{(-)_{\overline{B}}/B}[\frac{1}{\mu}] \colonequals \Prism_{(-)_{\overline{B}}/B} \otimes_{\rAinf(\mathcal{O}_{K_\infty})} \rAinf(\mathcal{O}_{K_\infty})[\frac{1}{\mu}],
	\]
	and similarly for the twisted version $\Prism^{(1)}_{(-)_{\overline{B}}/B}[\frac{1}{\mu}] = (\varphi_B^*\Prism_{(-)_{\overline{B}}/B})[\frac{1}{\mu}]$.
	As the invertible ideal $\mu\rAinf(\mathcal{O}_{K_\infty}) \subset \rAinf(\mathcal{O}_{K_\infty})$ is invariant under the action of $\Gal(K_\infty/K)$, the sheaf $\Prism_{(-)_{\overline{B}}/B}[\frac{1}{\mu}]$ then naturally descends to a sheaf onto $X_{\eta,\pe}$, which we call the \emph{rational prismatic period sheaf} over $(B,J)$.
	Note that by \Cref{const:Ainf_twist_by_mu}, the rational period sheaf $\Prism_{(-)_{\overline{B}}/B}[\frac{1}{\mu}]$ admits a canonical presentation
	\begin{equation}
		\label{eq:rational_period_sheaf_alt}
			\Prism_{(-)_{\overline{B}}/B}[\frac{1}{\mu}] \simeq \colim_{n\geq 0} \bigl( \Prism_{(-)_{\overline{B}}/B}\{n\}(-n) \bigr),
	\end{equation}
which in particular does not require the choice of $\mu$.

	\item For each $n\in \mathbb{Z}$, the twisted period sheaf $\mu^n\Prism_{(-)_{\overline{B}}/B}=\Prism_{(-)_{\overline{B}}/B}\otimes_{\mathbb{A}_{\inf}} \mathbb{A}_{\inf}\{-n\}(n)$ naturally admits a weak Frobenius structure (after inverting the ideal $J$)
	\[
	\mu^n\Prism_{(-)_{\overline{B}}/B} \longrightarrow \varphi_B^*( \mu^n\Prism_{(-)_{\overline{B}}/B}) \longrightarrow \mu^nJ^n\Prism_{(-)_{\overline{B}}/B}.
	\]
	In addition, the Frobenius-twisted object $\varphi_B^*( \mu^n\Prism_{(-)_{\overline{B}}/B})$ admits a canonical filtration $\Fil^\bullet \varphi_B^*( \mu^n\Prism_{(-)_{\overline{B}}/B})=\Fil^\bullet\Prism^{(1)}_{(-)_{\overline{B}}/B}\{-n\}(n)$ that is determined (up to $(p,J)$-complete sheafification) via the preimage of the ideals $J^\bullet \mu^n\Prism_{(-)_{\overline{B}}/B}$ along the Frobenius morphism.
	Both the Frobenius structures and the filtrations are compatible with respect to different $n$.	
	By taking the union of $\Fil^i \Prism^{(1)}_{(-)_{\overline{B}}/B}\{-n\}(n)$ with respect to $n$, we then obtain a filtration on $\Prism^{(1)}_{(-)_{\overline{B}}/B}[\frac{1}{\mu}]$.
	\end{enumerate}
\end{construction}
The following intersection formula will be used later.
\begin{lemma}
	\label{lem:rational_period_sheaves:filtration}
	Let $X$ be a regular $p$-adic formal scheme over $\mathcal{O}_K$ and let $(B,J)$ be a flat prism over $X_\Prism$.
	For each $n\in \mathbb{Z}$, we have
	\[
	\Fil^i \bigl( \Prism^{(1)}_{(-)_{\overline{B}}/B}[\frac{1}{\mu}] \bigr) \bigcap \mu^n\Prism^{(1)}_{(-)_{\overline{B}}/B} = \Fil^i \bigl( \mu^n\Prism^{(1)}_{(-)_{\overline{B}}/B}\bigr).
	\]
\end{lemma}
\begin{proof}
	It suffices to check the equality for the sections at a large enough perfectoid pro-\'etale objects $\Spf(S)_\eta$ over $X_\eta$, namely
	\[
	\Fil^i \bigl( \Prism^{(1)}_{S_{\overline{B}}/B}[\frac{1}{\mu}] \bigr) \bigcap \mu^n\Prism^{(1)}_{S_{\overline{B}}/B} = \Fil^i \bigl( \mu^n\Prism^{(1)}_{S_{\overline{B}}/B}\bigr),
	\]
	which follows from construction.
\end{proof}
	
Assembling various analyses, we obtain the following structural result on the rational prismatic period sheaf, which plays an essential role in the calculation of its global section.
	\begin{theorem}
		\label{thm:structure_of_rational_period_sheaf}
		Let $X$ be a regular $p$-adic formal scheme over $\mathcal{O}_K$, let $(B,J)\in X_\Prism$ be a flat prism, and let $n\in \mathbb{Z}$.
		\begin{enumerate}[label=\upshape{(\roman*)}]
			\item\label{thm:structure_of_rational_period_sheaf:graded}
			There are natural isomorphisms
			\begin{align*}
			\gr^\bullet \bigl( \mu^n\Prism^{(1)}_{(-)_{\overline{B}}/B} / \mu^{n+1} \Prism^{(1)}_{(-)_{\overline{B}}/B} [1/p] \bigr) & \simeq \gr^\bullet (\mu^n\Prism^{(1)}_{(-)_{\overline{B}}/B}\otimes_{\Ainf} \Ainf/\varphi^{-1} (I_{\Ainf}) [1/p]) \\ 
			& \simeq \mathcal{F}^\bullet \otimes_{\widehat{\mathcal{O}}} \widehat{\mathcal{O}}(n),
			\end{align*}
			where $\mathcal{F}^\bullet$ is a graded sheaf of $\widehat{\mathcal{O}}_{\overline{B}}$-vector bundles in degrees $\geq n$ on $X_{\eta,\pe}$, such that each $\mathcal{F}^i$ admits a finite filtration whose graded piece is an analytic vector bundle of the form $\bigl( \wedge^i \mathbb{L}_{X/\mathbb{Z}_p}\otimes_{\mathcal{O}_X} \overline{B}\{j\}\bigr) \otimes_{\overline{B}} \widehat{\mathcal{O}}_{\overline{B}}$.
			
			\item\label{thm:structure_of_rational_period_sheaf:sep}
			The filtered completion map below is an injection 
			\[
			\mu^n\Prism^{(1)}_{(-)_{\overline{B}}/B}\otimes_{\Ainf} \Ainf/\varphi^{-1} (I_{\Ainf}) [1/p] \longrightarrow \bigl( \mu^n\Prism^{(1)}_{(-)_{\overline{B}}/B}\otimes_{\Ainf} \Ainf/\varphi^{-1} (I_{\Ainf}) [1/p] \bigr)^\wedge_{\Fil}.
			\]
			
			\item\label{thm:structure_of_rational_period_sheaf:inj} Assume $(B,J)$ is a finite coproduct of framed regular prisms.
			The cokernel of the map $\widetilde{\iota}_{n+1,n}\colon \mu^{n+1}\Prism^{(1)}_{(-)_{\overline{B}}/B} \to \mu^n\Prism^{(1)}_{(-)_{\overline{B}}/B}$ naturally admits injections
			\begin{align*}
							\coker(\widetilde{\iota}_{n+1,n}) \hookrightarrow 
				\coker(\widetilde{\iota}_{n+1,n}) [1/p] \hookrightarrow &
				\prod_{a\geq 1} \prod_{1\leq b \leq a}\bigl( \mu^n\Prism^{(1)}_{(-)_{\overline{B}}/B}\otimes_{\Ainf} \Ainf/\varphi^{-b} (I_{\Ainf}) [1/p]\bigr) \\
				\hookrightarrow & \prod_{\text{countable}} \bigl( \mu^n\Prism^{(1)}_{(-)_{\overline{B}}/B}\otimes_{\Ainf} \Ainf/\varphi^{-1} (I_{\Ainf}) [1/p]\bigr)
			\end{align*}
		where the last arrow (on the individual factor) is induced by the $(b-1)$-th power of the absolute Frobenius morphism on $\Prism^{(1)}_{(-)_{\overline{B}}/B}$.
		\end{enumerate}
	\end{theorem}
\begin{proof}
	For \ref{thm:structure_of_rational_period_sheaf:graded}, as $\mu\Ainf$ is divisible by $\varphi^{-1}_{\Ainf} (I_{\Ainf})$, there is a map of graded modules
	\[
	\gr^\bullet(\Ainf/\mu\Ainf)[1/p] \longrightarrow \gr^\bullet(\Ainf/\varphi^{-1}_{\Ainf} (I_{\Ainf})[1/p]),
	\]
	which by \Cref{prop:Ainf_twist_difference}.\ref{prop:Ainf_twist_difference_cokernel_graded} is a graded isomorphism.
	In particular, by taking the base change, we get an induced isomorphism of graded objects
	\[
	\gr^\bullet \bigl( \mu^n\Prism^{(1)}_{(-)_{\overline{B}}/B}\otimes_{\Ainf} \Ainf/\mu\Ainf [1/p] \bigr)  \xrightarrow{\sim} \gr^\bullet \bigl( \mu^n\Prism^{(1)}_{(-)_{\overline{B}}/B}\otimes_{\Ainf} \Ainf/\varphi^{-1} (I_{\Ainf}) [1/p] \bigr)  = \gr^\bullet \coker(\widetilde{\iota}_{n+1,n})[1/p].
	\]
	Thus \ref{thm:structure_of_rational_period_sheaf:graded} follows from the canonical structure in \Cref{prop:gr_of_base_change_of_iota}.
	
	To show the injectivity in \ref{thm:structure_of_rational_period_sheaf:sep} and \ref{thm:structure_of_rational_period_sheaf:inj}, since the maps are all canonical, it suffices to do so after evaluating at objects in $X_{\eta,\pe}$ that are sufficiently large.
	So we may assume $X=\Spf(R)$ is affine and $S$ is a $p$-torsionfree integrally closed perfectoid algebra over $R_{\mathcal{O}_{K_\infty}}$ such that $\Spf(S)_\eta$ is a pro-\'etale object $\lim_i \Spa(R_i[1/p],R_i)$ for topologically finite type affinoid adic spaces $\Spa(R_i[1/p],R_i)$ in $X_{\eta,\pe}$.
	By enlarging $S$ if necessary, we also assume that $S$ is a perfectoid algebra over $\overline{A}_\perf$ for the perfection of some framed regular prism $A$.
	We let $d_S$ be a generator of the ideal $I_{\Prism_S}$.
	Then we make a claim on the intersection formula fo the ideal $\mu\Prism_S$, extending that of \cite[Lem.\ 3.23]{BMS1}.
	\begin{claim}
		\label{claim:muPrism_S}
		Let $S$ be the perfectoid $\mathcal{O}_{K_\infty}$-algebra be as above.
		There is an equality of ideals
		\[
		\mu\Prism_S = \bigcap_{a>0} \frac{\mu}{\varphi_{\Prism_S}^{-a}(\mu)} \Prism_S.
		\]
	\end{claim}
\begin{proof}[Proof of \Cref{claim:muPrism_S}]
	We first notice that by the assumption of $\Spf(S)_\eta$ and by \cite[Lem.\ 2.13]{GK22} together with the tilting correspondence, we know $S^\flat=(S^\flat[1/t])^\circ$, where $S^\flat$ is the tilt of $S$ and $t$ is a pseudo-uniformizer of $S^\flat$.
	In addition, by the arguments of \cite[Lem.\ 2.11]{GK22}, we know $S^\flat/tS^\flat$ has no almost zero elements with respect to $\mathfrak{m}^\flat$, where $\mathfrak{m}^\flat$ is the completion of the radical ideal of $tS^\flat$.
	Now to prove the equality, it suffices to check the inclusion for their the mod $p$ reductions is an equality:
	\[
	(\epsilon-1)S^\flat \subseteq \bigcap_{a>0} \frac{\epsilon-1}{\epsilon^{\frac{1}{p^a}}-1}S^\flat.
	\]
	The latter follows since the quotient $\bigl(\bigcap_{a>0} \frac{\epsilon-1}{\epsilon^{\frac{1}{p^a}}-1}S^\flat \bigr)/(\epsilon-1)S^\flat$ is an almost zero submodule in $S^\flat / (\epsilon-1)S^\flat$, which by the aforementioned inputs is trivial.
\end{proof}
	
	To show \ref{thm:structure_of_rational_period_sheaf:sep}, it suffices to show that each map in the factorization below is injective
	\[
	\mu^n\Prism^{(1)}_{S_{\overline{B}}/B} \otimes_{\Prism_S} \Prism_S/\varphi_{\Prism_S}^{-1}(I_{\Prism_S}) \longrightarrow \bigl( \mu^n\Prism^{(1)}_{S_{\overline{B}}/B} \otimes_{\Prism_S} \Prism_S/\varphi_{\Prism_S}^{-1}(I_{\Prism_S})\bigr)^\wedge_{\Fil} \longrightarrow \bigl( \mu^n\Prism^{(1)}_{S_{\overline{B}}/B} \otimes_{\Prism_S} \Prism_S/\varphi_{\Prism_S}^{-1}(I_{\Prism_S}) [1/p] \bigr)^\wedge_{\Fil}.
	\]
	The injectivity of the first map follows from the separatedness of the Nygaard filtration, as shown in the proof of \Cref{thm:global_section}.
	For the second map, since both the source and the target are filtered complete, it suffice to check the injectivity on individual graded pieces, which amounts to showing that the graded pieces of $\mu^n\Prism^{(1)}_{S_{\overline{B}}/B} \otimes_{\Prism_S} \Prism_S/\varphi_{\Prism_S}^{-1}(I_{\Prism_S})$ are $p$-torsionfree.
	We notice that the quotient ring $\Prism_S/\varphi_{\Prism_S}^{-1}(I_{\Prism_S})$ is naturally isomorphic to the zero-th graded piece of $\Fil^\bullet_N\Prism_S$.
	So we get
	\[
	\gr^\bullet \bigl( \mu^n\Prism^{(1)}_{S_{\overline{B}}/B} \otimes_{\Prism_S} \Prism_S/\varphi_{\Prism_S}^{-1}(I_{\Prism_S}) \bigr) \simeq \gr^0 \Prism_S \otimes_{\gr^\bullet_N \Prism_S} \gr^\bullet (\mu^n\Prism_S) \otimes_{\gr^\bullet_N \Prism_S} \gr^\bullet (\Prism^{(1)}_{S_{\overline{B}}/B}).
	\]
	By \Cref{prop:Ainf_twist_difference}.\ref{prop:Ainf_twist_difference_graded}, we know $\gr^\bullet (\mu^n\Prism_S)$ is naturally isomorphic to $(\gr^\bullet_N \Prism_S)\otimes_S S\{n\}$.
	Thus we reduce to checking that the graded module $S \otimes_{\gr^\bullet_N \Prism_S} \gr^\bullet (\Prism^{(1)}_{S_{\overline{B}}/B})\simeq \gr^\bullet \bigl( \Prism^{(1)}_{S_{\overline{B}/B}}\otimes_{\Prism_S} \Prism_S/\varphi^{-1}_{\Prism_S}(I_{\Prism_S}) \bigr)$ is $p$-torsionfree, as verified in the next claim.
	\begin{claim}
		\label{claim:torsionfree_of_mod_varphiinverse_I}
		The graded piece of the filtered base change $\Fil^\bullet_N \Prism^{(1)}_{S_{\overline{B}}/B} \otimes_{\Fil^\bullet_N \Prism_S} \gr^0_N \Prism_S$ is $p$-torsionfree.
	\end{claim}
	\begin{proof}[Proof of \Cref{claim:torsionfree_of_mod_varphiinverse_I}]
		Since the property of being $p$-torsionfree can be checked before/after a $p$-completely faithfully flat cover, by the base change formula of the Nyggard-filtered relative prismatic cohomology (\cite[Thm.\ 1.8.(5)]{BS22}), we may assume $(B,J)=(A,I)$ is a framed regular prism (as in the proof of \Cref{thm:coproduct_in_general} and Equation (\ref{eq:diagram_of_self_tensor_product_of_prisms})).
		In addition, since the Frobenius map $\varphi_A$ is faithfully flat (\Cref{prop:regular prism Frobenius}), we may further replace $(B,J)$ by the perfection $(A_\perf, IA_\perf)$, where we let $T=\overline{A}_\perf$ be the corresponding perfectoid ring.
		Then by applying \Cref{cor:coproduct:perfect_with_framed} twice, we obtain natural isomorphisms of prisms in $X_\Prism$:
		\[
		\Prism_{S_{\overline{B}}/B} \simeq (\Prism_S, I_{\Prism_S})\coprod (B,J) = (\Prism_S, I_{\Prism_S}) \coprod (\Prism_T, I_{\Prism_T}) \simeq \Prism_{T_S/\Prism_S}.\]
		
		To continue, we notice that since the Frobenius morphisms $\varphi_{\Prism_S}$ and $\varphi_{\Prism_T}$ are both isomorphic, by the definition of the relative Nyggard filtration, we obtain the following filtered base change isomorphisms
		\begin{equation}
			\label{diagram:two_relative_Nyggard_switch}
			\begin{tikzcd}
				B=\Prism_T \ar[d] &  \Prism_T \arrow[l, "\varphi_{\Prism_T}"', "\sim"] \ar[d] & \Prism_S \ar[d] \arrow[r, "\varphi_{\Prism_S}", "\sim"'] & \Prism_S \ar[d] \\
				\Fil^\bullet_N \Prism^{(1)}_{S_{\overline{B}}/B} & \Fil^\bullet \Prism_{S_{\overline{B}}/B}  \arrow[l, "\sim"] \arrow[r,equal] &  \Fil^\bullet \Prism_{T_S/\Prism_S}	\arrow[r, "\sim"'] & \Fil^\bullet_N \Prism^{(1)}_{{T_S}/\Prism_S},
			\end{tikzcd}
		\end{equation}
	    where the filtration on $\Prism_{S_{\overline{B}}/B} \simeq \Prism_{T_S/\Prism_S}$ is defined by the preimage of the $I$-adic filtration along its absolute Frobenius morphism.
	    Moreover, notice that the Frobenius morphism $\varphi_{\Prism_S}$ sends the absolute Nyggard filtration $\Fil^\bullet_N \Prism_S=\varphi^{-1}_{\Prism_S}(I_{\Prism_S})^\bullet$ onto the $I_{\Prism_S}$-adic filtration on $\Prism_S$.
	    Hence by taking the filtered reduction and by (\ref{diagram:two_relative_Nyggard_switch}), we obtain a filtered isomorphism
	    \[
	    \Fil^\bullet_N \Prism^{(1)}_{S_{\overline{B}}/B} \otimes_{\Fil^\bullet_N \Prism_S} \gr^0_N \Prism_S \simeq \Fil^\bullet_N \Prism^{(1)}_{T_S/\Prism_S} \otimes_{I_{\Prism_S}^\bullet} \Prism_S/I_{\Prism_S},
	    \]
	    which by the relative prismatic--de Rham comparison is further isomorphic to the Hodge-filtered de Rham cohomology $\Fil^\bullet_H \dR_{{T_S}/S} \simeq \bigl( \Fil^\bullet_H \dR_{T/\mathcal{O}_X}\otimes_{\mathcal{O}_X} S \bigr)^\wedge_p$.
	    As a consequence, by the locally freeness of the shifted cotangent complex $\mathbb{L}_{T/\mathcal{O}_X}[-1]=\mathbb{L}_{\overline{A}_{\perf}/\mathcal{O}_X}[-1]$ as in \Cref{prop:regular prism Frobenius} and the flatness of $S$ over $\mathcal{O}_X$, we see the graded pieces of the Hodge filtration is $p$-torsionfree.
	\end{proof}
	
	To show \ref{thm:structure_of_rational_period_sheaf:inj}, we notice that
	by the assumption of the complete faithful flatness of $\Prism^{(1)}_{S_{\overline{B}}/B}$ over $\Prism_S$ (cf. \Cref{prop:intersection_of_Frobenius_of_base_with_coproduct}), we know the sequence $(\mu, p)$ is a regular sequence in $\Prism^{(1)}_{S_{\overline{B}}/B}$.
	Hence $\coker(\widetilde{\iota}_{n+1,n})(\Spf(S)_\eta)$ is contained in its $p$-inverted localization $\coker(\widetilde{\iota}_{n+1,n})(\Spf(S)_\eta)[1/p]$.
	In addition, by \Cref{claim:muPrism_S} and the complete projectivity of the prism $(B,J)$ (cf. \Cref{thm:coproduct of prism faithful}.\ref{thm:coproduct of prism regular and regular}), we know the base change of $\Prism^{(1)}_{S_{\overline{B}}/B}$ along the reduction $\Prism_S \to \Prism_S/(p,\mu)\Prism_S$ is a projective module over $S^\flat/(\epsilon-1)S^\flat$ and in particular has no non-trivial almost zero elements.
	Thus as in the proof of \Cref{claim:muPrism_S} we get 
	\[
	\mu^{n+1}\Prism^{(1)}_{S_{\overline{B}}/B} = \bigcap_{a>0} \mu^n\varphi_{\Prism_S}^{-1}(I_{\Prism_S})\cdots \varphi_{\Prism_S}^{-a}(I_{\Prism_S}) \Prism^{(1)}_{S_{\overline{B}}/B}.
	\]
	Notice that the ideals $\varphi^{-j}(I_{\Prism_S})$ are disjoint after inverting $p$.
	Thus we get
	\[
	\coker(\widetilde{\iota}_{n+1,n})(\Spf(S)_\eta)[1/p] \subset \prod_{a>0} \prod_{1\leq b\leq a}\bigl( \mu^n\Prism^{(1)}_{S_{\overline{B}}/B} \otimes_{\Prism_S} \Prism_S/\varphi_{\Prism_S}^{-b}(I_{\Prism_S}) [1/p]\bigr).
	\]
	The final injection in \ref{thm:structure_of_rational_period_sheaf:inj} is given by the injectivity of (finite compositions of) the absolute Frobenius morphisms $\mu^n\Prism^{(1)}_{S_{\overline{B}}/B} \otimes_{\Prism_S} \Prism_S/\varphi_{\Prism_S}^{-b}(I_{\Prism_S}) [1/p]\to \mu^n\Prism^{(1)}_{S_{\overline{B}}/B} \otimes_{\Prism_S} \Prism_S/\varphi_{\Prism_S}^{-1}(I_{\Prism_S}) [1/p]$, which was proved in \Cref{prop:intersection_of_Frobenius_of_base_with_coproduct}.\ref{prop:intersection_of_Frobenius_of_base_with_coproduct:cohomology}.
\end{proof}

	Now we are ready to calculate the global section of the rational prismatic period sheaf.
\begin{theorem}
	\label{thm:global_section_rational}
	Let $X$ be a regular $p$-adic formal scheme over $\mathcal{O}_K$, and let $(B,J)\in X_\Prism$ be a flat prism.
	Assume either $(B,J)$ is a finite coproduct of framed regular prisms or $\varphi_B$ is completely faithfully flat.
	For each $n\in \mathbb{N}$, the structure map induces an isomorphism
	\[
	B \xrightarrow{\sim} \mathrm{H}^0_\pe(X_\eta, \mu^{-n}\Prism_{(-)_{\overline{B}}/B}).
	\]
\end{theorem}
\begin{proof}
	In the following, we use $\Prism_{(-)_{\overline{B}}/B}\{-m\}(m)$ to denote the double twisted sheaf $\Prism_{(-)_{\overline{B}}/B}\otimes_{\mathbb{A}_{\inf}} \mathbb{A}_{\inf}\{-m\}(m)\simeq \mu^m\Prism_{(-)_{\overline{B}}/B}$.
	We first notice that by applying the complete base change along $\varphi_B$, the canonical map $\Prism_{(-)_{\overline{B}}/B} \to \Prism_{(-)_{\overline{B}}/B}\{n\}(-n)$ naturally induces a map $\Prism^{(1)}_{(-)_{\overline{B}}/B}\to \Prism^{(1)}_{(-)_{\overline{B}}/B}\{n\}(-n)$, which by construction is equivariant under the relative Frobenii and filtrations.
	In addition, as in the proof of \Cref{thm:structure_of_rational_period_sheaf}, we choose a large enough perfectoid algebra $S$ that is $p$-completely flat over $X$ and is over 
	$\overline{A}_\perf$ for a framed regular prism $(A,I)$, such that the generic fiber covers $X_{\eta,\pe}$ by \Cref{lem:perfection_of_framed_regular_prism}).
	Under the assumption, we notice that $(\mu,p)$ is a regular sequence in $\Prism^{(1)}_{S_{\overline{B}}/B}$, thanks to the complete flatness of $\Prism^{(1)}_{S_{\overline{B}}/B}$ over $\Prism_S$.
	To see the latter, the case when $(B,J)$ is a finite coproduct of framed regular prisms follows from \Cref{prop:intersection_of_Frobenius_of_base_with_coproduct}.
	In general, the flatness of $(B,J)$ and \Cref{thm:coproduct_in_general}.\ref{thm:coproduct of prism faithful} implies that $\Prism_{S_{\overline{B}}/B}$ is completely flat over $\Prism_S$, and thus the other case when $\varphi_B$ is completely flat follows by taking the base change.
	Thus the induced map $\Prism^{(1)}_{S_{\overline{B}}/B}\to \Prism^{(1)}_{S_{\overline{B}}/B}\{n\}(-n)$ is injective.
	Consider the induced diagram for the filtered completion of the $p$-localized objects:
	\begin{equation}
		\label{eq:thm:global_section_rational_complete}
		\begin{tikzcd}
			\Prism^{(1)}_{(-)_{\overline{B}}/B} \arrow[r,hook] \ar[d] & \Prism^{(1)}_{(-)_{\overline{B}}/B}\{n\}(-n)  \ar[d]\\
			\bigl( \Prism^{(1)}_{(-)_{\overline{B}}/B} [1/p] \bigr)^\wedge_{\Fil} \ar[r] & \bigl( \Prism^{(1)}_{(-)_{\overline{B}}/B}\{n\}(-n) [1/p]\bigr)^\wedge_{\Fil}.
		\end{tikzcd}
	\end{equation}
Here we note that the other three arrows are injective as well: for the left vertical arrow, it follows from the separatedness of the Nygaard filtration (cf. proof of \Cref{thm:global_section}); for the right vertical arrow, it is the consequence of \Cref{thm:structure_of_rational_period_sheaf}.\ref{thm:structure_of_rational_period_sheaf:sep}; and for the bottom horizontal arrow, it follows from \Cref{prop:gr_of_base_change_of_iota}.

We then notice that the bottom arrow in (\ref{eq:thm:global_section_rational_complete}) induces an isomorphism on the global sections.
Indeed, by induction and the filtered completeness, it suffices to show that the graded pieces of the cokernel of the map $\iota\colonequals \bigl( \Prism^{(1)}_{(-)_{\overline{B}}/B} \{i\}(-i)[1/p] \bigr)^\wedge_{\Fil} \to  \bigl( \Prism^{(1)}_{(-)_{\overline{B}}/B}\{i+1\}(-i-1) [1/p]\bigr)^\wedge_{\Fil}$ have vanishing global sections for $i>0$.
In addition, by \Cref{thm:structure_of_rational_period_sheaf}.\ref{thm:structure_of_rational_period_sheaf:graded}, we know $\gr^\bullet(\coker(\iota))$ is of the form $\mathcal{F}^\bullet \otimes_{\widehat{\mathcal{O}}} \widehat{\mathcal{O}}(-i)$ where $\mathcal{F}^\bullet$ is a base change of locally free $\mathcal{O}_{X_{\overline{B}},\eta}$-modules.
So the vanishing of $\mathrm{H}^0_\pe(X_\eta,  \gr^\bullet(\coker(\iota)))$ for $i>0$ follows from the vanishing of the pro-\'etale cohomology of $\widehat{\mathcal{O}}(<0)$ in \Cref{prop:LZ_finiteness}, together with the projection formula in \Cref{thm:projection_formula}.

Now, by \Cref{thm:global_section}.\ref{thm:global_section_twisted_prismatic}, the twisted structure map induces an isomorphism
\[
(B[1/p])^\wedge_J \xrightarrow{\sim} \mathrm{H}^0_\pe(X_\eta, \bigl( \Prism^{(1)}_{(-)_{\overline{B}}/B} [1/p] \bigr)^\wedge_{\Fil}),
\]
which is further isomorphic to $\mathrm{H}^0_\pe(X_\eta, \bigl( \Prism^{(1)}_{(-)_{\overline{B}}/B}\{n\}(-n) [1/p]\bigr)^\wedge_{\Fil})$ by (\ref{eq:thm:global_section_rational_complete}).
As a consequence, by the injectivity of the right vertical arrow in (\ref{eq:thm:global_section_rational_complete}), we see 
\begin{equation}
	\label{eq:intersection_of_global_sec_and_localized_base_ring}
	\mathrm{H}^0_\pe(X_\eta, \Prism^{(1)}_{(-)_{\overline{B}}/B}\{n\}(-n))\subseteq (B[1/p])^\wedge_J \cap \mathrm{H}^0_\pe(X_\eta, \Prism^{(1)}_{(-)_{\overline{B}}/B}\{n\}(-n)),
\end{equation}
where the right hand side by \Cref{lem:intersection_of_prismatic_coh_with_base} is equal to $B$.
On the other hand, we know the left hand side in (\ref{eq:intersection_of_global_sec_and_localized_base_ring}) naturally contains $\mathrm{H}^0_\pe(X_\eta, \Prism^{(1)}_{(-)_{\overline{B}}/B})$, which under the twisted structure map is identified with $B$ itself (\Cref{thm:global_section}.\ref{thm:global_section_twisted_prismatic}).
Hence the twisted structure map $B\to \mathrm{H}^0_\pe(X_\eta, \Prism^{(1)}_{(-)_{\overline{B}}/B}\{n\}(-n))$ is an isomorphism.
Finally, by the same argument for \Cref{thm:global_section}.\ref{thm:global_section_prismatic}, we see the structure map $B\to \mathrm{H}^0_\pe(X_\eta, \Prism_{(-)_{\overline{B}}/B}\{n\}(-n))$ is an isomorphism as well, which finishes the proof.
\end{proof}
\begin{corollary}
	\label{cor:global_section_rational}
	Let $X$ be a regular $p$-adic formal scheme over $\mathcal{O}_K$, and let $(B,J)\in X_\Prism$ be a flat prism such that either $(B,J)$ is a finite coproduct of framed regular prisms or $\varphi_B$ is completely faithfully flat.
	The structure map below is an isomorphism
	\[
	B \xrightarrow{\sim} \mathrm{H}^0_\pe(X_\eta, \Prism_{(-)_{\overline{B}}/B}[\frac{1}{\mu}]).\]
\end{corollary}

\section{Canonical weak prismatic $F$-crystal}
\label{sec:F-crys}
Let $X$ be a regular $p$-adic formal scheme over $\mathcal{O}_K$, and let $X_\eta$ be its generic fiber.
In this section, we show that there is a canonical weak prismatic $F$-crystal over the flat prismatic site of $X$, for any $\mathbb{Z}_p$-local system over the generic fiber $X_\eta$.

Below we consider a subcategory of flat prisms that will serve as the parametrizing space for our construction.
\begin{definition}
	\label{def:flat_prismatic_site}
	Let $X$ be a regular $p$-adic formal scheme.
	The \emph{special prismatic site of $X$}, denoted as $X_\Prismsp$, is the full subcategory of $X_\Prism$ consisting of transversal prisms $(B,J)$ that are finite coproduct of framed regular prisms.
	We call such $(B,J)$ \emph{special}.
\end{definition}
It is easy to see that a special prism is in particular a flat prism.

\begin{definition}
	\label{def:family_of_wFrob_mod}
	Let $X$ be a regular $p$-adic formal scheme.
	A \emph{weak prismatic ($F$-)crystal} is a presheaf $\mathcal{E}$ of (resp. weak Frobenius) modules on $X_\Prism$ such that 
	\begin{itemize}
		\item For each $(B,J)\in X_\Prismsp$, the $B$-module $\mathcal{E}(B,J)$ is $(J,p)$-regular.
		\item For each noetherian prism $(B,J)\in X_\Prismsp$, the $B$-module $\mathcal{E}(B,J)$ is finitely presented.
		\item for each map of framed regular prisms $(B_1,J_1)\to (B_2,J_2)$ such that $\overline{B}_1/p\overline{B}_1\to \overline{B}_/p\overline{B}_2$ is flat, the induced linearized map of torsionfree modules 
		\[
		\bigl( B_2\otimes^L_{B_1} \mathcal{E}(B_1,J_1)\bigr)^\wedge_{(p,J_2)} \longrightarrow \mathcal{E}(B_2,J_2)
		\]
		is an injection.
	\end{itemize}
	The category of weak prismatic ($F$-)crystals is denoted as $\wCrys^{(\varphi)}(X_\Prismsp)$.
\end{definition}

We also remind the reader the following terminology for prismatic $F$-crystals. 
\begin{definition}
Let $X$ be a regular $p$-adic formal scheme.
 A prismatic ($F$-)crystal in complexes $\mathcal{E}$ on $X_\Prism$ is called \emph{reflexive} if for each framed regular prism $(A,I)$, the $A$-complex
 $\mathcal{E}(A,I)$ is represented by a reflexive $A$-module in the sense of \Cref{def:ref_Frob_mod}.
\end{definition}
We denote the category of reflexive prismatic ($F$-)crystals over $X$ as $\Coh^{(\varphi)}_{\refl}(X_\Prism)$.
\begin{remark}
When $X$ is smooth, the category of reflexive prismatic ($F$-)crystals is equivalent to the category of analytic prismatic ($F$-)crystals, as shown in \cite[Thm.\ 5.10]{GR24}.
\end{remark}

The following observation describes the relationship between prismatic $F$-crystals and weak prismatic $F$-crystals, and gives a formula of calculating the sections of a prismatic $F$-crystal through the subcategory $X_\Prismsp$.
\begin{lemma}
	\label{lem:weak_vs_non-weak}
	Let $X$ be a regular $p$-adic formal scheme.
	\begin{enumerate}[label=\upshape{(\roman*)}]
		\item\label{lem:weak_vs_non-weak_fully_faithful} There is a canonical fully faithful functor $\Coh^{(\varphi)}_{\refl}(X_\Prism) \to \wCrys^{(\varphi)}(X_\Prismsp)$.
		\item\label{lem:weak_vs_non-weak_etale_local} A weak prismatic ($F$-)crystal $\mathcal{E}$ is a (reflexive) prismatic ($F$-)crystal if and only if there is an \'etale cover $\{U_i\to X\}$ such that each restriction $\mathcal{E}|_{U_i}$ is so.
		\item\label{lem:weak_vs_non-weak_recovering_formula}
		Let $\mathcal{E}$ be a prismatic crystal in prefect complexes, let $(A,I)\in X_\Prism$ be a bounded prism, and let $(B,J)$ be a covering object in $X_\Prismsp$.
		Then the canonical map belows are isomorphisms
		\[
		\mathcal{E}(A,I) \longrightarrow \underset{(C,K)\in X_\Prismsp}{R\lim}\mathcal{E}((A,I)\coprod (C,K)) \longrightarrow R\lim_{[n]\in \Delta^\mathrm{op}} \mathcal{E}((A,I)\coprod (B^n,IB^n)).
		\]
		The same holds for Laurent $F$-crystals in vector bundles/perfect complexes.
	\end{enumerate}
\end{lemma}
\begin{proof}
	For Part \ref{lem:weak_vs_non-weak_fully_faithful}, as framed regular prisms (which are weak initial in $X_\Prism$) and their \v{C}ech nerves are contained in the subsite $X_\Prismsp$, we have a natural identification $\Coh_\refl^{(\varphi)}(X_\Prism) \simeq \Crys_\refl^{(\varphi)}(X_\Prismsp)$.
	Moreover, the canonical functor $\Crys^{(\varphi)}(X_\Prismsp)\to \wCrys^{(\varphi)}(X_\Prismsp)$ is fully faithful, which follows directly from the definition and \Cref{rmk:Frob_mod_vs_weak_Frob_mod}.
	For Part \ref{lem:weak_vs_non-weak_etale_local}, it suffices to check that for a given completely flat map of prisms $(B_1,J_1)\to (B_2,J_2)$ in $X_\Prismsp$, the induced map of $B_2$-modules $(\mathcal{E}(B_1,J_1)\otimes_{B_1} B_2)^\wedge_{(p,J_2)} \to \mathcal{E}(B_2,J_2)$ is an isomorphism.
	The latter can be checked $p$-completely \'etale locally, hence follows from the assumption.
	Part \ref{lem:weak_vs_non-weak_recovering_formula} follows from the complete faithfully flat descent of perfect complexes together with the flatness of the coproducts in  \Cref{thm:coproduct_in_general}.\ref{thm:coproduct of prism faithful}.
	The analogues for Laurent $F$-crystals follow from Drinfeld--Mathew's descent result \cite[Thm.\ 1.6, Thm.\ 1.7]{Mat22}.
\end{proof}

In the rest of the section, we will construct the canonical weak prismatic $F$-crystal associated to any $p$-adic local system $T\in \Loc_{\mathbb{Z}_p}(X_\eta)$ and prove \Cref{intro:thm:RH}.
Specifically, we give the construction and prove its finiteness in \Cref{sub:F-crys_const}, analyze its Nygaard filtration in \Cref{sub:weak_prismatic_F_crystal:filtration}.
We also study its integral de Rham realizations in \Cref{sub:dR}, its crystalline realization in \Cref{sub:crystalline}, and its compatibility with pullbacks and higher direct images in \Cref{sub:push_pull}. 
In the special case when the local system is pointwise crystalline, we prove in \Cref{sub:crys_loc_sys} that the equivalent descriptions in \Cref{intro:thm:crys_loc_sys} hold true, together with an equivalence of categories between pointwise crystalline local systems and reflexive prismatic $F$-crystals.

\subsection{Construction and the finiteness}
\label{sub:F-crys_const}
We start with the formula of the canonical weak prismatic $F$-crystal together with its finiteness.
Throughout the section, we let $X$ be a regular $p$-adic formal scheme over $\mathcal{O}_K$, and let $T\in \Loc_{\mathbb{Z}_p}(X_\eta)$.

\begin{construction}
\label{const:weak prismatic-F-crystal}
Let  $(B,J)\in X_\Prism$.
\begin{enumerate}[label=\upshape{(\roman*)}]
	\item We define the $B$-modules as follows
\begin{align*}
\label{eq:construction_weak_prismatic-F-crystal}	
\mathcal{E}_{\Prism,n,T}(B,J) &\colonequals \mathrm{H}^0_\pe(X_\eta, T\otimes_{\mathbb{Z}_p} \mu^n\Prism_{(-)_{\overline{B}}/B}), \\
\mathcal{E}_{\Prism,T}(B,J) & \colonequals \mathrm{H}^0_\pe(X_\eta, T\otimes_{\mathbb{Z}_p} \Prism_{(-)_{\overline{B}}/B}[1/\mu]) = \colim_{n\in \mathbb{N}} \mathcal{E}_{\Prism,n,T}(B,J),
\end{align*}
where $\mu^n\Prism_{(-)_{\overline{B}}/B}=\Prism_{(-)_{\overline{B}}/B}\{n\}(-n)$ is the $\mu^n$-twisted prismatic period sheaf as in \Cref{const:rational_period_sheaves}.
\item The Frobenius structure of the relative prismatic cohomology $\Prism_{(-)_{\overline{B}}/B}$ induces a natural map $\Prism_{(-)_{\overline{B}}/B}[1/\mu] \to \Prism_{(-)_{\overline{B}}/B}[1/\mu J]$,
where we use the equalities of ideals $\varphi_{\Ainf}(\mu\Ainf)=\mu\Ainf\cdot I_{\Ainf}$ and $I_{\Ainf}\Prism_{(-)_{\overline{B}}/B}=J\Prism_{(-)_{\overline{B}}/B}$.
By tensoring it with $T$ and then taking the cohomology, we obtain a $\varphi_B$-semi-linear morphism $\mathcal{E}_{\Prism,T}(B,J)\to \mathcal{E}_{\Prism,T}(B,J)[1/J]$.
We let $\mathcal{E}^{(1)}_{\Prism,T}$ be the Frobenius twist $\varphi_{\mathcal{O_\Prism}}^*\mathcal{E}_{\Prism,T}\colonequals(\mathcal{E}_{\Prism,T}(B,J)\otimes^L_{B, \varphi_B} B)^\wedge_{(p,J)}$.
Then inverting $J$ at the source, we obtain a canonical $B$-linear map
\[
\varphi_{\mathcal{E}_{\Prism,T}(B,J)} \colon \mathcal{E}^{(1)}_{\Prism,T}(B,J) [1/J] \to \mathcal{E}_{\Prism,T}(B,J)[1/J].
\]
\end{enumerate}
Putting the above together, we obtain a functor that sends a local system $T$ onto a module with weak Frobenius structure $(\mathcal{E}_{\Prism,T}(B,J), \varphi_{\mathcal{E}_{\Prism,T}(B,J)})$ over $B$,
which is also functorial in the prism $(B,J)\in X_\Prism$.
\end{construction}

To analyze $\mathcal{E}_{\Prism,T}(B,J)$, the following reduction plays an important role.
Namely, the localization at $\mu$ in the construction of $\mathcal{E}_{\Prism,T}$ can be replaced by a finite twist.
Note that the result uses the full strength of the structural results of the rational prismatic period sheaves.
\begin{proposition}[Rational-to-integral reduction]
	\label{prop:weak_prismatic_F_crystal_is_calculated_at_finite_mu_power}
	There is an integer $n\in \mathbb{Z}$ such that for each $(B,J)\in X_\Prismsp$, the inclusion map below is an equality
	\[
	\mathcal{E}_{\Prism,n,T}(B,J) \hookrightarrow \mathcal{E}_{\Prism,T}(B,J).
	\]
\end{proposition}
\begin{proof}
We first notice that the Frobenius structure $\varphi_B$ induces the following commutative diagram of injections
\[
\begin{tikzcd}
\mu^{n+1} \Prism_{(-)_{\overline{B}}/B} \ar[r] \ar[d] & \mu^n\Prism_{(-)_{\overline{B}}/B} \ar[d]\\
\mu^{n+1} \Prism^{(1)}_{(-)_{\overline{B}}/B}  \ar[r] & \mu^n\Prism^{(1)}_{(-)_{\overline{B}}/B}.
\end{tikzcd}
\]
Since $(B,J)$ is a finite coproduct of framed regular prisms, by the arguments of \Cref{prop:intersection_of_Frobenius_of_base_with_coproduct}, the diagram is cartesian.
So after tensoring it with the local system $T$ and then taking the global sections, it suffices to show that for small enough $n\in \mathbb{Z}$, the global section $\mathrm{H}^0_\pe(X_\eta, T\otimes_{\mathbb{Z}_p}\coker(\widetilde{\iota}_{n+1,n}))$ vanishes, where $\widetilde{\iota}_{n+1,n}$ is the map $\mu^{n+1}\Prism^{(1)}_{(-)_{\overline{B}}/B} \to \mu^n\Prism^{(1)}_{(-)_{\overline{B}}/B}$.
Moreover, by \Cref{thm:structure_of_rational_period_sheaf}.(\ref{thm:structure_of_rational_period_sheaf:inj} and \ref{thm:structure_of_rational_period_sheaf:sep}), it suffices to show that for small enough $n$, the cohomology below vanishes 
	\[
	\mathrm{H}^0_\pe(X_\eta,  \bigl( \mu^n\Prism^{(1)}_{(-)_{\overline{B}}/B}\otimes_{\Ainf} \Ainf/\varphi^{-1} (I_{\Ainf}) [1/p] \bigr)^\wedge_{\Fil} ),
	\]
which by the filtered completeness further reduces to the vanishing of its graded pieces.
Then the claim follows from the analysis of the graded pieces in \Cref{thm:structure_of_rational_period_sheaf}.\ref{thm:structure_of_rational_period_sheaf:graded}, together with Liu--Zhu's result in \Cref{thm:LZ_finiteness} and the projection formula in \Cref{thm:projection_formula}.\ref{thm:projection_formula_projective}.
\end{proof}
\begin{definition}
	The largest integer $n\in \mathbb{Z}$ that satisfies \Cref{prop:weak_prismatic_F_crystal_is_calculated_at_finite_mu_power}, if exists, is defined as the \emph{bottom height} of $T$.
\end{definition}
\begin{remark}
	It is still possible that $T$ does not have the bottom height.
	As we shall see in \Cref{prop:weak_prismatic_F_crystal:graded_piece}, this is equivalent to $\mathcal{E}_{\Prism,T}=0$.
\end{remark}

We now consider the structure of $\mathcal{E}_{\Prism,T}(B,J)$ as a $B$-module.
The first observation is the regularity of $(J,p)$.
\begin{lemma}[$(J,p)$-regularity]
	\label{lem:weak_prismatic_F_crystal:regularity}
	Assume $(B,J)\in X_\Prismsp$.
	The sequence of invertible ideals $(J,p)$ is a regular sequence for the $B$-module $\mathcal{E}_{\Prism,T}(B,J)$.
\end{lemma}
In particular, when $(B,J)$ is a two-dimensional complete localization of a regular prsim (for example, the classical Breuil--Kisin prism, and more generally the localization of transversal regular prisms in \Cref{sub:localization_of_prism}), the $B$-module $\mathcal{E}_{\Prism,T}(B,J)$ is locally free (cf. \Cref{lem:reflexive_is_saturated} and \Cref{cor:reflexive_is_locally_free}).
\begin{proof}
	As the construction of $\mathcal{E}_{\Prism,T}(B,J)$ is Zariski local with respect to $X$, we may assume $X$ is affine.
	In addition, by  \Cref{lem:perfection_of_framed_regular_prism}, we let $S$ be perfection of a framed regular prism over $X$, so that $\Spf(S)\to X$ is a faithfully flat cover and the generic fiber $\Spf(S)_\eta$ is a pro-\'etale cover of $X_\eta$.
	So by construction and \Cref{prop:weak_prismatic_F_crystal_is_calculated_at_finite_mu_power}, $\mathcal{E}_{\Prism,T}(B,J)=\mathrm{H}^0_\pe(X_\eta, T\otimes \mu^n\Prism_{(-)_{\overline{B}}/B})=\mathrm{H}^0_\pe(X_\eta, T\otimes \Prism_{(-)_{\overline{B}}/B}\{-n\}(n))$ is a submodule of $T(\Spf(S)_\eta)\otimes\Prism_{S_{\overline{B}}/B}\{-n\}(n)$, where the latter by \Cref{cor:coproduct:perfect_with_framed} is $J$-torsionfree.
	In addition, by the universal coefficient theorem, the mod $J$ reduction $\mathcal{E}_{\Prism,T}(B,J)/J\mathcal{E}_{\Prism,T}(B,J)$ is contained in the cohomology $\mathrm{H}^0_\pe(X_\eta, T\otimes \Prism_{(-)_{\overline{B}}/B}\{-n\}(n)\otimes_B \overline{B})=\mathrm{H}^0_\pe(X_\eta, T(n)\otimes \overline{\Prism}_{(-)_{\overline{B}}/\overline{B}}\otimes_{\overline{B}} \overline{B}\{-n\})$, where the latter is contained in $\overline{\Prism}_{S_{\overline{B}}/\overline{B}} \otimes_{\overline{B}} \overline{B}\{-n\}$.
	By \Cref{cor:conjuage_fil}.\ref{cor:conjuage_fil_non-canonical}, the graded pieces of the conjugate filtration of $\overline{\Prism}_{S_{\overline{B}}/\overline{B}} \otimes_{\overline{B}} \overline{B}\{-n\}$ is locally free over $S_{\overline{B}}$ and in particular are $p$-torsionfree.
	Hence the mod $J$ reduction of $\mathcal{E}_{\Prism,T}(B,J)$ is $p$-torsionfree as well.
\end{proof}

For later applications, we notice that the linearization map, or more generally a complete base change along a completely injective and flat morphism is injective.
\begin{lemma}
	\label{lem:weak_prismatic_F_crystal:inj_of_twist}
	Let $(B,J)\in X_\Prismsp$, let $M$ be a $(J,p)$-complete and $(J,p)$-regular module over $B$, and let $C$ be a $B$-algebra that is $(J,p)$-completely projective.
	\begin{enumerate}
		\item The $C$-complex $(M\otimes^L_B C)^\wedge_{(p,J)}$ lives in cohomological degree zero and is equal to $\lim_m(M\otimes_B C)/(p,J)^m$.
		\item The linearization map $M\to (M\otimes^L_B C)^\wedge_{(p,J)}$ is injective.
	\end{enumerate}
\end{lemma}
\begin{proof}
	The claim on the cohomological degree follows from \Cref{lem:complete_tensor_product_w_flat}.
	For the injectivity, we notice that the complete projectivity of the map $B\to C$ implies that the reduction $\overline{B}/p\overline{B}\to \overline{C}/p\overline{C}$ is faithfully flat and thus injective.
	So the claim follows from \Cref{lem:inj_and_completely_proj_mod}.
\end{proof}
\begin{lemma}
\label{lem:weak_prismatic_F_crystal:inj_of_twist_2}
Let $(B,J)\in X_\Prismsp$.
		The linearization $\mathcal{E}_{\Prism,T}(B,J)\to \mathcal{E}^{(1)}_{\Prism,T}(B,J)$ along $\varphi_B$ is injective.
\end{lemma}
\begin{proof}
The linearization map can be fitted into a composition
\[
\mathcal{E}_{\Prism,T}(B,J)\longrightarrow \mathcal{E}^{(1)}_{\Prism,T}(B,J) \longrightarrow \mathrm{H}^0_\pe(X_\eta,T\otimes \Prism^{(1)}_{(-)_{\overline{B}}/B}[1/\mu]),
\]
so it suffices to show that the composition is injective: this follows from the injectivity of the Frobenius twist map for the relative prismatic cohomology, as in \Cref{prop:intersection_of_Frobenius_of_base_with_coproduct}.\ref{prop:intersection_of_Frobenius_of_base_with_coproduct:cohomology}.
\end{proof}

Continue with the injectivity above, we also have the injectivity of the Frobenius structure.
\begin{lemma}[Frobenius is injective]
	\label{lem:weak_prismatic_F_crystal:Frob_inj}
	Let $(B,J)\in X_\Prismsp$ be a framed regular prism.
	The weak Frobenius structure $\mathcal{E}^{(1)}_{\Prism,T}(B,J)[1/J]\to \mathcal{E}_{\Prism,T}(B,J)[1/J]$ is injective.
\end{lemma}
\begin{proof}
	By construction, the map $\varphi_{\mathcal{E}_{\Prism,T}(B,J)}$ can be factored as below
	\[
	\varphi_B^* \mathrm{H}^0_\pe(X_\eta, T\otimes \Prism_{(-)_{\overline{B}}/B}[1/\mu]) \longrightarrow \mathrm{H}^0_\pe(X_\eta, T\otimes \Prism^{(1)}_{(-)_{\overline{B}}/B}[1/\mu]) \longrightarrow \mathrm{H}^0_\pe(X_\eta, T\otimes \Prism_{(-)_{\overline{B}}/B}[1/\mu])[1/J],
	\]
	where the first map is the linearlization of the morphism $\mathrm{H}^0_\pe(X_\eta, T\otimes \Prism_{(-)_{\overline{B}}/B}[1/\mu]) \to \mathrm{H}^0_\pe(X_\eta, T\otimes \Prism^{(1)}_{(-)_{\overline{B}}/B}[1/\mu])$, and the second map is induced from the Frobenius structure on the relative prismatic period sheaf.
	When $\varphi_B$ is finite flat, the first arrow is an isomorphism.
	The injectivity of the second arrow for a special $(B,J)$ follows from \Cref{prop:intersection_of_Frobenius_of_base_with_coproduct}.\ref{prop:intersection_of_Frobenius_of_base_with_coproduct:cohomology}. 
\end{proof}

Now we prove the finiteness $\mathcal{E}_{\Prism,T}(B,J)$ for a framed regular prism $(B,J)$.
\begin{theorem}[Finiteness]
	\label{prop:weak_prismatic_F_crystal:finiteness}
	Let $(B,J)\in X_\Prismsp$ be a framed regular (thus noetherian) prism.
	The $B$-module $\mathcal{E}_{\Prism,T}(B,J)\bigr)$ is finitely presented.
\end{theorem}
\begin{proof}
Recall from \Cref{prop:regular prism Frobenius} that the map $\varphi_B:B\to B$ is finite faithfully flat.
	So by the fact that the perfectness can be checked flat locally, it suffices to prove that $\mathcal{E}^{(1)}_{\Prism,T}(B,J)$ is a finitely presented $B$-module.
	In addition, as explained in the proof of \Cref{lem:weak_prismatic_F_crystal:Frob_inj}, the $B$-module $\mathcal{E}^{(1)}_{\Prism,T}(B,J)$  is naturally contained in the $(J,p)$-completely regular module $\mathrm{H}^0_\pe(X_\eta, T\otimes \Prism^{(1)}_{(-)_{\overline{B}}/B}\{-n\}(n))$.
	Thus by the universal coefficient theorem and the noetherian assumption, it suffices to show that the $p$-complete $p$-torsionfree module $\mathrm{H}^0_\pe(X_\eta, T\otimes \Prism^{(1)}_{(-)_{\overline{B}}/B}\{-n\}(n)\otimes_B \overline{B}) = \mathrm{H}^0_\pe(X_\eta, T(n)\otimes \dR_{(-)_{\overline{B}}/\overline{B}})$ is finitely generated over $\overline{B}$, where $n\in \mathbb{Z}$ is chosen as in \Cref{prop:weak_prismatic_F_crystal_is_calculated_at_finite_mu_power}.
	Moreover, since the Hodge filtration on the de Rham cohomology is separated (\Cref{prop:global_section_dR}.\ref{prop:global_section_dR_inj}), we may replace the de Rham period sheaf by its filtered completion.
	
	As the cohomology group is local with respect to the Zariski topology of $X$, we may assume $X=\Spf(R)$ is affine.
	Now by the rational finiteness in \Cref{prop:LZ_finiteness}.\ref{prop:LZ_finiteness_general} and by the integral finiteness \Cref{prop:finiteness_integral_general}, we know each cohomology group $\mathrm{H}^i_\pe(X_\eta, T(n)\otimes \gr^i_H \dR_{(-)/X})= \mathrm{H}^i_\pe(X_\eta, T(n)\otimes \wedge^i \mathbb{L}_{(-)/X}[-i])$ has bounded $p^\infty$ torsions and has finitely presented torsionfree quotient over $R$.
	So the projection formula \Cref{thm:projection_formula}.\ref{thm:projection_formula_H1} and the flatness assumption of the prism $(B,J)$ imply the isomorphism  
	\begin{equation}
		\label{eq:weak_prismatic_F_crystal:finiteness:tensor_product_formula}
			\mathrm{H}^i_\pe(X_\eta, T(n)\otimes \gr^i_H \dR_{(-)/X}) \otimes_R \overline{B} \simeq \mathrm{H}^0_\pe(X_\eta, T(n)\otimes \gr^i_H\dR_{(-)_{\overline{B}}/\overline{B}}).
	\end{equation}
	In addition, by the vanishing result in \Cref{prop:LZ_finiteness}.\ref{prop:LZ_finiteness_H0}, only finitely many graded pieces contribute to the cohomology $\mathrm{H}^0_\pe(X_\eta, T(n)\otimes \widehat{\dR}_{(-)_{\overline{B}}/\overline{B}})$.
	Hence the finitenss of $\mathrm{H}^0_\pe(X_\eta, T(n)\otimes \widehat{\dR}_{(-)_{\overline{B}}/\overline{B}})$ follows from the finiteness of the $R$-module $\mathrm{H}^0_\pe(X_\eta, T(n)\otimes \gr^i_H \dR_{(-)/X})$ mentioned above (which itself is $p$-torsionfree), together with the tensor product formula (\ref{eq:weak_prismatic_F_crystal:finiteness:tensor_product_formula}).
\end{proof}

We also observe that the despite the construction of the $B$-module $\mathcal{E}_{\Prism,T}(B,J)$ being global, it only depends on the image of $\Spf(\overline{B})\to X$ under the Zariski topology.
\begin{proposition}[The local nature of cohomology]
\label{prop:weak_prismatic_F_crystal:localization}
Let $(B,J)$ be a bounded prism over $X_\Prism$, and let $U\subseteq X$ be an open sub formal scheme.
Assume the structural map $\Spf(\overline{B})\to X$ factors through $U$.
Then for each $n\in \mathbb{Z}$, the natural map below is an isomorphism
\[
\mathrm{H}^0(X_\eta, T\otimes_{\mathbb{Z}_p} \mu^n\Prism_{(-)_{\overline{B}}/B}) \longrightarrow \mathrm{H}^0(U_\eta, T\otimes_{\mathbb{Z}_p} \mu^n\Prism_{(-)_{\overline{B}}/B}).
\]
\end{proposition}
\begin{proof}
By taking the \v{C}ech cohomology with respect to an open affine covering, it suffices to assume $X=\Spf(R)$ and $U=\Spf(R_1)$ are both affine and connected.
As in \Cref{const:algebrac_pi_1}, we let $\widetilde{X}=\Spf(S)\to X$ be the integral closed perfectoid scheme whose generic fiber is a fixed maximal connected pro-finite-\'etale cover of $X_\eta$.
Under the setup, $\widetilde{X}_\eta \to X_\eta$ is a Galois cover with Galois group $G\colonequals G_{X_\eta}$, and the restriction of the local system $T|_{\widetilde{X}_\eta}$ is trivial.
We also let $\widetilde{X}_U$ be the base change $\widetilde{X}\times_X U=\Spf(S_{R_1})$, so that $\widetilde{X}_{U,\eta}\to U_\eta$ is also a Galois cover of group $G$.
Then the localization map of the cohomology is identified with a map of continuous group cohomology as below
\begin{equation}
\label{eq:localization_of_global_section}
\bigl( T(\widetilde{X}_\eta) \otimes_{\mathbb{Z}_p(\widetilde{X}_\eta)} \mu^n\Prism_{(S\otimes_R{\overline{B}})^\wedge_p/B} \bigr)^G \longrightarrow \bigl( T(\widetilde{X}_{U,\eta}) \otimes_{\mathbb{Z}_p(\widetilde{X}_{U,\eta})} \mu^n\Prism_{(S_{R_1}\otimes_{R_1} \overline{B})^\wedge_p/B} \bigr)^G,
\end{equation}
where $\mathbb{Z}_p(\widetilde{X}_\eta)$ (resp. $\mathbb{Z}_p(\widetilde{X}_{U,\eta})$) is the evaluation of the $p$-complete constant sheaf $\mathbb{Z}_p$ at the adic spaces $\widetilde{X}_\eta$ (resp. $\widetilde{X}_{U,\eta}$).
By assumption, the $p$-complete tensor product $(S_{R_1}\otimes_{R_1} \overline{B})^\wedge_p$ is naturally equal to $(S\otimes_R \overline{B})^\wedge_p$, which induces an identification of their relative prismatic cohomology over $(B,J)$.
Moreover, we have a natural factorization of equivariant maps of rings
\[
\begin{tikzcd}
\mathbb{Z}_p(\widetilde{X}_\eta) \ar[r] \ar[d]& \Prism_{(S\otimes_R{\overline{B}})^\wedge_p/B} \arrow[d, equal] \\
\mathbb{Z}_p(\widetilde{X}_{U,\eta}) \ar[r]&\Prism_{(S_{R_1}\otimes_{R_1} \overline{B})^\wedge_p/B}.
\end{tikzcd}
\]
In addition, since the local system $T|_{\widetilde{X}_\eta}$ is trivial on $\widetilde{X}_\eta$, the pullback along the left vertical map above induces an equality of representations of $G$:
\[
T(\widetilde{X}_{U,\eta}) = T(\widetilde{X}_\eta)\otimes_{\mathbb{Z}_p(\widetilde{X}_\eta)} \mathbb{Z}_p(\widetilde{X}_{U,\eta}).
\]
As a consequence, by taking the base change of the representation $T(\widetilde{X}_\eta)$ (which is finitely generated over $\mathbb{Z}_p(\widetilde{X}_\eta)$) along the factorization diagram, the target Galois cohomology of (\ref{eq:localization_of_global_section}) is naturally identified with the source, which finishes the proof.
\end{proof}

\subsection{Nygaard filtration on the Frobenius twist}
\label{sub:weak_prismatic_F_crystal:filtration}
We then consider the filtration on the Frobenius twist $\mathcal{E}^{(1)}_{\Prism,T}$, evaluated at framed regular prisms.
As before, we assume $T$ is a $\mathbb{Z}_p$-local system on $X_\eta$, where $X$ is a regular $p$-adic formal scheme over $\mathcal{O}_K$.
For simplicity, we also assume $X$ is connected.

\begin{definition}
	\label{def:weak_prismatic_F-crystal_filtration}
	Let  $(B,J)\in X_\Prism$ be either a framed regular prism or a perfect prism.
	For each $i\in \mathbb{Z}$, we define \emph{the twisted (Nygaard) filtration} on $\mathcal{E}^{(1)}_{\Prism,T}(B,J)$ by 
	\[
	\Fil^i \mathcal{E}^{(1)}_{\Prism,T}(B,J) \colonequals \mathrm{H}^0_\pe\bigl(X_\eta, T\otimes_{\mathbb{Z}_p} \Fil^i (\Prism^{(1)}_{(-)_{\overline{B}}/B}[1/\mu])\bigr),
	\]
	where the filtration of the twisted prismatic period sheaf is defined in \Cref{const:rational_period_sheaves}.
\end{definition}
Here we note that we implicitly identifies the $B$-module $\mathcal{E}^{(1)}_{\Prism,T}(B,J)$ with the global section $\mathrm{H}^0_\pe (X_\eta, T\otimes_{\mathbb{Z}_p}  \Prism^{(1)}_{(-)_{\overline{B}}/B}[1/\mu])$, which is clear if $(B,J)$ is perfect and follows from the proof of \Cref{lem:weak_prismatic_F_crystal:Frob_inj} if $(B,J)$ is framed regular.
By construction, the filtration is functorial with respect to the local system and the prism $(B,J)$.

To analyze the filtration, we first notice that the localization by $\mu$ in the definition above can often be replaced by a finite $\mu$-twist, similar to \Cref{prop:weak_prismatic_F_crystal_is_calculated_at_finite_mu_power}.
\begin{lemma}
	\label{lem:weak_prismatic_F-crystal_filtration_is_calculated_at_finite_mu_power}
	Let $(B,J)$ be a framed regular prism, and let $n$ be an integer that satisfies \Cref{prop:weak_prismatic_F_crystal_is_calculated_at_finite_mu_power}.
	For each $i\in \mathbb{Z}$, the natural inclusion map below is an equality
	\[
	\mathrm{H}^0_\pe(X_\eta, T\otimes_{\mathbb{Z}_p} \Fil^i(\mu^n\Prism^{(1)}_{(-)_{\overline{B}}/B})) \xrightarrow{=} \Fil^i \mathcal{E}^{(1)}_{\Prism,T}(B,J).
	\]
	In particular, $\Fil^i \mathcal{E}^{(1)}_{\Prism,T}(B,J)=\mathcal{E}^{(1)}_{\Prism,T}(B,J)$ for $i\leq n$.
\end{lemma}
\begin{proof}
	We note from \Cref{lem:rational_period_sheaves:filtration} that the following diagram of injections is cartesian
	\[
	\begin{tikzcd}
		\Fil^i \bigl( \mu^n\Prism^{(1)}_{(-)_{\overline{B}}/B}\bigr) \ar[d] \ar[r] & \mu^n\Prism^{(1)}_{(-)_{\overline{B}}/B} \ar[d] \\
		\Fil^i \bigl( \Prism^{(1)}_{(-)_{\overline{B}}/B}[\frac{1}{\mu}] \bigr) \ar[r]  & \Prism^{(1)}_{(-)_{\overline{B}}/B}[1/\mu].
	\end{tikzcd}
	\]
	So the claim follows by taking the global section at the above diagram (after tensoring it with $T$), together with \Cref{prop:weak_prismatic_F_crystal_is_calculated_at_finite_mu_power}.
	The second claim on the eventual constancy of the filtration follows from the construction that $\Fil^i \bigl( \mu^n\Prism^{(1)}_{(-)_{\overline{B}}/B}\bigr)= \mu^n\Prism^{(1)}_{(-)_{\overline{B}}/B}$ for $i\leq n$.
\end{proof}
Recall from \Cref{lem:weak_prismatic_F_crystal:Frob_inj} that the weak Frobenius structure on $\mathcal{E}_{\Prism,T}(B,J)$ is injective when $(B,J)$ is framed regular.
Under the assumption, the twisted filtration is in fact determined by the weak Frobenius structure; or in other words it is saturated.
\begin{lemma}[Filtration and Frobenius]
	\label{lem:weak_prismatic_F_crystal:Frobenius_and_filtration}
	Let $(B,J)$ be a framed regular prism. For each $i\in \mathbb{Z}$, we have
		\[
		\Fil^i\mathcal{E}^{(1)}_{\Prism,T}(B,J) = \{x\in \mathcal{E}^{(1)}_{\Prism,T}(B,J)~|~\varphi_{\mathcal{E}_{\Prism,T}(B,J)}(x) \in J^i\mathcal{E}_{\Prism,T}(B,J)\}.
		\]
	In particular, the filtration on $\mathcal{E}^{(1)}_{\Prism,T}(B,J)$ is determined by the weak Frobenius structure.
\end{lemma}
\begin{proof}
	We let $n$ be any integer as in \Cref{lem:weak_prismatic_F-crystal_filtration_is_calculated_at_finite_mu_power} and assume $i\geq n$.
	Notice that for large perfectoid objects $\Spf(S)_\eta \to X_\eta$, the filtration on the $\mu$-twisted relative prismatic cohomology $\mu^n\Prism^{(1)}_{S_{\overline{B}}/B}$ is defined through the Frobenius preimage (\cite{BS22}).
	In particular, by ranging over those perfectoid objects, we see that for $i\geq n$, the following diagram of injections among period sheaves is in fact cartesian (cf. \Cref{prop:intersection_of_Frobenius_of_base_with_coproduct}.\ref{prop:intersection_of_Frobenius_of_base_with_coproduct:cohomology} for the injectivity)
	\[
	\begin{tikzcd}
		\Fil^i \bigl( \mu^n\Prism^{(1)}_{(-)_{\overline{B}}/B}\bigr) \arrow[r] \ar[d] & J^i\mu^n \Prism^{(1)}_{(-)_{\overline{B}}/B}\ar[d] \\
		\mu^n\Prism_{(-)_{\overline{B}}/B} \ar[r] & J^n\mu^n \Prism_{(-)_{\overline{B}}/B}.
	\end{tikzcd}
	\]
	So the claim follows by taking the global section and by \Cref{prop:weak_prismatic_F_crystal_is_calculated_at_finite_mu_power} and \Cref{lem:weak_prismatic_F-crystal_filtration_is_calculated_at_finite_mu_power}.
	For general $i\in \mathbb{Z}$, it suffices to notice that the integer $n$ can be as small as needed.
\end{proof}
Another observation is that the canonical filtration is separated.
\begin{lemma}[Separatedness]
	\label{lem:weak_prismatic_F_crystal:filtration_separatedness}
	Let $(B,J)$ be either a flat perfect prism or a framed regular prism. 
	The twisted Nygaard filtration on $\mathcal{E}^{(1)}_{\Prism,T}(B,J)$ is separated.
\end{lemma}
\begin{proof}
	As the statement is local with respect to the \'etale topology of $X$, we may assume $X$ is affine and admits a framed regular prism $(A,I)$ with $\Sigma\subset A^\times$.
	We let $S$ be the perfectoid ring given by the reduction of the perfection $(A_\perf,IA_\perf)$.
	Under the assumption, by \Cref{lem:perfection_of_framed_regular_prism}, the filtered $B$-module $\mathcal{E}^{(1)}_{\Prism,T}(B,J)= \mathrm{H}^0_\pe(X_\eta, T\otimes \Prism^{(1)}_{(-)_{\overline{B}}/B}\{-n\}(n))$ is contained in the Nygaard filtered relative prismatic cohomology $T(n)(\Spf(S)_\eta) \otimes \Prism^{(1)}_{S_{\overline{B}}/B}\{-n\}$, and it suffices to check that the latter is separated, which was proved in \Cref{cor:Nygarrd_filtration_separated}.
\end{proof}

Assembling the previous ingredients, we obtain the following observations on the graded pieces of the filtration on $\mathcal{E}^{(1)}_{\Prism,T}(B,J)$.
\begin{proposition}
	\label{prop:weak_prismatic_F_crystal:graded_piece}
	Let $(B,J)$ be a framed regular prism, and let $n$ be an integer that satisfies \Cref{prop:weak_prismatic_F_crystal_is_calculated_at_finite_mu_power}.
	\begin{enumerate}[label=\upshape{(\roman*)}]
		\item\label{prop:weak_prismatic_F_crystal:graded_piece:lower_pieces}For $i<n$, we have $\gr^i\mathcal{E}^{(1)}_{\Prism,T}(B,J)=0$.
		\item\label{prop:weak_prismatic_F_crystal:graded_piece:graded} For $i\geq n$, there is a natural injection of $\overline{B}$-modules
		\[
		\gr^i  \bigl( \mathcal{E}^{(1)}_{\Prism,T}(B,J) \bigr) \hookrightarrow \mathrm{H}^0_\pe \bigl(X_\eta,T(n) \otimes_{\mathbb{Z}_p} \gr^i_N (\Prism^{(1)}_{(-)_{\overline{B}}/\overline{B}}\{-n\}) \bigr),
		\]
		\item\label{prop:weak_prismatic_F_crystal:graded_piece:inj_of_coker} For each $i \geq n$, there exists a natural injection of $\overline{B}$-modules
		\[
		\coker\bigl( \Fil^{i-1}\mathcal{E}^{(1)}_{\Prism,T}(B,J) \otimes_B J \hookrightarrow \Fil^i\mathcal{E}^{(1)}_{\Prism,T}(B,J) \bigr) \hookrightarrow \mathrm{H}^0_\pe (X_\eta,T(n) \otimes_{\mathbb{Z}_p} \wedge^{i-n}\mathbb{L}_{(-)_{\overline{B}}/\overline{B}}[n-i]).
		\]
		\item\label{prop:weak_prismatic_F_crystal:graded_piece:vanishing} If $T$ does not admit the bottom height, then both $\mathcal{E}^{(1)}_{\Prism,T}(B,J)$ and thus $\mathcal{E}_{\Prism,T}(B,J)$ vanish.
	\end{enumerate}
\end{proposition}
\begin{proof}
	Part \ref{prop:weak_prismatic_F_crystal:graded_piece:lower_pieces} follows from \Cref{lem:weak_prismatic_F_crystal:Frobenius_and_filtration}.
	For \ref{prop:weak_prismatic_F_crystal:graded_piece:graded}, it follows from \Cref{lem:weak_prismatic_F-crystal_filtration_is_calculated_at_finite_mu_power} together with the left exact sequence of the global sections of the exact sequence of sheaves below
	\[
	0 \longrightarrow T(n)\otimes \Fil^{i+1}_N ( \Prism^{(1)}_{(-)_{\overline{B}}/B}\{-n\} )\longrightarrow T(n)\otimes \Fil^i_N (\Prism^{(1)}_{(-)_{\overline{B}}/B}\{-n\} ) \longrightarrow T(n)\otimes \gr^i_N (\Prism^{(1)}_{(-)_{\overline{B}}/B}\{-n\})  \longrightarrow 0.
	\]
	Similarly, Part \ref{prop:weak_prismatic_F_crystal:graded_piece:inj_of_coker} follows from \Cref{lem:weak_prismatic_F-crystal_filtration_is_calculated_at_finite_mu_power} and the following short exact sequence of sheaves
	\[
	0 \longrightarrow T(n)\otimes \Fil^{i-1}_N ( \Prism^{(1)}_{(-)_{\overline{B}}/B}\{-n\} )\otimes_B J \longrightarrow T(n)\otimes \Fil^i_N ( \Prism^{(1)}_{(-)_{\overline{B}}/B}\{-n\} ) \longrightarrow T(n)\otimes \gr^{i-n}_H \dR_{(-)_{\overline{B}}/\overline{B}} \longrightarrow 0,
	\]
	where the latter is a direct consequence of \cite[Cor.\ 5.2.8]{BL22a}.
	Finally, if $T$ does not admit the bottom height, then it means by definition that  any integer $n$ satisfies \Cref{prop:weak_prismatic_F_crystal_is_calculated_at_finite_mu_power}.
	Hence the vanishing of $\mathcal{E}^{(1)}_{\Prism,T}(B,J)$ in Part \ref{prop:weak_prismatic_F_crystal:graded_piece:vanishing} follows from Part \ref{prop:weak_prismatic_F_crystal:graded_piece:lower_pieces} together with the separatedness of the twisted Nygaard filtration (\Cref{lem:weak_prismatic_F_crystal:filtration_separatedness}).
	The vanishing of $\mathcal{E}_{\Prism,T}(B,J)$, on the other hand, is a consequence of its injection into $\mathcal{E}^{(1)}_{\Prism,T}(B,J)$ (\Cref{lem:weak_prismatic_F_crystal:inj_of_twist_2}).
\end{proof}
In particular, the filtration on the Frobenius twist $\mathcal{E}^{(1)}_{\Prism,T}$ is eventually $J$-adic.
\begin{corollary}
	\label{cor:weak_prismatic_F_crystal:eventual_adic}
	Let $n$ be an integer that satisfies \Cref{prop:weak_prismatic_F_crystal_is_calculated_at_finite_mu_power}.
	There is an integer $n'\geq n$, such that for each framed regular prism $(B,J)$, the filtered submodule $\Fil^{\geq n'} \mathcal{E}^{(1)}_{\Prism,T}(B,J)$ is the $J$-adic filtration $J^{\geq n'}\mathcal{E}^{(1)}_{\Prism,T}(B,J)$.
\end{corollary}
\begin{proof}
	It suffices to show that the inclusion map $\Fil^{i-1}\mathcal{E}^{(1)}_{\Prism,T}(B,J) \otimes_B J \hookrightarrow \Fil^i\mathcal{E}^{(1)}_{\Prism,T}(B,J)$ is an isomorphism for $i\gg 0$, which follows from \Cref{prop:weak_prismatic_F_crystal:graded_piece}.\ref{prop:weak_prismatic_F_crystal:graded_piece:inj_of_coker}, Liu--Zhu's finiteness and vanishing result in \Cref{prop:LZ_finiteness}, together with the projection formula in \Cref{thm:projection_formula}.
\end{proof}
Recall as in \Cref{sub:finiteness} that the pro-\'etale sheaf $\wedge^{i-n}\mathbb{L}_{\widehat{\mathcal{O}}/\mathcal{O}_{X_\eta}}(n)[n-i]$ is contained in the period sheaf $\mathcal{O}\mathbb{C}(i)=\gr^i\OBdR$.
So the $\overline{B}[1/p]$-module $\mathrm{H}^0_\pe (X_\eta,T(n) \otimes_{\mathbb{Z}_p} \wedge^{i-n}\mathbb{L}_{(-)_{\overline{B}}/\overline{B}}[n-i])[1/p]$ is a submodule in $\gr^i\mathrm{D}_\HT(T)|_{\Spf(\overline{B})_\eta}$.
Moreover, recall from the work of Shimizu \cite{Shi18} that there is a finite set of elements $\wt(T)\subset \overline{K}$, called \emph{generalized Hodge--Tate weights}, associated to the local system $T\in \Loc_{\mathbb{Z}_p}(X_\eta)$.
By the locally freeness of the generalized eigenspaces of Sen operator in \cite[Proof of Prop.\ 5.3]{Shi18}, we know $\gr^i\mathrm{D}_\HT(T)|_{\Spf(\overline{B})_\eta}=0$ if and only if $(-i)\notin \wt(T)$.
In particular, we get the following vanishing result.
\begin{corollary}
	\label{cor:weak_prismatic_F_crystal:vanishing}
	Assume  $\wt(T)\cap \mathbb{Z}=\emptyset$.
	The $B$-module $\mathcal{E}_{\Prism,T}(B,J)$ vanishes for any framed regular prism $(B,J)$ over $X$.
\end{corollary}

\subsection{Crystalline local systems}
\label{sub:crys_loc_sys}
In this section, we consider the special case when $T$ is a \emph{(pointwise) crystalline} local system.
We prove that under the assumption, the associated canonical weak prismatic $F$-crystal is in fact a reflexive prismatic $F$-crystal whose \'etale realization is $T$ and in addition satisfies the purity result.
In particular, the notion of crystallinity with respect to a regular integral model is as good as that for smooth integral models.

We start by recalling the Strong \'Etale Comparison Theorem of \cite[Thm.\ 1.21]{GR24}, which gives an enhanced association between prismatic $F$-crystals and \'etale local systems.
\begin{theorem}[Guo--Reinecke]
	\label{thm:GR_strong_'etale}
	Let $X$ be a regular $p$-adic formal scheme over $\mathcal{O}_K$, let $\mathcal{E}$ be a reflexive prismatic $F$-crystal over $X_\Prism$, and let $T$ be the associated \'etale local system over $X_\eta$.
	There is a canonical Frobenius equivariant isomorphism of pro-\'etale sheaves over $X_\eta$
	\[
	T\otimes_{\mathbb{Z}_p} \Ainf[1/\mu] \simeq \Ainf(\mathcal{E})[1/\mu],
	\]
	satisfying the following properties:
	\begin{itemize}
		\item The tautological map induces a natural short exact sequence of pro-\'etale sheaves
		\[
		0 \longrightarrow T \longrightarrow \Ainf(\mathcal{E})[1/\mu] \xrightarrow{\varphi_{\mathcal{E}}} \Ainf(\mathcal{E})[1/\varphi(\mu)] \longrightarrow 0.
		\]
		\item The base change of the isomorphism along $\Ainf[1/\mu]\to \Ainf \langle 1/I_{\Ainf} \rangle=W(\widehat{\mathcal{O}}^\flat)$ recovers the $p$-adic Riemann--Hilbert equivalence of Laurent $F$-crystals (\cite[Cor.\ 3.7]{BS23}).
		\item It is compatible with proper smooth direct image: $g:X\to Z$ be a proper smooth morphism of regular $p$-adic formal schemes over $\mathcal{O}_K$.
		There is a canonical Frobenius equivariant isomorphism of pro-\'etale complexes over $Z_\eta$
		\[
		(R g_{\eta,*} T)\otimes_{\mathbb{Z}_p} \Ainf[1/\mu] \simeq \Ainf(R g_{\Prism,*} \mathcal{E})[1/\mu].
		\]
	\end{itemize}
\end{theorem}
Here we recall from \cite[\S\ 9.1]{GR24} that $\Ainf(\mathcal{E})[1/\mu]$ is the canonical $\Ainf[1/\mu]$-linear pro-\'etale sheaf associated to the prismatic crystal $\mathcal{E}$, defined by sending each affinoid perfectoid object $\Spa(S[1/p],S)\in (X_{\eta, K_\infty})_\pe$ to the module $\mathcal{E}(\rAinf(S),I_{\rAinf(S)})[1/\mu]$.
\begin{proof}
For the first two properties,
	it suffices to show that for a covering family of affinoid perfectoid spaces $\Spa(S[1/p],S)\in (X_{\eta, K_\infty})_\pe$, there is a natural isomorphism $T\otimes_{\mathbb{Z}_p} \rAinf(S)[1/\mu] \simeq \mathcal{E}(\rAinf(S),I_{\rAinf(S)})[1/\mu]$ that satisfies the short exact sequence and recovers the $p$-adic Riemann--Hilbert equivalence of Bhatt--Scholze.
	In our case, we check those affinoid perfectoid spaces that admit a map from $\Spf(\overline{A}_\perf)_\eta$, where $(A,I)$ is some framed regular prism (the topos of such affinoid perfectoid spaces coincide with that of $X_{\eta,\pe}$ by \Cref{lem:perfection_of_framed_regular_prism}.)
	Under the assumption, each $\rAinf(S)$-module $\mathcal{E}(\rAinf(S),I_{\rAinf(S)})$, which is isomorphic to the tensor product $\mathcal{E}(A,I)\otimes_A \rAinf(S)$, is finitely presented, perfect, and becomes finite projective after inverting $p$.
	The rest of arguments is then identical to that of \cite[Thm.\ 9.1]{GR24}: by applying the arc descent at perfectoid rings, it suffices to produce a canonical isomorphism after replacing $S$ with a product of perfectoid valuation rings $\prod V_j$ over $\mathcal{O}_{\overline{K}}$.
	Then the claim follows from \cite[Prop.\ 9.12]{GR24}.
	
	For the last property, it suffices to notice that the argument of \cite[Thm.\ 9.1]{GR24} applies to the current setting as well: in \textit{loc. cit.}, the only inputs that made use of the smoothness assumption of $Z$ are the claims that $R^i g_{\Prism,*} \mathcal{E}$ is a prismatic $F$-crystal in perfect complexes on $Z_\Prism$, and $R^i g_{\Prism,*} \mathcal{E}[1/p]$ is a vector bundle over $(Z_\Prism, \mathcal{O}_{\Prism}[1/p])$.
	To see the former, the perfectness was proved in \cite[Cor.\ 5.16]{GR24}, and the Frobenius isogeny follows from either \cite[Rmk.\ 8.11]{GR24}, or the identical arguments of \cite[Thm.\ 8.1]{GR24} with the Breuil--Kisin prism replaced by the framed regular prism.
	To see the latter when $Z$ is a regular $p$-adic formal scheme, since $Z_\Prism$ is covered by framed regular prisms $(A,I)$, it suffices to show that for $M\colonequals (R^i g_{\Prism,*} \mathcal{E})(A,I)$, the localization $M[1/p]$ is finite projective over $A[1/p]$.
	By \Cref{thm:Frob_mod_is_analyticall_loc_free}, it suffices to check that for any given surjection $(A,I)\to (A',I')$ onto a two dimensional transversal regular prism, the base change $M\otimes_A A'[1/p]$ is finite projective over $A'[1/p]$.
	We let $\mathcal{O}_{K'}=\overline{A'}$ be the reduction, which under the assumption is a finite extension of $\mathcal{O}_K$, and let $z:\Spf(\mathcal{O}_{K'})\to Z$ be the structural morphism.
	Then the base change $M\otimes_A A'$ is naturally isomorphic to the evaluation of $\mathcal{F}\colonequals z_\Prism^*R^i g_{\Prism,*}\mathcal{E}=R^i (g_z)_{\Prism,*} (\mathcal{E}|_{X_z})$ at the prism $(A',I')\in (\mathcal{O}_{K'})_\Prism$, where $g_z: X_z\to \{z\}$ is the fiber of $g$ at $z\in Z$ and is proper smooth.
	By \cite[Thm.\ 1.18]{GR24}, $\mathcal{F}$ is a coherent prismatic $F$-crystal over $(\mathcal{O}_{K'})_\Prism$.
	In addition, by \Cref{claim:coherent_prismatic_F_crystal_over_O_K}, we know $\mathcal{F}[1/p]$ is locally free over $((\mathcal{O}_{K'})_\Prism, \mathcal{O}_\Prism[1/p])$.
	Hence the $A'[1/p]$-module $M\otimes_A A'[1/p] = \mathcal{F}(A',I')[1/p]$ is finite projective, which finishes the proof.
\end{proof}

The next statement analyzes the weak prismatic $F$-crystal $\mathcal{E}_{\Prism,T}$ assuming that the local system $T$ arises from a prismatic $F$-crystal.
\begin{proposition}
	\label{prop:uniqueness_of_the_associated_prismatic_F_crystal}
	Let $X$ be a regular $p$-adic formal scheme over $\mathcal{O}_K$, let $T\in \Loc_{\mathbb{Z}_p}(X_\eta)$, and let $\mathcal{E}$ be a reflexive prismatic $F$-crystal whose \'etale realization $\mathrm{T}_\et(\mathcal{E})$ admits an isomorphism $\alpha:\mathrm{T}_\et(\mathcal{E})\xrightarrow{\sim}T$ (\cite[\S 3]{BS23}).
	\begin{enumerate}[label=\upshape{(\roman*)}]
		\item There is an isomorphism of (weak) prismatic $F$-crystals $\mathcal{E}_{\Prism,T}\xrightarrow{\sim} \mathcal{E}$ induced from $\alpha$.
		\item The tautological linearization map $\mathcal{E}_{\Prism,T}\otimes_B \Prism_{(-)_{\overline{B}}/B}[1/\mu]\to T\otimes \Prism_{(-)_{\overline{B}}/B}[1/\mu]$ is an isomorphism of pro-\'etale sheaves.
	\end{enumerate}
\end{proposition}
\begin{proof}
	By \Cref{thm:GR_strong_'etale}, the map $\alpha$ induces an isomorphism of pro-\'etale sheaves $\alpha':T\otimes_{\mathbb{Z}_p} \Ainf[1/\mu] \simeq \Ainf(\mathcal{E})[1/\mu]$.
	In particular, for each prism $(B,J)\in X_\Prism$, by taking the base change along the pro-\'etale sheaves $\Ainf \to \Prism_{(-)_{\overline{B}}/B}$, we obtain an induced Frobenius equivariant isomorphism
	\[
	\beta_1 \colon T\otimes_{\mathbb{Z}_p} \Prism_{(-)_{\overline{B}}/B} [1/\mu] \simeq \Ainf(\mathcal{E}) \otimes_{\Ainf} \Prism_{(-)_{\overline{B}}/B}[1/\mu].
	\]
	Moreover, for each perfectoid ring $S$ over $X$, by applying the crystal structure of $\mathcal{E}$ at the following diagram of prisms
	\[
	(\Ainf(S),I_{\Ainf(S)}) \longrightarrow (\Prism_{S_{\overline{B}}/B} ,J\Prism_{(-)_{\overline{B}}/B}) \longleftarrow (B,J),
	\]
	we obtain a natural isomorphism to the constant pro-\'etale sheaf 
	\[
	\beta_2 \colon \Ainf(\mathcal{E}) \otimes_{\Ainf} \Prism_{(-)_{\overline{B}}/B}[1/\mu] \simeq \mathcal{E}(B,J) \otimes_B \Prism_{(-)_{\overline{B}}/B}[1/\mu].
	\]
	So by combining $\beta_1$ and $\beta_2$, we get an isomorphism of pro-\'etale sheaves
	\[
	\gamma \colon T \otimes_{\mathbb{Z}_p}\Prism_{(-)_{\overline{B}}/B} [1/\mu] \simeq  \mathcal{E}(B,J) \otimes_B \Prism_{(-)_{\overline{B}}/B}[1/\mu].
	\]
	
	Now for each prism $(B,J)\in X_\Prismsp$, by \Cref{const:weak prismatic-F-crystal}, the isomorphism $\gamma$ induces the formula
	\begin{equation}
		\label{eq:E_Prism_T_vs_E}
		\mathcal{E}_{\Prism,T}(B,J) \colonequals \mathrm{H}^0_\pe(X_\eta, T \otimes_{\mathbb{Z}_p}\Prism_{(-)_{\overline{B}}/B} [1/\mu]) \xrightarrow[\sim]{\mathrm{H}^0(\gamma)}   \mathrm{H}^0_\pe(X_\eta, \mathcal{E}(B,J) \otimes_B \Prism_{(-)_{\overline{B}}/B}[1/\mu]).
	\end{equation}
	We then claim that the right hand side above is equal to the $B$-module $\mathcal{E}(B,J)$ through the tautological map $\mathcal{E}(B,J)\to \mathcal{E}(B,J)\otimes_B \Prism_{(-)_{\overline{B}}/B}[1/\mu]$: since $\mathcal{E}$ is reflextive, we know $\mathcal{E}$ is locally free away from the locus $V(p,J)$ and $\mathcal{E}(B,J)=\mathcal{E}(B,J)[1/p]\cap \mathcal{E}(B,J)[1/J]$.
	In particular, by using the finite projectivity of $\mathcal{E}(B,J)[1/p]$ and $\mathcal{E}(B,J)[1/J]$, the map $\gamma$ induces isomorphisms
	\[
	\mathcal{E}_{\Prism,T}(B,J)[1/p] \simeq \mathcal{E}(B,J)[1/p],\quad \mathcal{E}_{\Prism,T}(B,J)[1/J] \simeq \mathcal{E}(B,J)[1/J].
	\]
	Hence the claim follows from \Cref{lem:weak_prismatic_F_crystal:regularity} and \Cref{cor:Koszul_reg_imply_sat}.
	As a consequence, the composition of the maps in (\ref{eq:E_Prism_T_vs_E}) and (the inverse of) the tautological map is a canonical isomorphism $\mathcal{E}_{\Prism,T}(B,J)\xrightarrow{\sim} \mathcal{E}(B,J)$.
	In addition, the canonical linearization map $\mathcal{E}_{\Prism,T}(B,J) \otimes_B \Prism_{(-)_{\overline{B}}/B}[1/\mu] \to T\otimes \Prism_{(-)_{\overline{B}}/B}[1/\mu]$ is in fact an isomorphism: this follows from the commutative diagram induced from $\gamma$:
	\[
	\begin{tikzcd}
		\mathcal{E}_{\Prism,T}(B,J)\otimes_B \Prism_{(-)_{\overline{B}}/B}[1/\mu ] \ar[r] \arrow[d, "{\mathrm{H}^0(\gamma)\otimes \Prism_{(-)_{\overline{B}}/B}[1/\mu]}", "\sim"'] & T\otimes \Prism_{(-)_{\overline{B}}/B}[1/\mu] \arrow[d, "\gamma", "\sim"']\\
		\mathcal{E}(B,J)\otimes_B \Prism_{(-)_{\overline{B}}/B} \arrow[r,"\sim"'] & \mathcal{E}(B,J)\otimes_B \Prism_{(-)_{\overline{B}}/B}[1/\mu].
	\end{tikzcd}
	\]
\end{proof}

By combining the Primitive Purity in \Cref{thm:primitive_purity_inf} with the properties of prismatic Riemann--Hilbert functor, we now prove that the crystallinity of the local system $T$ implies that $\mathcal{E}_{\Prism,T}$ is a reflexive prismatic $F$-crystal and is associated with $T$.
\begin{theorem}
	\label{thm:prismatic_RH_for_crystalline_local_system}
	Let $X$ be a regular $p$-adic formal scheme and let $T\in \Loc_{\mathbb{Z}_p}(X_\eta)$.
	Assume either of the following three equivalent conditions:
	\begin{enumerate}[label=\upshape{(\alph*)}]
		\item\label{thm:prismatic_RH_for_crystalline_local_system_open} There is a dense open smooth subscheme $U\subseteq X$ such that $T|_{U_\eta}$ is a crystalline local system.
		\item\label{thm:prismatic_RH_for_crystalline_local_system_purity} The restriction $T|_{\Spa(L)}$ is a crystalline representation for each Shilov point $\Spa(L)$ of $X_\eta$.
		\item\label{thm:prismatic_RH_for_crystalline_local_system_PC} There is a dense open smooth subscheme $U\subseteq X$ such that the restriction $T|_x$ is a crystalline representation for each classical point $x\in U_\eta$.
	\end{enumerate}
	Then we have the following:
	\begin{enumerate}[label=\upshape{(\roman*)}]
		\item\label{thm:prismatic_RH_for_crystalline_local_system_not_weak} The canonical weak prismatic $F$-crystal $\mathcal{E}_{\Prism,T}$ is a reflexive prismatic $F$-crystal.
		\item\label{thm:prismatic_RH_for_crystalline_local_system_linearization} For each $(B,J)\in X_\Prismsp$, the tautological map $\mathcal{E}_{\Prism,T}(B,J)=\mathrm{H}^0_\pe(X_\eta,T\otimes \Prism_{(-)_{\overline{B}}/B}[1/\mu])  \to T\otimes \Prism_{(-)_{\overline{B}}/B}[1/\mu]$ induces an isomorphism of pro-\'etale sheaves 
		\[
		\mathcal{E}_{\Prism,T}(B,J)\otimes_B \Prism_{(-)_{\overline{B}}/B}[1/\mu] \xrightarrow{\sim} T\otimes \Prism_{(-)_{\overline{B}}/B}[1/\mu].
		\]
		\item\label{thm:prismatic_RH_for_crystalline_local_system_etale_realization} The \'etale realization of $\mathcal{E}_{\Prism,T}$ is canonically isomorphic to $T$.
	\end{enumerate}
	In particular, the local system $T$ is \emph{crystalline} with respect to $X$ and is associated with the crystalline realization of $\mathcal{E}_{\Prism,T}$.
\end{theorem}
Here the crystallinity of the local system is defined in the style of Faltings, as we introduced in \cite[Def.\ 4.4]{GY24}.
For the convenience of discussions, we use $\Loc^\crys_{\mathbb{Z}_p}(X_\eta)$ to denote the full subcategory of local systems that satisfies the conditions in \Cref{thm:prismatic_RH_for_crystalline_local_system}.
\begin{proof}
	We first remind the reader of equivalences among the three assumptions: by taking the pullback, condition \ref{thm:prismatic_RH_for_crystalline_local_system_open} implies both \ref{thm:prismatic_RH_for_crystalline_local_system_purity} and \ref{thm:prismatic_RH_for_crystalline_local_system_PC}  (\cite[Thm.\ 5.3]{GY24}).
	The implication from \ref{thm:prismatic_RH_for_crystalline_local_system_purity} to \ref{thm:prismatic_RH_for_crystalline_local_system_open} is the purity result of Moon \cite[Thm.\ 1.2]{Moo24}.
	The implication from \ref{thm:prismatic_RH_for_crystalline_local_system_PC} to \ref{thm:prismatic_RH_for_crystalline_local_system_open} is the pointwise criterion of Guo--Yang in \cite[Thm.\ 1.1]{GY24}.
	
	Now we prove that the canonical weak prismatic $F$-crystal is a reflexive prismatic $F$-crystal in \ref{thm:prismatic_RH_for_crystalline_local_system_not_weak}.
	As the statement is \'etale local on $X$ (\Cref{lem:weak_vs_non-weak}.\ref{lem:weak_vs_non-weak_etale_local}), by \Cref{prop:uniqueness_of_the_associated_prismatic_F_crystal}, it suffices to show that for each closed point $x\in X$, there is a $p$-complete \'etale neighborhood $U$ together with a reflexive prismatic $F$-crystal $\mathcal{E}_U$ over $U_\Prism$ whose \'etale realization is $T|_{U_\eta}$.
	To see the latter, by \Cref{thm:regular prism}, we may assume $X=\Spf(R)$ is affine and there is a framed regular prism $(A,I)$ such that $\overline{A}=R$.
	Under the assumption, each Shilov point $\Spa(L)$ of $X_\eta$ corresponds to the complete localization $\Spa(L_\mathfrak{p},\mathcal{O}_{L_\mathfrak{p}})$, where $\mathcal{O}_{L_\mathfrak{p}}$ is the $p$-complete discrete valuation ring $(R_\mathfrak{p})^\wedge_p$ for an associated ideal $\mathfrak{p}\in\Ass(R/pR)$.
	Moreover, since $(A,I)$ covers the prismatic site (\Cref{cor:regular prism cover}), it suffices to show that there is a reflexive Frobenius module over $A$ together with a descent data over the \v{C}ech nerve $(A^n,IA^n)$.
	Now by the main results of \cite[Thm.\ A]{GR24} or \cite[Thm.\ 1.3]{DLMS24}, for each $\mathfrak{p}\in \Ass(R/pR)$, the crystalline assumption of $T|_{\Spa(L_\mathfrak{p})}$ implies that there is locally free prismatic $F$-crystal over $\mathcal{O}_{L_\mathfrak{p}}$ associated to $T|_{\Spa(L_\mathfrak{p}}$.
	More concretely, this prismatic $F$-crystal is given by a reflexive Frobenius module $M_{A_{L_\mathfrak{p}}}$ together with a descent isomorphism as below
	\[
	\gamma_{\mathfrak{p}}:M_{A_{L_\mathfrak{p}}}\otimes_{A_{L_{\mathfrak{p}}}} A^1_{L_\mathfrak{p}} \xrightarrow{\sim} A^1_{L_\mathfrak{p}}\otimes_{A_{L_{\mathfrak{p}}}} M_{A_{L_\mathfrak{p}}},
	\]
	where $A^\bullet_{L_\mathfrak{p}}$ is the \v{C}ech nerve and the isomorphism satisfies the cocycle condition over $A^{\geq 2}_{L_\mathfrak{p}}$.
	On the other hand, the $p$-adic Riemann--Hilbert equivalence of \cite[Cor.\ 3.7]{BS23} implies that the local system $T$ corresponds to a locally free Frobenius module $M_{A\langle 1/I \rangle}$ over $A\langle 1/I \rangle$ together with a descent isomorphism 
	\[
	\gamma_\et:M_{A\langle 1/I \rangle} \otimes_{A\langle 1/I \rangle} A^1\langle 1/I \rangle \xrightarrow{\sim} A^1\langle 1/I \rangle \otimes_{A\langle 1/I \rangle} M_{A\langle 1/I \rangle}.
	\]
	Notice that since the local system $T|_{\Spa(L_\mathfrak{p}}$ is produced by restricting $T\in \Loc_{\mathbb{Z}_p}(X_\eta)$ onto $\Spa(L_\eta)$,
	the isomorphisms $\gamma_{\mathfrak{p}}$ and $\gamma_\et$ are identical after their common base changes to $A^1_{L_\mathfrak{p}}\langle 1/I \rangle$.
	
	Next, by the Primitive Purity \Cref{thm:primitive_purity_inf}, the intersection $M\colonequals (\prod_{\mathfrak{p}} M_{A_{L_\mathfrak{p}}}) \cap  M_{A\langle 1/I \rangle}$ is a saturated Frobenius module over $A$ (and in particular is finitely presented).
	In addition, the tensor product $M \otimes_{A,p_j} A^n$ for each structure map $p_j:A\to A^n$ is naturally isomorphic to the intersection 
	\[
	\bigl( \prod_\mathfrak{p} M_{A_{L_{\mathfrak{p}}}} \otimes_{A_{L_{\mathfrak{p}}}} A^n_{L_{\mathfrak{p}}}\bigr) \bigcap \bigl(M_{A{\langle 1/I \rangle}} \otimes_{A\langle 1/I \rangle} A^n\langle 1/I \rangle\bigr).
	\]
	To see the latter, since each morphism $p_j: A\to A^n$ is completely free (\Cref{thm:coproduct_in_general}.\ref{thm:coproduct of prism regular and regular}), by choosing a basis $\{e_i\}_{i\in I}$ of $A^n$ over $A$, we may write the elements in $M\otimes_A A^n$ (which is complete thanks to the finiteness of $M$) uniquely as a countable sum $\sum_{i\in I} m_ie_i$ with $m_i\in M$ converging to zero.
	Similarly, elements in $M_{A_{L_{\mathfrak{p}}}} \otimes_{A_{L_{\mathfrak{p}}}} A^n_{L_{\mathfrak{p}}}$ are of the form $\sum_{i\in I} m_ie_i$ with $m_i\in M_{A_{L_\mathfrak{p}}}$, and elements in $M_{A\langle 1/I \rangle} \otimes_{A\langle 1/I \rangle} A^n\langle 1/I \rangle$ are of the form $\sum_{i\in I} m_i e_i$ with $m_i \in M_{A\langle 1/I \rangle}$.
	In particular, by looking at the coefficients in fron of each $e_i$ and by \Cref{thm:primitive_purity_inf}, an element $m=\sum_{i\in I} m_ie_i $ belongs to the intersection if and only if $m_i\in M$.
	As a consequence, by restricing the isomorphisms $\gamma_{\mathfrak{p}}$ and $\gamma_\et$ onto the intersection, we obtain the following isomorphism through the diagram
	\[
	\begin{tikzcd}
		M\otimes_A A^1 \arrow[r, dashed] \arrow[d, equal]& A^1\otimes_A M \arrow[d,equal]\\
		\bigl( \prod_\mathfrak{p} M_{A_{L_{\mathfrak{p}}}} \otimes_{A_{L_{\mathfrak{p}}}} A^1_{L_{\mathfrak{p}}}\bigr) \bigcap \bigl(M_{A\langle 1/I \rangle} \otimes_{A\langle 1/I \rangle} A^1\langle 1/I \rangle\bigr) \arrow[r,"\sim"] & \bigl( \prod_\mathfrak{p} A^1_{L_{\mathfrak{p}}} \otimes_{A_{L_{\mathfrak{p}}}} M_{A_{L_{\mathfrak{p}}}} \bigr) \bigcap \bigl( A^1\langle 1/I \rangle \otimes_{A\langle 1/I \rangle} M_{A\langle 1/I \rangle} \bigr).
	\end{tikzcd}
	\]
	The cocycle condition of this isomorphism can be checked through intersection.
	Hence we obtain a saturated prismatic $F$-crystal $\mathcal{E}$ over $X_\Prism$, whose \'etale realization coincides with the local system $T$ through the base change $A^n\to A^n\langle 1/I \rangle$ and \Cref{thm:primitive_purity_inf}.
	
	To continue, we want to make use of \Cref{prop:uniqueness_of_the_associated_prismatic_F_crystal}, and it remains to check that the prismatic $F$-crystal $\mathcal{E}$ is in fact reflexive.
	The latter amounts to showing that the Frobenius module $M=\mathcal{E}(A,I)$ is analytically locally free.
	By \Cref{thm:Frob_mod_is_analyticall_loc_free}, it suffices to check that for any given surjection $(A,I)\to (A',I')$ onto a two dimensional transversal regular prism, the base change $M\otimes_A A'[1/p]$ is finite projective over $A'[1/p]$.
	We let $\mathcal{O}_{K'}=\overline{A'}$ be the reduction, which under the assumption is a finite extension of $\mathcal{O}_K$, and let $f:R\to \mathcal{O}_{K'}$ be the structural morphism.
	Then the base change $M\otimes_A A'$ is naturally isomorphic to the evaluation of the coherent prismatic $F$-crystal $f_\Prism^*\mathcal{E}$ at the prism $(A',I')\in (\mathcal{O}_{K'})_\Prism$.
	Now we make the following claim.
	\begin{claim}
		\label{claim:coherent_prismatic_F_crystal_over_O_K}
		For a coherent prismatic $F$-crystal $\mathcal{F}$ over the $p$-adic discrete valuation ring $(\mathcal{O}_{K'})_\Prism$, the $p$-inverted localization $\mathcal{F}[1/p]$ is locally free.
	\end{claim}
	\begin{proof}[Proof of \Cref{claim:coherent_prismatic_F_crystal_over_O_K}]
		We let $(\mathfrak{S},I_\mathfrak{S})$ be a Breuil--Kisin prism of $\mathcal{O}_{K'}$, which covers the prismatic site $(\mathcal{O}_{K'})_\Prism$.
		It then suffices to check that the coherent Frobenius module $\mathcal{F}(\mathfrak{S},I_{\mathfrak{S}})$ over $\mathfrak{S}$, or in other words, the Breuil--Kisin module, is locally free after inverting by $p$.
		The latter was verified for example in \cite[Prop.\ 4.3]{BMS1}.
	\end{proof}
	By applying \Cref{claim:coherent_prismatic_F_crystal_over_O_K} at $\mathcal{F}=f_\Prism^*\mathcal{E}$, we know the localization $f_\Prism^*\mathcal{E}[1/p]$, and in particular the $A'[1/p]$-module $(f_\Prism^*\mathcal{E})(A',I')[1/p]$ is finite projective.
	Hence the $A'$-module $M\otimes_A A'[1/p] \simeq (f_\Prism^*\mathcal{E})(A',I')[1/p]$ is finite projective over $A'[1/p]$, which finishes the proof that $\mathcal{E}_{\Prism,T}$ is a reflexive prismatic $F$-crystal, namely Part \ref{thm:prismatic_RH_for_crystalline_local_system_not_weak}.
	
	Now, to check that the tautological map in \ref{thm:prismatic_RH_for_crystalline_local_system_linearization} is an isomorphism, it suffices to do so \'etale locally on $X$.
	We then notice that the above arguments have shown that \'etale locally on $X$, the \'etale realization of the prismatic $F$-crystal $\mathcal{E}_{\Prism,T}$ is isomorphic to $T$.
	So \ref{thm:prismatic_RH_for_crystalline_local_system_linearization} follows from \Cref{prop:uniqueness_of_the_associated_prismatic_F_crystal}.
	
	Finally, we prove that the \'etale realization of $\mathcal{E}_{\Prism,T}$ is canonically isomorphic to $T$.
	We let $\mathcal{F}\in \Vect^\varphi(X_\Prism,\mathcal{O}_\Prism \langle 1/I \rangle)$ be the Laurent $F$-crystal associated to the local system $T$.
	The claim then amounts to constructing a canonical equivariant isomorphism from $\mathcal{F}(\Prism_S,I_{\Prism_S})$ to $\mathcal{E}_{\Prism,T}(\Prism_S,I_{\Prism_S})\langle 1/I \rangle$ for large enough perfectoid algebras $S$ over $X$.
	We note that for each $(B,J)\in X_\Prismsp$, the maps of prisms $(\Prism_S,I_{\Prism_S})\to (\Prism_{S_{\overline{B}}/B}, I_{\Prism_{S_{\overline{B}}/B}})$ and $(B,J)\to(\Prism_{S_{\overline{B}}/B}, I_{\Prism_{S_{\overline{B}}/B}})$ induce the maps of modules
	\begin{align}
		\label{eq:relating_laurent_and_prismatic}
		\mathcal{F}(\Prism_S,I_{\Prism_S}) \longrightarrow &\mathcal{F}(\Prism_S, I_{\Prism_S})\otimes_{\Prism_S\langle 1/I_{\Prism_S} \rangle} \Prism_{S_{\overline{B}}/B}\langle 1/I_{\Prism_S} \rangle \\
		\simeq  & (T\otimes \Prism_{(-)_{\overline{B}}/B}\langle 1/J \rangle)(\Spf(S)_\eta) \\
		\label{eq:relating_laurent_and_prismatic_2}
		\simeq & \mathcal{E}_{\Prism,T}(B,J)\langle 1/J \rangle\otimes_{B\langle 1/J \rangle} \Prism_{S_{\overline{B}}/B} \langle 1/J \rangle \longleftarrow \mathcal{E}_{\Prism,T}(B,J)\langle 1/J \rangle,
	\end{align}
	which are functorial in $T$, $S$ and $(B,J)$.
	We also recall that the relative prismatic cohomology $\Prism_{S_{\overline{B}}/B}$ is naturally isomorphic to the coproduct $(\Prism_S,I_{\Prism_S})\coprod (B,J)$ by \Cref{cor:coproduct:perfect_with_framed}.
	So by taking the limit with respect to $(B,J)\in X_\Prismsp$ and by the formula in \Cref{lem:weak_vs_non-weak}.\ref{lem:weak_vs_non-weak_recovering_formula} for Laurent $F$-crystals,
	the above maps can be enhanced into canonical isomorphisms
	\begin{align}
		\label{eq:limit_formula_of_crystals}
		\mathcal{F}(\Prism_S,I_{\Prism_S}) & \simeq \underset{(B,J)\in X_\Prismsp}{R\lim} \bigl( \mathcal{F}(\Prism_S, I_{\Prism_S})\otimes_{\Prism_S\langle 1/I_{\Prism_S} \rangle} \Prism_{S_{\overline{B}}/B}\langle 1/I_{\Prism_S} \rangle \bigr),\\
		\mathcal{E}_{\Prism,T}(\Prism_S,I_{\Prism_S}) & \simeq \underset{(B,J)\in X_\Prismsp}{R\lim} \bigl( \mathcal{E}_{\Prism,T}(B,J)\langle 1/J \rangle\otimes_{B\langle 1/J \rangle} \Prism_{S_{\overline{B}}/B} \langle 1/J \rangle \bigr).
	\end{align}
	Hence by combining the maps in (\ref{eq:relating_laurent_and_prismatic}) and (\ref{eq:relating_laurent_and_prismatic_2}), we obtain a canonical isomorphisms of Laurent $F$-crystals
	\[
	\mathcal{F} \xrightarrow{\sim} \mathcal{E}_{\Prism,T} \langle 1/I_\Prism \rangle,
	\]
	which from construction is functorial in $T$.
\end{proof}
\begin{corollary}[Crystalline local systems are prismatic]
	\label{cor:etale_realization_is_equiv}
	Let $X$ be a regular $p$-adic formal scheme over $\mathcal{O}_K$.
	The prismatic Riemann--Hilbert functor induces an equivalence of categories
	\[
	\Loc^\crys_{\mathbb{Z}_p}(X_\eta) \longrightarrow \Coh^\varphi_\refl(X_\Prism),
	\]
	whose inverse is the \'etale realization.
\end{corollary}
\begin{proof}
	\Cref{thm:prismatic_RH_for_crystalline_local_system}.\ref{thm:prismatic_RH_for_crystalline_local_system_not_weak} shows that the canonical weak prismatic $F$-crystal of an object $T\in \Loc^\crys_{\mathbb{Z}_p}(X_\eta)$ is a reflexive prismatic $F$-crystal, hence inducing the functor as in the statement.
	In addition, by \Cref{thm:prismatic_RH_for_crystalline_local_system}.\ref{thm:prismatic_RH_for_crystalline_local_system_etale_realization} we know the composition of the above functor with the \'etale realization is an equivalence.
	To finish the proof, it then suffices to show that the \'etale realization functor of $\Coh^\varphi_\refl(X_\Prism)$ is faithful, which follows from the reflexivity and the fact that for a regular transversal prism $(A,I)$ that covers $X_\Prism$ (\'etale locally), the complete localization map $A\to A\langle 1/I \rangle$ is an injection (cf. (\ref{diagram:inj_of_prisms})).
\end{proof}
\begin{corollary}[Various notions of the crystallinity]
	\label{cor:various_notion_of_crystallinity}
	Let $X_\eta$ be a smooth rigid space over $K$ and let $T\in \Loc_{\mathbb{Z}_p}(X_\eta)$.
	Assume $X$ is a regular $p$-adic formal model of $X_\eta$ over $\mathcal{O}_K$.
	The following conditions are equivalent:
	\begin{enumerate}[label=\upshape{(\alph*)}]
		\item The canonical weak prismatic $F$-crystal $\mathcal{E}_{\Prism,T}$ is a prismatic $F$-crystal and is associated to $T$.
		\item There exists (hence unique by \Cref{cor:etale_realization_is_equiv}) a reflexive prismatic $F$-crystal over $X_\Prism$ whose \'etale realization is $T$.
		\item The local system $T$ is crystalline with respect to the integral model $X$ (in the sense of \cite[Def.\ 4.4]{GY24}).
		\item There is a dense open smooth subscheme $U\subseteq X$ such that $T|_{U_\eta}$ is a crystalline local system.
		\item The restriction $T|_{\Spa(L)}$ is a crystalline representation for each Shilov point $\Spa(L)$ of $X_\eta$.
		\item There is a dense open smooth subscheme $U\subseteq X$ such that the restriction $T|_x$ is a crystalline representation for each classical point $x\in U_\eta$.
	\end{enumerate}
	In particular, if above conditions hold true with respect to $X$, it then holds true for every regular model of the rigid space $X_\eta$.
\end{corollary}
\begin{remark}[Enhancement into $F$-gauges]
\label{rmk:F-gauge}
When $T\in \Loc^\crys_{\mathbb{Z}_p}(X_\eta)$, it worth mentioning that the prismatic $F$-crystal $\mathcal{E}_{\Prism,T}$ can be canonically enhanced into a coherent prismatic $F$-gauge over $X$, or in other words a coherent sheaf over the syntomification $X^{\mathrm{syn}}$, as introduced by Drinfeld \cite{Dri24} and Bhatt--Lurie (\cite{Bha23}).
Indeed, when $X$ is smooth, it was proved in our joint work with Li \cite[Thm.\ 3.32]{GL23} that there is a natural fully faithful functor
\[
\Coh_{\refl}^\varphi(X_\Prism) \longrightarrow \Coh(X^{\mathrm{syn}}),
\]
defined through the saturated filtration for the Frobenius structure.
Specifically, the twisted Nygaard filtrations on $\mathcal{E}_{\Prism,T}(B,J)$ for Breuil--Kisin prisms $(B,J)\in X_\Prism$, as introduced in \Cref{def:weak_prismatic_F-crystal_filtration}, generate (in the filtered sense) the Nygaard filtration on the associated $F$-gauge.
The latter was explained in \cite[Conv.\ 3.34, Lem.\ 3.37, and Prop.\ 3.38]{GL23}.
For general regular $p$-adic formal schemes $X$, the statement follows by applying the same arguments at the framed regular prisms $(B,J)\in X_\Prismsp$.
\end{remark}

\subsection{Integral de Rham realizations}
\label{sub:dR}
In this subsection, we construct an integral enhancement of Liu--Zhu's Riemann--Hilbert functor and relates it with the prismatic Riemann--Hilbert functor, for a regular $p$-adic formal scheme $X$ over $\mathcal{O}_K$.

As before, we let $\mathcal{O}_{K_\infty}$ be the ring of integers of the cyclotomic extension of $K$, and let $S$ be any perfectoid algebra over $X_{\mathcal{O}_{K_\infty}}$.
Recall that the absolute prismatic--de Rham realization map induces a canonical composition of injections
\[
\rAinf(S)=\Prism_S \longrightarrow \dR_{S/\mathbb{Z}_p} \longrightarrow \dR_{S/X},
\]
which is compatible with the filtrations.
So by twisting a power of the element $\mu\in \Fil^1_N \Prism_S$, for each $n\in \mathbb{Z}$, we obtain a filtered module $\Fil^\bullet \mu^n\dR_{S/X}$ together with natural injections $\Fil^\bullet \mu^{n+1}\dR_{S/X}\to \Fil^\bullet \mu^{n}\dR_{S/X}$.
As in \Cref{sub:rat_period_sheaf}, the construction naturally extends to a sequence of $p$-complete sheaves $\mu^{\bullet}\dR_{(-)/X}$ over $X_\eta$, which we call 
\emph{the twisted integral de Rham period sheaves}.
\begin{construction}
\label{const:de_Rham_functor}
Let $T\in \Loc_{\mathbb{Z}_p}(X_\eta)$, and let $n\in \mathbb{Z}$.\\
(1) 
We define a presheaf of $p$-complete $\mathcal{O}_X$-modules by sending an affine subscheme $U=\Spf(R)\subseteq X$ onto 
	\[
	\mathcal{E}_{\dR,n,T}(U)\colonequals \mathrm{H}^0_\pe(U_\eta, T\otimes_{\mathbb{Z}_p}\mu^{n} \dR_{(-)/X}).
	\]\\
(2) 
For each $i\in \mathbb{Z}$, we define a canonical filtration of submodules on $\mathcal{E}_{\dR,n,T}$ by 
	\[
	\Fil^i \mathcal{E}_{\dR,n,T}(U) \colonequals \mathrm{H}^0_\pe(U_\eta, T\otimes_{\mathbb{Z}_p} \Fil^i\mu^{n}\dR_{(-)/X}).
	\]\\
(3) 
	For each pair of integers $n\geq m$, the injection of twisted de Rham period sheaves induces an injection of presheaves of $\mathcal{O}_X$-modules 
	\[
	\mathcal{E}_{\dR,n,T} \longrightarrow \mathcal{E}_{\dR,m,T},
	\]
	and we let $\mathcal{E}_{\dR,T}$ be the union $\colim_{n\in \mathbb{Z}} \mathcal{E}_{\dR,n,T}$.	\\
(4)
Consider the derived connection on the de Rham period sheaf introduced in  \cite[Lem.\ 3.15]{GL21}
	\[
	\dR_{(-)/X} \longrightarrow \dR_{(-)/X}\otimes_{\mathcal{O}_X} \mathbb{L}_{\mathcal{O}_X/W},
	\]
which satisfies a derived Poincar\'e lemma as in \textit{loc.\ cit.}.
		
		Assume $X$ is smooth over $W$. 
		We define an $\mathcal{O}_K$-linear continuous connection $\nabla_{\dR,n,T}$ on $\mathcal{E}_{\dR,n,T}$ through the formula
		\[
		\mathrm{H}^0_\pe(U_\eta, T\otimes \mu^n\dR_{(-)/X}) \longrightarrow \mathrm{H}^0_\pe(U_\eta, T\otimes \mu^n\dR_{(-)/X}\otimes_{\mathcal{O}_X} \mathbb{L}_{X/W}) = \mathrm{H}^0_\pe(U_\eta, T\otimes \mu^n\dR_{(-)/X})\otimes_{R} \Omega^{1,\an}_{R/W},
		\]
		where $U=\Spf(R)\subseteq X$ is an affine open subscheme and $\Omega^{1,\an}_{R/W}=\mathbb{L}_{R/W}$ is the continuous differential (which is finite projective under the smoothness assumption).
		By the Griffiths transversality of the connection on the integral period sheaf, the connection $\nabla_{\dR,n,T}$ satisfies the Griffiths transversality as well.
	
		For a general regular $p$-adic formal scheme $X$, we define a $K$-linear continuous connection $\nabla_{\dR,T}:\mathcal{E}_{\dR,T}(U)[1/p]\to \mathcal{E}_{\dR,T}[1/p]\otimes_{R[1/p]} \Omega^{1,\an}_{R[1/p]/K}$ to be the map of global sections
		\[
		\mathrm{H}^0_\pe(U_\eta, T\otimes_{\mathbb{Z}_p} \dR_{(-)/X}[1/\mu]) [1/p]\longrightarrow \mathrm{H}^0_\pe(U_\eta, T\otimes_{\mathbb{Z}_p} \dR_{(-)/X}[1/\mu]\otimes_{\mathcal{O}_X} \mathbb{L}_{X/\mathbb{Z}_p})[1/p],
		\]
		where the latter is naturally isomorphic to $\mathrm{H}^0_\pe(U_\eta, T\otimes_{\mathbb{Z}_p} \dR_{(-)/X}[1/\mu])\otimes_{R[1/p]} \Omega^{1,\an}_{R[1/p]/K}$: thanks to the smoothness of the generic fiber $X_\eta$ over $K$, the analytic cotangent complex $\mathbb{L}^\an_{X_\eta/K}=\mathbb{L}_{\mathcal{O}_X/W}[1/p]$ is a vector bundle over the rigid space $X_\eta$ (cf. \cite[Prop.\ 5.2.12]{Guo25}).
\end{construction}

\begin{proposition}
	\label{prop:dR:finiteness_&_injective}
	Let $T\in \Loc_{\mathbb{Z}_p}(X_\eta)$, and let $n\in \mathbb{Z}$.
	\begin{enumerate}[label=\upshape{(\roman*)}]
		\item\label{prop:dR:finiteness_&_injective:finitenss} 
			Each $\mathcal{E}_{\dR,n,T}$ is a presheaf of finitely presented $p$-torsionfree modules over $\mathcal{O}_X$.
		\item\label{prop:dR:finiteness_&_injective:LZ}
		    For each affine open subscheme $U=\Spf(R)\subseteq X$, there is a canonical filtered injection of $R$-modules with continuous flat connections
		    \[
		    \mathcal{E}_{\dR,n,T}(U) \hookrightarrow \mathrm{D}_\dR(T)(U_\eta),
		    \]
		    where $\mathrm{D}_\dR(T)$ is the $p$-adic Riemann--Hilbert functor of Liu--Zhu \cite{LZ17}.
		\item\label{prop:dR:finiteness_&_injective:isogeny} 
		For $n\gg0$, the $p$-inverted localization of the inclusion $\mathcal{E}_{\dR,n+1,T}\to \mathcal{E}_{\dR,n,T}$ is an equality with both terms equal to $\mathcal{E}_{\dR,T}$.
	\end{enumerate}
\end{proposition}
\begin{proof}
    As in the proof of \Cref{prop:global_section_dR}, the filtered completion map $T\otimes_{\mathbb{Z}_p} \mu^n\dR_{(-)/X} \hookrightarrow T\otimes_{\mathbb{Z}_p} \mu^n\widehat{\dR}_{(-)/X}$ is an injection.
	Moreover, since the graded piece of $\dR_{(-)/X}$ is pro-\'etale locally free (\Cref{cor:conjuage_fil}) and is in particular $p$-torsionfree, we have a further injection $T\otimes_{\mathbb{Z}_p} \mu^n\widehat{\dR}_{(-)/X} \hookrightarrow T\otimes_{\mathbb{Z}_p} \mu^n(\dR_{(-)/X}[1/p])^\wedge_{\Fil} $.
	Notice that by \cite[Thm. 1.1]{GL21}, the period sheaf $(\dR_{(-)/X}[1/p])^\wedge_{\Fil}$ is naturally isomorphic to the de Rham period sheaf $\OB^+_\dR$ over $X_{\eta,\pe}$.
	Hence the composition of the above maps of pro-\'etale sheaves induce an injection
	\[
	\mathcal{E}_{\dR,n,T}(U) =\mathrm{H}^0_\pe(U_\eta, T\otimes_{\mathbb{Z}_p} \mu^n\dR_{(-)/X} ) \hookrightarrow \mathrm{H}^0_\pe(U_\eta, T\otimes_{\mathbb{Z}_p} \OB_\dR) = \mathrm{D}_\dR(T)(U_\eta),
	\]
	which finishes the proof of \ref{prop:dR:finiteness_&_injective:LZ}.
	
	To continue, we note that the pro-\'etale sheaf $T\otimes_{\mathbb{Z}_p} \mu^n\dR_{(-)/X}$ is $p$-torsionfree and $p$-complete, its global section is also a $p$-torsionfree $p$-complete module.
	On the other hand, by \cite[Thm.\ 3.9]{LZ17}, we know $\mathrm{D}_\dR(T)(U_\eta)$ is a finitely generated $R[1/p]$-module, which implies the same for the submodule $\mathcal{E}_{\dR,n,T}(U)[1/p]$.
	Hence the finiteness in \ref{prop:dR:finiteness_&_injective:finitenss} follows from \Cref{prop:finiteness_integral_general}.
	
	Finally, the equalities of $\mathcal{E}_{\dR,n+1,T}[1/p]=\mathcal{E}_{\dR,n,T}[1/p]=\mathcal{E}_{\dR,T}[1/p]$ follows from the noetherian property of the ring $R[1/p]$ and \ref{prop:dR:finiteness_&_injective:LZ}.
	To finish the proof of \ref{prop:dR:finiteness_&_injective:isogeny} , it suffices to check that $\mathcal{E}_{\dR,T}$ is automatically $p$-inverted (and thus $\mathcal{E}_{\dR,T}[1/p]=\mathcal{E}_{\dR,T}$): by the functoriality of the derived de Rham complex, the period sheaf $\dR_{(-)/X}$ locally on $X_{\eta,\pe}$ admits a map from $\dR_{\mathcal{O}_{K_\infty}/W}$, where the latter by Beilinson and Bhatt (cf. \cite[Thm.\ 1.4]{GL21}) is naturally isomorphic to the ring $\rAcrys(\mathcal{O}_{K_\infty})$.
	So the claim follows from a known fact that the element $p$ is inverted in $\rAcrys(\mathcal{O}_{K_\infty})[1/\mu]=\mathrm{B}_\crys(\mathcal{O}_{K_\infty})$.	
\end{proof}

To relate the prismatic Riemann--Hilbert functor with the integral de Rham functor, we have the following natural observation.
\begin{proposition}[de Rham realization]
	\label{prop:weak_prismatic_F_crystal:deRham}
	Let $(B,J)\in X_\Prism$, let $n$ be an integer, and let $U=\Spf(R)\subseteq X$ be the image of the morphism $\Spf(\overline{B})\to X$.
	\begin{enumerate}[label=\upshape{(\roman*)}]
	\item\label{prop:weak_prismatic_F_crystal:deRham:inj} Assume $(B,J)$ is a framed regular prism or its perfection. 
	There is a canonical injection of filtered $\overline{B}$-modules that are compatible with the continuous connections (either when $X$ is smooth or after inverting $p$)
		\[
		(\mathcal{E}^{(1)}_{\Prism,n,T}(B,J))\otimes_B \overline{B} \hookrightarrow \mathcal{E}_{\dR,n,T}\otimes_R \overline{B},
		\]
	\item Assume $T\in \Loc^\crys_{\mathbb{Z}_p}(X_\eta)$ and $(B,J)$ is special.
	Then $\mathcal{E}_{\dR,T}$ is a sheaf of finitely presented modules over $\mathcal{O}_X[1/p]$ and the above map is an isomorphism after inverting $p$.
	Moreover, the canonical injection $\mathcal{E}_{\dR,T}\to \mathrm{D}_{\dR}(T)$ in \Cref{prop:dR:finiteness_&_injective}.\ref{prop:dR:finiteness_&_injective:LZ} is an isomorphism.
	\end{enumerate}
\end{proposition}
\begin{proof}
    We first notice that it suffices to replace $X$ by the open subscheme $U$, thanks to the localization property in \Cref{prop:weak_prismatic_F_crystal:localization} (where the same arguments apply verbatim to the period sheaves $\dR_{(-)_{\overline{B}}/\overline{B}}$ and $\dR_{(-)/X}$).
	By the proof of \Cref{lem:weak_prismatic_F_crystal:Frob_inj}, we know the Frobenius twist $\varphi_B^*\mathrm{H}^0_\pe(X_\eta, T\otimes \Prism_{(-)_{\overline{B}}/B}\cdot \mu^n)$ is naturally equal to $\mathrm{H}^0_\pe(X_\eta, T\otimes \Prism^{(1)}_{(-)_{\overline{B}}/B}\cdot \mu^n)$.
	In addition, since the ideal $J$ is an invertible $B$-module, the universal coefficient theorem implies an injection
	\[
	\mathcal{E}^{(1)}_{\Prism,T}(B,J)\otimes_B \overline{B} = \bigl( \varphi_B^*\mathrm{H}^0_\pe(X_\eta, T\otimes \Prism_{(-)_{\overline{B}}/B}\cdot \mu^n) \bigr)\otimes_B \overline{B} \hookrightarrow \mathrm{H}^0_\pe(X_\eta, T\otimes \dR_{(-)_{\overline{B}}/\overline{B}}\cdot \mu^n).
	\]
	On the other hand, by the complete projectivity assumption of $\overline{B}$ over $R$, the finiteness of $\mathrm{H}^0_\pe(X_\eta, T\otimes \dR_{(-)/X}\cdot \mu^n)$ over $R$ (\Cref{prop:dR:finiteness_&_injective}.\ref{prop:dR:finiteness_&_injective:finitenss}), and the projection formula (\Cref{thm:projection_formula}.\ref{thm:projection_formula_projective}), we see the following linearization map is an isomorphism
	\[
	\mathrm{H}^0_\pe(X_\eta, T\otimes \dR_{(-)/X}\cdot \mu^n) \otimes_R \overline{B}\longrightarrow \mathrm{H}^0_\pe(X_\eta, T\otimes \dR_{(-)_{\overline{B}}/\overline{B}}\cdot \mu^n).
	\]
	Hence by combining the above two equations, we obtain a canonical injection as in the statement.
	
	Now we assume $T\in \Loc^\crys_{\mathbb{Z}_p}(X_\eta)$, and for simplicity we assume the rank of the local system $T$ is $m$.
	We apply the base change of the linearization isomorphism in \Cref{thm:prismatic_RH_for_crystalline_local_system}.\ref{thm:prismatic_RH_for_crystalline_local_system_linearization} along $B\xrightarrow{\varphi_B} B \to \overline{B}$, to get an isomorphism of period sheaves on $X_{\eta,\pe}$:
	$$(\mathcal{E}^{(1)}_{\Prism,T}(B,J))\otimes_B \overline{B} \otimes_{\overline{B}} \dR_{(-)_{\overline{B}}/\overline{B}}[1/\mu]\simeq T\otimes \dR_{(-)_{\overline{B}}/\overline{B}}[1/\mu].$$
	So by taking the global section and by the projection formula as above (\Cref{thm:projection_formula}.\ref{thm:projection_formula_projective}), we get the expected isomorphism
	\[
	(\mathcal{E}^{(1)}_{\Prism,T}(B,J))\otimes_B \overline{B}[1/p]\simeq \mathcal{E}_{\dR,T}\otimes_{R[1/p]}\overline{B}[1/p].
	\]
	The last isomorphism in particular implies that $\mathcal{E}_{\dR,T}$ is a flat connection of rank $m$.
	As a consequence, since $\mathrm{D}_{\dR}(T)$ has rank $m$ as well (\cite{LZ17}) and since the category of flat connections is abelian, we know the injection $\mathcal{E}_{\dR,T}\to \mathrm{D}_{\dR}(T)$ is an isomorphism of flat connections.
\end{proof}

\subsection{Crystalline realization}
\label{sub:crystalline}
In this section, we consider the prismatic Riemann--Hilbert functor for a collection of crystalline prisms in $X_\Prism$.
As before, we let $X$ be a regular $p$-adic formal scheme over $\mathcal{O}_K$.

Consider the following collection of crystalline prisms:
\begin{definition}
	\label{def:strongly_crystalline_prism}
	We define a crystalline prism $(D,(p))\in X_\Prism$ to be \emph{special} if either of the following is true:
	\begin{enumerate}[label=\upshape{(\alph*)}]
		\item The prism $(D,(p))$ is of the form $(D_B(J),(p))$, where $(B,J)$ is a prism $(B,J)\in X_\Prismsp$ and $D_B(J)$ is the $p$-complete $p$-adic divided power envelope for the surjection $B\to \overline{B}$.
		\item The prism $(D,(p))$ is of the form $(R_0, (p))$, where the ring $R_0$ is $p$-completely smooth over $W$, and the structural map $\Spec(R_0/pR_0)\to X_{p=0}$ can be lifted to a $p$-completely \'etale morphism $\Spf(R_{0,\mathcal{O}_K})\to X$.
		\footnote{Informally, the ring $\Spf(R_0)$ is an unramified smooth model of an \'etale object over (the smooth locus of) $X$, equipped with a Frobenius structure.}
	\end{enumerate}
\end{definition}
\begin{remark}
	By the identical arguments of \Cref{eg:crystalline_prism}, a special crystalline prism is in particular a flat crystalline prism in the sense of \Cref{ass:flat}.
\end{remark}

We now calculate the global section of the prismatic period sheaves.
\begin{proposition}
	\label{prop:global_section_crystalline}
	Let $(D,(p))$ be a special crystalline prism in $X_\Prism$.
	The tautological maps below are isomorphisms:
	\[
	D\to \mathrm{H}^0_\pe(X_\eta, \Prism_{(-)_{\overline{D}}/D}),\quad D[1/p]\to \mathrm{H}^0_\pe(X_\eta, \Prism_{(-)_{\overline{D}}/D}[\frac{1}{\mu}]).
	\]
\end{proposition}
\begin{proof}
We first consider the case when $D=D_B(J)$ for a special prism $(B,J)$.
As the statement is \'etale local on $X$, we may assume $X=\Spf(R)$ is affine and connected, and following \Cref{const:algebrac_pi_1} we also let $\widetilde{X}_\eta=\Spf(S[1/p],S)\to X_\eta$ be the maximal connected Galois cover with Galois group $G_{X_\eta}$.
It then suffices to show that the tautological map $D\to \Prism_{S_{\overline{D}}/D}$ induces an isomorphism onto $G_{X_\eta}$-invariant of the target.
By the base change formula of prismatic cohomology, the target $\Prism_{S_{\overline{D}}/D}$ is isomorphic to the base change of $\Prism_{S_{\overline{B}}/B}$ along the map of \emph{prisms} $(B,J)\to (D,(p))$.
Here we notice that the underlying map of rings $B\to D$ is naturally factored as $B \xrightarrow{\text{incl}} D \xrightarrow{\varphi_D} D$,
where the first arrow is the canonical inclusion for the divided power envelope $B\to D=D_B(J)$.
So by interpreting $\Prism_{S_{\overline{D}}/D}\simeq (\Prism_{S_{\overline{B}}/B}\otimes_B D)^\wedge_p$ as the coproduct (\Cref{cor:coproduct:perfect_with_framed}), and by further identifying this coproduct as the divided power envelope (\cite[Cor.\ 2.39]{BS22}), we see $\Prism_{S_{\overline{D}}/D}$ is naturally isomorphic to the $p$-complete divided power envelope $D(S)$ of the l.c.i surjection
\[
(\Ainf(S)\otimes_W B)^\wedge_p \longrightarrow S_{\overline{B}}.
\]
Here the ring $D(S)$ is naturally equipped with the divided power filtration, which is compatible with that of $D=D_B(J)$ and is separated by the proof of \Cref{lem:dR_is_sep}.
So it suffices to show that the canonical map $D\to D(S)$ induces an isomorphism $\gr^iD \simeq (\gr^iD(S))^G_{X_\eta}$.

To continue, we notice that by the structure of the divided power envelope with respect to l.c.i surjections, the higher graded pieces are the divided powers of the first graded pieces, where the latter is the shifted ($p$-complete) cotangent complexes.
So we need to prove that 
\[
\overline{B} \xrightarrow{\sim} S_{\overline{B}}^{G_{X_\eta}},~\text{and}~J/J^2 \xrightarrow{\sim} \bigl( \mathbb{L}_{S_{\overline{B}}/(\Ainf(S)\otimes_W B)^\wedge_p}[-1] \bigr)^{G_{X_\eta}}.
\]
The first equality follows from \Cref{lem:global_sec_of_Ohat} and the projection formula in \Cref{thm:projection_formula}.\ref{thm:projection_formula_projective}, since $\overline{B}$ is completely projective.
For the second map, we note that by the relateive perfectness of $\Ainf(S)$ over $W$, the cotangent complex $\mathbb{L}_{S_{\overline{B}}/(\Ainf(S)\otimes_W B)^\wedge_p}$ is naturally isomorphic to $\mathbb{L}_{S_{\overline{B}}/B}$.
In addition, the factorization $B\to \overline{B} \to S_{\overline{B}}$ naturally induces a distinguished triangle of finite projective $S_{\overline{B}}$-modules
\[
 \mathbb{L}_{\overline{B}/B}[-1] \otimes_{\overline{B}} S_{\overline{B}} \longrightarrow \mathbb{L}_{S_{\overline{B}}/B}[-1] \longrightarrow \bigl(\mathbb{L}_{S/R} \otimes_R \overline{B} \bigr)^\wedge_p [-1],
 \]
where by the functoriality $J/J^2\simeq \mathbb{L}_{\overline{B}/B}[-1]$ identifies the first term in the triangle as the base extension.
In particular, by the same reasoning above, the Galois invariant of $\mathbb{L}_{\overline{B}/B}[-1]\otimes_{\overline{B}} S_{\overline{B}}$ is naturally isomorphic to $\mathbb{L}_{\overline{B}/B}[-1]$ through the functoriality map.
So to finish the calculation, it suffices to notice that $\bigl(\mathbb{L}_{S/R} \otimes_R \overline{B} \bigr)^\wedge_p [-1]$ has no $G_{X_\eta}$-invariant, thanks to vanishing result of \Cref{prop:LZ_finiteness}.\ref{prop:LZ_finiteness_H0} and the projection formula.

To calculate the global section for the $p$-inverted period sheaf, we first notice that the ring $D(S)[1/\mu]$, being an algebra over $\Acrys(S)[1/\mu]$, is $p$-inverted.
So to prove the equality $D[1/p]\simeq \mathrm{H}^0_\pe(X_\eta, \Prism_{(-)_{\overline{D}}/D}[1/p])$, it suffices to check that map
\[
\mu^{n+1}D(S)[1/p] \longrightarrow \mu^n D(S)[1/p]
\]
induces an isomorphism on Galois invariants for each $n<0$.
In addition, by the separatedness of the divided power filtrations, we may do so by considering the Galois invariant of the graded pieces.
To continue, we use the calculation from \Cref{prop:Ainf_twist_difference} to get that the $i$-th graded pieces of the cokernel of the above map is isomorphic to $S(n)\otimes_S \gr^{i-n}D(S)[1/p]\simeq S(n)\otimes_S \wedge^{i-n}\mathbb{L}_{S_{\overline{B}}/B}[n-i][1/p]$ (or zero if $i<n$).
Now we claim that the above module has no Galois invariant for each $n<0$.
By the calculation of the cotangent complex, the above $S_{\overline{B}}[1/p]$-module admits a finite filtration whose graded pieces are $\Gamma^j (\mathbb{L}_{\overline{B}/B}[-1]) \otimes_{\overline{B}} \Gamma^{i-n-j} \bigl( \mathbb{L}_{S/R} \otimes_R \overline{B} \bigr)^\wedge_p [-1]\bigr)\otimes_S S(n)[1/p]$ for $0\leq j\leq i-n$.
Note that the tensor product with $\Gamma^\bullet\mathbb{L}_{\overline{B}/B}[-1]$ and $\overline{B}$ have no influence on the vanishing of the Galois invariant, since both are completely projective $R$-modules.
So to check the vanishing, it suffices to recall that as in the proof of \Cref{prop:LZ_finiteness}.\ref{prop:LZ_finiteness_H0}, the Galois invariant of each each $\Gamma^{i-n-j} (\mathbb{L}_{S/R}  [-1])\otimes_S S(n)[1/p]$ is contained in $\mathrm{H}^0_\pe(X_\eta, \widehat{\mathcal{O}}(n))$, which is zero for $n<0$.

Next, we consider the case when $D=R_0$ is a smooth model of $X=\Spf(R)$ over $W$, where the Frobenius structure $\varphi_{R_0}$ is finite flat.
Under the assuption, to show that the canonical injection $D\to \Prism_{S_{\overline{R}_0}/R_0}^{G_{X_\eta}}$ is isomorphic, it suffices to show it after a pullback along $\varphi_{R_0}$.
In particular, by the prismatic--crystalline comparison (\cite[Thm.\ 5.2]{BS22}), it suffices to consider the map $R_0\to R\Gamma_\crys(S_{\overline{R}_0}/R_0)^{G_{X_\eta}}$.
Note that by \cite[Thm. 1.4]{GL21} and the crystalline--de Rham comparison, the crystalline cohomology $R\Gamma_\crys(S_{\overline{R}_0}/R_0)$ is naturally isomorphic to $\mathcal{O}\mathbb{A}_{\crys,R_0}(S[1/p],S)$.
In addition, the rational analogue $R_0[1/p]\to \bigl( \mathcal{O}\mathbb{A}_{\crys,R_0}(S[1/p],S)[1/\mu]\bigr) ^{G_{X_\eta}}$ is an isomorphism (\cite[Thm.\ 3.21]{GY24}).
As a consequence, since $\mathcal{O}\mathbb{A}_{\crys,R_0}(S[1/p],S)^{G_{X_\eta}}$ is a compact subalgebra of $R_0[1/p]$ that contains $R_0$, the Galois invariant has to be $R_0$ itself.
\end{proof}

We then notice that the prismatic Riemann--Hilbert functor satisfies the base change formula when the base prisms are special crystalline.
\begin{lemma}
	\label{lem:weak_prismatic_F_crystal:crystalline_base_change}
	Let $T\in \Loc_{\mathbb{Z}_p}(X_\eta)$, and let $(D_1,(p))\to (D_2,(p))$ be a completely projective map of crystalline prisms.
	The linearization map $\bigl( \mathcal{E}_{\Prism,n,T}(D_1,(p))\otimes_{D_1} D_2 \bigr)^\wedge_p \to \mathcal{E}_{\Prism,n,T}(D_2,(p))$ is an isomorphism.
\end{lemma}
\begin{proof}
This follows immerdiately from the projection formula in \Cref{thm:projection_formula}.\ref{thm:projection_formula_projective}.
\end{proof}
\begin{corollary}[Compatibility with the crystalline Riemann--Hilbert functor]
	\label{cor:weak_prismatic_F_crystal:crystalline:local_gluing}
	Let $U\subset X$ be an affine open subscheme, and let $T\in \Loc_{\mathbb{Z}_p}(X_\eta)$.
	\begin{enumerate}[label=\upshape{(\roman*)}]
		\item 	By ranging over all the special crystalline prisms $(D,(p))$ such that $\overline{D}$ is projective over $U_{p=0}$, the $D[1/p]$-modules $\mathcal{E}_{\Prism,n,T}(D,(p))$ naturally glue to a crystal $\mathcal{F}_{n,U}$ with a weak Frobenius structure over $U_{p=0}$.
		\item Assume $X$ is smooth over $\mathcal{O}_K$. 
		Then the crystal $\mathcal{F}_{n,U}$ is coherent, the inclusion map $\mathcal{F}_{n,U}\to \mathcal{F}_{n-1,U}$ is an isogeny for $n\ll 0$, and the isocrystal $\colim_{n\in \mathbb{Z}}\mathcal{F}_{n,U}$ with weak Frobenius structure is naturally equivalent to $\mathcal{E}_{\crys,T}(U)$ of Guo--Yang \cite[Thm.\ 1.10]{GY24}.
	\end{enumerate}
\end{corollary}
We expect the individual crystal $\mathcal{F}_{n,U}$ to be finitely generated in the first case as well, though we do not pursue it in this article due to the mixed characteristic theme and the length of the article.
\begin{proof}
	The claim follows from the base change formula in \Cref{lem:weak_prismatic_F_crystal:crystalline_base_change} and the prismatic--crystalline equivalence for coefficients in \cite[Thm.\ 6.4]{GR24}.
\end{proof}

Our next result relates the prismatic Riemann--Hilbert functor for transversal and crystalline prisms.
\begin{proposition}
\label{prop:weak_prismatic_F_crystal:transversal_to_crystalline}
	Let $(D,(p))$ be a special crystalline prism in $X_\Prism$, and let $T\in \Loc_{\mathbb{Z}_p}(X_\eta)$.
	\begin{enumerate}[label=\upshape{(\roman*)}]
		\item\label{prop:weak_prismatic_F_crystal:transversal_to_crystalline:type1} Assume $D=D_B(J)$ for a framed regular prism $(B,J)$.
		The linearization map below is injective 
		\[
		\mathcal{E}_{\Prism,n,T}(B,J)\otimes_B D \to \mathcal{E}_{\Prism,n,T}(D,(p)).
		\]
		\item \label{prop:weak_prismatic_F_crystal:transversal_to_crystalline:type2} Assume $D=R_0$ is a smooth model of $U\in X_\et$ over $W$, and let $\mathfrak{S}=R_0 \llbracket u \rrbracket$ be the associated Breuil--Kisin prism.
		For each $n\in \mathbb{Z}$, the mod $u$ reduction map $\mathfrak{S}\to R_0$ induces an injection
		\[
		\mathcal{E}_{\Prism,n,T}(\mathfrak{S},(E(u)))\otimes_{\mathfrak{S}} R_0 \to \mathcal{E}_{\Prism,n,T}(R_0,(p)).
		\]
		\item\label{prop:weak_prismatic_F_crystal:transversal_to_crystalline:for_crys_loc} Assume $T\in \Loc^\crys_{\mathbb{Z}_p}(X_\eta)$.
		Then the above maps become isomorphic after taking the unions with respect to $n$ and inverting $p$.
		In particular, the $F$-isocrystal $\mathcal{F}_U$ in \Cref{cor:weak_prismatic_F_crystal:crystalline:local_gluing} is canonically isomorphic to the crystalline realization of the prismatic $F$-crystal $\mathcal{E}_{\Prism,T}$.
	\end{enumerate}
\end{proposition}
\begin{proof}
For \ref{prop:weak_prismatic_F_crystal:transversal_to_crystalline:type1}, we let $\widehat{D}$ be the filered completion of $D$ with respect to the divided power filtration $\Fil^\bullet D=J^{[\geq \bullet]}$.
There is then a commutative diagram enlarging the linearization map
\begin{equation}
\label{eq:diagram_of_framed_and_its_crystalline_prism}
\begin{tikzcd}
\mathcal{E}_{\Prism,n,T}(B,J)\otimes_B D\ar[r] \ar[d] &  \mathcal{E}_{\Prism,n,T}(D,(p)) \ar[d]\\
\mathcal{E}_{\Prism,n,T}(B,J)\otimes_B \widehat{D} \ar[r] & \mathrm{H}^0_\pe(X_\eta, T\otimes \mu^n (\Prism_{(-)_{\overline{B}}/B} \otimes_B D)^\wedge_{1\otimes \Fil^\bullet D}),
\end{tikzcd}
\end{equation}
where the period sheaf $(\Prism_{(-)_{\overline{B}}/B} \otimes_B D)^\wedge_{1\otimes \Fil^\bullet D}$ is the filtered completion of $\Prism_{(-)_{\overline{B}}/B} \otimes_B D$ with respect to the divided power filtration on $D$.
So to show the injectivity of the claimed arrow, it suffices to show the injectivity of the rest of the arrows in (\ref{eq:diagram_of_framed_and_its_crystalline_prism}).
Notice that since the map of ring $B\to D$ is given by the composition $B\xrightarrow{\varphi_B} B \xrightarrow{\mathrm{incl}} D$ (cf. proof of \Cref{prop:global_section_crystalline}), the left column of (\ref{eq:diagram_of_framed_and_its_crystalline_prism}) can be identified with the base change of $\mathcal{E}^{(1)}_{\Prism,T}(B,J)$ along $B\xrightarrow{\mathrm{incl}} D$.
So the injectivity of the left vertical arrow in (\ref{eq:diagram_of_framed_and_its_crystalline_prism}) follows from the finiteness of the filtered module $\mathcal{E}_{\Prism,T}(B,J)$ over $J^\bullet B$ (as in \Cref{prop:weak_prismatic_F_crystal:finiteness} and \Cref{cor:weak_prismatic_F_crystal:eventual_adic}), together with the separatedness of the filterer ring $\Fil^\bullet D$ (\Cref{lem:dR_is_sep}).
For the right vertical arrow of (\ref{eq:diagram_of_framed_and_its_crystalline_prism}), we notice that it is given by the global section of the map of sheaves $T\otimes \mu^n (\Prism_{(-)_{\overline{B}}/B} \otimes_B D)^\wedge_p \to T\otimes \mu^n (\Prism_{(-)_{\overline{B}}/B} \otimes_B D)^\wedge_{1\otimes \Fil^\bullet D}$.
We let $S=\overline{A}_\perf$ be the reduction of the perfection of a framed regular prism $(A,I)$ as in \Cref{lem:perfection_of_framed_regular_prism}.
Then to show the injectivity of the right vertical arrow in (\ref{eq:diagram_of_framed_and_its_crystalline_prism}), it suffices to show that the map of rings $(\Prism_{S_{\overline{B}}/B}\otimes_B D)^\wedge_p \to (\Prism_{S_{\overline{B}}/B}\otimes_B D)^\wedge_{1\otimes \Fil^\bullet D}$ is injective: for the latter, by \Cref{thm:coproduct_in_general}.\ref{thm:coproduct of prism regular and regular}, both source and target are completely free over $D$ and $\widehat{D}$ respectively, so it follows by the injectivity of $D\to \widehat{D}$ in \Cref{lem:dR_is_sep}.
Finally, to prove that the bottom arrow in (\ref{eq:diagram_of_framed_and_its_crystalline_prism}) is injective, by the filtered completeness, it suffices to check it for each graded piece.
The $i$-th graded piece of the bottom arrow is 
\[
\mathcal{E}^{(1)}_{\Prism,T}(B,J) \otimes_B \overline{B} \otimes_{\overline{B}} \Gamma^i (J/J^2) \longrightarrow \mathrm{H}^0_\pe(X_\eta, T\otimes \mu^n \Prism_{(-)_{\overline{B}}/B} \otimes_B \overline{B} \otimes_{\overline{B}} \Gamma^i (J/J^2)).
\]
So by factoring out the invertible $\overline{B}$-module $\Gamma^i (J/J^2)$, its injectivity follows from that of the de Rham realization functor in \Cref{prop:weak_prismatic_F_crystal:deRham}.\ref{prop:weak_prismatic_F_crystal:deRham:inj}.

	Part \ref{prop:weak_prismatic_F_crystal:transversal_to_crystalline:type2} follows from the universal coefficient theorem, since the element $u$ is regular in $\mathfrak{S}$ and hence in $\Prism_{(-)_{\overline{\mathfrak{S}}}/\mathfrak{S}}$.
	For \ref{prop:weak_prismatic_F_crystal:transversal_to_crystalline:for_crys_loc}, we notice that by \Cref{thm:prismatic_RH_for_crystalline_local_system}, the pro-\'etale sheaf $T\otimes \Prism_{(-)_{\overline{B}}/B}[1/\mu]$ is naturally isomorphic to $\mathcal{E}_{\Prism,T}(B,J)\otimes_B \Prism_{(-)_{\overline{B}}/B}[1/\mu]$, where $\mathcal{E}_{\Prism,T}(B,J)$ is a reflexive module over $B$.
	By taking the base change along $\Prism_{(-)_{\overline{B}}/B}\to \Prism_{(-)_{\overline{D}}/D}$, the above map induces an isomorphism
	\[
	\mathcal{E}_{\Prism,T}(B,J)\otimes_B D \otimes_D \Prism_{(-)_{\overline{D}}/D}[1/\mu] \xrightarrow{\sim} T\otimes \Prism_{(-)_{\overline{D}}/D}[1/\mu].
	\]
	Hence the claim follows by taking the global section, thanks to \Cref{prop:global_section_crystalline} and the finite projectivity of $\mathcal{E}_{\Prism,T}(B,J)\otimes_B D[1/p]$ over $D[1/p]$.
\end{proof}

\subsection{Functoriality}
\label{sub:push_pull}
In this subsection, we analyze how the prismatic Riemann--Hilbert functor and the integral de Rham functor interact with the pullback.
For crystalline local systems, we also consider its higher direct image along proper smooth morphisms.

We first consider the pullback for the flat connection $\mathcal{E}_{\dR,T}$.
\begin{proposition}[Pullback and de Rham functor]
\label{prop:dR:pull_back_inj}
	Let $f:Y\to X$ be a map of regular $p$-adic formal schemes over $\mathcal{O}_K$, let $T\in \Loc_{\mathbb{Z}_p}(X_\eta)$.
	\begin{enumerate}[label=\upshape{(\roman*)}]
		\item\label{prop:dR:pull_back_inj:rational} The canonical filtered map $f^*\mathcal{E}_{\dR,n,T} \to \mathcal{E}_{\dR,n,f_\eta^{-1}T}$ is injective after inverting $p$.
		\item\label{prop:dR:pull_back_inj:integral_condition} Assume $f^*\mathcal{E}_{\dR,n,T}$ is $p$-torsionfree. 
		Then the above filtered map of $\mathcal{O}_Y$-modules is injective integrally.
	\end{enumerate}
\end{proposition}
\begin{proof}
	We let $U=\Spf(R_2)\to V=\Spf(R_1)$ be an induced map of open affine subschemes in $Y$ and $X$ respectively, and consider the associated map of filtered $R_2[1/p]$-modules with connections (cf. \Cref{const:de_Rham_functor})
	\[
	\mathcal{E}_{\dR,n,T}(V)\otimes_{R_1} R_2 [1/p]\longrightarrow \mathcal{E}_{\dR,n,f_\eta^{-1}T} (U)[1/p].
	\]
	By the comparison with Liu--Zhu's Riemann--Hilbert functor (\Cref{prop:dR:finiteness_&_injective}.\ref{prop:dR:finiteness_&_injective:LZ}), the above map fits into a commutative diagram
	\[
	\begin{tikzcd}
	\mathcal{E}_{\dR,n,T}(V)\otimes_{R_1} R_2[1/p] \ar[r] \ar[d] & \mathcal{E}_{\dR,n,f_\eta^{-1}T} (U)[1/p] \ar[d]\\
	\mathrm{D}_{\dR,X_\eta}(T)(V)\otimes_{R_1[1/p]} R_2[1/p] \ar[r] & \mathrm{D}_{\dR,Y_\eta}(f_\eta^{-1}T) (U),
	\end{tikzcd}
	\]
	where the right vertical map is injective by \textit{loc.\ cit.} and the bottom map is an isomorphism by \cite[Thm.\ 3.9.(ii)]{LZ17}.
	In addition, since $\mathcal{E}_{\dR,n,T}(V)[1/p]\to \mathrm{D}_{\dR,X_\eta}(V)$ is an injection of flat connections over the smooth rigid space $X_\eta$, its pullback is also injective.
	Hence by the commutative diagram above, we get the injectivity of the top horizontal map.

	Now assume that $f^*\mathcal{E}_{\dR,n,T}$ has no $p$-torsion,
	By definition, the map $\mathcal{E}_{\dR,n,T}(V)\otimes_{R_1} R_2\to \mathcal{E}_{\dR,n,T}(V)\otimes_{R_1} R_2[1/p]$ is injective.
	On the other hand, the $R_2$-module $\mathcal{E}_{\dR,n,f_\eta^{-1}T} (U)$ by \Cref{prop:dR:finiteness_&_injective}.\ref{prop:dR:finiteness_&_injective:LZ} has no $p$-torsion.
	Hence the injectivity of $\mathcal{E}_{\dR,n,T}(V)\otimes_{R_1} R_2 \to \mathcal{E}_{\dR,n,f_\eta^{-1}T} (U)$ follows from the injectivity after inverting $p$.
\end{proof}
Below we give some examples where the torsionfree assumption in \Cref{prop:dR:pull_back_inj}.\ref{prop:dR:pull_back_inj:integral_condition} is satisfied.
We recall that a l.c.i morphism has \emph{embedded codimension} $r$ if locally it can be factored as a composition of a smooth morphism followed by a regular closed immersion of codimension $r$.
We refer the reader to \cite[Def.\ 5.6]{GL21} for the precise definition and discussion.
\begin{corollary}
	\label{prop:dR:Tor_condition}
	Let $f:Y\to X$ be a l.c.i. morphism of regular $p$-adic formal schemes satisfying either of the following conditions:
	\begin{itemize}
		\item the map $f$ is flat;
		\item the embedded codimension of $f$ is one.
	\end{itemize}
	Then for each $T\in \Loc_{\mathbb{Z}_p}(X_\eta)$, the pullback $f^*\mathcal{E}_{\dR,n,T}$ is $p$-torsionfree. 
	In particular, the canonical filtered map $f^*\mathcal{E}_{\dR,n,T} \to \mathcal{E}_{\dR,n,f_\eta^{-1}T}$ is injective.
\end{corollary}
\begin{proof}
	Since $\mathcal{E}_{\dR,n,T}$ has no $p$-torsions, it suffices to consider the case when $f:Y=\Spf(R_2)\to X=\Spf(R_1)$ is a regular closed immersion of codimension one: namely $R_2=R_1/gR_1$ for a regular element $g\in R_1$.
	Note that by our assumption, the quotient $R_1/gR_1$ is $p$-torsionfree.
	By shrinking $X$ if necessary, we let $S$ be a perfectoid faithfully flat $R_1$-algebra such that the generic fiber $\Spf(S)_\eta$ covers $X_\eta$ (\Cref{lem:perfection_of_framed_regular_prism}).
	Then one has $\mathcal{E}_{\dR,n,T}(X)\subset T(\Spf(S)_\eta)\otimes \mu^n\dR_{S/R_1}$, where the latter by the separatedness of the Hodge filtration (\Cref{lem:dR_is_sep}) is contained in $T(\Spf(S)_\eta)\otimes_{\mathbb{Z}_p(\Spf(S)_\eta)} \mu^n\widehat{\dR}_{S/R_1}$.
	Since each graded pieces of $\widehat{\dR}_{S/R_1}$ is a finite projective $S$-module (\Cref{prop:relative_prismatic_coh_of_perfectoid}) and since $S$ is faithfully flat over $R_1$, the graded pieces of $\widehat{\dR}_{S/R_1}\otimes_{R_1} R_1/gR_1$ is flat over $R_1/gR_1$ and has no $p$-torsion.
	Hence by taking the limit, the entire filtered complete ring $\widehat{\dR}_{S/R_1}\otimes_{R_1}R_1/gR_1$ has no $p$-torsion, and so is its subring $\dR_{S/R_1}\otimes_{R_1}R_1/gR_1$.
	As a consequence, by the universal coefficient theorem, we obtain an injection 
	\[
	\mathcal{E}_{\dR,n,T}(X)\otimes_{R_1} R_1/gR_1 \hookrightarrow \mathrm{H}^0_\pe(X_\eta, T\otimes \mu^n \dR_{(-)/X}\otimes_{R_1}R_1/gR_1),
	\]
	where the latter has no $p$-torsion.
	Thus $\mathcal{E}_{\dR,n,T}(X)\otimes_{R_1} R_1/gR_1$ is $p$-torsionfree.
\end{proof}

We now discuss the injectivity for Frobenius modules over a compatible map of prisms.
Below for given prisms $(A,I)\in X_\Prism$ and $(B,J)\in Y_\Prism$, we say a map of prisms $(A,I)\to (B,J)$  is \emph{over} a map of $p$-adic formal schemes $f:Y\to X$ if the map of the reductions $\overline{A}\to \overline{B}$ is compatible with $f$.
\begin{theorem}[Pullback and prismatic Riemann--Hilbert functor]
\label{thm:weak_prismatic_F_crystal:pull_back_wrt_prisms}
Let $(A,I)\to (B,J)$ be a map of prisms over a l.c.i. map of regular $p$-adic formal schemes $f:Y\to X$ and let $T\in \Loc_{\mathbb{Z}_p}(X_\eta)$. 
Let $g:\mathcal{E}_{\Prism,T}(A,I)\otimes_{A}B \to \mathcal{E}_{\Prism,f_\eta^{-1}T}(B,J)$ be the induced map of $B$-modules, and let $g^{(1)}=\varphi_{B}^*(g)$ be the completed Frobenius twist.

Assume for now that $(A,I)\in X_\Prism$ is framed regular, that $\overline{B}$ is completely projective over the image of $\Spf(\overline{B})$ in $Y$, and $\varphi_{B}$ is completely projective:
\begin{enumerate}[label=\upshape{(\roman*)}]
\item\label{thm:weak_prismatic_F_crystal:pull_back_wrt_prisms:generic_twisted} The complete base change of $g^{(1)}$ along $B\to B[1/p]^\wedge_{J}$ is injective.
\item\label{thm:weak_prismatic_F_crystal:pull_back_wrt_prisms:twisted} The maps $g^{(1)}$ and $g$ are injective if the map $f$ is flat.
\end{enumerate}
Assume for now that $(A,I)\in X_\Prismsp$, and either $(B,J)\in Y_\Prismsp$ or $\varphi_B$ is completely projective:
\begin{enumerate}[resume]
\item[\upshape{(iii)}]
\label{thm:weak_prismatic_F_crystal:pull_back_wrt_prisms:iso} 
The maps $g^{(1)}$ and $g$ are isomorphisms if either $T\in \Loc_{\mathbb{Z}_p}^\crys(X_\eta)$ or $A\to B$ is finite \'etale.
\end{enumerate}
\end{theorem}
Here we recall that the assumptions for $(B,J)$ in \Cref{thm:weak_prismatic_F_crystal:pull_back_wrt_prisms}.(\ref{thm:weak_prismatic_F_crystal:pull_back_wrt_prisms:generic_twisted} and \ref{thm:weak_prismatic_F_crystal:pull_back_wrt_prisms:twisted}) are satisfied for example when $(B,J)$ is a framed regular prism or its perfection over $Y$.
We also note that under the assumption of \Cref{thm:weak_prismatic_F_crystal:pull_back_wrt_prisms}.\ref{thm:weak_prismatic_F_crystal:pull_back_wrt_prisms:iso}, the map $f:Y\to X$ is \'etale itself.
\begin{proof}
	As the claims are Zariski local on $Y$ and $X$, we may assume $X=\Spf(R_1)$ and $Y=\Spf(R_2)$.	
	We first consider the Frobenius twist $g^{(1)}$.
	By the $J$-completeness of the map $g^{(1)}$ and $(g^{(1)}[1/p])^\wedge_{J}$, to show their injectivity, it suffices to show it for their mod $J$ reductions.
	We consider the following commutative diagram that relates the reduction of $g^{(1)}$ with the base change of the integral de Rham functors
	\begin{equation}
	\label{eq:diagram_of_prism_and_dR}
	\begin{tikzcd}
	\mathcal{E}^{(1)}_{\Prism,n,T}(A,I)\otimes_{A} \overline{B} \arrow[r, "g^{(1)}"] \ar[d] & \mathcal{E}^{(1)}_{\Prism,n,f^{-1}T}(B,J) \otimes_B \overline{B}\ar[d]\\
	\bigl( \mathcal{E}_{\dR,n,T}(X) \otimes_{R_1} \overline{B}\bigr)^\wedge_p \ar[r] & \bigl(\mathcal{E}_{\dR,n,f^{-1}T} (Y)\otimes_{R_2} \overline{B}\bigr)^\wedge_p,
	\end{tikzcd}
	\end{equation}
	where we use the finiteness of $\mathcal{E}^{(1)}_{\Prism,n,T}(A,I)$ to ensure the completeness of the left top object.
	In the diagram, the right vertical arrow is injective thanks to the complete projectivity assumptions of $(B,J)$ and \Cref{prop:weak_prismatic_F_crystal:deRham}.
	For the bottom arrow, by \Cref{lem:inj_and_completely_proj_mod}, \Cref{prop:dR:Tor_condition}, and the complete projectivity of $\overline{B}$, we know it is injective if $f$ is flat and is always injective after inverting $p$.
	So to prove both \ref{thm:weak_prismatic_F_crystal:pull_back_wrt_prisms:generic_twisted} and \ref{thm:weak_prismatic_F_crystal:pull_back_wrt_prisms:twisted}, it suffices to prove the injectivity of the left vertical arrow in the two situations.
	
	We now analyze the left vertical arrow.
	By construction, the map is identical to the base change of the map of $\overline{A}$-modules $h\colonequals\mathcal{E}^{(1)}_{\Prism,n,T}(A,I)\otimes_{A} \overline{A} \to \mathcal{E}_{\dR,n,T}(X)\otimes_{R_1} \overline{A}$ along $\overline{A}\to \overline{B}$, where $\overline{A}$ is $p$-completely \'etale over $R_1$ and is in particular topologically of finite type over $\mathcal{O}_K$.
	Note that by \Cref{prop:weak_prismatic_F_crystal:deRham}, the map $h$ is injective and is compatible with their flat connections after inverting $p$.
	So if the map $f:Y\to X$ is flat, then the base change $h\otimes_{R_1} R_2$ is an injection of finite presented modules over $\overline{A}\otimes_{R_1}R_2$.
	For general $f$, since the injectivity of a map of flat connections over a rigid space is preserved under base changes, the map $h\otimes_{R_1} R_2[1/p]$ is an injection of flat connections over the smooth rigid space $\Spf(\overline{A}\otimes_{R_1} R_2)_\eta$.
	Hence by the complete projective assumption of $\overline{B}$ over $R_2$ and by taking the base change of $h\otimes_{R_1} R_2$ along $\overline{A}\otimes_{R_1} R_2 \to \overline{B}$, we see the left vertical arrow in (\ref{eq:diagram_of_prism_and_dR}) is injective either when $f$ is flat, or in general after inverting $p$.

	To get the injectivity of $g$, we consider the following commutative diagram 
	\[
	\begin{tikzcd}
	\mathcal{E}_{\Prism,n,T}(A,I)\otimes_{A}B \arrow[r, "g"]  \ar[d]&  \mathcal{E}_{\Prism,n,f_\eta^{-1}T}(B,J) \ar[d] \\ 
	\mathcal{E}^{(1)}_{\Prism,n,T}(A,I)\otimes_{A}B \arrow[r, "g^{(1)}"] &  \mathcal{E}^{(1)}_{\Prism,n,f_\eta^{-1}T}(B,J),
	\end{tikzcd}
	\]
	where the vertical maps are linearlization for $\varphi_{B}:B\to B$.
	So the injectivity of $g$ follows from that of $g^{(1)}$ together with \Cref{lem:inj_and_completely_proj_mod}.
	
	Finally, we assume $A\to B$ is finite \'etale and prove that $g$ is an isomorphism.
	We first assume $T\in \Loc_{\mathbb{Z}_p}^\crys(X_\eta)$.
	By \Cref{thm:prismatic_RH_for_crystalline_local_system}. (\ref{thm:prismatic_RH_for_crystalline_local_system_not_weak} and \ref{thm:prismatic_RH_for_crystalline_local_system_linearization}), the weak $F$-isocrystal $\mathcal{E}_{\Prism,T}$ is in fact a reflexive prismatic $F$-crystal and there is a canonical isomorphism $\mathcal{E}_{\Prism,T}(A,I)\otimes_A \Prism_{(-)_{\overline{A}}/A}[1/\mu] \xrightarrow{\sim} T\otimes \Prism_{(-)_{\overline{A}}/A}[1/\mu]$.
	In particular, we get a canonical isomorphism of sheaves of $B$-modules over $Y_{\eta,\pe}$ through the pullback along the map of ringed sites $(Y_{\eta,\pe}, \Prism_{(-)_{\overline{B}}/B}[1/\mu])\to (X_{\eta,\pe}, \Prism_{(-)_{\overline{A}}/A}[1/\mu])$:
	\[
	\mathcal{E}_{\Prism,T}(A,I)\otimes_A B \otimes_B \Prism_{(-)_{\overline{B}}/B}[1/\mu] \xrightarrow{\sim} f^{-1}T\otimes \Prism_{(-)_{\overline{B}}/B}[1/\mu].
	\]
	So by taking the global section functor $\mathrm{H}_\pe^0(Y_\eta, -)$ and by \Cref{cor:global_section_rational}, we get an isomorphism
	\[
	\mathcal{E}_{\Prism,T}(A,I)\otimes_A B \simeq \mathcal{E}_{\Prism,f^{-1}T}(B,J).
	\]
	
	Assume from now that $T\in \Loc_{\mathbb{Z}_p}(X_\eta)$ and $A\to B$ is finite \'etale.
	Under the assumption, both $(A,I)$ and $(B,J)$ are framed regular prisms over $X$, so $\Spf(\overline{A})$ is \'etale over $X$ and $\Spf(\overline{B})$ is \'etale over $Y$.
	In particular, we know the map $f$ is \'etale itself.
	Notice that by the localization property in \Cref{prop:weak_prismatic_F_crystal:localization}, we may assume the reduction $\Spf(\overline{A})$ is surjective onto $X=\Spf(R_1)$ and similarly for $\Spf(\overline{B})\to Y=\Spf(R_2)$. 
	Thus by taking the composition, it suffices to discuss two separate cases: in the first case $\overline{B}$ is the base change of $\overline{A}$ along a finite \'etale map $f$, and in the second case $X=Y$.
	Moreover, by the injectivity of $g$ in  \ref{thm:weak_prismatic_F_crystal:pull_back_wrt_prisms:twisted} for flat maps, we may assume $f$ is finite Galois.
	
	As in \Cref{const:algebrac_pi_1}, we let $\widetilde{X}=\Spf(S)\to X$ be the integral model of a maximal connected Galois cover of the generic fiber, let $G=G_{X_\eta}$ be the Galois group, and let $H$ be the kernel of $G\to \Gal(Y/X)$.
	So the canonical map $g$ can be translated to the following map of Galois invariants:
	\begin{equation}
		\label{eq:two_galois_inv_1}
	(T(\widetilde{X}_\eta)\otimes \Prism_{S\otimes_{R_1} {\overline{A}}/A})^G \otimes_A B \longrightarrow (T(\widetilde{X}_\eta)\otimes \Prism_{S\otimes_{R_2}{\overline{B}}/B})^H.
	\end{equation}
	In the first case when $\overline{B}=\overline{A}\otimes_{R_1} R_2$, the tensor product $S\otimes_{R_1} \overline{A}$ is equal to $S\otimes_{R_2} \overline{B}$.
	Moreover, by \cite[Lem.\ 7.12]{GR24}, the \'etaleness implies that the relative prismatic cohomology $\Prism_{S\otimes_{R_2} \overline{B}/B}=\Prism_{S\otimes_{R_1} \overline{A}/B}$ is naturally isomorphic to $\Prism_{S\otimes_{R_1} \overline{A}/A}$.
	\footnote{Though the statement in \textit{loc.\ cit.} assume the smoothness, the arguments apply verbatimally in general.}
	We let $M$ be the $A$-module $\bigl( T(\widetilde{X}_\eta)\otimes \Prism_{S\otimes_{R_1} {\overline{A}}/A} \bigr)^H$ with the induced $G/H$-action.
	Then (\ref{eq:two_galois_inv_1}) can be rewritten as a map of $B$-modules
		\[
		M^{G/H} \otimes_A B \to M.
		\]
	Note that by \cite[Lem.\ 2.18]{BS22}, the map $\Spf(B)\to \Spf(A)$ is also a finite Galois $G/H$-cover that is compatible with $\Spf(\overline{B})\to \Spf(\overline{A})$.
	Thus the claim follows from the usual Galois descent.
		
	We then consider the second case when $X=Y$ and $A\to B$ is finite \'etale.
	Note that since $B$ is finite projective over $A$ and has the trivial $G$-action, by the projection formula for finite projective modules, we may replace the source of (\ref{eq:two_galois_inv_1}) as below
	\begin{equation}
	\label{eq:two_galois_inv_2}
	(T(\widetilde{X}_\eta)\otimes \Prism_{S\otimes_{R_1} {\overline{A}}/A} \otimes_A B)^G  \longrightarrow (T(\widetilde{X}_\eta)\otimes \Prism_{S\otimes_{R_2} \overline{B}/B})^H.
	\end{equation}
	So the map in (\ref{eq:two_galois_inv_2}) is an isomorphism, thanks to the base change property of relative prismatic cohomology.
\end{proof}

In the special case when $Y=X$, we yield a change of prisms property.
\begin{corollary}
Let $(B_1,J_1)\to (B_2,J_2)$ be a map of two prisms over $X$, and let $T\in \Loc_{\mathbb{Z}_p}(X_\eta)$.
Assume $(B_1,J_1)$ is framed regular, the map $\varphi_{B_2}$ is completely projective, and $\Spf(\overline{B}_2)$ is completely projective over its image in $X$.
The induced map of weak Frobenius modules $\mathcal{E}_{\Prism,T}(B_1,J_1)\otimes_{B_1}B_2 \to \mathcal{E}_{\Prism,T}(B_2,J_2)$ is injective.
\end{corollary}

Note that one can construct a natural pullback functor for presheaves of weak Frobenius modules.
\begin{construction}
\label{const:weak_prismatic_F_crystal:push_pull_functor}
Let $f:Y\to X$ be a map of regular $p$-adic formal schemes over $\mathcal{O}_K$.
For a weak prismatic ($F$-)crystal $\mathcal{E}\in \wCrys^{(\varphi)}(X_\Prism)$, we define the \emph{pullback of $\mathcal{E}$} to be the functor on $Y_\Prismsp$ such that
\[
Y_\Prismsp \ni (B,J) \longmapsto (f_\Prism^* \mathcal{F})(B,J) \colonequals \lim (\mathcal{E}(A,I)\otimes_{A} B')^\wedge_{(p,I)},
\]
where the limit is ranging over all the diagrams $\bigl(
(B,J) \to  (B',J') \leftarrow (A,I) \bigr)$ such that $(B,J)\to (B',J')$ is completely flat and $(A,I)$ is completely flat over $X$.
Here we note that the index category of the limit is in fact filtered thanks to \Cref{thm:coproduct_in_general}.
\end{construction}
So applying \Cref{thm:weak_prismatic_F_crystal:pull_back_wrt_prisms} and the limit formula in \Cref{lem:weak_vs_non-weak}, we obtain the following pullback compatibility.
\begin{corollary}
\label{cor:weak_prismatic_F_crystal:pull_back}
Let $f:Y\to X$ be a map of regular $p$-adic formal schemes over $\mathcal{O}_K$, let $T\in \Loc_{\mathbb{Z}_p}(X_\eta)$.
Assume either $T\in \Loc_{\mathbb{Z}_p}^\crys(X_\eta)$ or $f$ is finite \'etale.
There is a canonical isomorphism of weak prismatic $F$-crystals on $Y_\Prismsp$
\[
f_\Prism^* \mathcal{E}_{\Prism,T} \xrightarrow{\sim} \mathcal{E}_{\Prism,f^{-1}T}.
\]
\end{corollary}

Finally, we show that the prismatic Riemann--Hilbert functor is naturally compatible with the derived direct image for $T\in \Loc^\crys_{\mathbb{Z}_p}(X_\eta)$.
Before the statement, we recall that $(-)_\tf$ is defined as the $p$-torsionfree quotient.
For an analytically locally free coherent prismatic $F$-crystal $\mathcal{F}$ over a regular $p$-adic formal scheme, we define its \emph{saturation} $\mathcal{F}'$ to be the intersection of $\mathcal{F}[1/p]\cap \mathcal{F}[1/\mathcal{I}_\Prism]$ inside of $\mathcal{F}[1/p\mathcal{I}_\Prism]$.
By \Cref{cor:reflexive_in_ca} and the assumption that $\mathcal{F}$ is analytically locally free, the saturation is also equal to the reflexive hull of $\mathcal{F}$.
In particular, the cokernel of $\mathcal{F}\hookrightarrow \mathcal{F}'$ is supported at the non-analytic locus $V(p,\mathcal{I}_\Prism)$ in $X_\Prism$ and is killed by a power of $p$.
\begin{proposition}[Pushforward and prismatic Riemann--Hilbert functor]
\label{prop:weak_prismatic_F_crystal:direct_image}
Let $g:X\to Z$ be a smooth proper map of regular $p$-adic formal schemes over $\mathcal{O}_K$, and let $T\in \Loc^\crys_{\mathbb{Z}_p}(X_\eta)$.
For each $i\in \mathbb{N}$, there is a canonical injection of coherent prismatic $F$-crystal over $Z_\Prism$:
\[
(R^i g_{\Prism,*} \mathcal{E})_\tf \hookrightarrow \mathcal{E}_{\Prism, (R^i g_{\eta,*} T)_\tf},
\]
which identifies the target as the saturation of the source.
\end{proposition}
\begin{proof}
By the Strong \'Etale Comparison of Guo--Reinecke in \Cref{thm:GR_strong_'etale}, we know there is a canonical Frobenius equivariant isomorphism of pro-\'etale sheaves over $Z_\eta$
\begin{equation}
\label{eq:iso_strong_etale}
		(R g_{\eta,*} T)\otimes_{\mathbb{Z}_p} \Ainf[1/\mu] \simeq \Ainf(R g_{\Prism,*} \mathcal{E})[1/\mu].
\end{equation}
So for each prism $(B,J)\in Z_\Prismsp$, the base change along $\Ainf[1/\mu] \to \Prism_{(-)_{\overline{B}}/B}[1/\mu]$ induces a Frobenius equivariant isomorphism
\[
(R g_{\eta,*} T)\otimes_{\mathbb{Z}_p} \Prism_{(-)_{\overline{B}}/B}[1/\mu] \simeq \Ainf(R g_{\Prism,*} \mathcal{E})[1/\mu]\otimes_{\Ainf[1/\mu]} \Prism_{(-)_{\overline{B}}/B}[1/\mu].
\]
Moreover, by applying the crystal property of $R g_{\Prism,*} \mathcal{E}$ along the maps of prisms $(\Ainf, \ker(\widetilde{\theta})) \to (\Prism_{(-)_{\overline{B}}/B} ,I_{\Prism_{(-)_{\overline{B}}/B}}) \leftarrow (B,J)$, the right hand side above is naturally isomorphic to 
\[
(R g_{\Prism,*}\mathcal{E})(B,J) \otimes_B \Prism_{(-)_{\overline{B}}/B}[1/\mu].
\]
So by combining the above two identifications, we obtain a natural isomorphism 
\[
(R g_{\eta,*} T)\otimes_{\mathbb{Z}_p} \Prism_{(-)_{\overline{B}}/B}[1/\mu] \simeq (R g_{\Prism,*}\mathcal{E})(B,J) \otimes_B \Prism_{(-)_{\overline{B}}/B}[1/\mu].
\]
The $i$-th cohomology yields a canonical isomorphism of pro-\'etale sheaves over $Z_\eta$
\begin{equation}
\label{eq:iso_i-th_cohomology}
(R^i g_{\eta,*} T) \otimes_{\mathbb{Z}_p}  \Prism_{(-)_{\overline{B}}/B}[1/\mu]  \simeq (R^i g_{\Prism,*}\mathcal{E}) (B,J) \otimes_B \Prism_{(-)_{\overline{B}}/B}[1/\mu].
\end{equation}

Now, by the proper smooth assumption of $g:X\to Z$, the derived direct image of $\mathcal{E}$ is a prismatic $F$-crystal in perfect complexes (\cite[Cor.\ 5.16, Rmk.\ 8.11]{GR24}) and each $R^i g_{\Prism,*}\mathcal{E}$ is a coherent prismatic $F$-crystal over $Z_\eta$.
Moreover, as in the proof of \Cref{thm:GR_strong_'etale}, the prismatic $F$-crystal $R^i g_{\Prism,*}\mathcal{E}$ is locally free after inverting $p$.
In addition, by taking the $p$-completion at the $i$-th cohomology of (\ref{eq:iso_strong_etale}), we see the $p$-torsionfree quotient $(R^i g_{\Prism,*}\mathcal{E})_\tf$ becomes locally free after base change along $\mathcal{O}_\Prism \to \mathcal{O}_\Prism\langle 1/\mathcal{I}_\Prism \rangle$.
As a consequence, we know $(R^i g_{\Prism,*}\mathcal{E})_\tf$ is analytically locally free and is contained in its saturation, where the latter is a reflexive coherent prismatic $F$-crystal over $Z$.
In particular, by \Cref{thm:global_section_rational}, we know for each special prism $(B,J)\in Z_\Prismsp$, there is an injection
\[
(R^i g_{\Prism,*}\mathcal{E})_\tf (B,J) \longrightarrow \mathrm{H}^0_\pe\bigl( Z_\eta, (R^i g_{\Prism,*}\mathcal{E})_\tf (B,J) \otimes_B \Prism_{(-)_{\overline{B}}/B}[1/\mu]\bigr)
\]
that induces an isomorphism after inverting either $p$ or $I$.
Hence by combining this injection with (\ref{eq:iso_i-th_cohomology}), we obtain a natural isomorphism
\[
\mathcal{E}_{\Prism, (R^i g_{\eta,*} T)_\tf}(B,J) = \mathrm{H}^0_\pe\bigl(Z_\eta, (R^i g_{\eta,*} T)_\tf \otimes_{\mathbb{Z}_p}  \Prism_{(-)_{\overline{B}}/B}[1/\mu]  \bigr) \simeq (R^i g_{\Prism,*}\mathcal{E})_{\tf,\sat} (B,J),
\]
where $(R^i g_{\Prism,*}\mathcal{E})_{\tf,\sat}$ is the saturation of $(R^i g_{\Prism,*}\mathcal{E})_\tf$ and is equal to the intersection $(R^i g_{\Prism,*}\mathcal{E})_\tf[1/p] \cap (R^i g_{\Prism,*}\mathcal{E})_\tf[1/\mathcal{I}_\Prism]$.
So the statement follows by ranging over $(B,J)\in Z_\Prismsp$.
\end{proof}

\bibliographystyle{amsalpha}
\bibliography{ref_Ogus}
\end{document}